\documentclass[11pt, letterpaper, reqno]{amsart}
\usepackage[utf8]{inputenc} 	
\usepackage{microtype} 			
\usepackage{geometry} 			
\usepackage{amsmath}			
\usepackage{amsthm}		 		
\usepackage{amssymb}	 		
\usepackage{bm}					
\usepackage{mathrsfs}			
\usepackage{thmtools}
\usepackage{xcolor}				
\definecolor{darkyellow}{rgb}{0.7, 0.7, 0}
\definecolor{darkred}{rgb}{0.6, 0.1, 0.1}
\definecolor{darkblue}{rgb}{0.2, 0.2, 0.7}
\definecolor{darkgreen}{rgb}{0.1, 0.4, 0.1}
\definecolor{bettercyan}{rgb}{0.1, 0.4, 0.7}
\usepackage{tikz}				
\usetikzlibrary{decorations.markings}
\usetikzlibrary{patterns}
\usepackage[all]{xy}			
\usepackage{graphicx}			
\usepackage{pgfplots}			
\usepackage{enumitem} 			
\usepackage{array}
\usepackage{subcaption}
\usepackage{lmodern}			
\usepackage[T1]{fontenc}		
\usepackage[breaklinks=true, colorlinks=true, citecolor=bettercyan, urlcolor=magenta, linkcolor=darkred]{hyperref}
\usepackage{cleveref}
\usepackage[backend=bibtex, style=alphabetic, backref=true, firstinits=true, isbn=false, date=year]{biblatex}
\makeatletter
\renewbibmacro{in:}{}
\DeclareFieldFormat{pages}{#1}
\renewcommand*{\bibnamedash}{%
	\leavevmode\raise +0.6ex\hbox to 5.5ex{\hrulefill}.\space\space}

\InitializeBibliographyStyle{\global\undef\bbx@lasthash}

\newbibmacro*{bbx:savehash}{\savefield{fullhash}{\bbx@lasthash}}

\renewbibmacro*{author}{%
	\ifboolexpr{
		test \ifuseauthor
		and
		not test {\ifnameundef{author}}
	}
	{%
		\iffieldequals{fullhash}{\bbx@lasthash}
		{\bibnamedash\addcomma\space}
		{\printnames{author}}%
		\usebibmacro{bbx:savehash}%
		\iffieldundef{authortype}
		{}
		{%
			\setunit{\addcomma\space}%
			\usebibmacro{authorstrg}%
		}%
	}
	{\global\undef\bbx@lasthash}%
}
\makeatother

\newtheorem{propositionx}{Proposition}[section]
\newenvironment{proposition}
{\pushQED{\qed}\propositionx}
{\popQED\endpropositionx}
\newenvironment{propositionp}
{\pushQED{\qed}\propositionx}
{\popQED\endpropositionx}
\newtheorem*{theorem*}{Theorem}
\newtheorem{theorem}{Theorem}

\newtheorem*{corollaryx}{Corollary}
\newenvironment{corollary}
{\pushQED{\qed}\corollaryx}
{\popQED\endcorollaryx}
\newtheorem{lemmax}[propositionx]{Lemma}
\newenvironment{lemma}
{\pushQED{\qed}\lemmax}
{\popQED\endlemmax}
\newenvironment{lemmap}
{\pushQED{\qed}\lemmax}
{\popQED\endlemmax}

\theoremstyle{remark}
\newtheorem{remark}[propositionx]{Remark}
\newtheorem*{remark*}{Remark}
\newtheorem*{example}{Example} 
\makeatletter
\@namedef{subjclassname@2020}{\textup{2020} Mathematics Subject Classification}
\makeatother

\newcommand{\bbC}{\mathbb{C}}

\newcommand{\bbN}{\mathbb{N}}

\newcommand{\bbR}{\mathbb{R}}
\newcommand{\bbS}{\mathbb{S}}

\newcommand{\bbZ}{\mathbb{Z}}

\newcommand{\calA}{\mathcal{A}}

\newcommand{\calC}{\mathcal{C}}
\newcommand{\calD}{\mathcal{D}}
\newcommand{\calE}{\mathcal{E}}
\newcommand{\calF}{\mathcal{F}}
\newcommand{\calG}{\mathcal{G}}

\newcommand{\calI}{\mathcal{I}}
\newcommand{\calJ}{\mathcal{J}}
\newcommand{\calK}{\mathcal{K}}

\newcommand{\calM}{\mathcal{M}}

\newcommand{\calQ}{\mathcal{Q}}
\newcommand{\calR}{\mathcal{R}}
\newcommand{\calS}{\mathcal{S}}

\newcommand{\calX}{\mathcal{X}}

\newcommand{\frakW}{\mathfrak{W}}

\newcommand{\dd}{\,\mathrm{d}}

\title{Full semiclassical asymptotics near transition points}
\author{Ethan Sussman}
\date{March 25th, 2025 (Last update). October 23, 2023 (Draft).}
\email{ethanws@stanford.edu}
\email{ethan.sussman@northwestern.edu}
\address{Department of Mathematics, Stanford University, California, USA}
\address{Department of Mathematics, Northwestern University, Illinois, USA}

\begin{document}

\begin{abstract}
	We construct complete asymptotic expansions of solutions of the 1D semiclassical Schr\"odinger equation near transition points. There are three main novelties: (1) transition points of order $\kappa\geq 2$ (i.e.\ trapped points --- the simple turning point is $\kappa=1$, the simple pole is $\kappa=-1$) are handled, (2) various terms in the operator are allowed to have controlled singularities of a form compatible with the geometric structure of the problem (some applications are given in the text), and (3) the term-by-term differentiability of the expansions with respect to the semiclassical parameter is included. 
	We prove that any solution to the semiclassical ODE with initial data of exponential type is of exponential-polyhomogeneous type on a suitable manifold-with-corners compactifying the $h\to 0^+$ regime. Consequently, such a solution has an atlas of full asymptotic expansions in terms of elementary functions, and these expansions are well-behaved. The Airy and Bessel functions show up in the expected way, as the asymptotic data at one boundary edge. We are able to handle cases that Langer--Olver could not because the framework of polyhomogeneous functions on manifolds-with-corners provides more flexibility (two matched $h\to 0^+$ expansions, possibly with logarithms, in this case) than that employed by Langer--Olver (one uniform $h\to 0^+$ expansion without logarithms). We work entirely in the $C^\infty$ category. No analyticity is ever assumed, nor proven.
\end{abstract}
	
\maketitle

\tableofcontents

\section{Introduction}

In this note, we revisit the old problem of producing asymptotic expansions of solutions of semiclassical ODEs near transition points, where the classical Liouville--Green theory breaks down.
Consider the one-dimensional semiclassical Schr\"odinger operator
\begin{equation}
P = \{P(h)\}_{h>0} = - h^2 \frac{\partial^2}{\partial z^2} + \varsigma z^\kappa W(z)  + h^2 \psi(z,h)
\label{eq:1}
\end{equation}
on the interval $[0,Z]_z$,
where $\varsigma \in \{-1,+1\}$ is a sign ($\varsigma>0$ is known as the ``classically forbidden'' case, and $\varsigma<0$ is the ``classically allowed'' case), $\kappa \in \{-1\} \cup \bbN$, $W \in C^\infty([0,Z]_z;\bbR^+)$, 
and $\psi \in C^\infty( (0,Z]_z\times [0,\infty)_{h^2} ;\bbC)$ is drawn from some suitable class of admissible functions of the independent variable $z$ and the semiclassical parameter $h$. 
We will be more precise later on about the meaning of ``admissible'' in the previous sentence; for now, it suffices to note that any function $\psi: (0,Z]_z\times [0,\infty)_h\to \bbC$ of the form
\begin{equation}
\psi = \frac{\nu}{z^2} + \frac{\varphi(h)}{z} + G(z,h)
\label{eq:psi_init}
\end{equation}
for $\nu\in \bbC$, $\varphi(h) \in C^\infty([0,\infty)_{h^2};\bbC)$, and $G\in C^\infty([0,Z]_z\times [0,\infty)_{h^2};\bbC)$ is allowed. However, this does not exhaust the admissible $\psi$. The reader is invited to take $\psi=0$ if simplification is desired, though some of the intended applications, not to mention our proofs, require greater generality.
 
Restated, the problem is producing $h\to 0^+$ asymptotics of solutions $u=\{u(-;h)\}_{h>0}$ to $Pu=0$.
The key structural feature of $P$ is the \emph{transition point} at $z=0$, at which the potential $V(z)=\varsigma z^\kappa W(z)$
may have a zero or singularity. The Liouville--Green theory suffices away from the transition point \cite[\S6]{OlverBook}. Moreover, for each individual $h>0$, the operator $P(h) \in \operatorname{Diff}^2(0,Z)$ is (at worst) a regular singular ordinary differential operator depending smoothly on $h$ (without varying indicial roots), so the behavior of solutions follows from general theory \cite{MelroseAPS}. It is only the behavior near the corner 
\begin{equation} 
	\{z=0,h=0\}\in [0,Z]_z\times [0,\infty)_{h}
\end{equation} 
that needs to be further understood. Put differently, the issue at hand is asymptotics as $h,z\to 0^+$ \emph{together}.
The use of the term ``corner'' is the first instance of our emphasis on geometric structures.

Our main theorem will be phrased (somewhat imprecisely) in terms of the notion of \emph{exponential-polyhomogeneous type}. We will give a brief introduction to polyhomogeneous and exponential-polyhomogeneous functions in \S\ref{subsec:primer}. Roughly, an exponential-polyhomogeneous function is one which admits, in the relevant limits, a full atlas of matched (well-behaved) asymptotic expansions in terms of elementary functions. ``Exponential-polyhomogeneity'' is just a way to express the qualitative fact that asymptotic expansions exist and behave in the expected way under differentiation without having to write out explicitly the actual expansions. This is very convenient when trying to state a general-purpose theorem.

We say that a function $u:(0,Z]_z\times (0,\infty)_h  \to \bbC$ with $C^1$ slices $u|_{h=h_0} \in C^1(0,Z]$ has ``initial data of exponential-polyhomogeneous type'' if the restrictions 
\begin{align}
	\begin{split}
		u(Z,h)&:(0,\infty)_h  \to \bbC, \\
		u'(z,h)|_{z=Z}&:(0,\infty)_h  \to \bbC
	\end{split}
\end{align}
are of exponential-polyhomogeneous type on $[0,\infty)_h$. This restricts their $h\to 0^+$ behavior while saying nothing about the irrelevant $h\to \infty$ regime. If the reader prefers, they may consider the case where $u(Z,h),u'(Z,h)$ are independent of $h$; $h$-independent initial data is of exponential-polyhomogeneous type.

Our main theorem is a constructive version of: 
\begin{theorem}
	If $Pu=0$ and $u$ has initial data of exponential-polyhomogeneous type, then $u$ is of corresponding exponential-polyhomogeneous type on a certain compactification $M\hookleftarrow (0,Z)_z\times (0,\infty)_h$
	defined below, in \S\ref{subsec:compactification}. 
	\label{thm1}
\end{theorem}

See \Cref{fig} for a depiction of $M$. Skip directly to \S\ref{sec:examples}, \S\ref{subsec:more_examples} for examples.

The theorem states the existence of full, well-behaved asymptotic expansions of solutions in suitable asymptotic regimes which suffice to cover all possible ways of following $u$ along some  smooth graph $(\Gamma(h),h): [0,\infty)_h \to [0,Z]_z\times [0,\infty)_h$
as $h\to 0^+$. A key point is that the asymptotic expansions in powers of the boundary-defining-functions of the edges of $M$ do not depend on the angle or other aspects of the manner via which $\Gamma$ approaches $\partial M$. Only the endpoint $\lim_{h\to 0^+} (\Gamma(h),h) \in \partial M$ in $M$ matters. This sort of behavior can be contrasted with the behavior of the polar angle 
\begin{equation} 
	\theta = \operatorname{arctan}(y/x) : ([0,\infty)_x\times [0,\infty)_y)\backslash \{0\}\to [0,\pi/2],
\end{equation} 
the limiting value of which when followed along a curve ending at the origin depends on the angle of approach. The difference is that $\theta$ is not polyhomogeneous on the punctured quadrant but rather on the ``polar coordinates'' blowup $[[0,\infty)^2_{x,y}; (0,0)]$.

A more explicit and precise version of \Cref{thm1} appears below, in \S\ref{sec:thm}. See \Cref{thm2}. The proof is constructive, in the sense that it provides an algorithm for computing all asymptotic expansions, as well as joint asymptotic expansions at the corners. That algorithm is our main result. We will deduce \Cref{thm1} from \Cref{thm2} as a corollary.

\begin{figure}[t]
	\begin{tikzpicture}[scale=1]
		\fill[gray!5] (5,1.5) -- (0,1.5) -- (0,0)  arc(90:0:1.5) -- (5,-1.5) -- cycle;
		\draw[dashed] (5,1.5) -- (5,-1.5);
		\node (ff) at (.75,-.75) {$\mathrm{fe}$};
		\node (zfp) at (-.35,.75) {$\mathrm{ze}$};
		\node (pf) at (2.5,1.8) {$\mathrm{ie}=\Gamma^{(Z)}$};
		\node (mf) at (3.25,-1.8) {$\mathrm{be}$};
		\draw[->, darkred] (.1,.1) -- (.1,.5) node[right] {$z$};
		\draw[->, darkred] (.1,1.4) -- (1.7,1.4) node[below] {$h^2$};
		\draw[->, darkred] (.1,1.4) -- (.1,.9) node[right] {$Z-z$};
		\draw[->, darkred] (.1,.1) to[out=0, in=150] (1,-.23) node[above] {$\quad\qquad\lambda^{-(\kappa+2)}$};
		\draw[->, darkred] (1.6,-1.4) -- (2.2,-1.4) node[above] {$\qquad\;h^{2/(2+\kappa)}$};
		\draw[->, darkred] (4.9,-1.4) -- (4.2,-1.4) node[above] {$h_0-h$};
		\draw[->, darkred] (4.9,-1.4) -- (4.9,-.4) node[left] {$z$};
		\draw[->, darkred] (1.6,-1.4) to[out=90, in=-60] (1.375,-.65) node[right] {$\lambda$};
		\draw[->, darkred] (4.9,1.4) -- (4.2,1.4) node[below left] {$h_0-h$};
		\draw[->, darkred] (4.9,1.4) -- (4.9,.7) node[below left] {$Z-z$};
		\draw (5,1.5) -- (0,1.5) -- (0,0)  arc(90:0:1.5)  -- (5,-1.5);
	\end{tikzpicture}
	\quad 
	\begin{tikzpicture}[scale=1]
		\fill[gray!5] (5,1.5) -- (0,1.5) -- (0,0)  arc(90:0:1.5) -- (5,-1.5) -- cycle;
		\draw[dashed] (5,1.5) -- (5,-1.5);
		\node[white] (ff) at (.75,-.75) {ff};
		\node[white] (zfp) at (-.35,.75) {$\mathrm{ze}$};
		\node[white] (pf) at (2.5,1.8) {$\mathrm{ie}$};
		\node[white] (mf) at (3.25,-1.8) {$\mathrm{be}$};
		\draw[dashed] (1.3,-.76) to[out=30,in=185] (5,0);
		\draw[dashed] (0,0) to[out=30,in=150] (2,.4) to[out=-30,in=180] (5,.3);
		\draw[dashed] (0,.85) to[out=0,in=180] (5,.6);
		\draw[dashed] (1.5,-1.5) to[out=5,in=150] (3,-1.3) to[out=-30,in=180] (5,-1);
		\draw (5,1.5) -- (0,1.5) -- (0,0)  arc(90:0:1.5)  -- (5,-1.5);
		\node[below] (?) at (3.4,1.4) {${\Gamma^{(1)}= \{z=1\}}$};
		\node (?) at (1.5,.3) {${\Gamma_{\mathrm{H}}}$};
		\node (?) at (3,-.6) {${\Gamma_{1}}$};
		\node (?) at (4.5,-.8) {${\Gamma_{\mathrm{L}}}$};
	\end{tikzpicture}
	\caption{\textit{Left:} The manifold-with-corners $M$ (the portion with $h<h_0$, for $h_0>0$ arbitrary). Some local coordinate charts are depicted. 
	Here, as elsewhere in the paper, $\smash{\lambda=z/h^{2/(\kappa+2)}}$. The quasihomogeneous blowup used to create $M$ is the one which resolves curves $\Gamma_{\lambda_0}=\{\lambda=\lambda_0\}$ of constant $\lambda$, defined in \cref{eq:Gamma_curve}. 
	\textit{Right}: some curves in $M$, including (the lift of) $\Gamma_1$, which hits the front face of the blowup. Also shown is a level set of $z$, a curve $\Gamma_{\mathrm{H}}$ probing the ``intermediate'' regime $\mathrm{ze}\cap \mathrm{fe}$ (the high corner of $\mathrm{fe}$), and one, ${\Gamma_{\mathrm{L}}}$, probing the other such regime $\mathrm{fe}\cap \mathrm{be}$ (the low corner of $\mathrm{fe}$). One of the upshots of this paper is that we can understand the asymptotics of solutions of the ODE along each of these curves, even $\Gamma_{\mathrm{H}},\Gamma_{\mathrm{L}}$. }
	\label{fig} 
\end{figure}

Given how well-trodden this subject is, it may be surprising that there is something left to say. \Cref{thm1} improves on the existing literature in three ways:
\begin{enumerate}
	\item we handle the case $\kappa\geq 3$, for which no full expansions had previously been known (see the remarks of Olver quoted below), and the expansion in the $\kappa=2$ case has been given a proof.
	\item  Poincar\'e expandability is improved to polyhomogeneity. This implies control of all derivatives in the semiclassical parameter, filling an apparent hole in the earlier literature even in the otherwise well-understood cases.
	\item The potential is allowed to have a natural sort of singularity at the transition point, as arises in several applications discussed below. See \S\ref{subsec:more_examples}.
\end{enumerate}
Each of these points will be elaborated upon later.

A key point to stress vis-\`a-vis (1) and (3) is that we are able to handle cases that Olver explicitly states in \cite[\S12.14]{OlverBook} cannot be handled using traditional methods. (Point (2) is more of a technical improvement.) What allows us to handle the cases that Olver cannot is that the set of polyhomogeneous functions on $M$ is bigger than the set of polyhomogeneous functions on the rectangle $[0,Z]_z\times [0,\infty)_{h^2}$. Because $M$ has \emph{two} boundary edges over $\{h=0\}$, in contrast to $[0,Z]_z\times [0,\infty)_{h^2}$, which just has one, polyhomogeneous functions on $M$ have two different $h\to 0^+$ asymptotic expansions associated to two different $z$-scales --- see the exposition in \S\ref{subsec:compactification}. This flexibility turns out to be exactly what is required to push the analysis through. The Langer--Olver expansion prescribed by \cite[\S12.14]{OlverBook} will then be interpreted as \emph{one} of the two asymptotic expansions. The other we will need to construct, but it is arguably the easier of the two. 

In short, \Cref{thm1} provides two matched $h\to 0^+$ expansions (possibly with logarithmic terms), whereas Langer and Olver demanded a single uniform $h\to 0^+$ expansion (without logarithms). This is why our analysis is able to go further.

The rest of this introduction consists of subsections which can be read in any order or skipped entirely, except \S\ref{subsec:compactification}, which must be read and read first.  The subsections are:
\begin{itemize}
	\item \S\ref{subsec:compactification} contains the definition of the compactification appearing in \Cref{thm1}. What we denote `$M$' is a manifold-with-corners (mwc). In this subsection, we will say a few words about why mwcs are the natural setting for multi-scale analysis.
	\item In \S\ref{subsec:JWKB}, we discuss the ur-example, the case of simple turning points, i.e.\ $\kappa=1$, for which our results reduce to (and slightly strengthen) the JWKB theory of Airy-function patching. This basic example does not serve to illustrate the novelty of our results, but it must be understood by any reader. It is hoped moreover that a familiar example will illuminate the geometric formalism we adopt and its use for multi-scale analysis.
	\item \S\ref{subsec:history} is our too-brief literature review. Our goal is not to be comprehensive.
	\item \S\ref{subsec:primer} is a primer on the notion of exponential-polyhomogeneous type appearing in \Cref{thm1}. A rough idea suffices for reading the first few sections of this paper.
	The technical details are important only for readers interested in the technical details.
	\item Related topics in the theory of PDE are described in \S\ref{subsec:PDE}. 
	\item Finally, \S\ref{subsec:conc} contains some concluding remarks, including a list of some directions in which this work could be expanded.
\end{itemize}
Readers may benefit from skipping directly to \S\ref{sec:examples}, which contains a number of examples.
Then, \S\ref{sec:thm} contains more, somewhat technical, introductory material (including a few more examples in \S\ref{subsec:more_examples}), together with an outline of the remainder of the paper. \Cref{thm2} appears there, as does a summary of its proof.

\begin{remark}[Reduction from the general second-order case to that above]
	Any second-order semiclassical ordinary differential operator 
	\begin{equation} 
		a(z,h) h^2 \frac{\partial^2}{\partial z^2} + b(z,h) h \frac{\partial}{\partial z} + c(z,h)
	\end{equation} 
	on the real line (with coefficients which are smooth save for isolated poles at locations not depending on $h$) is, assuming that the leading coefficient $a$ is smooth and nonvanishing, equivalent modulo conjugation (in the sense of possessing identical null space) to 
	\begin{equation}
		P=-h^2 \frac{\partial^2}{\partial z^2} + V(z)+ hQ(z,h)
		\label{eq:misc_klo}
	\end{equation}
	for some $V,Q$. 
	This is why we assumed, in \cref{eq:1}, that $P$ lacks first-order terms.
	
	A very common situation is $Q=h \psi$ for $\psi$ a smooth function of $h^2$. It is this common situation to which we restrict attention. 
	
	Generically, we should expect that, insofar as the potential $V$ in \cref{eq:misc_klo} vanishes, it vanishes only at isolated points, at each of which the function vanishes to some finite order. This is the case when $V$ is meromorphic and not identically zero. 
	
	If one is concerned with the local properties of solutions, then it suffices to restrict attention to closed intervals $I\subset \bbR$ of $z$'s containing at most a single zero/singularity of $V$ or singularity of $\psi$. Without loss of generality, it may be assumed that the interval $I$ is given by $I=[0,Z]$ for some $Z>0$, with the zero/singularity (if it exists at all) at the endpoint $z=0$. 
	Then, we can write the potential as $V(z) = \varsigma z^\kappa W(z)$
	for $\varsigma \in \{-1,+1\}$, $W$ as above, and $\kappa\in \bbZ$.
	\label{rem:reduction}
\end{remark}
\begin{remark}[$\kappa\leq -2$]
	We do not consider $\kappa \leq -2$, as then either $P$ has a regular singularity with variable indicial roots (if $\kappa=-2$) or an irregular singularity (if $\kappa\leq -3$). These cases are of a rather different character than those handled here. We believe that an irregular singularity is actually easier to handle than the case here, and that exponential-polyhomogeneity holds on the original rectangle $[0,Z]_z\times [0,\infty)_h$. No blowup is necessary. 
	This should be possible to prove using a variant of the argument in \S\ref{sec:collapsing}. See \cite[\S10,4/5, \S12.14.4]{OlverBook}\cite{OlverTransition} for Olver's discussion of the case of an irregular singularity. Olver's conclusion is that, under certain reasonable constraints on $\psi$, the Langer--Olver method utilized in \cite{Langer31, Langer32,Langer35}\cite{OlverOriginal, OlverBook} applies. Other treatments of different aspects of the problem come to the same conclusion \cite[\S31.(c)]{Wasow}.

	For the $\kappa=-2$ case, we believe that exponential-polyhomogeneity holds on $M\backslash \mathrm{be}$ if $M$ is defined as above \emph{for any different $\kappa\geq -1$}. This case is known to be recalcitrant (or at least it appears that way to us). It is conspicuously absent from \cite[\S12.14.4]{OlverBook}, though partial results (e.g.\ approximations lacking an extension to full expansions) exist, see e.g.\ \cite{Cashwell}\cite{McKelvey}.
	We expect the argument in \S\ref{sec:collapsing} to yield this after the ``Langer correction'' of the Liouville--Green ansatz is taken into account. 
	\label{rem:extension}
\end{remark}

\begin{remark}[Fractional $\kappa$]
	We believe that \Cref{thm1}, \Cref{thm2} hold for any real $\kappa>-2$ and that the argument below, with small modifications (in particular, more complicated index sets), can handle this level of generality, though we have not checked the details. See \cite{OlverFrac} for Olver's discussion of fractional $\kappa$. It is worth noting that Olver provides only an approximation, not an expansion --- see the remark from \cite{OlverFrac} quoted below. 
\end{remark}

\subsection{The compactification}
\label{subsec:compactification}

Let $M$ denote the manifold-with-corners\footnote{We use ``manifold-with-corners'' in the sense of Melrose \cite{MelroseCorners}, though the precise definition is not important here. Roughly, a $d$-dimensional mwc is a space smoothly modeled on $\bbR^k\times [0,\infty)^{d-k}$ for $k\leq d$. } (mwc), depending on $\kappa$, constructed by performing a quasihomogeneous blowup of the corner $\{z=0,h=0\}$ of the rectangle $[0,Z]_z\times [0,\infty)_{h^2}$ so as to separate the family $\{\Gamma_\lambda\}_{\lambda>0}$ of curves 
\begin{equation}
\Gamma_\lambda = \{ z= \lambda h^{2/(2+\kappa)} \}.
\label{eq:Gamma_curve}
\end{equation}
See \Cref{fig:blowup}.
This blowup resolves the ratio 
\begin{equation} 
	\lambda=z/h^{2/(\kappa+2)},
\end{equation} 
which becomes a smooth coordinate $\lambda \in C^\infty(\mathrm{fe}^\circ)$ parametrizing the front edge $\mathrm{fe}^\circ$ of the blowup and extending smoothly down to one boundary point. 
Besides $\mathrm{fe}$, the other edges of $M$ are 
\begin{itemize}
	\item 
	$\mathrm{ze} =\mathrm{cl}_M \{h=0,z>0\}$, the ``zero edge,'' the edge left over from the edge of the rectangle $[0,Z]_z\times [0,\infty)_{h^2}$ where the semiclassical parameter $h$ vanishes,
	\item $\mathrm{be} = \mathrm{cl}_M\{h>0,z=0\}$, the ``boundary edge'' (or the locus of b-analysis, where, as in \cite{MelroseAPS}, b- means associated with regular singular differential equations),
	\item and $\mathrm{ie} = \mathrm{cl}_M\{z=Z\} = \{z=Z\}$, the ``initial edge,'' where the initial data is specified.\footnote{Throughout this paper, we will identify submanifolds disjoint from a neighborhood of a blown up locus with their lifts to the resolved space.
	This holds, in particular, for points in the interior of the mwc. (We will not blow up any interior submanifolds in this paper.) Thus, we will consider blowups to leave unchanged the interior, at the level of sets.}
\end{itemize}  
Again, see \Cref{fig} for a depiction of $M$.

\begin{figure}
	\begin{tikzpicture}[scale=.8]
		\fill[gray!5] (4,1.5) -- (0,1.5) -- (0,-1.5) -- (4,-1.5) -- cycle;
		\draw[dashed] (4,1.5) -- (4,-1.5);
		\node[white] (ff) at (.75,-.75) {ff};
		\node[white] (zfp) at (-.35,.75) {$\mathrm{ze}$};
		\node[white] (pf) at (2.5,1.8) {$\mathrm{ie}$};
		\node[white] (mf) at (3.25,-1.8) {$\mathrm{be}$}; to[out=-30,in=180] (5,-1);
		\draw (4,1.5) -- (0,1.5) -- (0,-1.5) -- (4,-1.5);
		\node[right] (?) at (4,1.2) {${\Gamma_{3}}$};
		\node[right] (?) at (4,.7) {${\Gamma_{2}}$};
		\node[right] (?) at (4,.2) {${\Gamma_{1}}$};
		\draw[dashed] (0,-1.5) to[out=90,in=182] (4,.2);
		\draw[dashed] (0,-1.5) to[out=90,in=185] (4,.7);
		\draw[dashed] (0,-1.5) to[out=90,in=185] (4,1.2);
		\draw[dotted] (0,-1) -- (4,-1) node[right] {$\Gamma^{(1)}$};
		\draw[dotted] (0,-.5) -- (4,-.5) node[right] {$\Gamma^{(2)}$};
		\draw[dotted] (2,-1.5) -- (2,1.5) node[above] {$\{h=h_0\}$};
		\node[above] () at (0,1.5) {$\{h=0\}$};
		\node[above] () at (4,1.5) {$\{h=h_1\}$};
		\node[right] () at (4,-1.5) {$\Gamma^{(0)}$};
		\draw[darkred,->] (-.1,-1.6) -- (-.1,-.4) node[left] {$z$};
		\draw[darkred,->] (-.1,-1.6) -- (1.1,-1.6) node[below] {$\;\;h^2$};
	\end{tikzpicture}
	\begin{tikzpicture}[scale=.8]
		\fill[gray!5] (4,1.5) -- (0,1.5) -- (0,-1.5) -- (4,-1.5) -- cycle;
		\draw[dashed] (4,1.5) -- (4,-1.5);
		\node[white] (ff) at (.75,-.75) {ff};
		\node[white] (zfp) at (-.35,.75) {$\mathrm{ze}$};
		\node[white] (pf) at (2.5,1.8) {$\mathrm{ie}$};
		\node[white] (mf) at (3.25,-1.8) {$\mathrm{be}$}; to[out=-30,in=180] (5,-1);
		\draw (4,1.5) -- (0,1.5) -- (0,-1.5) -- (4,-1.5);
		\node[right] (?) at (4,1) {${\Gamma_{3}}[\kappa=0]$};
		\node[right] (?) at (4,0) {${\Gamma_{2}}[\kappa=0]$};
		\node[right] (?) at (4,-1) {${\Gamma_{1}}[\kappa=0]$};
		\draw[dashed] (0,-1.5) -- (4,-1);
		\draw[dashed] (0,-1.5) -- (4,0);
		\draw[dashed] (0,-1.5) -- (4,1);
		\draw[dashed] (0,-1.5) to[out=0, in=250] (3.8,1.5) node[above] {$\Gamma_1[\kappa=-1]$};
		\draw[dashed] (0,-1.5) to[out=0, in=265] (1.6,1.5) node[above] {$\Gamma_3[\kappa=-1]\qquad$};
		\draw[darkred,->] (-.1,-1.6) -- (-.1,-.4) node[left] {$z$};
		\draw[darkred,->] (-.1,-1.6) -- (1.1,-1.6) node[below] {$\;\;h^2$};
	\end{tikzpicture}
	\begin{tikzpicture}[scale=.8]
		\fill[gray!5] (4,1.5) -- (0,1.5) -- (0,0)  arc(90:0:1.5) -- (4,-1.5) -- cycle;
		\draw[dashed] (4,1.5) -- (4,-1.5);
		\node (ff) at (.75,-.75) {$\mathrm{fe}$};
		\node (zfp) at (-.35,.75) {$\mathrm{ze}$};
		\node (pf) at (2.5,1.8) {$\mathrm{ie}$};
		\node (mf) at (3.25,-1.8) {$\mathrm{be}$};
		\draw (4,1.5) -- (0,1.5) -- (0,0)  arc(90:0:1.5)  -- (4,-1.5);
		\node[white] () at (0,-2.2) {};
		\draw[dotted] (3.5,-1.5) to[out=90,in=270] (3.25,1.5);
		\draw[dotted] (3,-1.5) to[out=90,in=275] (2.125,1.5);
		\draw[dotted] (2.5,-1.5) to[out=90,in=280] (1,1.5);
		\draw[dotted] (0,1.2) to[out=0,in=180] (4,1);
		\draw[dotted] (0,.7125) to[out=0,in=180] (4,0);
		\draw[dotted] (0,.3) to[out=0,in=180] (4,-1);
		\begin{scope}
			\clip  (4,1.5) -- (0,1.5) -- (0,0)  arc(90:0:1.5) -- (4,-1.5) -- cycle;
			\draw[dashed] (0,-1.5) -- (1.3,1.5);
			\draw[dashed] (0,-1.5) -- (4,1.5);
			\draw[dashed] (0,-1.5) -- (4,-.5) node[above left] {$\Gamma_1$};
		\end{scope}
		\node[above] () at (4,1.5) {$\Gamma_2$};
		\node[above] () at (1.3,1.5) {$\Gamma_3$};
	\end{tikzpicture}
	\caption{\textit{Left:} the rectangle $[0,Z]_z\times [0,\infty)_{h^2}$ whose lower-left corner is blown up (quasihomogeneously) to create $M$. Three families of curves are shown in $M$. These are the level sets $\Gamma^{(z_0)}=\{z=z_0\}$ of $z$ (horizontal lines), the level sets of $h$ (vertical lines), and the level sets $\Gamma_\lambda$ of $\lambda$. We depict the $\kappa\geq 1$ case. \textit{Middle:} some of the curves $\Gamma_\lambda$ in the $\kappa=-1,0$ cases, where their convexity differs from the $\kappa\geq 1$ case. In the $\kappa=0$ case, they are lines. In the $\kappa=-1$ case, they are parabolas.
	\textit{Right:} the same curves in $M$. Note that, in this compactification, the $\Gamma_\lambda$ limit to distinct points of $\mathrm{fe}^\circ$. This is what the quasihomogeneous blowup of the corner $\{z=0,h=0\}$ accomplishes. Note that the level sets of $h,z$ (dotted lines) appear distorted in this rendering.}
	\label{fig:blowup} 
\end{figure}

Insofar as we have a choice of smooth structure, that choice does not affect the spaces of polyhomogeneous functions and is therefore not important for what we do here. For definiteness, we fix a choice: choose the smooth structure at the front face $\mathrm{fe}$ of the blowup  such that 
\begin{equation} 
	\varrho_{\mathrm{fe}}  = z + \smash{h^{2/(\kappa+2)}}
	\label{eq:fe_bdf}
\end{equation} 
becomes a (boundary-)defining-function (bdf) of it. 
Defining functions of the faces $\mathrm{ze}$, $\mathrm{be}$, $\mathrm{ie}$ are then given by
\begin{equation}
	\varrho_{\mathrm{ze}}  = h^2/ \varrho_{\mathrm{fe}}^{\kappa+2},\quad \varrho_{\mathrm{be}} = z \varrho_{\mathrm{fe}}^{-1},\quad \varrho_{\mathrm{ie}} = Z-z.
	\label{eq:other_bdfs}
\end{equation}
Other choices are possible. These are simply the most convenient.

When restricting attention to local coordinate charts, it is almost always possible to work instead with local bdfs:
\begin{itemize}
	\item Outside of any neighborhood of $\mathrm{ze}$, the ratio $z/h^{2/(\kappa+2)}$ is a bdf for $\mathrm{be}$ and $h^{2/(\kappa+2)}$ is a bdf for $\mathrm{fe}$, and
	\item  outside of any neighborhood of $\mathrm{be}$, $h^2 / z^{\kappa+2}$ is a local bdf for $\mathrm{ze}$ and $z$ is a bdf for $\mathrm{fe}$. 
\end{itemize}

The mwcs we care about arise as compactifications of manifolds-without-boundary, and we would like to keep track of these compactifications. This is usually done through the use of subscripts which indicate the compactifying map. For example, when we write ``$[0,Z]_z\times [0,\infty)_h$,'' we mean the compactification of $(0,Z)_z\times (0,\infty)_h$ which is just the inclusion of the latter into the former, i.e.\ the identity map. Likewise, when we write ``$[0,Z]_z\times [0,\infty)_{h^{2/(\kappa+2)}}$,'' the compactification is 
\begin{equation} 
(0,Z)\times \bbR^+\ni 	(z,h)\mapsto (z,h^{2/(\kappa+2)}) \in [0,Z]\times [0,\infty).
\end{equation}
Even though $[0,Z]_z\times [0,\infty)_h$ and $[0,Z]_z\times [0,\infty)_{h^{2/(\kappa+2)}}$ are diffeomorphic as mwc --- they are both rectangles with one open side --- they are inequivalent in the category of $C^\infty$-compactifications of $(0,Z)\times (0,\infty)$. (Two compactifications are equivalent in this category if there exists a diffeomorphism which restricts to the identity map on the space compactified, in which case that diffeomorphism is canonical. This just means equivalence in the category of topological compactifications with the added requirement that the resultant homeomorphism be a diffeomorphism.) None of the mwcs we work with are topologically complicated --- their diffeomorphism classes are just polygons --- but the way in which they arise as compactifications of $[0,Z]_z\times (0,\infty)_h$ matters.

Figures such as \Cref{fig} should be understood as illustrating the combinatorial structure of $M$, i.e.\ the boundary components/edges and how they intersect. 

 In the depicted atlas of coordinate charts, the coordinates are identified with their restrictions to the interior
\begin{equation}
	M^\circ = (0,Z)_z\times (0,\infty)_{h}.
	\label{eq:misc_04}
\end{equation}
Any smooth, or more generally polyhomogeneous, function on $M$ is uniquely determined by its restriction to the interior $M^\circ$ and can therefore be identified with this restriction.\footnote{Polyhomogeneous functions are certain distributions on $M$, meaning elements of the dual of $C_{\mathrm{c}}^\infty(M)$, and they are smooth on $M^\circ$. (Note that elements of $C_{\mathrm{c}}^\infty(M)$ can be nonvanishing at the boundary.) The only distributions on $M$ which, when restricted to the interior, become zero are those supported on the boundary. These involve Dirac $\delta$-distributions or their derivatives and are not counted among the polyhomogeneous \emph{functions}. See \cite{MelroseCorners}.} Consequently, we can think of polyhomogeneous functions on $M$ as certain smooth functions of $z,h>0$ (which may not extend smoothly down to $z=0$ or $h=0$). 
We will employ this sort of identification whenever thinking about different compactifications and polyhomogeneous functions on them. Functions are always identified with their restrictions to the interior.

The reader may wish to think of $M$ as being some explicit subset of $\bbR^2$. (Unfortunately, then \cref{eq:misc_04} no longer holds literally, at the level of sets, but it still makes sense as a canonical diffeomorphism.) For example, one can choose an embedding $\iota:M\hookrightarrow \bbR^2_{x,y}$ with image
\begin{equation}
	\Omega = \{(x,y) \in \bbR^2: y\in [0,2], x\geq 0, x^2+y^2\geq 1\}.
	\label{eq:Omega}
\end{equation}
The set $\Omega$ is a sub-mwc of $\bbR^2$. (This should hold given any reasonable notion of ``mwc.'') The edge ie is taken as $\{y=2\}$, be is taken as $\{y=0\}$, ze is taken as $\{x=0\}$, and fe is taken as $\{x^2+y^2=1\}$. 
An advantage of thinking about $M$ in this way is that then the smooth functions on $M$ are exactly the restrictions of smooth functions on the ambient $\bbR^2$: smooth means extendable. In addition, diagrams such as those in \Cref{fig} can then be interpreted literally, showing the subset of $\bbR^2$ (identifying $\bbR^2$ with the plane of the page) we are identifying with $M$.

The fact that $M$ has two boundary edges at $h=0$, namely ze and fe, means that analysis on $M$ is a form of \emph{multi-scale analysis}. There are two scales relevant in the $h\to 0^+$ limit, 
\begin{itemize}
	\item the $z=\Omega(1)$ scale (associated with ze) and
	\item the $z= O(h^{2/(\kappa+2)})$ scale (associated with fe). 
\end{itemize}
To say that a function is polyhomogeneous on $M$  means very roughly that only those two scales are relevant in understanding its $h\to 0^+$ asymptotics. For example, polyhomogeneity at the corner $\mathrm{ze}\cap\mathrm{fe}$ means that no ``intermediate scales'' between $z=\Omega(1)$ and $z=O(h^{2/(\kappa+2)})$ need to be resolved to fully understand the $h\to 0^+$ limit. Polyhomogeneity at $\mathrm{be}\cap\mathrm{fe}$ means that no scale smaller than $z=O(h^{2/(\kappa+2)})$ poses a problem either. No interesting behavior is hidden in a corner of $M$, unlike on the original rectangle $[0,Z]_z\times [0,\infty)_{h^2}$, where interesting behavior is hidden in the corner $\{z=0,h=0\}$.

The passage to the blowup is a specific instance of a general strategy in geometric singular analysis, that of \emph{trading analytic complexity for geometric complexity}. Rather than working with complicated function spaces defined on simple spaces, one can choose to work with simple function spaces on complicated spaces.\footnote{In this case, the polyhomogeneous function spaces are considered simple. If the reader is not familiar with these function spaces, they may not appear simple, but it should be stressed that polyhomogeneity is only a mild generalization of smoothness. Given that $C^\infty$ is a familiar function space, the spaces of polyhomogeneous functions should not be considered too exotic.}

\begin{example}
	As a simple example of a function on $(0,Z)_z\times \bbR^+_h$ with ``multi-scale dependence'' resolved by working on the $\kappa=0$ version of $M$, consider
	\begin{equation}
		f(z,h) = \operatorname{arctan}\Big(\frac{z}{h} \Big) -z-h.
		\label{eq:f_ex}
	\end{equation}
	Intuitively, if one thinks of $f$ as a one-parameter family $\{f(-,h)\}_{h>0}$ of functions of $z$, it is varying on two scales. The term $-z$ is varying on the $z=O(1)$ scale, whereas $\operatorname{arctan}(z/h)$ is varying fast, on the $\Omega(h)$ length scale.
	A contour plot of $f$ is shown in \Cref{fig:fur_Nick}. In that same figure, we show $f\circ T:\Omega\to \bbR$ for $T:(x,y)\mapsto (z,h)$ for 
	\begin{align}
		\begin{split} 
			h &= \frac{x(-1+\sqrt{x^2+y^2})}{\sqrt{x^2+y^2}} \\ 
			z&= \frac{y (-1+\sqrt{x^2+y^2})\sqrt{x^2+4}}{2\sqrt{x^2+y^2}(-1+\sqrt{x^2+4})}.
		\end{split} 
		\label{eq:fur_Nick}
	\end{align} 
	This map $T$ does not extend to a diffeomorphism $\Omega \to M$, but it does extend to a homeomorphism. (This is just a matter of the choice of smooth structure. The choice that is most convenient for the asymptotic analysis below is not the most convenient for constructing an explicit diffeomorphism $M\cong \Omega.$)
	
	Notice how $f\circ T$ is continuous on $\Omega$, whereas $f$ is not continuous on $[0,Z]_z\times [0,\infty)_{h^2}$, specifically at the corner $\{z=0,h=0\}$. 
	This demonstrates the utility of the geometric perspective --- passing to $M$ or $\Omega$, it is easy to describe the regularity of $f$, it is just continuous. How would one describe the behavior of $f$ working directly with the $z,h$ variables? No clean solution presents itself. 
	This is the philosophy of geometric singular analysis in action: we locate $f$ in a simple function space on a complicated object, namely the space of continuous functions (a simple function space) on $M$ (a ``complicated'' object, though not a very complicated one), rather than in some complicated function space defined directly in terms of the $z,h$ variables.
\end{example}

\begin{figure}
	\centering
	\includegraphics[scale=.45]{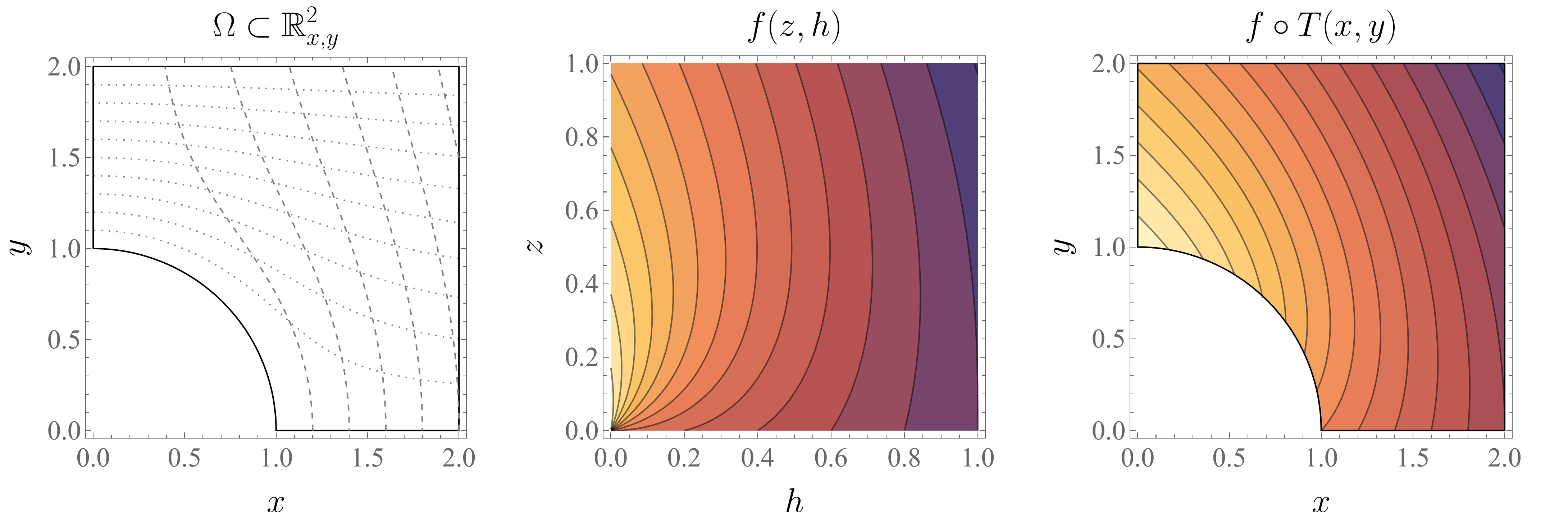}
	\caption{\textit{Left:} $\Omega\subset \bbR^{2}_{x,y}$, where $\Omega$ is as in \cref{eq:Omega}. Level curves of $h(x,y)$, $z(x,y)$ are shown. This example was computed using \cref{eq:fur_Nick}. \textit{Middle:} The function $f$ defined in \cref{eq:f_ex} plotted against $h,z$. Note how many contour lines converge at the bottom-left corner. This shows that $f$ is not continuous there. \textit{Right:} the pullback $f\circ T$ of $f$ to $\Omega$, via $T:(x,y)\mapsto (z,h)$. As is evident from the behavior of the contour lines, $f\circ T$ is continuous on $\Omega$. }
	\label{fig:fur_Nick}
\end{figure}

When we say that a function is polyhomogeneous on $M$, what we are allowing is multi-scale dependence like that seen in the previous example.

\begin{example}
	As another example of this sort of multi-scale dependence, consider 
	\begin{equation}
		f(z,h) = \cos\Big( \frac{1}{h} \Big)\Big(1 - \frac{h \log h}{z+h} \Big).
		\label{eq:misc_017}
	\end{equation}
	How might we think about the behavior of this function in the $h\to 0^+$ limit? It is tempting to identify the $h (z+h)^{-1}\log h$ term as ``lower-order,'' and indeed it is $O(h)$ in $\{z>\varepsilon\}$ for any $\varepsilon$, but if we take $h,z\to 0^+$ together, specifically while fixing the ratio $h/z>0$, then it diverges logarithmically; not only is it not $O(h)$, it is not even $O(1)$. So, in this limit, we should consider $h (z+h)^{-1} \log h$ as the \emph{main} term in \cref{eq:misc_017}, not as a subleading term.

	Now, $f$ is highly oscillatory, so it is not polyhomogeneous on $M$. However, $\cos(1/h)$ is an exponential function of $1/h$, which is polyhomogeneous, so $f(z,h)$ is of \emph{exponential}-polyhomogeneous type if
	\begin{equation} 
		f_0(z,h) = 1- h(z+h)^{-1} \log h
	\end{equation} 
	is polyhomogeneous. This is not polyhomogeneous on $[0,Z]_z\times [0,\infty)_{h^2}$, nor is it continuous on $M$, but it is polyhomogeneous on $M$ (when $\kappa=0$, as in the previous example). In order to see this, note that, in terms of the boundary-defining-functions in \cref{eq:fe_bdf}, \cref{eq:other_bdfs}, 
	\begin{equation}
		f_0(z,h) = 1 - \varrho_{\mathrm{ze}}^{1/2} \log (\varrho_{\mathrm{ze}}^{1/2} \varrho_{\mathrm{fe}} ) = 1 - \frac{\varrho_{\mathrm{ze}}^{1/2}}{2} \log \varrho_{\mathrm{ze}} - \varrho_{\mathrm{ze}}^{1/2} \log \varrho_{\mathrm{fe}}.
		\label{eq:misc_019}
	\end{equation}
	On any manifold-with-corners, bdfs, their powers, and powers of their logarithms are all polyhomogeneous, so $f_0$ is indeed polyhomogeneous on $M$. 
	
	A key feature of polyhomogeneous functions on $M$ is that they can be expanded in either of the bdfs $\varrho_{\mathrm{ze}},\varrho_{\mathrm{fe}}$, and the two expansions do not need to have much to do with each other, except that compound expansions have to agree at the corner $\mathrm{fe}\cap\mathrm{ze}$ (a basic compatibility criterion). So, polyhomogeneous functions on $M$ have \emph{two} different expansions in the $h\to 0^+$ limit, one for each scale, or equivalently for each edge $\mathrm{fe},\mathrm{ze}$ of $M$ at $\{h=0\}$. Concretely, in the case at hand, the two expansions of $f_0(z,h)$ are as follows:
	\begin{itemize}
		\item For the expansion at $\mathrm{ze}$, the coefficients should be functions on $\mathrm{ze}$, i.e.\ functions of $z$. (Really, what we are doing is using the canonical identification of a neighborhood of $\mathrm{ze}$ in $M$ with $[0,\varepsilon)_{\varrho_{\mathrm{ze}}}\times [0,Z]_z$.) Thus, we rewrite $f_0$ in terms of $\varrho_{\mathrm{ze}},z$, which gives 
		\begin{equation}
			f_0 = 1 - \varrho_{\mathrm{ze}}^{1/2} \log \bigg( \frac{\varrho_{\mathrm{ze}}^{1/2} z}{1-\varrho_{\mathrm{ze}}^{1/2}}\bigg)=1 - \frac{\varrho_{\mathrm{ze}}^{1/2}}{2}\log \varrho_{\mathrm{ze}}- \varrho_{\mathrm{ze}}^{1/2} \log z -\varrho_{\mathrm{ze}} \sum_{j=0}^\infty \frac{\varrho_{\mathrm{ze}}^{j/2}}{j+1}.
		\end{equation}
		\item For the expansion at $\mathrm{fe}$, the coefficients should be functions on $\mathrm{fe}$, which we can parameterize using $\varrho_{\mathrm{ze}}$. (Similar to above, really one should be using an identification of a neighborhood of $\mathrm{fe}$ in $M$ with $[0,\varepsilon)_{\varrho_{\mathrm{fe}}}\times [0,1]_{\varrho_{\mathrm{ze}}}$.) Thus, \cref{eq:misc_019} is already presented in the form of an expansion at $\mathrm{fe}$, we just need to rearrange:
		\begin{equation}
			f_0 = - \varrho_{\mathrm{ze}}^{1/2} \log \varrho_{\mathrm{fe}} + \Big(1 - \frac{\varrho_{\mathrm{ze}}^{1/2}}{2} \log \varrho_{\mathrm{ze}}  \Big).
		\end{equation}
	\end{itemize}
	The added flexibility due to having two different $h\to 0^+$ expansions will be crucial below.
\end{example}

\subsection{The JWKB case}
\label{subsec:JWKB}
Consider the semiclassical Schr\"odinger operator 
\begin{equation}
	P = -h^2 \frac{\partial^2}{\partial z^2} - E + V
\end{equation}
for fixed $E\in \bbR$ (this having nothing to do with what we called $E$ previously) and $V\in C^\infty(\bbR)$, and suppose that $E$ is a regular value of $V$, so that $V(z)-E$ has only simple zeros, say at $z_0<\dots<z_N$, within some interval $[-Z,+Z]$, where $Z>0$ and $-Z$ are not in $V^{-1}(\{E\})$. 
We will consider the ODE $Pu=0$ with initial data posed at $z=Z$. The operator $P$ has the form considered in \cref{eq:1} for $\kappa=1$ and $\psi=0$ on small intervals of $z$, e.g.\ $[z_j,z_j+(z_{j+1}-z_j)/2]$.

Then, $z_0,\dots,z_N$ are called (simple) \emph{turning points}, because if one considers a Newtonian particle with total energy $E$ in the presence of some potential $V$, then the $z_n$'s are the locations where the particle stops and reverses direction -- that is, where it turns around. The regions $\{V(z)<E\}\subset \bbR$ are known as ``classically allowed regions,'' because it is possible for a Newtonian particle in this region to have total energy $E$. The regions $\{V(z)>E\}$ are the ``classically forbidden'' regions, because any Newtonian particle in this region must, by the non-negativity of kinetic energy, have energy strictly greater than $V$.

Consider the mwc $X$ that results from performing a quasihomogeneous blowup of each 
\begin{equation}
	(z_0,0),\dots,(z_N,0)\in [-Z,Z]_z\times [0,\infty)_{h^2}
\end{equation}
on the $\{h=0\}$ boundary of the rectangle.
The specific quasihomogeneous blowup is the one such that $X$ is, near each blowup, constructed from attaching at the edge labeled `be' two copies of the $\kappa=1$ version of $M$. Concretely, this means that the interior of the front face that results from blowing up  the $n$th point, $(z_n,0)$, is parameterized by 
\begin{equation} 
	\lambda=h^{-2/3}(z-z_n) .
\end{equation}

Then, \Cref{thm1} says that, if we impose exponential-polyhomogeneous initial data at $z=Z$ (for example, if our initial data is just independent of $h$), then the solution is of exponential-polyhomogeneous type on $X$.
Somewhat remarkably, it appears that no proof of this fact, or one amounting to it, has yet appeared in the literature. The fact that asymptotic expansions exist and hold in the sense of Poincar\'e goes back to at least the 1930s, as does the fact that you can differentiate the expansions term-by-term in the argument $z$. However, that you can \emph{also} differentiate in the semiclassical parameter $h$, and that the resultant expansions are asymptotic expansions for $\partial u/\partial h$, seems not to have been formulated or given a proof. At least, we are not aware of the existence of such a result. Olver and the other asymptotic analysts of the surrounding generations
tended to state results in terms of Poincar\'e expandability. The pertinent theorems state that expansions hold in the sense of Poincar\'e, but this does not imply that they can be differentiated in the semiclassical parameter. We do not know whether Olver or others considered this question. It is somewhat unnatural from the perspective of semiclassical analysis. Regardless, they could have provided a proof with the tools at their disposal (for reasons we explain below). This differentiability in $h$ is not an essential novelty of this paper, but we felt it important to record.

\Cref{thm1} is qualitative. \Cref{thm2} gives more precise information. It shows that the solution is approximated by an Airy function near each turning point. 
The edges of $X$ called `ze' in this paper are where the Liouville--Green theory works (this is essentially what polyhomogeneity at $\mathrm{ze}^\circ$ says), and the edges ``fe,'' i.e.\ the front faces of the blowups, are the transitional regions where the Airy function must be used to interpolate.
This old idea of using Airy functions to patch together the Liouville--Green ansatzes in the classically allowed and classically forbidden regions goes back to the 1920s, in the works of Jeffreys--Wentzel--Kramers--Brillouin (JWKB) \cite{Jeffreys}\cite{Wentzel}\cite{Kramers}\cite{Brillouin}, after whom the whole industry of semiclassical expansions has come to be referred. We have labeled this subsection as the ``JWKB'' case, because, even though physicists often refer to the Liouville--Green approximation as the ``WKB approximation,'' the essential novelty in the works of JWKB was the treatment of turning points.

The geometric formulation of \Cref{thm1} makes clear that the Airy function patching is to occur in an $\Omega(h^{2/3-\epsilon})$ neighborhood of each turning point, for arbitrary $\epsilon\in (0,2/3)$.  Polyhomogeneity at $\mathrm{fe}^\circ$ guarantees that, within any 
$O(h^{2/3})$ neighborhood of each turning point, the approximation can be promoted to a full asymptotic expansion in ascending fractional powers of $h$, with coefficients which are smooth functions of $\lambda$. Polyhomogeneity at the corner  $\mathrm{ze}\cap \mathrm{fe}$ gives us a way of thinking about the uniformity of the approximation and the way in which it dovetails with the Liouville--Green expansion in the classically allowed/forbidden regions. 

This is all more or less well-known, and there exists, in this case (the observation we are currently making will not generalize), a better method of thinking about the uniformity of the approximation. The better method was proposed by Langer \cite{Langer31,  Langer32, Langer49} and developed rigorously by Langer and Olver \cite{OlverOriginal, OlverBook}, so we will refer to it after those two authors. This is the ``uniform'' JWKB (we prefer ``Langer--Olver'') expansion, in which the solution of interest is approximated near the turning point $z_n$ using an asymptotic expansion of the form
\begin{equation}
	(1+ h^2 \beta) A\Big( \frac{\zeta}{h^{2/3}} \Big)  + h^{2/3} \gamma  A'\Big( \frac{\zeta}{h^{2/3}} \Big)\label{eq:JWKB_uniform}
\end{equation}
for some Airy function $A$ and formal series $\beta,\gamma \in C^\infty(\bbR_z)[[h^2]]$, where $\zeta$ is given by \cref{eq:zeta}, which in this case reads 
\begin{equation}
	\zeta(z) = \Big[ \pm \frac{3}{2} \int_{z_n}^z \sqrt{E-V(\omega)} \dd \omega \Big]^{2/3},
\end{equation}
where the sign $\pm$ depends on whether $z_n$ is to the left or the right of a classically allowed region.
See \cite[Thm.\  7.1]{OlverBook} for a rigorous statement. Numerous examples of this sort of expansion exist in the literature on special functions. Many are listed in the National Institute of Standards and Technology's (NIST) online ``DLMF'' database \cite{NIST} (at the moment of writing) --- see \cite[\href{http://dlmf.nist.gov/10.20.i}{\S10.20(i)}]{NIST}\cite[\href{http://dlmf.nist.gov/12.10.vii}{\S12.10(vii)}]{NIST}\cite[\href{http://dlmf.nist.gov/12.10.viii}{\S12.10(viii)}]{NIST}\cite[\href{http://dlmf.nist.gov/13.21.iii}{\S13.21(iii)}]{NIST}\cite[\href{http://dlmf.nist.gov/18.15.E22}{\S18.15.22}]{NIST}\cite[\href{http://dlmf.nist.gov/33.12.i}{\S33.12(i)}]{NIST}.

\begin{figure}
	\begin{center}
		\begin{tikzpicture}
			\draw[dashed] (-.7,-2.1) rectangle (4.6,2.05);
			\fill[fill=lightgray!20] (3.9,1.45) -- (.1,1.45) -- (.1,-1.5) -- (3.9,-1.5);
			\fill[fill=lightgray!40] (3.9,.65) -- (.1,.65) -- (.1,-.65) -- (3.9,-.65) -- cycle;
			\draw (3.9,1.45) -- (.1,1.45) -- (.1,-1.5) -- (3.9,-1.5);
			\draw[dashed] (3.9,1.45) -- (3.9,-1.5);
			\node () at (-.3,1.7) {(a)};
			\node () at (2,-1.8) {$\{z=-Z\}$};
			\node () at (2,1.7) {$\{z=Z\}$};
			\draw[dashed] (.1,.65) -- (3.9,.65) node[right] {$z_2$};
			\draw[dashed] (.1,-.65) -- (3.9,-.65) node[right] {$z_1$};
			\draw[dashed] (.1,.65) to[out=70, in=184] (1,1.45) node[below] {$\Gamma_{2,1}$};
			\draw[dashed] (.1,-.65) to[out=-70, in=176] (1,-1.45) node[above] {$\;\Gamma_{1,-1}$};
			\fill[black] (.1,.65) circle (2pt);
			\fill[black] (.1,-.65) circle (2pt);
			\draw[->,darkred] (.2,-.4) -- (.9,-.4) node[right] {$h$};
			\draw[->,darkred] (.2,-.4) -- (.2,.4) node[right] {$z$};
		\end{tikzpicture}
		\begin{tikzpicture}
			\draw[dashed] (-.7,-2.1) rectangle (4.4,2.05);
			\fill[fill=lightgray!20] (3.9,1.45) -- (.1,1.45) -- (.1,.95) arc(90:-90:.3) -- (.1,-.35) arc(90:-90:.3) -- (.1,-1.5) -- (3.9,-1.5);
			\begin{scope}
				\clip (3.9,1.45) -- (.1,1.45) -- (.1,.95) arc(90:-90:.3) -- (.1,-.35) arc(90:-90:.3) -- (.1,-1.5) -- (3.9,-1.5) -- cycle;
				\draw[dashed] (.1,.65) -- (1,1.45) node[below right] {$\Gamma_{2,1}$};
				\draw[dashed] (.1,-.65) -- (1,-1.45) node[above right] {$\Gamma_{1,-1}$};
			\end{scope}
			\fill[fill=lightgray!40] (3.9,.65) -- (.4,.65) arc(0:-90:.3) -- (.1,-.35) arc(90:0:.3) -- (3.9,-.65) -- cycle;
			\draw (3.9,1.45) -- (.1,1.45) -- (.1,.95) arc(90:-90:.3) -- (.1,-.35) arc(90:-90:.3) -- (.1,-1.5) -- (3.9,-1.5);
			\draw[dashed] (3.9,1.45) -- (3.9,-1.5);
			\node () at (-.3,1.7) {(b)};
			\node () at (.15,.65) {fe};
			\node () at (.15,-.625) {fe};
			\node () at (-.35,0) {ze};
			\draw[dotted] (-.35,-.25) -- (-.35,-1.25) -- (0,-1.25);
			\draw[dotted] (-.35,.25) -- (-.35,1.25) -- (0,1.25);
			\draw[dotted] (-.15,0) -- (.05,0);
			\node () at (2,-1.8) {$\{z=-Z\}$};
			\node () at (2,1.7) {$\{z=Z\}$};
		\end{tikzpicture}
	\end{center}
	\caption{(a) The rectangle $[-Z,Z]_z\times [0,\infty)_{h^2}$ and (b) the mwc $X$ that results from blowing up the points $(z_n,0)$, as described in \S\ref{subsec:JWKB}. Here, $\Gamma_{n,\lambda}=\{z= z_n+\lambda h^{2/3}\}$. }
\end{figure}

The Langer--Olver expansion \cref{eq:JWKB_uniform} is the prototype for the representation of solutions in \Cref{thm2}.
That theorem states, in full generality, the existence of a basis of solutions of the desired form, except the coefficients $\beta,\gamma$ are only stated to be polyhomogeneous on $M$. 
Now, \Cref{thm2} is much more general than what Olver proves, but it is neither stronger nor weaker when applied to the JWKB case:
\begin{itemize}
	\item 
	In one way, our theorem is stronger than \cite[Thm.\  7.1]{OlverBook}, because the former gives an actual solution of the ODE while the latter only gives, for each $N\in \bbN$, an  approximation with error $O(h^N)$. This is just a matter of ``asymptotic summation'' of the Langer--Olver expansion, and we do not consider it to be an essential point.
	\item In another, more important way, it is weaker, because $\beta,\gamma$'s smoothness (and polyhomogeneity is even worse) on $M$ does not imply that they can be expanded in powers of $h$ with coefficients which are smooth functions of $z$. That would be smoothness on the \emph{rectangle} $[0,Z]_z\times [0,\infty)_{h^2}$. Smoothness on the rectangle is what is suggested by the Langer--Olver analysis (which only falls short of proving this because it does not show differentiability of the approximation's error in the semiclassical parameter). 
	Blowing up a submanifold enlarges the set of polyhomogeneous functions (e.g. any smooth function on $\bbR^2$ is a smooth function of polar coordinates $r,\theta$), so polyhomogeneity on $M$ is weaker than polyhomogeneity on the rectangle.
\end{itemize} 
 
Polyhomogeneity on $M$, rather than on the rectangle $[0,Z]_z\times [0,\infty)_{h^2}$, is the best that one can expect in general (see \S\ref{subsubsec:coulombic_high_energy}, \S\ref{subsec:parabolic}, \cite[\S12.14]{OlverBook}). But in the case we are currently considering, and in the others for which Olver proves a uniform JWKB expansion, it is desirable to prove that \cref{eq:JWKB_uniform} holds for 
\begin{equation} 
	\beta,\gamma \in C^\infty([0,Z]_z\times [0,\infty)_{h^2}).
\end{equation} 
Indeed, this can be proven by combining the Langer--Olver argument with the the error-analysis in \S\ref{sec:error}. We discuss this in \S\ref{sec:collapsing}. The resultant theorem, \Cref{thm:collapsed} also treats $\kappa=-1,0$, other cases dealt with in \cite{OlverBook}.

Note that the presentation \cref{eq:JWKB_uniform} is \emph{not} of exponential-polyhomogeneous type on $[0,Z]_z\times [0,\infty)_{h^2}$. It is on $M$, in accordance with \Cref{thm1}, but blowing up the corner is necessary. (We might say instead that \cref{eq:JWKB_uniform} is of ``Airy-polyhomogeneous''-type on $[0,Z]_z\times [0,\infty)_{h^2}$.)
It is therefore not true that \Cref{thm1} can be simplified in this case. The sharper \Cref{thm2} can, in the manner discussed above, but not \Cref{thm1}.
The miraculous feature of the Langer--Olver ansatz \cref{eq:JWKB_uniform} that holds in the JWKB case is that it takes all of the delicate $\lambda=z/h^{2/3}$ dependence which a typical polyhomogeneous function on $M$ has and organizes it into just two special functions, an Airy function and its derivative. No such miracle will hold in general, as Olver himself emphasizes in \cite[\S12.14]{OlverBook}.

\subsection{Bibliographic remarks} 
\label{subsec:history}
It is difficult to construct an adequate history of a problem as old as this one, so the following brief remarks will have to do.

We have already discussed the $\kappa=1$ (JWKB) case in 
\S\ref{subsec:JWKB} and will not say more, except to repeat that the definitive treatment of Poincar\'e expandability is due to Langer--Olver and the better part of a century old.  
It is quite satisfactory, except in that it stops at Poincar\'e expandibility. The standard account is Chp.\ 11 of Olver's \emph{Asymptotics and Special Functions} \cite{OlverBook}.
The $\kappa=-1$ case has been brought to a similarly refined level by the same authors. This can also be found in Olver's text, specifically \cite[Chp.\ 12]{OlverBook}. It is what he calls ``Case III.'' (Liouville--Green is Case I, and a simple turning point is Case II in Olver's numbering scheme.) The first work on this case is apparently \cite{Langer35}. Olver's work is \cite{OlverTransition}.

When $\kappa\geq 2$, the potential vanishes degenerately at $z=0$. 
One might call $\{z=0\}$ a \emph{trapped point}, since, in addition to being a zero of the potential, it is a point of equilibrium of the associated classical dynamics, at which the force $-V'$ on a Newtonian particle vanishes. Varied terminology has appeared in the literature. Unlike in the $\kappa\in \{-1,0,1\}$ case, only very incomplete results can be found. Olver describes the situation like so:
\begin{quote}
	\textit{In a region containing a turning point of [integer] multiplicity $\kappa$, uniform asymptotic approximations to the solutions can be constructed in terms of} [Bessel functions of order $1/(\kappa+2)$]... \textit{When $\kappa>1$, however, there is no straightforward extension from asymptotic approximations to asymptotic expansions} \cite[\S4.3]{OlverUnsolved}.  
\end{quote}
Here, ``asymptotic approximation'' refers to a statement such as $u=(1+O(h))u_0$ for explicit $u_0$ (in this case written in terms of Bessel functions), and ``asymptotic expansion'' refers to the production of a full series in powers of $h$. The precise no-go theorem that Olver is alluding to here is presented in full detail in \cite[\S12.14]{OlverBook}, based on \cite{OlverTransition}. 

Another, even clearer, statement of the problem can be found in Olver's discussion of the $\kappa\notin \bbZ$ case:
\begin{quote}
	\textit{It needs to be emphasized, however, that each of [the results in \cite{OlverFrac}] is confined to the first approximations
		to the wanted solutions. Except in the cases $\kappa = 0,\pm 1$ it is not known how to construct asymptotic expansions for the solutions in descending powers of the large parameter that involve only functions of a single variable.}  \cite{OlverFrac}
\end{quote}
Note that Olver's discussion in the quoted section is not confined to $\kappa \in (-2,2)$, as he mentions $\kappa=2$.

For the $\kappa=2$ case, Olver tried a different ansatz using the parabolic cylinder functions in \cite{OlverWell} (this is the reason for the ``functions of a single variable'' caveat in the quote above), though the idea predates his work. The hope was that, by beginning with a more accurate ansatz, the same method that works in the $\kappa\leq 1$ case would work in the $\kappa=2$ case. Though not stated outright, it is implied in \cite[\S11]{OlverWell} that the resultant asymptotic approximation can be promoted to a full expansion as long as uniformity is not demanded as $z\to\infty$. The argument does not appear to have been written down, but we believe that it works. 

For $\kappa\geq 3$, in which the trapping is degenerate, it seems that little is known beyond what Olver already states in \cite{OlverOriginal, OlverUnsolved, OlverBook}. 
In particular, I am unable to locate any asymptotic expansions including lower order terms, though it is hard to rule out the existence of folk theorems known to experts on the subject. There are several reasons why this case is less explored. One is that few natural examples exhibit it, so the motivation is lacking. (Though, at least one natural example is presented in \S\ref{subsec:anharmonic}.) Another is that, without severe restrictions on $\psi$, the Langer--Olver method does not apply --- see \cite[\S12.14, eq. 14.17]{OlverBook}. One could invent new special functions -- the cubic, quadric, etc.\ analogues of parabolic cylinder functions -- and attempt what Olver attempted in \cite{OlverWell}, but this violates one of the principles of asymptotic analysis: the approximating function must be ``less recondite than the solutions of the original equation'' \cite[p.\ 475]{OlverBook}.

For a pedagogical account of the aforementioned developments, we recommend Olver's expository works. 
Olver was not the first mathematician to work on this problem, but his works are some of the most systematic, and the fact that some of them have been incorporated into textbook form makes him the standard reference.
An encyclopedic treatment appears in \cite[Chp.\ 4]{Fedoryuk}. An extensive summary of the research literature, including the citations omitted above, can be found in \cite[\S31.2]{Wasow}. Wasow's account does not cover the last few decades, but it still seems current as far as our particular problem is concerned. 

An idea close to that pursued here is that of ``stretched/matched asymptotic expansions.'' It has been utilized extensively by Wasow \cite{Wasow59, Wasow61, Wasow62, WasowFinal}. 

\subsection{Primer on exponential-polyhomogeneity}
\label{subsec:primer}

In this subsection, we quickly review a few definitions from the geometric singular analysis literature, of which \cite{MelroseCorners, MelroseAPS, GrieserBasics, SherBessel} is a relevant sample. For the precise definition of the notion of manifolds-with-corners, we defer to these references. A key fact is that if $M$ is a mwc, then every boundary hypersurface of $M$ is a mwc. We will only work with mwc with finitely many boundary hypersurfaces.

Let $M$ denote a mwc.
A smooth function $f\in C^\infty(M^\circ)$ is said to be \emph{conormal} of order 0 on $M$ if $f\in L^\infty_{\mathrm{loc}}$ and, given any number $N\in \bbN$ of smooth vector fields $V_1,\dots,V_N$ on $M$ tangent to all of the boundary hypersurfaces of $M$ (a.k.a.\ b-vector fields), 
\begin{equation}
	V_1\cdots V_N f \in L^\infty_{\mathrm{loc}}. 
\end{equation}
The set of such functions is denoted $\calA^{0,0\cdots}(M)$, where there is one `$0$' in the superscript for each boundary hypersurface $\mathrm{f}$ of $M$. Then, for any $\alpha_1,\alpha_2\dots\in \bbR$,
\begin{equation}
	\calA^{\alpha_1,\alpha_2\cdots }(M) = \prod_n \varrho_{\mathrm{f}_n}^{\alpha_n} \calA^{0,0\cdots}(M). 
\end{equation}
If $M$ is compact, then these are Fr\'echet spaces, with the seminorms given by the $L^\infty$-norms of the various b-derivatives of $f$.
For any Fr\'echet space $\calX$, one can define the spaces $\calA^{\alpha_1,\alpha_2\dots}(M;\calX)$ of conormal $\calX$-valued  functions.

A subset $\calE\subseteq \bbC\times \bbN$ is called an \emph{index set} if 
\begin{itemize}
	\item for every $\alpha\in \bbR$, the set $\{(j,k)\in \calE: \Re j <\alpha\}$ is finite, 
	\item $(j,k)\in \calE\Rightarrow (j+1,k)\in \calE$, 
	\item $(j,k+1)\in \calE \Rightarrow (j,k)\in \calE$. 
\end{itemize}
Then, the set $\calA^{\calE}([0,1); \calX)$ of polyhomogeneous $\calX$-valued functions is defined to be the subset of $f\in C^\infty((0,1);\calX)$ such that there exist $f_{j,k}\in \calX$ for the $(j,k)\in \calE)$ such that, for any $\alpha\in \bbR$, 
\begin{equation}
	f(x) - \sum_{(j,k)\in \calE, \Re j\leq \alpha} f_{j,k} x^j (\log x)^k \in \calA^\alpha([0,1);\calX). 
\end{equation}
Then, if $M$ is a mwc, and if $\calE_{\mathrm{f}}$ are index sets, one for each boundary hypersurface $\mathrm{f}$ of $M$, $\calA^{\calE_1,\calE_2,\dots}(M)$ consists of the smooth functions on $M$ such that, for any boundary hypersurface $\mathrm{f}$, and for any point $p\in\mathrm{f}$, $f$ is locally of the form 
\begin{equation}
	f\in \calA^{\calE_{\mathrm{f}}}( [0,\varepsilon)_{\varrho_{\mathrm{f}}}; \calA^{\cdots}(\mathrm{f}) ), 
\end{equation}
where the $\cdots$ is the list of those index sets of the boundary hypersurfaces of $M$ adjacent to $\mathrm{f}$. 

One can define spaces of partially polyhomogeneous functions similarly. The notation that we use below will be defined as needed.

Polyhomogeneity is a precisification of the notion of having full asymptotic expansions in terms of powers and logarithms (the particular combinations that are allowed to appear being specified by the index set), with 
\begin{itemize}
	\item each boundary hypersurface (a.k.a. facet) of our mwc $M$ corresponding to one asymptotic regime, and 
	\item the corners corresponding to ``intermediate'' asymptotic regimes.
\end{itemize}
These expansions are differentiable term-by-term, in every direction, and the results agree with the derivatives of the function. This is useful to have, and polyhomogeneity guarantees it. So, one can think of polyhomogeneity as being a slight weakening/generalization of smoothness, namely smoothness up to the fact that one's ``Taylor series'' (polyhomogeneous expansions) have logarithms and possibly non-natural powers of the boundary-defining-functions. 

Polyhomogeneity at the corners ensures that the asymptotic expansions at the various adjacent facets match up, and consequently one has well-defined \emph{joint} asymptotic expansions, a point stressed in \cite{SherBessel}. For an application to the Bessel functions, see \cite[\S2.4]{SherBessel}.
The coefficients in the (joint) expansion at a face $\mathrm{f}$ (not necessarily a facet; it could be a corner) are polyhomogeneous functions on $\mathrm{f}$, and the expansion of the derivative along $\mathrm{f}$ is the derivative of the expansion. This is part of what it means for the expansion to be differentiable term-by-term, but the latter is stronger because it applies to normal derivatives as well. 

A function $u$ on $M$ is said to be of \emph{exponential-polyhomogeneous type} if it can be written as 
\begin{equation}
u =  \sum_{n=1}^N u_n e^{\theta_n } \label{eq:misc_032}
\end{equation}
for some $N\in \bbN$ and polyhomogeneous $\theta_n,u_n:M\to \bbC$. Say that $u$ is of \emph{nontrivial} exponential-polyhomogeneous type if, in some presentation \cref{eq:misc_032} of $u$ with $N$ minimal, the $u_n$ have nonempty index sets at each boundary hypersurface of $M$ (and similarly at each corner). 
Thus, the notion of (nontrivial) exponential-polyhomogeneous type is one particular precisification of the notion of having full, well-behaved asymptotic expansions in terms of elementary functions, i.e.\ polynomials, logarithms, \emph{and exponentials}.

One interesting aspect of the theorem above, \Cref{thm1}, is the possible presence of logarithmic terms. This stands in contrast to the situation, described in Appendix \S\ref{sec:collapsing}, when $\kappa \in \{-1,0,1\}$ and $\psi$ is sufficiently nice. 
An example, with $\kappa=2$, in which such logarithmic terms appear is presented in Appendix \S\ref{sec:second}.

One advantage of the use of the function spaces above is that we can avoid the use of the theory of special functions entirely. 
Rather than work with parabolic cylinder functions (or some alternative such as Weber functions), as Olver does in \cite{OlverWell} to handle the $\kappa=2$ case, it suffices here to work entirely in terms of elementary functions. After all, it is only in terms of elementary functions that the notion of exponential-polyhomogeneity is phrased. 
Insofar as special functions appear below, they appear as the solutions to some special differential equations --- see \S\ref{sec:examples} --- and all of the properties required of them are proven directly from the differential equations they satisfy. No integral representations are used.

One disadvantage, for instance in the $\kappa=2$ case, is that the possible logarithmic terms, which in the usual treatment are hidden in asymptotics of the parabolic cylinder functions (see Appendix \S\ref{sec:second}), become explicit and  worse organized. This is especially inconvenient when trying to relate the expansions on opposite sides of the transition point. Without care, the result is proving polyhomogeneity statements with unnecessarily large index sets.

\subsection{Analogues in PDE}
\label{subsec:PDE}
Many of the phenomena we study here exist in the PDE setting, but results are less complete. We will haphazardly list a few works that can serve as a jumping off point to the wider literature for the interested reader. As in the previous few subsections, no attempt is made to be comprehensive.

Work by Buchal--Keller \cite{BuchalKeller}, Ludwig \cite{Ludwig}, and Guillemin--Schaeffer \cite{Guillemin} addresses the case of a simple turning point, $\kappa=1$. 
(This is an incomplete sample of the relevant literature.)
The cited works essentially solve the problem for $\kappa=1$, at least if $\psi$ is non-singular, giving full asymptotics in the classically allowed region, albeit only superpolynomial (rather than exponential) decay in the classically forbidden region.
The key idea is the representation of solutions of the equation as oscillatory integrals modeled on the integral formula for the Airy functions. The cited works prove Poincar\'e-type asymptotic expansions, but it is surely the case that e.g.\ \cite[Eq. E]{Guillemin} can be sharpened to include control on all derivatives. 

Whether similar ideas apply in the $\kappa\geq 2$ case remains to be seen, one difficulty being that, in this case, the semiclassical characteristic set is not a smooth submanifold of the semiclassical cotangent bundle. Thus, this falls under the header of semiclassical asymptotics for singular Lagrangians. Some of the relevant theory is developed in \cite{CdV1,CdV2}.

The $\kappa=-1$ case gives a prototype for Lagrangian distributions for Lagrangian submanifolds which hit fiber infinity. This situation has been encountered in microlocal studies of the Klein--Gordon \cite{SussmanKGI} and time-dependent Schr\"odinger equations \cite{HassellNL}.

The $\kappa=2$ case is one of the simplest examples of a semiclassical differential operator whose associated Hamiltonian flow displays normally hyperbolic trapping, the origin of the cotangent bundle over the transition point being the trapped set. Microlocal estimates near trapped sets have been a popular theme in recent years.  We will not attempt to summarize the literature. See \cite{WunschZworski}\cite{HintzVasy}\cite{ZworskiDyatlov}\cite{HintzTrapping} for a sample; for a non-microlocal perspective, see also \cite{Schlag}. Parabolic cylinder functions/Weber functions, or various microlocal tools, have been used to construct quasimodes near trapped bicharacteristics. In the case of unstable trapping, there are global obstructions to improving quasimodes beyond the order in \cite{Eswarathasan}. In the case of stable trapping studied by Ralston \cite{Ralston71, Ralston76}, the quasimodes are exceptionally good.

The utility of a transitional regime in the analysis of the Helmholtz--Schr\"odinger equation at high energy near cone points has been noticed by Hintz \cite{HintzCone1, HintzCone2}, who has established a number of powerful estimates for the analysis of this regime. See \S\ref{subsec:Coulomb}, \S\ref{subsubsec:coulombic_high_energy} for an ODE example, namely the Schr\"odinger equation at high positive/negative energies in the presence of a Coulomb potential. We will demonstrate the necessity of the transitional regime in this case --- see \Cref{fig:cone_test}.

Hydrogen ionization, i.e.\ low-energy scattering theory for an attractive Coulomb-like potential, has been studied in \cite{SussmanACL}, which established joint $E\to 0^+,r\to\infty$ asymptotics  without the symmetry assumptions that would allow a reduction to the ODE case. Our discussion in \S\ref{subsubsec:Rydberg} of the ODE case goes further (albeit only for the exact Coulomb potential) by treating not just hydrogen ionization, but the other three cases that differ by the sign of $E,\mathsf{Z}$. So, the Rydberg series, low-energy scattering off a repulsive Coulomb potential, and low-energy elliptic estimates for a repulsive Coulomb potential are all handled. This last case can be dealt with using the tools in \cite{SussmanACL}, but the Rydberg series and low-energy scattering off a repulsive Coulomb potential are more difficult because of the appearance of a simple turning point ($\kappa=1$) in the classical dynamics. Understanding the Rydberg series at large angular momentum poses an additional challenge that we do not deal with here, namely that of coalescing turning points. 

\subsection{Concluding remarks}
\label{subsec:conc}

This work began with the goal of providing a proof of the main result of \cite{SherBessel} using only the structural features of Bessel's ODE. The results here are not sufficient for this purpose.
The missing piece is an analysis of the $\kappa=-2$ case, in which the ODE possesses a regular singularity. However, the clause of Sher's result stating the exponential-polyhomogeneity of the Bessel functions in a neighborhood of the locus of their Airy function asymptotics is a special case of \Cref{thm:collapsed}, with $\kappa=1$. See \Cref{rem:extension} for our thoughts on the $\kappa\leq -2$ case.

If the coefficients $W,\psi$ appearing in the definition of $P$, \cref{eq:1}, depend on some parameters $p \in \bbR^d$ (or more generally $p$ valued in some manifold-with-corners), then e.g.\ smoothness in $p$ implies smoothness of the solutions constructed in \Cref{thm2} in $p$ in a suitable topology. Similar results hold for other $L^\infty$-based notions of regularity. The family $\calQ = \operatorname{ker} N(P)$ of quasimodes appearing in the statement of \Cref{thm2} depends smoothly on the coefficients of $P$ with respect to some function space topologies. It is therefore possible to articulate a theorem to the effect that, given a smooth section's worth of chosen quasimodes, the function $u$ constructed in the theorem is also smoothly varying in some suitable sense.

We end this introduction with a few remarks regarding potential extensions: 
\begin{itemize}
	\item For simplicity, we have restricted attention to integral $\kappa$, but we believe that the arguments below work for all real $\kappa >-2$, except the specific index sets appearing may differ.
	\item One can handle $\psi$ polyhomogeneous on $M$ with appropriate index sets at the various faces. If $\psi$ has only partial regularity (partially polyhomogeneous, with a conormal error, or something similar), then partial asymptotic expansions can still be constructed, with the error controlled in a corresponding conormal space. In the argument, the final step, in which the remaining error is solved away via the method of variation of parameters, becomes more delicate, as the forcing is no longer $O(h^\infty)$ but instead only a decaying conormal function. However, it still seems possible to make do.
	\item The previous point also applies if $W$ is polyhomogeneous or partially polyhomogeneous at $0$, except now the Langer transformation is no longer a diffeomorphism. This issue is not critical, as our use of the Langer transformation is only to simplify the algebra. Besides this one complication, the argument works as expected to produce the desired expansions.
	\item One particularly natural extension is to allow $\psi$ to have a pole of the form $\sim 1/h$, in which case the term $h^2 \psi$ in $P$ has an $O(h)$ contribution. Unfortunately, this can dominate the potential at $\mathrm{fe}$. For instance, at $\mathrm{fe}$, the potential is essentially $z^\kappa \sim h^{2\kappa/(\kappa+2)}$,
	and the right-hand side is smaller than $h$ if $h<1$ and $\kappa\geq 3$. 
	
	For $\kappa\in \{-1,0,1,2\}$ at least, where such issues do not occur, this problem has been considered by Langer and others, and there are standard tools which apply. Thus, one can expect results in the vein of the theorems in this paper. 
	\item Complex $h$ can be considered essentially without modification if $\psi$ has appropriate analyticity. In fact, considering $\Im h=p$ as a parameter, this is a special case of smooth parameter dependence. It can be shown that the constructed solutions depend analytically on $h$ in some neighborhood $U\subset \bbC_h$ of $(0,\infty)_h\subset \bbC_h$. Of course, $U$ cannot contain the origin, because the Liouville--Green ansatzes have essential singularities there.

	If $P$ has appropriate analytic structure in $z$, then one can allow complex $z$ as well, but this requires a bit more care, and one has to restrict attention to $z$ in a sector depending on the order of vanishing of the potential. We are not experts on complex WKB, so we will refrain from speculating more.
	\item One common occurrence when considering parameters is that the transition point disappears or changes character when a parameter is varied. Consider for instance $V(r;\theta) = 1-r e^{i\theta}$, where $r$ is the independent variable and $\theta\in \bbR$ is a parameter. The semiclassical ODE $Pu=0$ for $P= - h^2 \partial_r^2 + V$ has a turning point at $r=1$ if $\theta\in 2 \pi \bbZ$ but not otherwise. A more interesting example is that of coalescing turning points, where the potential $V(z;a)$ depends on a parameter $a\geq 0$ and has $N\in \bbN^{\geq 2}$ turning points for each $a>0$, with the turning points coalescing in some way as $a\to 0^+$. The $N=2$ case is the setting of \cite{OlverWell}. 
	
	These sorts of problems cannot be readily addressed with the techniques below. 
\end{itemize}

We mention finally the recent thesis of Sobotta \cite{Sobotta}, who, in ongoing work with Grieser, has constructed $O(h^\infty)$-quasimodes for a wide class of semiclassical ODEs, including many of those considered here. While of a similar spirit to ours, their methods are significantly more general and not restricted to second-order ODEs. They include the use of iterated blowups organized via Newton polygons. When applied to the second-order case, only a single blowup is required, as is true here. It is expected that the constructed quasimodes can be upgraded to full solutions, but this has yet to be done. 

\section{Examples}
\label{sec:examples}
This section consists of a number of examples which can be read in any order or skipped entirely. Their purpose is to aid the reader in understanding the statements of \Cref{thm1}, \Cref{thm2}. We have elected to focus on cases where the solutions can be written in terms of special functions and  the term $\psi$ in \cref{eq:1} is of the form in \cref{eq:psi_init}.
\begin{itemize}
	\item \S\ref{subsec:model} discusses the ``homogeneous'' (a.k.a.\ ``model'') problem, when the ODE can be solved explicitly in terms of Bessel functions. Like that discussed in \S\ref{subsec:JWKB}, this example is important to understand, even though it is basic.
	\item \S\ref{subsec:Coulomb}, discusses the hydrogen atom, or more generally the Schr\"odinger equation in the presence of a Coulomb potential, which is allowed to be either attractive or repulsive. 
	This ODE too can be solved in terms of special functions, 
	specifically \emph{Whittaker functions} (which can in turn be written in terms of confluent hypergeometric functions). 
	Applied to this special case, our theorems give new results about the Whittaker functions; in particular, we will state the existence of an asymptotic expansion in an asymptotic regime where previously none seemed known. We will see that the reason why we are able to do this is that \emph{polyhomogeneity on $M$ is the notion needed to go beyond leading order} in connecting the asymptotics in $\{r>\varepsilon\}$ with the asymptotics in $\{r|E|^{1/2}<C\}$, where $C\gg 1 \gg \varepsilon>0$.
	
	Due to the length of this example, many of the details are deferred to the appendix \S\ref{sec:hydrogen_details}.
	\item In \S\ref{subsec:parabolic}, we discuss another example where our theorems give new results about special functions, in this case the parabolic cylinder functions: the simple harmonic oscillator. This example involves the $\kappa=0,1$ cases when initial data is provided at small $r$.  
	
	The discussion will continue in \S\ref{sec:second}, where we discuss providing initial data at large $r$. This will involve the $\kappa=2$ case.
	One feature of this example is that we can work out, algebraically, the appearance of log terms in asymptotic expansions at the edge $\mathrm{fe}\subset M$. This explains our insistence on speaking of ``polyhomogeneity'' rather than smoothness/one-step polyhomogeneity. The former permits log terms which, as this example demonstrates, can  actually appear. (Log terms will also appear in \S\ref{subsubsec:coulombic_high_energy}, but we will only probe them numerically.)
	\item \S\ref{subsec:anharmonic} discusses the small and large-anharmonicity limit of the anharmonic oscillator, an influential topic in the physics literature -- e.g.\ the famous paper \cite{BenderWu} -- because its relation to the $\phi^4_d$ quantum field theory. (The anharmonic oscillator can be thought of as the $\phi^4_1$ theory.) We will see the $\kappa=2,4$ cases arise in this example, one in the limit of small anharmonicity and one in the limit of large anharmonicity.
	\item In \S\ref{subsec:Bessel}, we discuss an example which motivated this work, the Bessel ODE at large order \emph{and} large argument. (Bessel functions also appear in \S\ref{subsec:model}, but only for one individual order that depends on $\kappa,\alpha$.) We discuss the relation between our work and the recent work of Sher \cite{SherBessel}, which proved (by completely different singular-geometric, partially special-function-theoretic means) the polyhomogeneity of the Bessel functions $J_\nu(z),Y_\nu(z)$, considered as functions of both $z,\nu$. 
	\item Finally, in \S\ref{subsec:Regge_Wheeler}, we discuss an application to black hole physics, the Regge--Wheeler equation. The solutions cannot be written in terms of special functions, as far as we are aware.
\end{itemize}

In this section, we will describe quasihomogeneous blowups without specifying a choice of smooth structure for the new mwc at the front face. The ambiguity will be whether one function or a power thereof will be a boundary-defining-function for the front face. 
Since such a choice does not affect the sets of polyhomogeneous functions (and therefore exponential-polyhomogeneous functions), no harm will come from this imprecision. Similarly, when specifying local coordinate charts, the listed local coordinates may be a power of a local boundary-defining-function.

\subsection{The model problem} \label{subsec:model}
The simplest example is when $W=1$ and $\psi(z)=(\alpha^2-1/4)/z^2$ for some $\alpha\in \bbC$. (We write $\psi$ in terms of $\alpha$ rather than what we called $\nu$ in \cref{eq:psi_init} because we will use `$\nu$' below for something else.) Then, the operator $P$ defined by \cref{eq:1} is
\begin{equation}
	P=- h^2 \frac{\partial^2 }{\partial z^2} + \varsigma z^\kappa  + \frac{h^2}{z^2} \Big(\alpha^2-\frac{1}{4} \Big) =0.
	\label{eq:misc_o12}
\end{equation}
The corresponding ODE is $Pu=0$. In physicists' parlance, this is the zero-energy, time-independent Schr\"odinger equation on the half-line in the presence of a potential 
\begin{equation} 
	V_{\mathrm{eff}}=\varsigma z^\kappa + h^2(\alpha^2-1/4) /z^2.
\end{equation} 
One way this arises is from separating variables in the zero-energy, time-independent Schr\"odinger equation on $\bbR^3$ in the presence of a central $\varsigma r^\kappa$ potential. Then, physicists usually write `$r$' in place of `$z$' and `$\hbar$' in place of $h$. The semiclassical parameter is the (reduced) Planck constant. The term $(\alpha^2-1/4)/z^2=\ell(\ell+1)/r^2$ is the angular momentum term, i.e.\ the effective centrifugal potential, which can be thought of as repelling particles with nonzero angular momentum away from the origin; $\ell\in \bbN$ is the ``azimuthal quantum number,'' related to $\alpha$ by $\ell = -2^{-1}\pm \alpha$. The reason why $(\alpha^2-1/4)/z^2=\ell(\ell+1)/r^2$ has an $h^2$ in front of it in \cref{eq:misc_o12} is because the angular Laplacian $\triangle_{\bbS^2}$ has an $h^2$ in front of it in the original 3d Schr\"odinger equation, so after separation of variables the resultant term in the effective potential $V_{\mathrm{eff}}$ continues to carry that $h^2$.

The ODE $Pu=0$ can be solved in terms of Bessel functions (as observed and utilized by Langer in a number of works):
\begin{equation}
	\ker P(h) = \bigg\{ z^{1/2} I \bigg( \frac{2 z^{(\kappa+2)/2}}{h(\kappa+2)} \bigg) : I(t)\text{ a solution to }t^2 \frac{\mathrm{d}^2 I}{\mathrm{d} t^2} + t \frac{\mathrm{d} I}{\mathrm{d} t} - (\varsigma t^2 + \nu^2) I(t) = 0 \bigg\},
\end{equation}
where now `$\nu$' denotes $\nu=2\alpha/(\kappa+2)$. The ODE satisfied by $I(t)$ is Bessel's ODE of order $\nu$, so $I$ is a (modified, if $\varsigma>0$, and unmodified otherwise) Bessel function of order $\nu$.
Two special cases with $\alpha=1/2$ are worth noting: 
\begin{itemize}
	\item (Liouville--Green) if $\kappa=0$, then $\nu=1/2$, and the Bessel functions of order $\nu=1/2$ are just trigonometric functions, $I(t) \in \operatorname{span}_\bbC\{t^{-1/2} \sin(t), t^{-1/2} \cos(t)\}$. Then, 
	\begin{equation}
		\frac{z^{1/2}}{h^{1/(\kappa+2)}} I \bigg( \frac{2 z^{(\kappa+2)/2}}{h(\kappa+2)} \bigg)\in \operatorname{span}_\bbC\big\{ e^{ \sqrt{\varsigma} z/h}, e^{- \sqrt{\varsigma} z/h} \big\}.
	\end{equation}
	Indeed, in this case, $Pu=0$ is just $h^2 u''(z)=\varsigma u(z) $, whose solutions are linear combinations of $\exp(\pm z/h)$ in the $\varsigma>0$ (classically allowed) case and $\exp(\pm i z/h)$ in the $\varsigma<0$ (classically forbidden) case.
	\item (Jeffreys--Wentzel--Brillouin--Kramers) if $\kappa=1$, then $\nu=1/3$, and the Bessel functions of order $\nu=1/3$ can be written in terms of Airy functions: 
	\begin{equation} 
		I(t) \in \operatorname{span}_\bbC \{t^{-1/3} \operatorname{Ai}( \varsigma (3z/2)^{2/3} ),t^{-1/3} \operatorname{Bi}( \varsigma (3z/2)^{2/3} ) \},
	\end{equation} 
	\begin{equation}
		\frac{z^{1/2}}{h^{1/(\kappa+2)}} I \bigg( \frac{2 z^{(\kappa+2)/2}}{h(\kappa+2)} \bigg)\in \operatorname{span}_\bbC\Big\{ \operatorname{Ai}\Big( \frac{\varsigma z}{h^{2/3}}\Big), \operatorname{Bi} \Big(\frac{\varsigma z}{h^{2/3}}\Big)\Big\}.
	\end{equation}
	Indeed, in this case, $Pu=0$ is just $h^2 u''(z) = \varsigma z u(z)$, which is Airy's ODE.

	\begin{figure}[h]
		\includegraphics[scale=.525]{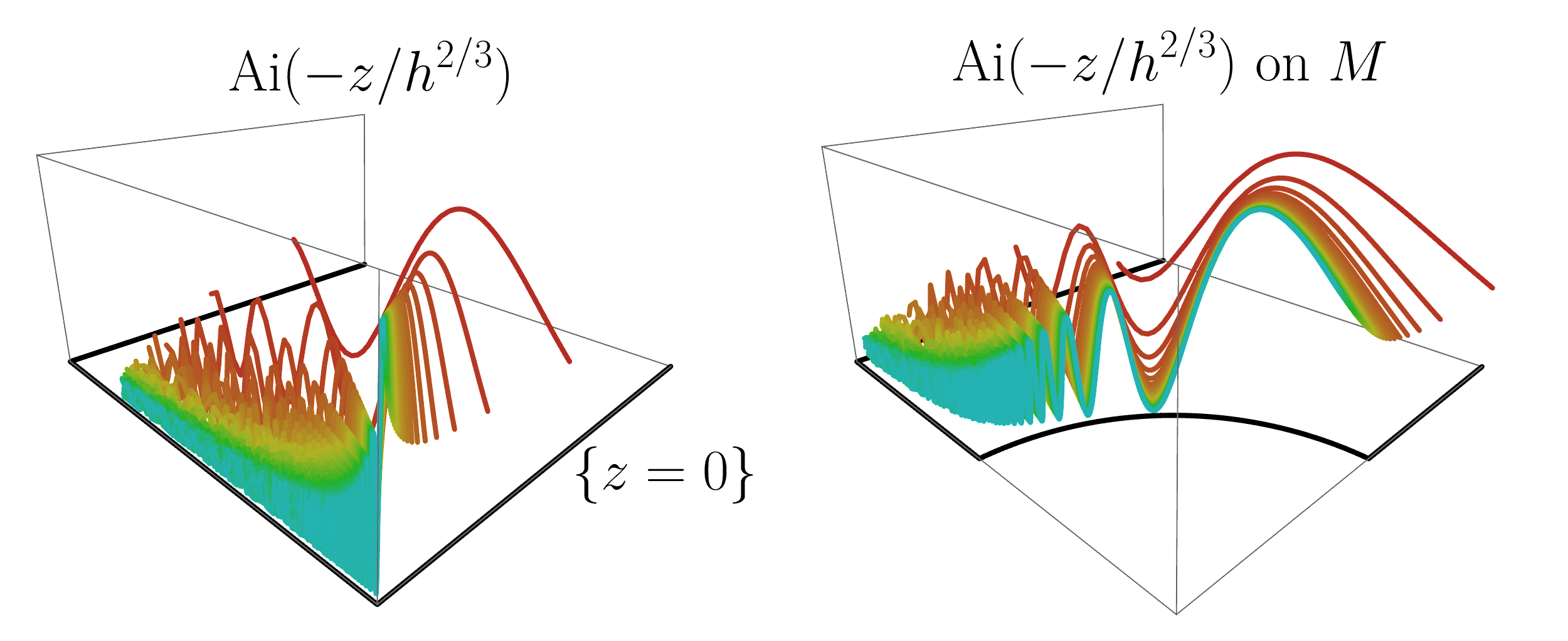}
		\caption{The Airy function $\operatorname{Ai}(-z/h^{2/3})$, plotted for several values of $h$, vs. $z$ (\textit{left}) and over $M$ (\textit{right}). Observe how the blowup of the corner $\{z=0,h=0\}$ helps to resolve the behavior of the function there. The curves for different values of $h$ are plotted with different colors.}
	\end{figure}
\end{itemize}

Recalling the rescaled variable $\lambda=z/h^{2/(\kappa+2)}$ defined above, 
\begin{equation}
	\frac{z^{1/2}}{h^{1/(\kappa+2)}} I \bigg( \frac{2 z^{(\kappa+2)/2}}{h(\kappa+2)} \bigg) = \lambda^{1/2} I\Big( \frac{2 \lambda^{(\kappa+2)/2}}{\kappa+2}\bigg)
\end{equation}
is a function of $\lambda$ alone. (This will not usually be the case and reflects the homogeneities of the various terms in \cref{eq:misc_o12}. Were $W\neq 1$, or were $\psi$ more complicated, then the elements of $\ker P(h)$ would probably not just be dilations of elements of $\ker P(1)$.) Now we recall two facts about Bessel functions:
\begin{itemize}
	\item (small argument polyhomogeneity) $I(t)$ has an asymptotic expansion as $t\to 0$ in powers of $t$ and possibly $\log (t)$ \cite[ \href{http://dlmf.nist.gov/10.8}{\S10.8}]{NIST}. (These are often given as the definitions of the various Bessel functions, because they actually converge.)
	\item (large argument exponential-polyhomogeneity) $I(t)$ has an asymptotic expansion as $t\to\infty$ in powers of $1/t$, with an overall exponential factor out front. These expansions are often called ``Hankel's expansions'' \cite[\href{http://dlmf.nist.gov/10.17.i}{\S10.17(i)}]{NIST}\cite[\href{http://dlmf.nist.gov/10.40.i}{\S10.40(i)}]{NIST}. The Hankel expansions can be proven in many ways. One traditional method is extracting them from integral representations of the Bessel functions via the method of stationary phase.
\end{itemize}
The reference \cite{NIST} describes these expansions as asymptotic expansions in the sense of Poincar\'e, but it also gives the expansions of the derivatives of the Bessel functions, and the latter result from formally differentiating the former. Using Bessel's ODE, the analogous statement holds for all higher derivatives as well. So the expansions in \cite{NIST} are actually Taylor series in $\rho=1/\lambda^{(\kappa+2)/2}$ around $\rho=0$, not just Poincar\'e.

So, the solutions of $Pu=0$ (or at least those that just depend on $\lambda$; we could always multiply by a function $c(h)$ of $h$ alone) are of exponential-polyhomogeneous type on 
\begin{equation}
	M_1=[0,\infty]_\lambda \times [0,\infty)_{h^{2/(\kappa+2)}}.
\end{equation}
This implies exponential-polyhomogeneity on the mwc $M$ in \Cref{thm1}, because any polyhomogeneous function on $M_1$ is also polyhomogeneous on $M$. The reason  is that $M$ arises as (more precisely, is equivalent as a compactification to) a sub-mwc of the polar blowup
\begin{equation}
	M_2 = [M_1; \{\lambda=\infty,h=0\};\kappa+2],
\end{equation}
where the $\kappa+2$ just means that, after the blowup, the smooth structure at the front face of the blowup is modified so that $h^2$ becomes a boundary-defining-function of its interior. Let 
\begin{equation}
	M' = \mathrm{cl}_{M_2} \{\lambda h^{2/(\kappa+2)} \leq Z\} \subseteq M_2.
\end{equation}
The map $(0,Z)_z\times (0,\infty)_h \ni (z,h) \mapsto (z/h^{2/(\kappa+2)},h )$
extends to a diffeomorphism $M\to M'$. This diffeomorphism identifies $\{\lambda=0\} \subset M'$ with $\mathrm{be}$, $\mathrm{cl}_{M'}\{\lambda<\infty,h=0\}= \mathrm{cl}_{M_2}\{\lambda<\infty,h=0\}$ with $\mathrm{fe}$, the portion within $M'$ of the front face of the blowup with $\mathrm{ze}$, and the curve 
\begin{equation} 
	\mathrm{cl}_{M_2} \{\lambda h^{2/(\kappa+2)} = Z\}
\end{equation} 
with $\mathrm{ie}$. 
See \Cref{fig_alt}.

\begin{figure}[h]
	\begin{tikzpicture}
		\fill[gray!5] (5,-1.5) -- (0,-1.5) -- (0,0)  arc(-90:0:1.5) -- (5,1.5) -- cycle;
		\draw[dashed] (5,-1.5) -- (5,1.5);
		\node[white] (ff) at (.75,.75) {ff};
		\node (zfp) at (-.35,-.75) {$\mathrm{fe}$};
		\node (pf) at (2.5,-1.8) {$\mathrm{be}$};
		\node[gray] (ie) at (3.25,.3) {$\mathrm{ie}$};
		\node (mf) at (.7,.5) {$\mathrm{ze}$};
		\filldraw[dashed, fill=darkgray!20] (0,0) arc(-90:-22:1.5)  -- (5,-1) -- (5,-1.5) -- (0,-1.5) --  cycle;
		\draw (5,-1.5) -- (0,-1.5) -- (0,0)  arc(-90:0:1.5)  -- (5,1.5);
		\node[darkgray] at (2.5,-.4) {$M'\cong M$};
		\draw[->, darkred] (.1,-.1) -- (.1,-.5) node[right] {$\lambda^{-(\kappa+2)}$};
		\draw[->, darkred] (.1,-1.4) -- (.7,-1.4) node[above right] {$h^{2/(\kappa+2)}$};
		\draw[->, darkred] (.1,-1.4) -- (.1,-1.1) node[right] {$\lambda$};
		\draw[->, darkred] (.1,-.1) to[out=0, in=-150] (1,.23) node[below] {$z$};
		\draw[->, darkred] (1.6,1.4) -- (2.2,1.4) node[below] {$h^2$};
		\draw[->, darkred] (4.9,-1.4) -- (4.2,-1.4) node[above] {$h_0-h$};
		\draw[->, darkred] (4.9,-1.4) -- (4.9,-.4) node[left] {$\lambda$};
		\draw[->, darkred] (1.6,1.4) to[out=-90, in=50] (1.195,.41) node[right] {$1/z$};
		\draw[->, darkred] (4.9,1.4) -- (4.2,1.4) node[below] {$h_0-h$};
		\draw[->, darkred] (4.9,1.4) -- (4.9,.4) node[left] {$\lambda^{-1}$};
	\end{tikzpicture}
	\caption{The mwc $M_2$, with the subset $M\cong M'\subset M_2$ labeled, with the edges labeled by the corresponding edges in $M$. A smooth atlas is depicted.}
	\label{fig_alt}
\end{figure}

So, in this example, \Cref{thm1} is true and just about the exponential-polyhomogeneity of the Bessel functions. The necessity of the transitional regime fe is made manifest by the dependence of the elements of $\ker P$ on $\lambda=z /h^{2/(\kappa+2)}$.
\begin{remark}
	\label{rem:exception}
	An exception to this last sentence is when $\nu$ is a half of an odd integer, because then the Bessel functions are a linear combinations of exponentials, and so $\smash{I(2\lambda^{(\kappa+2)/2}/(\kappa+2))}$ is of exponential-polyhomogeneous type on $[0,Z]_z\times [0,\infty)_h$ without needing to do a blowup. For instance, the Liouville--Green example $\exp(\pm i z/h)$ is of exponential-polyhomogeneous type without needing to do a blowup, because $z/h$ is polyhomogeneous without needing to do a blowup. In contrast, $J_0(\lambda)$ is only of nontrivial exponential-polyhomogeneous type \emph{after} doing the blowup. The latter case is generic, but the exceptional case includes many important examples.
\end{remark}

\subsection{The hydrogen atom}
\label{subsec:Coulomb}

Consider the (time-independent) Schr\"odinger equation for the hydrogen atom, after separation of variables:
\begin{equation}
	- \frac{\partial^2 u}{\partial r^2} +\Big( - \frac{\mathsf{Z}}{r} + \frac{\ell(\ell+1) }{r^2}  \Big) u = E u
	\label{eq:Whittaker}
\end{equation}
where $u$ is the ``wavefunction,'' $r$ is the Euclidean radial coordinate (distance from the point-like nucleus), $E\in \bbR^\times$ is the energy, $\mathsf{Z}\in \bbR^\times$ is the atomic number (we allow $\mathsf{Z}<0$, in which case this is really modeling Coulomb repulsion) and $\ell\in \bbN$ is the azimuthal mode number. This is a form of Whittaker's ODE, and we can immediately write down the solution space:
\begin{equation}
	\operatorname{span}\{  \operatorname{WhittM}_{-i\mathsf{Z}/2\sqrt{E},2^{-1}+\ell}(2i \sqrt{E} r) , \operatorname{WhittW}_{-i\mathsf{Z}/2\sqrt{E},2^{-1}+\ell}(2i \sqrt{E} r) \},
	\label{eq:misc_ak3}
\end{equation}
(fixing one branch of $\sqrt{E}$)
where $\operatorname{WhittW}_{\kappa,\mu},\operatorname{WhittM}_{\kappa,\mu}$ are the Whittaker functions \cite{NIST} (which can be written in terms of the confluent hypergeometric functions).

We are interested in all possible asymptotic regimes, $r\to 0$, $r\to \infty$, $E\to 0^\pm$ (low energy), $E \to \pm \infty$ (high energy), including regimes where these vary jointly.\footnote{The angular momentum $\ell$ will be taken as fixed. One could also vary $\mathsf{Z}$, but this is of less physical relevance, and anyways, a change in $|\mathsf{Z}|$ can be absorbed into a rescaling of $E,r$.} It turns out that the existent literature does not discuss every possible way of taking $r,E$ to their limiting values. (See \S\ref{subsubsec:coulombic_high_energy} for an example.) Recently, Black--Toprak--Vergara--Zou \cite{Toprak} discussed the $\ell=0$, $\mathsf{Z}<0$ (i.e.\ repulsive) case. Their interest was proving dispersive estimates for the time-dependent Schr\"odinger equation in the presence of the Coulomb potential. The literature did not suffice for their purposes, which required understanding all possible asymptotic regimes. Using special-function-theoretic means, they established some symbolic estimates on the Whittaker-M function and some partial polyhomogeneous expansions. They did not establish full expansions.

It turns out that, using our methods, one can analyze \cref{eq:Whittaker} on the compactification $X_1 \hookleftarrow \bbR^\times_E\times \bbR^+_r$ that comes from taking 
\begin{equation} 
	X_0=([-\infty,0]\sqcup [0,\infty])_E\times [0,\infty]_r
\end{equation} 
and performing a polar blowup of the low-energy, large-radii corners $\{E=0,r=\infty\}$. That is, $X_1 = [X_0; \{ E=0,r=\infty \}]$.
The conclusion is that the solutions of \cref{eq:Whittaker} with exponential-polyhomogeneous (say, $E$-independent) initial data at some fixed value of $r$ are polyhomogeneous on the compactification $X\hookleftarrow \bbR^\times_E\times \bbR^+_r$ that results from performing a parabolic blowup of the $\{|E|=\infty,r=0\}$ corners of $X_1$, so as to resolve the ratio $r|E|^{1/2}$, and (either before or afterwards) performing a particular quasihomogeneous blowup of one particular remaining point in the interior of the front face of the blowup $X_0\rightsquigarrow X_1$: 
\begin{equation}
	X=[X_0 ;\{|E|=\infty,r=0\},  (\partial X_1)\cap \mathrm{cl}_{X_1}\{-rE=\mathsf{Z}\} ]_{\mathrm{quasihomogeneous}}
	\label{eq:misc_033}
\end{equation}
See \Cref{fig:X}. We will discuss this last blowup in more detail below.

\begin{figure}[h!]
	\begin{tikzpicture} 
		\draw[dashed] (-5.2,-2.2) rectangle (5.2,1.7);
		\filldraw[fill=lightgray!20] (.2,-2) -- (.2,0) to[out=0,in=270] (1.7,1.5) -- (5,1.5) -- (5,-.5) to[out=180,in=90] (3.5,-2) -- cycle;
		\filldraw[fill=lightgray!20] (-.2,-2) -- (-.2,0) to[out=180,in=270] (-1.7,1.5) -- (-5,1.5) -- (-5,-.5) to[out=0,in=90] (-3.5,-2) -- cycle;
		\begin{scope}
			\clip (-.2,-2) -- (-.2,0) to[out=180,in=270] (-1.7,1.5) -- (-5,1.5) -- (-5,-.5) to[out=0,in=90] (-3.5,-2) -- cycle;
			\filldraw[fill=white] (-1.3,.6) circle (.4);
		\end{scope}
		\fill[white] (-1.3,.6) circle (.39);
		\node () at (0,1.2) {$X$};
		\draw[darkred, ->] (3.4,-1.9) --(2.7,-1.9) node[above] {$E^{-1}$};
		\draw[darkred, ->] (-3.4,-1.9) --(-2.7,-1.9) node[above] {$-E^{-1}$};
		\draw[darkred, ->] (3.4,-1.9) to[out=90, in=231] (3.7,-1.05) node[left] {$rE^{1/2}$};
		\draw[darkred, ->] (-3.4,-1.9) to[out=90, in=309] (-3.7,-1.05) node[right] {$r|E|^{1/2}$};
		\draw[darkred, ->] (-.3,-1.9) --(-1,-1.9) node[above] {$-E$};
		\draw[darkred, ->] (-.3,-1.9) --(-.3,-1.2) node[left] {$r$};
		\draw[darkred, ->] (.3,-1.9) --(1,-1.9) node[above] {$E$};
		\draw[darkred, ->] (.3,-1.9) --(.3,-1.2) node[right] {$r$};
		\draw[darkred, ->] (-4.9,-.4) --(-4.9,.3) node[right] {$r$};
		\draw[darkred, ->] (-4.9,-.4) to[out=0,in=155] (-4.2,-.6) node[above] {$\;\;\;\;\;\;\;\;\quad(r|E|^{1/2})^{-1}$};
		\draw[darkred, ->] (4.9,-.4) --(4.9,.3) node[left] {$r$};
		\draw[darkred, ->] (4.9,-.4) to[out=180,in=25] (4.2,-.6) node[above] {$(rE^{1/2})^{-1}\quad$};
		\draw[darkred, ->] (4.9,1.4) --(4.9,.7) node[left] {$r^{-1}$};
		\draw[darkred, ->] (4.9,1.4) --(4.2,1.4) node[below] {$E^{-1}$};
		\draw[darkred, ->] (1.8,1.4) --(2.5,1.4) node[below] {$E$};
		\draw[darkred, ->] (1.8,1.4) to[out=-90, in=60] (1.6,.7) node[right] {$(rE)^{-1}$};
		\draw[darkred, ->] (-4.9,1.4) --(-4.2,1.4) node[below] {$-E^{-1}$};
		\draw[darkred, ->] (-4.9,1.4) --(-4.9,.7) node[right] {$r^{-1}$};
		\draw[darkred, ->] (.3,-.1) --(.3,-.8) node[right] {$r^{-1}$};
		\draw[darkred, ->] (.3,-.1) to[out=0, in=210] (1,.1) node[right] {$rE$};
		\draw[darkred, ->] (-1.7,.9) to[out=120,in=270] (-1.8,1.3) node[left] {$\varrho$};
		\draw[darkred, ->] (-1.7,.9) to[out=210,in=90] (-1.825,.5) node[left] {$s$};
		\draw[darkred, ->] (-1.05,.11) to[out=-25,in=175] (-.5,-.1) node[below] {$\varrho$};
		\draw[darkred, ->] (-1.05,.11) to[out=200,in=-20] (-1.5,.11) node[below] {$s^{-1}$};
	\end{tikzpicture}
	\caption{The (disconnected) manifold-with-corners $X$ appearing in \cref{eq:misc_033} in the $\mathsf{Z}>0$ (attractive) case. If $\mathsf{Z}<0$ (the repulsive case), then instead the notch (the front face of the final blowup) is on the other component, so the figure should be reflected across the vertical axis (except for the signs on $\pm E$). Above, $\varrho = |rE+\mathsf{Z}|$ and $s=|E|/|rE+\mathsf{Z}|^3$. Due to lack of space, not all coordinate charts are shown. Away from the notch, the two components are identical, except for the sign of $E$, so the coordinate charts not depicted on one component can be read off the other.}
	\label{fig:X}
\end{figure}

The reason why our results apply is that we can cover $X_1$ by copies of the rectangle $M_0=[0,Z]_z\times [0,\infty)_{h^2}$ or $M$ such that, on each copy, the equation \cref{eq:Whittaker} has the form $Pu$ for $P$ as in \cref{eq:1}, for $\kappa \in \{-4,-1,0,1\}$, after a suitable rescaling and change-of-variables.
As already mentioned in \Cref{rem:extension}, we are not discussing $\kappa=-4$ here. This is an irregular singularity, and irregular singularities are actually (unlike the difficult $\kappa=-2$ case) easier to handle than the cases we consider here; indeed, the Coulomb wavefunctions do not have delicate compound asymptotics at the corners of $X$ where $\kappa=-4$ is the relevant case \cite{OlverBook, NIST, SussmanACL, Toprak}. So, we will focus on the other cases.\footnote{To see that $\kappa=-4$ is the correct description, let $z=1/r$. Then, $\partial_r^2 = (z^2\partial_z)^2 = z^4 \partial_z^2 +2 z^3\partial_z$, $\mathsf{Z}/r=\mathsf{Z}z$, $\ell(\ell+1)/r^2 = \ell(\ell+1)z^2$. After rewriting the ODE according to \Cref{rem:reduction}, $z^\kappa W(z)=\pm 1/ z^4$. This is the term coming from the energy in \cref{eq:Whittaker}. It is the most important term when $\min\{r,Er\}\gg 1$ because $-\mathsf{Z}/r+\ell(\ell+1)/r^2$ is decaying as $r\to\infty$.}

We have depicted the locations of the various copies of $M_0$ in \Cref{fig:covering} (slightly shrunken down, for visibility's sake; they can be enlarged to cover the space). 
\begin{itemize}
	\item The high-energy, low-radii regime will be discussed first, in \S\ref{subsubsec:coulombic_high_energy}. This is the regime covered by the copy of $M_0$ labeled ``$M_0^{(\mathrm{I})}$'' in \Cref{fig:covering}.  
	Here, $\kappa=0$ is the relevant case.
	\item 
	The low-energy regime will then be discussed in \S\ref{subsubsec:Rydberg}. This includes the subsets
	\begin{equation}
		M^{(\mathrm{II})},M_0^{(\mathrm{III})},M_0^{(\mathrm{IV})}\subset X_1
		\label{eq:misc_036}
	\end{equation}
	in \Cref{fig:covering}. 
	Here, we have one or two cases to consider, depending on whether $\operatorname{sign}(E)=\operatorname{sign}(\mathsf{Z})$. The more complicated case, when the signs disagree is what is depicted in \Cref{fig:covering}. Then, $\kappa=1,-1$ are relevant. When the signs of $E,\mathsf{Z}$ agree, only $M^{(\mathrm{II})}$ is needed (and can be taken to cover the other two areas in \cref{eq:misc_036}). Then, only $\kappa=-1$ is relevant.
\end{itemize}
The joint low-energy, large-radii regime is associated with the name ``Rydberg'' in the physics literature.

\begin{figure}[t!]
	\begin{center}
		\begin{tikzpicture}[xscale=-1]
			\filldraw[fill=lightgray!20] (-.2,-2) -- (-.2,0) to[out=180,in=270] (-1.7,1.5) -- (-6,1.5) -- (-6,-2) -- cycle;
			\begin{scope}
				\clip (-.2,-2) -- (-.2,0) to[out=180,in=270] (-1.7,1.5) -- (-6,1.5) -- (-6,-2) -- (-6.5,-2) -- cycle;
				\fill[darkgray!20] (0,-2) -- (-2,-2) -- (0,1) -- cycle;
				\fill[darkgray!20] (-6,0) -- (-6,-2) -- (-4,-2) -- (-4,0) -- cycle;
				\fill[darkgray!20] (-6,.15) -- (-9,2) -- (-4,2) -- (-4,.15) -- cycle;
				\fill[darkgray!20] (0,1.5) -- (-3.9,-1.9) -- (-2.1,-1.9) -- (0,1) -- cycle;
				\fill[darkgray!20] (0,1.5) -- (-3.9,-1.7) -- (-3.9,0) -- (0,1.5) -- cycle;
				\fill[darkgray!20] (-3.9,.15) -- (-3.9,2) -- (0,1.5) -- cycle;
				\node () at (-5,-.5) {$M^{(\mathrm{I})}_0$};
				\node () at (-5.125,-.9) {$\kappa=0$};
				\node () at (-5,1) {$M^{(\mathrm{I})}_0 $};
				\node () at (-5,.5) {$\kappa=-4$};
				\node () at (-.9,-1.2) {$M^{(\mathrm{II})} $};
				\node () at (-.9,-1.7) {$\kappa=-1$};
				\node () at (-2.5,-1.5) {$\kappa=1$};
				\node () at (-2.1,-1) {$M^{(\mathrm{III})}_0 $};
				\node () at (-2.6,0) {$M^{(\mathrm{IV})}_0 $};
				\node () at (-3.35,.69) {$M^{(\mathrm{V})}_0 $};
				\node () at (-3.3,-.5) {$\kappa=1$};
				\node () at (-2.5,1.2) {$\kappa=-4$};
				\draw (-.2,-2) -- (-.2,0) to[out=180,in=270] (-1.7,1.5) -- (-6,1.5) -- (-6,-2) -- cycle;
			\end{scope}
			\node () at (-6.4,1.8) {$X_1$};
			\draw[darkred, ->] (-.1,-2.1) -- (-.9,-2.1) node[below] {$|E|$};
			\draw[darkred, ->] (-.1,-2.1) -- (-.1,-1.2) node[left] {$r$};
			\draw[darkred, ->] (-.1,.1) -- (-.1,-.7) node[left] {$r^{-1}$};
			\draw[darkred, ->] (-.1,.1) to[out=180,in=-25] (-.8,.25) node[above left] {$r|E|$};
			\draw[darkred, ->] (-6.1,1.6) -- (-5.3,1.6) node[above] {$|E|^{-1}$};
			\draw[darkred, ->] (-1.6,1.6) -- (-2.4,1.6) node[above] {$|E|$};
			\draw[darkred, ->] (-1.6,1.6) to[out=-90,in=120] (-1.4,.8) node[above] {$(r|E|)^{-1}\qquad\quad$};
			\draw[darkred, ->] (-6.1,1.6) -- (-6.1,.8) node[right] {$r^{-1}$};
			\draw[darkred, ->] (-6.1,-2.1) -- (-5.3,-2.1) node[below] {$|E|^{-1}$};
			\draw[darkred, ->] (-6.1,-2.1) -- (-6.1,-1.2) node[right] {$r$};
		\end{tikzpicture}
	\end{center}
	\caption{One component of $X_1$ (stretched to aid visibility; cf. \Cref{fig:X}) where $\operatorname{sign}(E)\neq \operatorname{sign}(\mathsf{Z})$, covered by copies of $M$ and the rectangle $M_0$. The copies are labeled with Roman numerals.  In each region, the hydrogen atom falls into \cref{eq:1} with $\kappa$ as indicated. The copies of $M,M_0$ can be taken to cover all of $X_1$, but to aid the reader in distinguishing them, we have shrunk them so that they do not overlap. The situation in the $\operatorname{sign}(E)= \operatorname{sign}(\mathsf{Z})$ component of $X_1$ is simpler, because the $\kappa=1$ copies of $M_0$ can be omitted (see \S\ref{subsubsec:Rydberg}).}
	\label{fig:covering}
\end{figure}

The details are deferred to the appendix --- see \S\ref{sec:hydrogen_details}.

\subsection{The simple harmonic oscillator}
\label{subsec:parabolic} 

The next example we discuss is the simple harmonic oscillator: 
\begin{align}
	P = -  \frac{\partial^2}{\partial r^2}  +  \mathsf{k} r^2- E + \frac{ \ell(\ell+1)}{r^2} 
	\label{eq:misc_049}
\end{align}
where $\mathsf{k} \in \{-1,+1\}$, $E\in \bbR$, $\ell\in \bbN$. (If $\mathsf{k}<0$, this is more properly called the \emph{inverted} simple harmonic oscillator.) The kernel of $P$ can be solved for explicitly:
\begin{equation}
	\ker P = 
	\begin{cases}
		\operatorname{span}_\bbC\{r^{-1/2} \operatorname{WhittW}_{E/4,\mu}(r^2), r^{-1/2} \operatorname{WhittM}_{E/4,\mu}(r^2) \} & (\mathsf{k}=1), \\ 
		\operatorname{span}_\bbC\{r^{-1/2} \operatorname{WhittW}_{iE/4,\mu}(-ir^2), r^{-1/2} \operatorname{WhittM}_{iE/4,\mu}(-ir^2) \} & (\mathsf{k}=-1),
	\end{cases}
	\label{eq:SHO_soln}
\end{equation}
for $\mu = 4^{-1}(4(\ell^2+\ell)+1)^{1/2}$.

We are interested in all possible limits as $r\to 0^+, r\to\infty$ and $E\to \pm \infty$, possibly together. 

In order to study this ODE, we begin with the compactification 
\begin{equation} 
	X_0=[-\infty,\infty]_E \times [0,\infty]_{r/ \langle E \rangle^{1/2}} \hookleftarrow \bbR^+_r\times \bbR_E
\end{equation}  
depicted in \Cref{fig:parabolic_X}. Near the closure of $\{E=0\}$ in this compactification, $P$ is just a smooth family of ODEs with an irregular singularity of fixed order $\kappa=-6$.
So, it suffices to analyze the ODE near the sides $\{E=\pm \infty\}$ of $X_0$. Here, near one side, $X_0\cong [0,\infty]_{\hat{r}}\times [0,1)_{|E|^{-1}}$ canonically, where $\hat{r} = r/ |E|^{1/2}$. To study the ODE nearby, we can rewrite the ODE in terms of $\hat{r}$:
\begin{equation}
	|E|^{-1} P = -\frac{1}{|E|^2} \frac{\partial^2}{\partial \hat{r}^2} + \mathsf{k} \hat{r}^2  \mp 1 + \frac{1}{|E|^2} \frac{\ell(\ell+1)}{\hat{r}^2}.
	\label{eq:misc_052}
\end{equation}
This has the form \cref{eq:1} with $z=\hat{r}$ and $h=1/|E|$. Now, near $\hat{r}=\infty$, this is a semiclassical ODE with a fixed irregular singularity of order $-6$, which is beyond our purview (but considered well-understood), so we only need to discuss finite $\hat{r}$. 
We have at most three transition points at finite $\hat{r}$:
\begin{itemize}
	\item unless $\ell=0$, we always have a transition point with $\kappa=0$ at $\hat{r}=0$ (this gives two transition points because $E$ can be $\pm \infty$), and
	\item if $\mathsf{k}E>0$, then we have a simple turning point, $\kappa=1$, at $\hat{r}=|\mathsf{k}|$. 
\end{itemize}
Thus, we get from \Cref{thm1} that solutions of the ODE $Pu=0$ with say constant initial data at fixed $\hat{r}$ will be exponential-polyhomogeneous on the mwc $X$ constructed by blowing up the various transition points. See \Cref{fig:parabolic_X}.

This is another example where \Cref{thm:collapsed} applies to the $\kappa=1$ transition point.

\begin{figure}[h!]
	\begin{tikzpicture}
		\filldraw[fill=lightgray!20] (-2,0) -- (2,0) -- (2,3) -- (-2,3) -- cycle;
		\draw[dashed] (.5,1.5) -- (2,1.5) node[right] {$\{\hat{r} = \mathsf{k}\}$};
		\fill[black] (2,1.5) circle (1.5pt);
		\draw[darkred, ->] (-1.9,.1) -- (-1.9,.8) node[right] {$\hat{r}$};
		\draw[darkred, ->] (-1.9,.1) -- (-1.1,.1) node[above] {$|E|^{-1}$};
		\draw[darkred, ->] (-1.9,2.9) -- (-1.9,2.1) node[right] {$\hat{r}^{-1}$};
		\draw[darkred, ->] (-1.9,2.9) -- (-1.1,2.9) node[below] {$|E|^{-1}$};
		\draw[darkred, ->] (1.9,.1) -- (1.9,.8) node[above] {$\hat{r}$};
		\draw[darkred, ->] (1.9,.1) -- (1.1,.1) node[above] {$E^{-1}$};
		\draw[darkred, ->] (1.9,2.9) -- (1.9,2.1) node[left] {$\hat{r}^{-1}$};
		\draw[darkred, ->] (1.9,2.9) -- (1.1,2.9) node[below] {$E^{-1}$};
		\draw[darkred, ->] (-.2,.1) -- (-.2,.8) node[left] {$\hat{r}$};
		\draw[darkred, ->] (-.2,.1) -- (.6,.1) node[above left] {$E$};
		\draw[darkred, ->] (-.2,2.9) -- (-.2,2.1) node[left] {$\hat{r}^{-1}$};
		\draw[darkred, ->] (-.2,2.9) -- (.6,2.9) node[below left] {$E$};
		\node () at (0,1.5) {$X_0$};
	\end{tikzpicture}
	\begin{tikzpicture}
		\filldraw[fill=lightgray!20] (-2,0) -- (2,0) -- (2,3) -- (-2,3) -- cycle;
		\fill[fill=darkgray!20] (-2,0) -- (-2,1) -- (-.5,1) -- (-.5,0) -- cycle;
		\fill[fill=darkgray!20] (-2,3) -- (-2,2) -- (-.5,2) -- (-.5,3) -- cycle;
		\fill[fill=darkgray!20] (2,3) -- (2,2.25) -- (.5,2.25) -- (.5,3) -- cycle;
		\fill[fill=darkgray!20] (2,0) -- (2,.75) -- (.5,.75) -- (.5,0) -- cycle;
		\fill[fill=darkgray!20] (2,1) -- (2,2) -- (.5,2) -- (.5,1) -- cycle;
		\draw (-2,0) -- (2,0) -- (2,3) -- (-2,3) -- cycle;
		\node () at (1.25,1.5) {$\kappa=1$};
		\node () at (1.25,2.6) {$\kappa=-6$};
		\node () at (-1.25,2.5) {$\kappa=-6$};
		\node () at (1.25,.4) {$\kappa=0$};
		\node () at (-1.25,.5) {$\kappa=0$};
	\end{tikzpicture}
	\qquad\quad
	\begin{tikzpicture}
	\filldraw[fill=lightgray!20] (-1,.5) arc(90:0:.5)  -- (2.5,0) arc(180:90:.5) -- (3,1.5) arc(-90:-270:.5) -- (3,3) -- (-1,3) -- cycle;
	\node () at (.15,2) {$X_1$};
	\draw[dashed] (1.5,1.5) to[out=0,in=180] (2.5,2);
	\draw[darkred, ->] (-.9,.6) -- (-.9,1.2) node[right] {$\hat{r}$};
	\draw[darkred, ->] (-.9,.6) to[out=0,in=135] (-.55,.45) node[right] {$r^{-1}|E|^{-1/2}$};
	\draw[darkred, ->] (2.4,.1) -- (1.8,.1) node[above] {$E^{-1}$};
	\draw[darkred, ->] (2.4,.1) to[out=90,in=210] (2.7,.55) node[above] {$rE^{1/2}\;\;\;\;$};
	\draw[darkred, ->] (2.6,1.55) to[out=140,in=230] (2.57,2.45) node[left] {$(\hat{r}-\mathsf{k})E^{2/3}$};
	\end{tikzpicture}
	\caption{\textit{Left and right}: The mwcs $X_0,X$ used to study the simple harmonic oscillator in \S\ref{subsec:parabolic}. Away from the blowups, $X_1\cong X_0$, so we have not repeated coordinate charts where avoidable. (We have depicted the $\mathsf{k}>0$ case. If $\mathsf{k}<0$, the figure is flipped across the vertical axis.) \textit{Middle:} the relevant cases of $\kappa$.}
	\label{fig:parabolic_X}
\end{figure}

\begin{remark}
	The appearance of Whittaker functions in the solutions of $Pu=0$ suggests a relation to the hydrogen atom. Indeed, defining $z=r^2$, $Pu=0$ can be rewritten 
	\begin{equation}
	\Big[	  \frac{\partial^2}{\partial z^2}  - \frac{\mathsf{k}}{4} + \frac{E }{4z} - \frac{1}{4}\Big(\ell(\ell+1) - \frac{3}{4} \Big) \frac{1}{z^2} \Big]  (z^{1/4} u) = 0,
	\label{eq:misc_054}
	\end{equation}
	which is a form of Whittaker's ODE and how \cref{eq:SHO_soln} is derived.
	So, studying the $E\to\pm \infty$ behavior of the simple harmonic oscillator corresponds to studying the hydrogen atom / repulsive Coulomb potential as $\mathsf{Z}\to\pm \infty$, with a possibly non-integral angular momentum number.
	
	An interesting point is that \Cref{thm:collapsed} applies to the $\kappa=-1$ transition point of \cref{eq:misc_054} at $z=0$. It follows that the \emph{conclusion} of \Cref{thm:collapsed} actually applies to the $\kappa=0$ transition points of \cref{eq:misc_052}, even though, owing to the possibly nonzero $\ell(\ell+1)/r^2$ term, $P$ does not satisfy the hypotheses of that theorem. It turns out that \cref{eq:misc_052} is one of a few known exceptional cases for which this holds. See \Cref{rem:exsimp}.
	\label{rem:SHO_Coulomb_relation}
\end{remark}

Some further discussion of the simple harmonic oscillator can be found in \S\ref{sec:second}.

\subsection{The anharmonic oscillator}\label{subsec:anharmonic}
Next, consider the \emph{anharmonic oscillator},
\begin{equation}
	P=-\frac{\partial^2 }{\partial r^2} + \frac{\ell(\ell+1)}{r^2} + \mathsf{E}+  \mathsf{k} r^2 + \lambda r^4  ,
\end{equation}
for $\ell\in \bbN$, $\mathsf{k}\in \{-1,+1\}$, and $\mathsf{E}\in \bbR$  all fixed, and for $\lambda>0$.\footnote{Note that, in this context, we use $\lambda$ to denote a coefficient in the ODE, not $z/h^{2/(\kappa+2)}$ as elsewhere in this paper.} We are interested in all possible asymptotic regimes, including $r\to 0^+$, $r\to\infty$, $\lambda \to 0^+$ (the small anharmonicity limit), and $\lambda\to \infty$ (the large anharmonicity limit), as well as combinations thereof. 

When $\ell=0$, the ODE $Pu=0$ is a form of the tri-confluent Heun equation, and therefore elements of $\ker P$ can be written in terms of the tri-confluent Heun functions. Indeed, using  $\operatorname{HeunT}_{q,\alpha,\gamma,\delta,\epsilon}(z)$ to denote the element of the kernel of 
\begin{equation}
	\partial_{z}^2 + (\gamma+\delta z+\epsilon z^2) \partial_z + (\alpha z-q), 
\end{equation}
the tri-confluent Heun operator, 
with initial value $\operatorname{HeunT}_{q,\alpha,\gamma,\delta,\epsilon}(0)=1$ and initial derivative $\operatorname{HeunT}_{q,\alpha,\gamma,\delta,\epsilon}'(0)=0$, 
\begin{equation}
	\ker P = \operatorname{span}_{\bbC}\Big\{ \exp\Big( \pm \Big( \frac{r\mathsf{k}}{2\sqrt{\lambda}}+\frac{r^3\sqrt{\lambda}}{3} \Big)\Big) \operatorname{HeunT}_{\mathsf{E}-\frac{1}{4\lambda},\pm 2\sqrt{\lambda},\pm \mathsf{k} \lambda^{-1/2},0,\pm 2\sqrt{\lambda}}(r) \Big\},
\end{equation}
where all of the signs above are the same. (That is, one basis vector is given by replacing all of the `$\pm$' by `$+$,' and the other is given by replacing them by `$-$.')

Construct a mwc $X_0$ by beginning with the rectangle 
$ [0,\infty]_\lambda \times [0,\infty]_r$ and then performing a  quasihomogeneous blowup of the corner $\{\lambda=0,r=\infty\}$ so as to resolve the ratio $\hat{r}=\lambda^{1/2} r$. See \Cref{fig:X_anharmonic}. Then, we can cover $X_0$ by copies of the rectangle $M_0=[0,Z]_z \times [0,\infty)_{h^2}$ or $M$ such that, on each, the ODE has the form \cref{eq:1}, for the $\kappa$ depicted in \Cref{fig:X_anharmonic}, except in one corner where this is only true if $\mathsf{E}=0$ (we will comment on this below). There are two cases, depending on the sign of $\mathsf{k}$. The more complicated case is $\mathsf{k}=-1$, which is what is depicted in the figure. If $\mathsf{k}=+1$ then the $\kappa=1$ transition point is absent.

\begin{figure}[h!]
	\begin{tikzpicture}
		\filldraw[fill=lightgray!20] (0,0) -- (0,1.5) arc(-90:0:1.5) -- (4,3) -- (4,0) -- cycle;
		\draw[darkred,->] (3.9,2.9) -- (3.9,2) node[left] {$1/r$};
		\draw[darkred,->] (3.9,2.9) -- (3.2,2.9) node[below] {$1/\lambda$};
		\draw[darkred,->] (3.9,.1) -- (3.9,1) node[left] {$r$};
		\draw[darkred,->] (3.9,.1) -- (3.2,.1) node[above] {$1/\lambda$};
		\draw[darkred,->] (.1,.1) -- (.1,.5) node[right] {$r$};
		\draw[darkred,->] (.1,1.4) -- (.1,.9) node[right] {$1/r$};
		\draw[darkred,->] (.1,1.4) to[out=0,in=210] (.8,1.6) node[right] {$\hat{r}$};
		\draw[darkred,->] (.1,.1) -- (.9,.1) node[above] {$\lambda$};
		\draw[darkred,->] (1.6,2.9) -- (2.2,2.9) node[below] {$\lambda$};
		\draw[darkred,->] (1.6,2.9) to[out=-90, in=70] (1.4,2.2) node[right] {$1/\hat{r}$};
		\node () at (2,1) {$X_0$};
	\end{tikzpicture}
	\quad 
	\begin{tikzpicture}
		\filldraw[fill=lightgray!20] (0,0) -- (0,1.5) arc(-90:0:1.5) -- (5,3) -- (5,0) -- cycle;
		\begin{scope}
			\clip (0,0) -- (0,1.5) arc(-90:0:1.5) -- (5,3) -- (5,0) -- cycle;
			\fill[darkgray!20] (0,3) -- (1.5,0) -- (0,0) -- cycle;
			\fill[darkgray!20] (0,3) -- (1.6,.1) -- (3,.1) -- cycle;
			\fill[darkgray!20] (0,3) -- (3.1,.2) -- (3.1,1.6) -- cycle;
			\fill[darkgray!20] (0,3) --  (3.1,1.7) -- (3.1,3) -- cycle;
			\fill[darkgray!20] (3.2,0) -- (3.2,1.5) -- (5,1.5) -- (5,0) -- cycle;	
			\fill[darkgray!20] (3.2,3) -- (3.2,1.6) -- (5,1.6) -- (5,3) -- cycle;
		\end{scope}
		\draw (0,0) -- (0,1.5) arc(-90:0:1.5) -- (5,3) -- (5,0) -- cycle;
		\node () at (4.125,2) {$\kappa=-8$};
		\node () at (4.125,2.5) {$M_0^{(\mathrm{I})}$};
		\node () at (4.125,.5) {$\kappa=4$};
		\node () at (4.125,1) {$M_0^{(\mathrm{II})}$};
		\node () at (.55,.75) {$M^{(\mathrm{III})}$};
		\node () at (1.7,.75) {$M^{(\mathrm{IV})}_0$};
		\node () at (1.9,1.75) {$M^{(\mathrm{V})}_0$};
		\node () at (2.9,2.2) {$M^{(\mathrm{VI})}_0$};
		\node () at (.55,.25) {$\kappa=2$};
		\node () at (2.1,.3) {$\kappa=1$};
		\node () at (2.6,1.3) {$\kappa=1$};
		\node () at (2.3,2.7) {$\kappa=-8$};
	\end{tikzpicture}
	\quad
	\begin{tikzpicture}
		\filldraw[fill=lightgray!20] (0,0) -- (0,1.5) arc(-90:0:1.5) -- (4,3) -- (4,1) arc(90:180:1) -- cycle;
		\begin{scope}
			\clip (0,0) -- (0,1.505) arc(-90:0:1.495) -- (4,3) -- (4,1) arc(90:180:1) -- cycle;
			\filldraw[fill=white] (.9,2.1) circle (.5);
		\end{scope}
		\fill[fill=white] (.9,2.1) circle (.492);
		\node () at (.5,.5) {$X$};
		\draw[darkred, ->] (2.9,.1) to[out=90, in=240] (3.105,.7) node[above] {$r \lambda^{1/6}$};
		\draw[darkred, ->] (2.9,.1) -- (2.2,.1) node[above] {$\lambda^{-1/6}$};
		\draw[darkred, ->] (4.1,.9) to[out=180, in=40] (3.6,.75) node[below right] {$r^{-1} \lambda^{-1/6}$};
		\draw[darkred, ->] (4.1,.9) -- (4.1,1.5) node[right] {$r$};
		\draw[darkred, ->] (1,1.5) to[out=10,in=-100] (1.5,2) node[right] {$\lambda^{-2/3}(\hat{r}-1)$};
	\end{tikzpicture}
	\caption{\textit{Left:} the mwc $X_0$ described in \S\ref{subsec:anharmonic}. \textit{Middle:} the same mwc, covered by copies of $M_0,M$, showing the kind of transition point present in each. Only the $\mathsf{k}<0$ case is shown. \textit{Right:} the mwc that results from performing the appropriate quasihomogeneous blowup of the $\kappa=1,4$ transition points in $X_0$.}
	\label{fig:X_anharmonic}
\end{figure}

So, solutions whose initial data is supplied along a curve $\lambda\mapsto (\lambda,r(\lambda))$ not impacting a corner of $X_0$ (and of exponential-polyhomogeneous type on that curve) are of exponential-polyhomogeneous type on the mwc depicted in $X$.   

The $\kappa=1$ turning point present in the $\mathsf{k}=-1$ case is covered by \Cref{thm:collapsed} (and Olver's exposition in \cite{OlverBook}). The other two are not:
\begin{itemize}
	\item Consider the $\kappa=4$ transition point, in the $M_0^{(\mathrm{II})}$ region in \Cref{fig:X_anharmonic} (large anharmonicity). Here, we use $z=r$ and $h=\lambda^{-1/2}$, so $P$ is given by
	\begin{equation}
		\lambda^{-1} P = - h^2\frac{\partial^2}{\partial z^2}  + z^4 +h^2\Big(\frac{\ell(\ell+1)}{z^2}  + \mathsf{E} + \mathsf{k} z^2\Big) .
	\end{equation}
	This is not one of Olver's exceptional cases listed in \cite[Eq.\ 12.14.17]{OlverBook}, unless $\mathsf{E},\mathsf{k}=0$. 
	\item Consider the $\kappa=2$ transition point, in the $M^{(\mathrm{III})}$ region in \Cref{fig:X_anharmonic} (small anharmonicity). In terms of $\hat{r}=\lambda^{1/2} r$, 
	\begin{equation}
		\lambda P = - \lambda^2 \frac{\partial^2}{\partial \hat{r}^2} + \mathsf{k} \hat{r}^2 + \hat{r}^4+ \frac{\lambda^2 \ell(\ell+1)}{\hat{r}^2} +\lambda \mathsf{E}.
		\label{eq:misc_058}
	\end{equation}
	This has the form \cref{eq:1} with $h=\lambda$, except that $\lambda \mathsf{E} = h \mathsf{E}$ is not $O(h^2)$ if $\mathsf{E}\neq 0$. Consequently, to apply our various theorems to this transition point, we need to assume that $\mathsf{E}=0$. It is likely that the conclusion of \Cref{thm1} is still true, in this case, even if $\mathsf{E}\neq 0$, but \Cref{thm2} likely requires modification: near the transition point in question, the approximation 
	\begin{equation}
		u \approx \sqrt[4]{\frac{\xi^2}{W} } Q\Big( \frac{\zeta}{h^{1/2}} \Big) 
	\end{equation}
	that \cref{eq:misc_16} says holds should, in fact, not hold. The reason is that the $\lambda \mathsf{E}$ term in \cref{eq:misc_058} is leading order at $\mathrm{fe}$, but $Q$ solves an ODE that does not involve $\mathsf{E}$. Likely, a correct treatment of this case will involve a correction to the definition of $Q$ involving $\mathsf{E}$. Then, we expect \Cref{thm2} to hold, except the index sets involved may become more complicated. This is just to say that the reason why we cannot handle $\mathsf{E}\neq 0$ does not appear fundamental --- it is just because we considered $\psi$ with the simplest possible index set. See our remark in \S\ref{subsec:conc} regarding allowing $\psi$ to have an $O(1/h)$ term in it.
	
	However, if $\mathsf{E} =0$, then we can apply \Cref{thm1}, \Cref{thm2} directly.
\end{itemize}

\subsection{Bessel's ODE}\label{subsec:Bessel}
Consider Bessel's ODE  
\begin{equation}
	x^2 \frac{\partial^2 u}{\partial x^2} + x \frac{\partial u}{\partial x} + ( \pm x^2 -\epsilon \nu^2 )u = 0  
	\label{eq:Bessel} 
\end{equation}
on $\bbR^+_x$, 
in which $\nu>0$ is a parameter to be varied and $\epsilon \in \{-1,+1\}$ is fixed. Here, the sign $\pm$ governs whether this is the Bessel ODE or the modified Bessel ODE. It is the unmodified/modified Bessel ODE of order $\nu$ if $\epsilon=+1$ and of order $i\nu$ if $\epsilon =-1$.

\begin{remark}
	Because we are concerned in this paper with ODE of a real variable and one real parameter, we do not consider Bessel's ODE at complex (meaning neither purely real nor purely imaginary) argument/order.
\end{remark}

We convert \cref{eq:Bessel} into the form \cref{eq:1} in two steps. First, following \Cref{rem:reduction} to remove the first-order terms: letting $v=x^{1/2}u$, \cref{eq:Bessel} is equivalent to 
\begin{equation}
	\frac{\partial^2 v}{\partial x^2} + \frac{1}{x^2}\Big(\pm x^2-\epsilon \nu^2 + \frac{1}{4} \Big) v = 0.
	\label{eq:Bessel_2}
\end{equation}
Now let $z = x /\nu $, $h=1/\nu$. Then, \cref{eq:Bessel_2} becomes 
\begin{equation}
	h^2 \frac{\partial^2 v}{\partial z^2} + \Big(\pm 1 - \frac{\epsilon }{z^2} + \frac{h^2}{4z^2}\Big)v.
\end{equation}
This is now manifestly semiclassical, with three transition points: an irregular singularity with $\kappa=-4$ at $z=\infty$ (as revealed by using $1/z$ as a coordinate), a possible simple turning point, $\kappa=1$, at $z=1$ (this being absent in two of the four cases of signs),
and a regular singularity, $\kappa=-2$, at $z=0$.

Consider the compactification $X \hookleftarrow \bbR^+_z\times [1,\infty)_\nu$ gotten by taking $[0,\infty]_z\times [0,1]_h$ and then performing the $\kappa=1$ quasihomogeneous blowup of the simple turning point (if present) and then performing the $\kappa=0$ blowup (i.e.\ a polar blowup) of the corner $\{z=0,h=0\}$. The reason for choosing this particular last blowup is that it resolves $z/h = x $, the original variable, which thus parameterizes the interior of the front face of the blowup. 

The resulting mwc is depicted in \Cref{fig:Bessel}; see also \cite[Fig.\ 1]{SherBessel}. The expansions of $J_\nu(x),Y_\nu(x)$ at each edge are well-known:
\begin{itemize}
	\item The expansion at the pure large-argument edge are the already mentioned Hankel expansions \cite[\href{http://dlmf.nist.gov/10.17}{\S10.17}]{NIST}. (See \cite[\href{http://dlmf.nist.gov/10.40}{\S10.40}]{NIST} for the modified Bessel analogues.)
	\item Except near the curve $\{\nu=x\}$, the large-$\nu$ expansions are the Debye expansions \cite[\href{http://dlmf.nist.gov/10.19.ii}{\S10.19.ii}]{NIST}. (See \cite[\href{http://dlmf.nist.gov/10.41.ii}{\S10.41.ii}]{NIST} for the modified Bessel analogues.)
	\item Near the curve $\{\nu=x\}$, we have a turning point. Here, the ``transitional'' expansions \cite[\href{http://dlmf.nist.gov/10.19.iii}{\S10.19.iii}]{NIST}\cite[\href{http://dlmf.nist.gov/10.20.i}{\S10.20(i)}]{NIST} are required.
\end{itemize}
Our treatment here (assuming that the $\kappa=-2$ case works in the manner described in \Cref{rem:extension}) makes clear that the Bessel functions $K_{i\nu}(x), I_{i\nu}(x)$ are (at least when normalized appropriately) of nontrivial exponentially-polyhomogeneous type on the same mwc. Even at the level of Poincar\'e-type asymptotic expansions, this fact is a bit difficult to extract from the literature; NIST's coverage \cite[\href{http://dlmf.nist.gov/10.45}{\S10.45}]{NIST} of this case is a bit lacking, for example.

The cases $J_{i\nu}(x)$, $Y_{i\nu}(x)$, $K_\nu(x)$, $I_\nu(x)$ are even simpler, because of the lack of a turning point at $\nu=x$.

\begin{figure}[h!]
	\qquad
	\begin{tikzpicture}[scale=1.25]
		\filldraw[fill=lightgray!20] (0,-1) -- (4,-1) -- (4,1) arc(-90:-180:2) -- (0,3) -- cycle;
		\begin{scope}
			\clip (0,0) -- (4,0) -- (4,1) arc(-90:-180:2) -- (0,3) -- cycle;
			\filldraw[fill=white] (2.5,1.7) circle (.5);
		\end{scope}
		\draw (0,-1) -- (4,-1) -- (4,1) arc(-90:-180:2) -- (0,3) -- cycle;
		\fill[fill=white] (2.5,1.7) circle (.49);
		\draw[darkred,->] (3.9,-.9) -- (3.9,-.3) node[left] {$x-1$};
		\draw[darkred,->] (3.9,.9) -- (3.9,.2) node[left] {$1/x$};
		\draw[darkred,->] (3.9,.9) to[out=180,in=-10] (3.45,.97) node[below] {$x/\nu $};
		\draw[darkred,->] (3.9,-.9) -- (3.3,-.9) node[above] {$1/\nu$};
		\draw[darkred,->] (.1,-.9) -- (.1,-.3) node[right] {$x-1$};
		\draw[darkred,->] (.1,-.9) -- (.7,-.9) node[above] {$\nu$};
		\draw[darkred,->] (.1,2.9) -- (.7,2.9) node[below] {$\nu$};
		\draw[darkred,->] (.1,2.9) -- (.1,2.3) node[right] {$1/x$};
		\draw[darkred,->] (1.9,2.9) -- (1.3,2.9) node[below] {$\nu^{-1}$};
		\draw[darkred,->] (1.9,2.9) to[out=-90,in=105] (2,2.3) node[left] {$\nu/x$};
		\draw[darkred,->] (2.85,1.45) to[out=-30,in=160] (3.3,1.22) node[above] {\qquad$1-x/\nu$};
		\draw[darkred,->] (2.85,1.45) to[out=230,in=20] (2.6,1.28) node[above] {$s$};
		\draw[darkred,->] (2.35,2.1) to[out=210,in=70] (2.1,1.8) node[right] {$s=\nu^{1/3}|x-\nu|^{-1}$};
		\draw[darkred,->] (2.35,2.1) to[out=120,in=-70] (2.2,2.4) node[right] {$x/\nu-1$};
		\draw[darkred, ->] (2.4,1.11) to[out=170, in=-90] (1.9,1.7) node[left] {$ \frac{x-\nu}{\nu^{1/3}}$};
	\end{tikzpicture}
	\begin{tikzpicture}[scale=1.25]
		\filldraw[fill=lightgray!20] (0,-1) -- (4,-1) -- (4,1) arc(-90:-180:2) -- (0,3) -- cycle;
		\begin{scope}
			\clip (0,-1) -- (4,-1) -- (4,1) arc(-90:-180:2) -- (0,3) -- cycle;
			\fill[darkgray!20] (4,3) -- (2.1,-1) -- (4,-1) -- cycle;
			\fill[darkgray!20] (4,3) -- (2.0,-.9) -- (.1,-.9) -- cycle;
			\fill[darkgray!20] (4,3) -- (.1,-.7) -- (.1,1.4) -- cycle;
			\fill[darkgray!20] (4,3) -- (.1,1.5) -- (.1,3) -- cycle;
			\filldraw[fill=white] (2.5,1.7) circle (.5);
			\node () at (1.8,-0) {$M^{(\mathrm{II})}$};
			\node () at (1.4,-.4) {$\kappa=1$};
			\node () at (3.15,-.7) {$M^{(\mathrm{I})}[\kappa=0]$};
			\node () at (3.3,-.25) {$\kappa=-2$};
			\node () at (1.2,1.4) {$M^{(\mathrm{III})}$};
			\node () at (.8,1) {$\kappa=1$};
			\node () at (.7,2.1) {$M^{(\mathrm{IV})}_0$};
			\node () at (1.1,2.55) {$\kappa=-4$};
		\end{scope}
		\draw (0,-1) -- (4,-1) -- (4,1) arc(-90:-180:2) -- (0,3) -- cycle;
		\fill[fill=white] (2.5,1.7) circle (.49);
		\node () at (1,3.2) {Hankel: \cite[\href{http://dlmf.nist.gov/10.17}{\S10.17}]{NIST}};
		\node[right] (deb) at (2.3,2.6) {Debye: \cite[\href{http://dlmf.nist.gov/10.19.ii}{\S10.19.ii}]{NIST}};
		\draw[dotted] (deb) -- (5,1.2) -- (4.1,.5);
		\draw[dotted] (5,1.2) -- (3.4,1.2);
		\draw[dotted] (deb) -- (2.1,2.6);
		\node[right] () at (2,1.6) {\cite[\href{http://dlmf.nist.gov/10.19.iii}{\S10.19.iii}]{NIST}};
	\end{tikzpicture}
	\caption{\textit{Left:} the mwc on which Sher \cite{SherBessel} proved that the Bessel functions are exponential-polyhomogeneous. \textit{Right:} the same mwc, covered by copies of $M$ or the rectangle on which $P$ has the form \cref{eq:1}. The edges have been labeled by which sections in \cite{NIST} provide asymptotics there (for $J_\nu,Y_\nu$).
	}
	\label{fig:Bessel}
\end{figure}

In \cite{SherBessel}, Sher proved $J_\nu,Y_\nu$ that the Bessel functions (at real order) are of nontrivial exponential-polyhomogeneous type on $X$. 
Assuming that we already know that the Bessel functions have nontrivial exponential-polyhomogeneous data along some curve $\{z=z_0\}$ for $z_0\in \bbR^+\backslash \{1\}$, the part of this result stating nontrivial exponential-polyhomogeneity near the transition point at $z=1$ follows from \Cref{thm1}, \Cref{thm2}. In fact, \Cref{thm:collapsed} gives something slightly stronger, namely the term-by-term differentiability in each direction of the standard uniform expansion for the Bessel functions near the turning point \cite[\S11.10]{OlverBook}\cite[\href{http://dlmf.nist.gov/10.20.i}{\S10.20(i)}]{NIST}. Note that this applies also to the modified Bessel functions at imaginary order (the $\epsilon=-1$ case).

The assumptions in the previous paragraph are essentially regarding the normalization of the Bessel functions. Even if we did not know that the Bessel functions have exponential-polyhomogeneous initial data along a suitable curve, we could still use \Cref{thm2} to produce a basis of solutions of Bessel's ODE that do.

\begin{figure}[h!]
	\includegraphics[scale=.47]{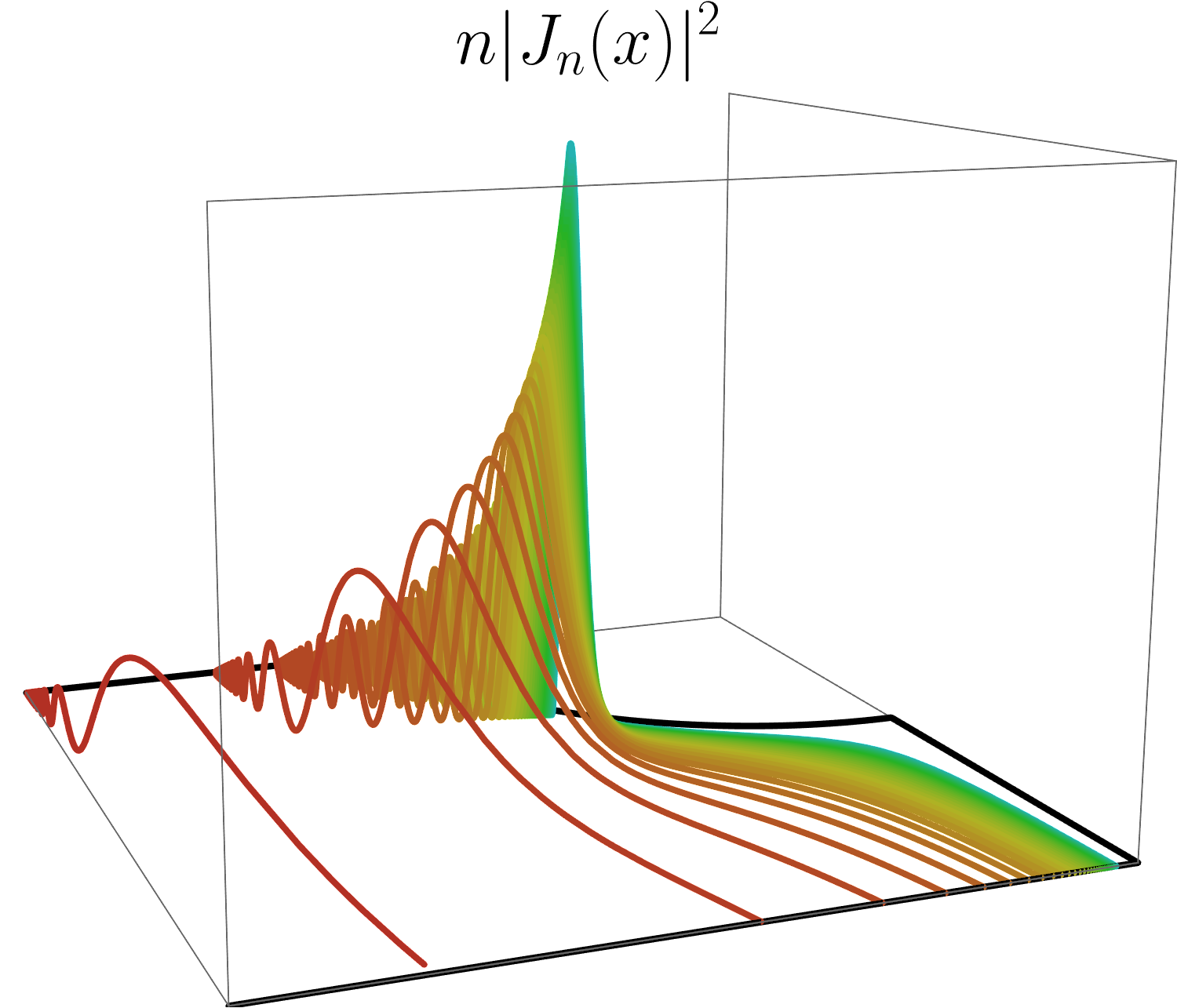}
	\caption{A plot, over $X$ sans the $\kappa=1$ blowup, of $n|J_n(x)|^2$, for $n\in \{1,\dots,35\}$. Note the difference in behavior in the regions $\{x>n\}$ and $\{x<n\}$ separated by the turning point at which $n|J_n(x)|^2$ displays a large Airy-like spike. Similar figures (with different behavior in different regions) could be made for $Y_n(x)$, $K_{in}(x)$, $I_{in}(x)$. }
\end{figure}

If \Cref{rem:extension} is correct regarding the $\kappa\leq -2$ cases, then the full exponential-polyhomogeneity result stated in \cite{SherBessel} follows, assuming that the Bessel functions can be shown to either have exponential-polyhomogeneous ``scattering data'' (i.e.\ leading coefficient in the large-argument expansion) or exponential-polyhomogeneous data at some finite $x$. This underscores the importance of understanding the $\kappa=-2$ case.

\subsection{The Regge--Wheeler equation}
\label{subsec:Regge_Wheeler}

Consider the \textit{Regge--Wheeler} equation, which describes linear (axial) perturbations of the metric of the exterior of the Schwarzschild spacetime. After some manipulations (cf.\ \Cref{rem:reduction}), the differential operator in question can be written \cite[\S1.1]{Schlag2} in the form
\begin{equation}
-\Big( 1 - \frac{r_{\mathrm{H}}}{r} \Big)^2 \frac{\partial^2 }{\partial r^2} - \sigma^2  + \Big( 1 - \frac{r_{\mathrm{H}}}{r} \Big) \frac{\ell(\ell+1)}{r^2}+ V_{\mathrm{RW}}(r)  , 
\label{eq:Regge_wheeler}
\end{equation}
where $r_{\mathrm{H}}>0$ is the location of the black hole horizon, $\sigma$ is the perturbation's temporal frequency, and $V_{\mathrm{RW}} \in (r-r_{\mathrm{H}}) C^\infty[r_{\mathrm{H}},\infty)_r$ is some ``effective potential,'' and $\ell\in \bbN$ is the azimuthal quantum number. We are interested in the large angular momentum limit, $\ell\to\infty$.

To write the operator in \cref{eq:Regge_wheeler} in the form considered above, define $z=r-r_{\mathrm{H}}$ and $h^{-2} = \ell(\ell+1)$. 
Multiplying through by $h^2 (1-r_{\mathrm{H}}/r)^{-2} = h^2 (z+r_{\mathrm{H}})^2/z^2$, the result is 
\begin{equation}
- h^2 \frac{\partial^2}{\partial z^2} + \frac{1}{z} \frac{1}{(z+r_{\mathrm{H}}) } +h^2 \frac{(z+r_{\mathrm{H}})^2}{z^2} \Big[ - \sigma^2 + V_{\mathrm{RW}}(z+r_{\mathrm{H}}) \Big]. 
\end{equation}
This has the form \cref{eq:1} studied above, with $\kappa=-1$,  $W(z)=1/(z+r_{\mathrm{H}})$, and $\psi = z^{-2}(z+r_{\mathrm{H}})^2 (-\sigma^2 + V_{\mathrm{RW}}(z+r_{\mathrm{H}}) )$. So, \Cref{thm1}, \Cref{thm2}, and \Cref{thm:collapsed} all apply to this transition point.

\section{A more precise theorem}
\label{sec:thm}

\subsection{Allowed singularities in the operator coefficients}
\label{subsec:allowed}

We describe now the allowed singularities of the term $\psi$ in \cref{eq:1}, in full generality. 
The assumptions to be placed on $\psi$ are
\begin{itemize}
	\item $\psi \in \varrho_{\mathrm{be}}^{-2}\varrho_{\mathrm{fe}}^{-2} C^\infty(M) = z^{-2} C^\infty(M)$, and
	\item $z^2 \psi |_{\mathrm{be}}:\mathrm{be}\to\bbC$ is constant. 
\end{itemize}
Allowing the singularity at $\mathrm{be}$ is mostly useful when considering the semiclassical ODE arising from partial wave analysis for semiclassical Schr\"odinger operators with spherical symmetry. The second of the two requirements is placed to avoid having to deal with regular singular differential equations with variable indicial roots. 

More interesting is the singularity at $\mathrm{fe}$. Methods such as Olver--Langer's do not seem able to handle this sort of singularity. In fact, they do not even seem to be able to handle 
\begin{equation} 
	\psi \in z^{-2} C^\infty([0,Z]_z\times [0,\infty)_{h^2} ;\bbC),
\end{equation} 
except for the $\kappa=-1$ case and the exceptional cases listed in \cite[\S12.14.4]{OlverBook}. In contrast, the singular geometric methods below are sufficiently robust.

We now define some notation used below.
Given any function $\psi$ satisfying the conditions above, choose $\alpha\in \bbC$ with $\Re \alpha\geq 0$ such that $z^2\psi |_{\mathrm{be}} = \alpha^2-1/4$, and let $\Psi = z^2 \psi|_{\mathrm{fe}} - \alpha^2+1/4$. We assume for the remainder of the paper that $\Re \alpha>0$. The arguments below all go through if $\Re \alpha = 0$, with minor modifications and some additional casework.

Notice that $\Psi \in \varrho_{\mathrm{be}} C^\infty(\mathrm{fe})$. Identifying $\mathrm{fe}\backslash \mathrm{ze}$ with $[0,\infty)_\lambda$ for $\lambda = z/h^{2/(\kappa+2)}$, the previous sentence says, more prosaically, that $\Psi(\lambda) \in \lambda C^\infty[0,\infty)_\lambda$ and satisfies $\smash{\Psi(\rho^{-2/(\kappa+2)})} \in C^\infty [0,\infty)_{\rho^2}$.
Throughout this paper, we will identify $\mathrm{fe}\backslash \mathrm{ze}$ with $[0,\infty)_\lambda$. Now define $E(z,h)$ by 
\begin{equation}
	\psi(z,h) = \frac{1}{z^2}\Big(\alpha^2-\frac{1}{4}\Big) + \frac{1}{z^2} \Psi\Big( \frac{z}{h^{2/(\kappa+2)}} \Big) + E(z,h), 
	\label{eq:psi_form}
\end{equation}
and in addition let $\phi = \psi - E$. Since $z^2 E$ vanishes at $\mathrm{be}\cup\mathrm{fe}$, it must lie in $\varrho_{\mathrm{be}}\varrho_{\mathrm{fe}} C^\infty(M)$. That is, $E \in \varrho_{\mathrm{be}}^{-1}\varrho_{\mathrm{fe}}^{-1} C^\infty(M)= z^{-1} C^\infty(M)$. As it turns out, the term $h^2 \phi$ in $P$, though subleading at $\mathrm{ze}$, is of comparable order to the main terms at $\mathrm{fe}$ and must therefore be taken into account in order to understand the $h\to 0^+$ limit. (We have already seen examples of this, with $\Psi=0$, in \S\ref{sec:examples}.) This is why we devote the symbol `$\phi$' to it; it will play a prominent role later.

Let 
\begin{equation}
	N(P) = -h^2 \frac{\partial^2}{\partial z^2} + \varsigma  z^\kappa W(z) + \frac{h^2}{z^2} \Big(\alpha^2 - \frac{1}{4} \Big) + \frac{h^2}{z^2} \Psi\Big( \frac{z}{h^{2/(\kappa+2)}} \Big). \label{eq:NP_Def}
\end{equation} 
Thus, $P-N(P) \in (h^2/z) C^\infty(M)$, so $N(P)$ arises from omitting from $P$ terms of order $O(h^2/z)$ and better. Thus, solutions $Q$ to $N(P)Q=0$ an be considered as $O(h^2/z)$-quasimodes for the original semiclassical ODE $Pu=0$, hence the use of the notation `$Q$' rather than `$u$.'
These quasimodes will be studied later.

\subsection{Applications with nonzero $\Psi$}
\label{subsec:more_examples}
All of the examples discussed in \S\ref{sec:examples} have $\Psi=0$. One reason for allowing nonzero $\Psi$ is to apply the theory to families $\{P(r;\sigma)\}_{\sigma>0}$ of ODEs on the whole real line $\bbR_r$ depending non-semiclassically on a parameter $\sigma$, in which the semiclassical regime arises as an artificial transitional regime in the joint large $r$, low $\sigma$ limit. 
The coordinate $r$ is to be related to $z$ in the following way: the blown up face $\mathrm{fe}\subset M$ corresponds to the original face $[0,\infty]_r\times \{0\}_\sigma\subset \bbR_r\times \{0\}_\sigma$. This is the reason why $\Psi$ may be nonzero; it captures the variation of $z^2 \psi$ in $r$ that cannot be absorbed into the constant $\alpha^2-1/4$. 

We saw examples in \S\ref{subsubsec:Rydberg}, \S\ref{subsec:anharmonic} where the semiclassical regime arose in this way, but in those examples the terms in the PDE were monomials with definite homogeneities, which resulted in the $\Psi$ term being unnecessary. However, as soon as we add a term to the potential with a different homogeneity, the original homogeneities no longer apply, and so $\Psi$ becomes necessary:

\begin{example}[\cite{SussmanACL}]	
	Consider a 1D Schr\"odinger operator with Coulomb-like potential, similar to \cref{eq:Whittaker}:
	\begin{equation}
		- \frac{\partial^2 u}{\partial r^2} +\Big[ - \frac{\mathsf{Z}}{r} + \frac{\ell(\ell+1) }{r^2}  + V(r) \Big] u = E u
		\label{eq:Whittaker_general}
	\end{equation}
	where $\mathsf{Z}\in \bbR^\times$, $\ell\in \bbN$, and $V\in \langle r \rangle^{-2} C^\infty([0,\infty]_r ) $ are fixed, and where we are interested in the $E\to 0^\pm$ limits. The way we studied the $V=0$ case in \S\ref{subsubsec:Rydberg} was to introduce the coordinate $\hat{r}=r|E|$ and to rewrite the ODE in terms of this rescaled coordinate:
	\begin{equation}
		- |E|\frac{\partial^2 u}{\partial \hat{r}^2} +\Big[ - \frac{\mathsf{Z}}{\hat{r}} + |E|\Big( \frac{ \ell(\ell+1) }{\hat{r}^2}  + \frac{1}{|E|^2}V\Big(\frac{\hat{r}}{E}\Big) \Big) \Big] u = \pm  u.
	\end{equation}
	This is of the form \cref{eq:1} for $z=\hat{r}$ and $h=|E|^{1/2}$, if we take 
	\begin{equation}
		\psi = \frac{\ell(\ell+1)}{\hat{r}^2} + \frac{1}{|E|^2}V\Big(\frac{\hat{r}}{E}\Big) = \frac{\ell(\ell+1)}{z^2} + \underbrace{\frac{1}{h^4} V \Big( \frac{z}{h^2} \Big)}_{\in \varrho_{\mathrm{fe}}^{-2}  C^\infty(M) }.
	\end{equation}
	We have a transition point with $\kappa=-1$ at $\hat{r}=0$. 
	Then, the definition of $M$ is such that $z/h^2$ parameterizes the interior of the front face of the blowup, and so that $\varrho_{\mathrm{fe}} = z+h^2$.
	So, $\psi$ has the form described in the previous subsection, with $z^2 \psi|_{\mathrm{be}} = \ell(\ell+1)$ and with $\Psi(\lambda) = \lambda^2 V(\lambda)$, and this is nonzero if $V$ is nonzero. 
\end{example}
For applications to massive wave propagation on black hole spacetimes, see \cite{BarrancoETAL}\cite{Moortel}, where the problem amounts asymptotics for the confluent Heun equation in some asymptotic regime. In fact, in \cite[\S IV.A]{BarrancoETAL}, Barranco et.\ al.\ comment on their inability to find, in the existing literature on special functions, asymptotics in the relevant regime. As the previous example shows, \Cref{thm2} provides such asymptotics away from the horizon (and it is then straightforward to get asymptotics near the horizon). 

\begin{example}
	Consider the anharmonic oscillator from \S\ref{subsec:anharmonic}, at zero energy, and now allow a $C_{\mathrm{c}}^\infty$ potential term:
	\begin{equation}
	P = - \frac{\partial^2}{\partial r^2} + \frac{\ell(\ell+1)}{r^2} + \mathsf{k} r^2 + \lambda r^4 + V(r), 
	\end{equation}
	where $\ell\in \bbN, \mathsf{k}\in \{-1,+1\}$ are fixed, $V\in C_{\mathrm{c}}^\infty(\bbR)$, and $\lambda>0$. We are interested here in the $\lambda \to 0^+$ limit. In \S\ref{subsec:anharmonic}, we studied this by introducing the coordinate $\hat{r} = \lambda^{1/2} r$, in terms of which 
	\begin{equation}
	\lambda P = - h^2 \frac{\partial^2}{\partial \hat{r}^2} + \mathsf{k} \hat{r}^2 + \hat{r}^4 + h^2\Big(\frac{\ell(\ell+1)}{\hat{r}^2}+ \underbrace{h^{-1}V\Big( \frac{\hat{r}}{h^{1/2}}\Big)}_{\in \varrho_{\mathrm{fe}}^{-2} C^\infty(M) }\Big) 
	\end{equation}
	where $h=\lambda$. The transition point at $\hat{r}=0$ has type $\kappa=2$; it is a double turning point. So, $\lambda P$ has the form we require of it near $\hat{r}=0$, and $\Psi(t)=t^2 V(t)$. Again, we see that $\Psi$ is nonzero if $V$ is.
\end{example}

\subsection{Statement of main theorem}
Let 
\begin{equation}
	\zeta(z) = \Big( \frac{\kappa+2}{2} \int_0^z \omega^{\kappa/2} \sqrt{W(\omega)} \dd \omega  \Big)^{2/(\kappa+2)}.
	\label{eq:zeta}
\end{equation}
This lies in $z C^\infty([0,Z]_z; \bbR^+)$. Let $\xi\in C^\infty([0,Z]_z; \bbR^+)$ be defined by $\zeta = z \xi$. 

For each $j\in \bbC$, we use $(j,0)$ as an abbreviation for the index set $\{(j+n,0):n\in \bbN\}$. In particular, $(0,0)$, which we also abbreviate ``$\bbN$,'' is the index set for which polyhomogeneity means smoothness. Recall also that ``$\infty$'' denotes the empty index set.

Letting $\calE_{\mathrm{ie}},\calE_{\mathrm{ze}},\calE_{\mathrm{fe}},\calE_{\mathrm{be}}$ denote index sets, we use $\calA^{\calE_{\mathrm{ie}},\calE_{\mathrm{ze}},\calE_{\mathrm{fe}},\calE_{\mathrm{be}}}(M) \subset C^\infty(M^\circ;\bbC)$
to denote the set of (complex-valued) polyhomogeneous functions on $M^\circ$ with index set $\calE_{\mathrm{e}}$ at each edge $\mathrm{e} \in \{\mathrm{ie},\mathrm{ze},\mathrm{fe},\mathrm{be}\}$. To avoid having to write too many index sets, we use $\calA^{\calE,\calF}(M)$ as an abbreviation for this set when $\calE_{\mathrm{ie}},\calE_{\mathrm{ze}}=(0,0)$, with $\calE = \calE_{\mathrm{fe}}$ and $\calF=\calE_{\mathrm{be}}$. When working with functions supported away from $\mathrm{be}$, or where the behavior there is unimportant, we will sometimes omit the `$\calF$' from ``$\calA^{\calE,\calF}(M)$.''

Below, we use the index set\footnote{This index set is likely not optimal. It seems in particular to overestimate the logarithmic terms when $\kappa \notin 2\bbN$. Our argument does not take into account some algebraic cancellations that seem to occur and kill the logarithmic terms in this case.}
\begin{equation} 
	\calE_0 =  
	(0,0)\cup
	\begin{cases}
		\bigcup_{j=1}^\infty ((\kappa+2)j-1,j) & (\kappa\notin 2\bbN), \\ 
		(\kappa/2,1)\cup \bigcup_{j=1}^\infty ((\kappa+2)j-1,2j) \cup ((\kappa+2)j+\kappa/2,2j+1) & (\kappa \in 2 \bbN).
	\end{cases}  
	\label{eq:calEdef}
\end{equation}

Let $\calQ\subset C^\infty(\bbR^+_\lambda)$ denote 
\begin{equation}
	\calQ  = \Big\{ v(\lambda)\in C^\infty(\bbR^+_\lambda): \frac{\partial^2 v}{\partial \lambda^2} = \Big[ \varsigma \lambda^\kappa + \frac{1}{\lambda^2} \Big( \alpha^2 - \frac{1}{4} \Big) + \frac{1}{\lambda^2} \Psi \Big( \frac{\lambda}{\sqrt[\kappa+2]{W(0)}} \Big) \Big]v(\lambda) \Big\}.
	\label{eq:quasimode}
\end{equation}
We call these $O(h^2/\zeta)$-quasimodes, for reasons discussed later.
See \S\ref{subsec:model}, \S\ref{subsec:examples} for the $\Psi=0$ case, wherein the elements of $\calQ$ end up being weighted Bessel functions; for example, when $\kappa=1$ and $\alpha=1/2$, then $\calQ$ consists of Airy functions.
\begin{theorem}
	For any $Q\in \calQ$, there exist 
	\begin{itemize} 
		\item an index set $\calG$,
		\item functions $\beta, \gamma \in \calA^{\calE_0}(M)$ with $\operatorname{supp} \beta,\operatorname{supp} \gamma$ disjoint from $\mathrm{be}$, and 
		\item $\delta \in \calA^{\calE_0,\calG}(M)$ with $\operatorname{supp} \delta \cap (\mathrm{ie}\cup \mathrm{ze})= \varnothing$ 
	\end{itemize}
	such that the function $u$ defined by
	\begin{equation}
		u = \sqrt[4]{\frac{\xi^\kappa}{W}} \, \bigg[ (1+\varrho_{\mathrm{ze}} \varrho_{\mathrm{fe}} \beta) Q\Big( \frac{\zeta}{h^{2/(\kappa+2)}} \Big) +  \varrho_{\mathrm{ze}}^{(\kappa+1)/(\kappa+2)} \varrho_{\mathrm{fe}} \gamma Q'\Big( \frac{\zeta}{h^{2/(\kappa+2)}} \Big)\bigg] + \varrho_{\mathrm{be}}^{1/2-\alpha} \varrho_{\mathrm{fe}} \delta 
		\label{eq:misc_16} 
	\end{equation}
	solves $Pu=0$ in $\{h<h_0\}$ for some $h_0>0$. 
	\label{thm2}
\end{theorem}

In \cite{NIST}, one can find numerous Poincar\'e-type expansions stated for various special functions that have the form \cref{eq:misc_16} with $Q$ essentially a Bessel function; see \S\ref{subsec:JWKB} for a discussion of the $\kappa=1$ case. For a examples with $\kappa=-1$, see \cite[\href{http://dlmf.nist.gov/33.20}{\S33.20}]{NIST}\cite[\S12.11]{OlverBook} (in which a recurrence relation is used to write $Q'$ in terms of other Bessel functions).
\Cref{thm2} is a serious generalization. In \Cref{thm:collapsed}, we state a stronger and simpler version of the previous theorem that applies to the cases in \cite{NIST}. It is, in fact, exactly these cases to which the arguments of the sort in \cite{OlverBook} apply. (However, \Cref{thm:collapsed} is a technical improvement upon the corresponding results in the existing literature, because it allows differentiation in the semiclassical parameter.)

\begin{remark}[The index set at $\mathrm{be}$]
	The index set $\calG$, describing the behavior of $\delta$ at $\mathrm{be}$ can be extracted from the argument below, but we will not be explicit. This does not substantially affect applicability, because in practice a sharp upper bound on $\calG$ is known, coming from an analysis of the ODE at the regular singular point $z=0$ for each individual $h>0$.
	
	Because $Q,Q'$ cannot vanish simultaneously without being identically zero, it is possible to absorb the $\delta$ term into the main term in \cref{eq:misc_16}, at the cost of enlarging the support of $\beta,\gamma$.
\end{remark}

Using \Cref{prop:quasimode_large_argument} to express $Q,Q'$ in exponential-polyhomogeneous form, 
\begin{multline}
	u -  \sqrt[4]{\frac{\xi^\kappa}{W}} Q\Big( \frac{\zeta}{h^{2/(\kappa+2)}} \Big) \in \exp \Big( - \frac{2\sqrt{\varsigma} \chi}{\kappa+2} \frac{\zeta^{(\kappa+2)/2}}{h}  \Big) \sqrt{\varrho_{\mathrm{ze}}}\varrho_{\mathrm{fe}} \calA^{\bbN,2^{-1} \bbN,\calE_0,\infty}(M)   \\ +  \exp \Big(  \frac{2\sqrt{\varsigma} \chi}{\kappa+2} \frac{\zeta^{(\kappa+2)/2}}{h}  \Big) \sqrt{\varrho_{\mathrm{ze}}}\varrho_{\mathrm{fe}} \calA^{\bbN,2^{-1}\bbN,\calE_0,\infty}(M) + \varrho_{\mathrm{fe}}\varrho_{\mathrm{be}}^{1/2-\alpha}\calA^{\infty,\infty,\calE_0,\calG}(M) 
	\label{eq:misc_022}
\end{multline}
for $\chi \in C^\infty(M;[0,1])$ identically $1$ near $\mathrm{ie}\cup \mathrm{ze}$ and identically vanishing near $\mathrm{be}$.

In particular, for $h$ sufficiently small, there exists a solution $u$ to the semiclassical ODE $Pu=0$ such that 
\begin{equation}
	u \approx 
	\sqrt[4]{\frac{\xi(z)^\kappa}{W(z)} } Q \Big(\frac{\zeta}{h^{2/(\kappa+2)}} \Big).
\end{equation}
This is just Langer's approximation \cite{Langer31}. If $|z|\gg h^{2/(\kappa+2)}$, then, as $h\to 0^+$, the large-argument asymptotics of the $Q\in \calQ$ allow $Q$ to be approximated by a linear combination of exponentials, in which case the approximation above becomes an instance of the Liouville--Green ansatz. If instead $|z| = O(h^{2/(\kappa+2)})$, then $\zeta \approx W(0)^{1/(\kappa+2)} z$, and so 
\begin{equation} 
	u\approx W(0)^{-(\kappa+1)/(4\kappa+8)} Q\Big(\frac{z W(0)^{1/(\kappa+2)}}{h^{2/(\kappa+2)}}\Big),
\end{equation} 
which is an ansatz generalizing that appearing in the JWKB connection analysis. (Since $Q$ can be replaced by $cQ$ for any $c\neq 0$, the multiplier out front is not important.)

If $Q_1,Q_2\in \calQ$ are linearly independent modes, then the functions $u[Q_1](-,h)$, $u[Q_2](-,h)$ produced by the previous theorem are linearly independent for $h$ sufficiently small, so any solution $u$ to $Pu=0$ can be written as $u = c_1 u[Q_1] + c_2 u[Q_2]$
for $h$ sufficiently small
for some functions $c_1,c_2 : (0,\infty)_h\to \bbC$. Given the values of $u|_{\mathrm{ie}}$ and the derivative $u'|_{\mathrm{ie}}$, the $h\to 0^+$ asymptotics of the coefficients $c_1,c_2$ can be computed straightforwardly from the asymptotics of the $Q\in \calQ$. The preceding theorem therefore contains within it the means of producing semiclassical expansions for \emph{any} such $u$ with prescribed initial data. Thus, \Cref{thm1} follows from \Cref{thm2}. For completeness, the deduction is included in the appendices. See \S\ref{sec:ivp}. 

\subsection{Sketch of proof}

The proof of \Cref{thm2} consists of four steps. All steps are carried out in the $W=1$ case, a simplification which suffices, as originally observed by Langer, for reasons discussed in \S\ref{sec:Langer}. Then, as preparation for the later sections, some properties of the quasimodes $Q\in \calQ$ are proven in \S\ref{sec:quasimodes}. 
The four steps of the main argument begin afterwards, each comprising one section:
\begin{enumerate}
	\item The first step of the main argument, carried out in \S\ref{sec:ze}, is to produce a solution to the ODE of the desired form modulo an error of the form $f Q +  g Q'$ for 
	\begin{equation} 
		f,g \in \varrho_{\mathrm{ze}}^\infty\varrho_{\mathrm{fe}} \calA^{\calE_0,\infty}(M),
	\end{equation} 
	i.e.\ accurate to infinitely many orders at $\mathrm{ze}$, and supported away from $\mathrm{be}$. The proof involves re-interpreting Langer--Olver's asymptotic series in \cite{OlverOriginal, OlverBook}, which fails to define a uniform expansion down to $\mathrm{be}^\circ$, as an asymptotic expansion at $\mathrm{ze}$, which can then be asymptotically summed in suitable function spaces. 
	\item The next step is to solve away the error from the previous step near $ \mathrm{ze}\cap \mathrm{fe}$ 
	modulo $\varrho_{\mathrm{ze}}^\infty\varrho_{\mathrm{fe}}^\infty  C^\infty(M) = h^\infty C^\infty([0,Z]_z\times [0,\infty)_h ;\bbC)$
	remainders supported away from $\mathrm{be}$. The argument  involves the inversion of the \emph{normal operator} $\smash{\hat{N}(P)}\in \operatorname{Diff}^2(\mathrm{fe}^\circ)$ in order to produce the approximate solution term-by-term. This step is in \S\ref{sec:fe}.
	\item In \S\ref{sec:error}, the remaining $h^\infty C^\infty$ error near $\mathrm{ze}$ is solved away using the standard method of variation of parameters \cite{Erdelyi}\cite[Chp. 6]{OlverBook}\cite[\S15.5]{SimonComplex}. This is the most technical part of the argument, but the error is as amenable as possible --- as long as we stay away from $\mathrm{be}$, the $O(h^\infty)$ suppression is sufficient to kill any large negative powers of $h,\zeta$ that arise in the computations. It is necessary to keep careful track of the exponential weights arising in the classically forbidden case, but this is the only delicate point. Estimating derivatives complicates the formulae but is totally straightforward.
	\item For various reasons, the previous steps allow an error supported away from $\mathrm{ze}$ (that is, ``near'' be). The final step of the argument is to solve away this error completely in $\{h<h_0\}$. Because $P$ is a nondegenerating, smooth family of regular singular differential operator on $\mathrm{fe}\backslash \mathrm{ze}$, this follows from the general theory of regular singular ODE and requires only a short argument. See \S\ref{sec:low}.
\end{enumerate}
The various parts of the argument are combined in \S\ref{sec:main}.

\section{The Langer diffeomorphism}
\label{sec:Langer}

As a first step in the proof of the main theorem, we follow Langer and Olver in the use of a variant of the ``Langer diffeomorphism'' --- see \cite{Langer31}\cite{OlverBook} ---  to reduce to $W=1$ case. This serves to simplify the computations. While not strictly necessary in the singular geometric approach, the analysis is somewhat shortened. 

Let $M[Z]$ denote the mwc constructed in the introduction and depicted in \Cref{fig}, where we are now making the dependence on $Z$ explicit. For any $Z_0>0$, $M[Z]\cong M[Z_0]$, so this is only a distinction at the level of sets. This construction was coordinate invariant in the following sense:

\begin{lemma}
	If $\zeta:[0,Z] \to [0,\zeta(Z)]$ is any diffeomorphism, then the diffeomorphism $[0,Z]_z\times [0,\infty)_{h^{2/(\kappa+2)}}\to [0,\zeta(Z)]_\zeta\times [0,\infty)_{h^{2/(\kappa+2)}}$
	given by $(z,h)\mapsto (\zeta(z),h)$ lifts to a diffeomorphism $M[Z]\to M[\zeta(Z)]$.
\end{lemma}

\begin{proof}
	The polar blowup of a corner is a coordinate invariant notion, so $\iota:(z,h)\mapsto (\zeta(z),h)$ lifts to a diffeomorphism 
	\begin{equation}
	[[0,Z]_z\times [0,\infty)_{h^{2/(\kappa+2)}} ; \{0\}_z\times \{0\}_h] \to [[0,\zeta(Z)]_\zeta\times [0,\infty)_{h^{2/(\kappa+2)}} ; \{0\}_\zeta\times \{0\}_h].
	\label{eq:misc_031}
	\end{equation}
	This is the desired diffeomorphism $M[Z]\to M[\zeta(Z)]$, but we do not yet know that it is a diffeomorphism, because, although the domain and codomain in \cref{eq:misc_031} agree with $M[Z],M[\zeta(Z)]$ at the level of sets, they differ in terms of smooth structure at $\mathrm{ze}$. 
	
	It therefore suffices to verify that the map remains a diffeomorphism after the changes of smooth structures at $\mathrm{ze}$. 
	We only check smoothness, as smoothness of the inverse map is proven analogously.
	
	Concretely, it suffices to check that the functions
	\begin{align}
	\begin{split}
	\varrho_{\mathrm{ze}}(\zeta(z),h) &= h^2(\zeta(z) + h^{2/(\kappa+2)})^{-(\kappa+2)}, \\ 
	\varrho_{\mathrm{fe}}(\zeta(z),h) &= \zeta(z) + h^{2/(\kappa+2)}
	\end{split} 
	\end{align}
	are smooth functions on $M[Z]$. Indeed, writing $\zeta(z) = z \xi(z)$, we have $\xi \in C^\infty([0,Z]_z;\bbR^+)$, and then the identities 
	\begin{align}
	\begin{split}
	\varrho_{\mathrm{ze}}(\zeta(z),h) &=
	\varrho_{\mathrm{ze}}(z,h)(\varrho_{\mathrm{be}}(z,h)  (\xi(z) - 1)+1)^{-(\kappa+2)}, \\
	\varrho_{\mathrm{fe}}(\zeta(z),h) &= \varrho_{\mathrm{fe}}(z,h) \varrho_{\mathrm{be}}(z,h) (\xi(z)-1) + \varrho_{\mathrm{fe}}(z,h)
	\end{split}
	\end{align}
	hold. The second of these is manifestly smooth on $M[Z]$. The first is a smooth function of $\varrho_{\mathrm{ze}},\varrho_{\mathrm{be}},z$ away from the set $\{\varrho_{\mathrm{be}} (\xi-1)= -1\}$, but this set is avoided on $M[Z]$, because $\xi(z)>0$ and $\varrho_{\mathrm{be}}(z,h) \in [0,1]$.
\end{proof}

We apply this lemma to the map defined by \cref{eq:zeta}. That this is a diffeomorphism $[0,Z]\to [0,\zeta(Z)]$ follows from 
\begin{equation}
	\frac{\mathrm{d} \zeta}{\mathrm{d} z} =  \xi^{-\kappa/2} \sqrt{W(z)} \in C^\infty([0,Z];\bbR^+). 
\end{equation}
Thus, the map $(z,h)\mapsto (\zeta(z),h)$ lifts to a diffeomorphism $M[Z]\to M[\zeta(Z)]$. This lift is what we refer to as the \emph{Langer diffeomorphism}. It will be used implicitly below.

Let $P(h)$ be as in the introduction; that is, $P(h) = - h^2 \partial_z^2 + \varsigma z^\kappa W(z) + h^2 \psi(z,h)$ for $\varsigma \in \{-1,+1\}$, $\kappa \in (-2,\infty)$, $W\in C^\infty([0,Z]_z;\bbR^+)$, and $\psi \in \varrho_{\mathrm{be}}^{-2}\varrho_{\mathrm{fe}}^{-2}C^\infty(M[Z])$ of the form specified in \cref{eq:psi_form}.
In terms of $\zeta$, the operator $P(h)$ can be written as 
\begin{align}
	\begin{split} 
	\frac{\xi^\kappa}{ W} 	P(h) &= -h^2 \frac{\partial^2}{\partial \zeta^2} + \frac{h^2}{2} \Big( \frac{\kappa}{\xi} \frac{\partial \xi}{\partial \zeta}-\frac{1}{W} \frac{\partial W}{\partial \zeta}\Big)\frac{\partial}{\partial \zeta} + \varsigma \zeta^\kappa   + \frac{h^2}{\zeta^2}\frac{\xi^{\kappa+2}}{W} \Big[\Big(\alpha^2-\frac{1}{4}\Big) + \Psi\Big( \frac{\zeta}{ \xi h^{2/(\kappa+2)}} \Big)\Big] + h^2 \tilde{E}  \\ 
	&\overset{\mathrm{def}}{=} P_0(h),
	\end{split}  
\end{align}
where $\alpha,\Psi$ are as in \cref{eq:psi_form} and $\tilde{E} =  \xi^\kappa E /W \in \varrho_{\mathrm{be}}^{-1} \varrho_{\mathrm{fe}}^{-1} C^\infty(M[\zeta(Z)])$. 
Viewing $P_0(h)$ as a differential operator on $[0,Z]_z$, $Pu=0$ is equivalent to $P_0 u=0$. 

In order to simplify the expression, and to facilitate comparison of $P_0$ with $P$, we can rearrange some terms; defining
\begin{multline}
E_0 = \tilde{E} + \frac{1}{\zeta^2} \Big( \frac{\xi^{\kappa+2}}{W} - 1 \Big)\Big[\Big(\alpha^2-\frac{1}{4}\Big) + \Psi\Big( \frac{\zeta}{ \xi h^{2/(\kappa+2)}} \Big)\Big] \\ + \frac{h^2}{\zeta^2} \Big[ \Psi\Big( \frac{\zeta}{ \xi h^{2/(\kappa+2)}} \Big) -\Psi\Big( \frac{\zeta}{  \sqrt[\kappa+2]{W(0) h^2} } \Big) \Big],
\label{eq:misc_lok}
\end{multline}
we have, owing to the observation that $\xi(0)^{\kappa+2} = W(0)$, that $E_0 \in \varrho_{\mathrm{be}}^{-1} \varrho_{\mathrm{fe}}^{-1} C^\infty(M[\zeta(Z)])$. Indeed, 
\begin{equation} 
(\xi^{\kappa+2}W^{-1}-1 ) \in z C^\infty([0,Z]_z;\bbC) \subset \varrho_{\mathrm{be}}\varrho_{\mathrm{fe}} C^\infty(M[\zeta(Z)]).
\end{equation} 
Since $\zeta^{-2} \in \varrho_{\mathrm{be}}^{-2} \varrho_{\mathrm{fe}}^{-2}C^\infty(M[\zeta(Z)])$ and 
\begin{equation} 
\Psi(\zeta \xi^{-1} h^{-2/(\kappa+2)}) \in \varrho_{\mathrm{be}}C^\infty(M[\zeta(Z)]),
\end{equation} 
overall the first line on the right-hand side of \cref{eq:misc_lok} is in $\varrho_{\mathrm{be}}^{-1} \varrho_{\mathrm{fe}}^{-1} C^\infty(M[\zeta(Z)])$. Similarly, 
\begin{equation}
 \Psi\Big( \frac{\zeta}{ \xi h^{2/(\kappa+2)}} \Big) -\Psi\Big( \frac{\zeta}{  \sqrt[\kappa+2]{W(0) h^2} } \Big)  \in \varrho_{\mathrm{be}} \varrho_{\mathrm{fe}} C^\infty(M[\zeta(Z)]), 
\end{equation}
so the second line is in $\varrho_{\mathrm{be}}^{-1} \varrho_{\mathrm{fe}}^{-1} C^\infty(M[\zeta(Z)])$ as well.

In terms of $E_0$, the operator $P_0$ can be written 
\begin{equation}
P_0 =  -h^2 \frac{\partial^2}{\partial \zeta^2} + \frac{h^2}{2} \Big( \frac{\kappa}{\xi} \frac{\partial \xi}{\partial \zeta}-\frac{1}{W} \frac{\partial W}{\partial \zeta}\Big)\frac{\partial}{\partial \zeta} + \varsigma \zeta^\kappa   + \frac{h^2}{\zeta^2} \Big[\Big(\alpha^2-\frac{1}{4}\Big) +\Psi\Big( \frac{\zeta}{  \sqrt[\kappa+2]{W(0) h^2} } \Big)\Big] + h^2 E_0 .
\end{equation}
As $P_0$ has a first-order term, unlike $P$, it is useful to consider the conjugation 
\begin{equation} 
P_1 = M_{W^{1/4} \xi^{-\kappa/4}} P_0 M_{W^{-1/4} \xi^{\kappa/4}},
\end{equation} 
where $M_\bullet$ denotes the multiplication operator $u\mapsto \bullet u$. A computation yields
\begin{equation}
P_1  = - h^2 \frac{\partial^2}{\partial \zeta^2} + \varsigma \zeta^\kappa   +  \frac{h^2}{\zeta^2}\Big(\alpha^2-\frac{1}{4}\Big) + \frac{h^2}{\zeta^2} \Psi\Big( \frac{\zeta}{  \sqrt[\kappa+2]{W(0) h^2} } \Big)  +  h^2 E_1,
\end{equation}
where
\begin{equation}
	E_1= E_0 + \frac{1}{4} \frac{\partial^2}{\partial \zeta^2} \log \Big(\frac{W}{\xi^\kappa}\Big) + \frac{1}{16} \Big( \frac{\partial}{\partial \zeta} \log \Big(\frac{W}{\xi^\kappa}\Big)\Big)^2 \in \varrho_{\mathrm{be}}^{-1}\varrho_{\mathrm{fe}}^{-1} C^\infty(M[\zeta(Z)];\bbC).
	\label{eq:misc_039}
\end{equation} 
The operator $P_1$ is therefore of the same form as $P$, in that it also satisfies the hypotheses of \Cref{thm2}.

So, if we know the result in the $W=1$ case, then we can apply it to $P_1$. 
Let $\calQ$ 
be defined by \cref{eq:quasimode}.  
Given any $Q\in \calQ$, \Cref{thm2} applied to $P_1$ gives a solution $u_1$ to $P_1 u_1=0$ of the form
\begin{equation}
u_1(\zeta,h) = (1+\varrho_{\mathrm{ze}} \varrho_{\mathrm{fe}} \beta) Q\Big( \frac{\zeta}{h^{2/(\kappa+2)}} \Big) +  \varrho_{\mathrm{ze}}^{(\kappa+1)/(\kappa+2)} \varrho_{\mathrm{fe}} \gamma Q'\Big( \frac{\zeta}{h^{2/(\kappa+2)}} \Big) + \varrho_{\mathrm{be}}^{1/2-\alpha} \varrho_{\mathrm{fe}} \delta
\end{equation}
for $\beta,\gamma,\delta$ as in the theorem. 
Define $u(z,h)$ by $u(z,h) = W^{-1/4} \xi^{\kappa/4} u_1(\zeta(z),h)$.
This satisfies $0=P_1u_1 = W^{1/4} \xi^{-\kappa/4} P_0 u $, so $P_0u=0$, and therefore $Pu=0$. 
The form of $u$ specified in \Cref{thm2} follows from the form of $u_1$. Thus, if we know the result in the $W=1$ case, then we can deduce the result in general. 

Below, we restrict attention to the case $W=1$, in which case $\zeta=z$, and we will mostly write $\zeta$ in place of $z$.

\section{$O(h^2/\zeta)$-quasimodes and their properties}
\label{sec:quasimodes}

Consider 
\begin{equation}
P = - h^2 \frac{\partial^2}{\partial \zeta^2} + \varsigma \zeta^\kappa + h^2 \psi  
\end{equation}
for $\psi$ as in \cref{eq:psi_form}, i.e.\ $\psi(\zeta,h) = \zeta^{-2}(\alpha^2-1/4) + \zeta^{-2} \Psi(\zeta/h^{2/(\kappa+2)})+ E$ for $\alpha\in \bbC$ with $\Re \alpha > 0$, $\Psi \in \varrho_{\mathrm{be}}|_{\mathrm{fe}}C^\infty(\mathrm{fe})$, and $E\in \zeta^{-1} C^\infty(M)$. Now, \cref{eq:NP_Def} reads
\begin{equation}
N(P) = -h^2 \frac{\partial^2}{\partial \zeta^2} + \varsigma \zeta^\kappa + \frac{h^2}{\zeta^2} \Big(\alpha^2 - \frac{1}{4} \Big) + \frac{h^2}{\zeta^2} \Psi\Big( \frac{\zeta}{h^{2/(\kappa+2)}} \Big) = P-h^2 E.
\end{equation} 
As discussed in the introduction, solutions $Q$ to $N(P)Q=0$ can be considered as $O(h^2/\zeta)$-quasimodes for the original semiclassical ODE $Pu=0$. This section is devoted to studying the properties of these quasimodes. 

Conversely, given $\alpha \in \bbC$ with $\Re \alpha\geq 0$ and $\Psi \in \varrho_{\mathrm{be}}|_{\mathrm{fe}}C^\infty(\mathrm{fe})$, we can consider the semiclassical ordinary differential operator defined by
\begin{equation}
	N_0 = -h^2 \frac{\partial^2}{\partial \zeta^2} + \varsigma \zeta^\kappa + \frac{h^2}{\zeta^2} \Big(\alpha^2 - \frac{1}{4} \Big) + \frac{h^2}{\zeta^2} \Psi\Big( \frac{\zeta}{h^{2/(\kappa+2)}} \Big).
\end{equation} 
Setting $P=N_0$, the operator $P$ satisfies the hypotheses of our setup, and $N(P)=N$. So, the set of all $N_0$ of the stipulated form is the set of all $N(P)$ arising as above. Forgetting $P$, the ODE under investigation is therefore $N_0Q=0$.
Per the preceding discussion, this can be thought of either as a first step in studying the original semiclassical problem or as a special case of the general problem. 

So, though this section is required preparation for the proofs in later sections, it can also be read as providing examples.

The key observation in studying $N_0$ is that it can be considered as a homogeneous family of operators on $\mathrm{fe}$. 
Changing coordinates from $\zeta$ to $\lambda = \zeta / \smash{h^{2/(\kappa+2)}}$, we have $N_0\propto N$ for $N$ the ordinary differential operator defined by
\begin{equation}
	N= - \frac{\partial^2}{\partial \lambda^2} + \varsigma \lambda^\kappa + \frac{1}{\lambda^2} \Big( \alpha^2 - \frac{1}{4} \Big)  + \frac{\Psi(\lambda)}{\lambda^2},
	\label{eq:barN}
\end{equation}
where `$\propto$' means proportionality up to a nonvanishing function of $h$. 
This computation shows that $N_0$, considered as a family of ordinary differential operators on the positive real axis, has kernel $\ker N_0=
\ker N$ independent of $h$.

Let $\calQ = \{Q\in C^\infty (\bbR^+_\lambda): N Q = 0 \}$ 
denote the kernel of $N$, thought of as a subset of $C^\infty (\bbR^+_\lambda)$ via $h$-independence. Cf.\ \cref{eq:quasimode}.

\subsection{Asymptotics of quasimodes}

Asymptotics for small argument:
\begin{proposition}
	If $Q\in \calQ$, then $Q(\lambda)\in \lambda^{1/2-\alpha} C^\infty[0,\infty)_\lambda +  \lambda^{1/2+\alpha} C^\infty[0,\infty)_\lambda$ if $\alpha\notin 2^{-1}\bbZ$.
	Otherwise, $Q(\lambda) \in \lambda^{1/2-\alpha} C^\infty [0,\infty)_\lambda + \lambda^{1/2+\alpha} \log(\lambda) C^\infty [0,\infty)_\lambda$. 
	In either case, there exists a nonzero 
	\begin{equation} 
	Q_0 \in \calQ\cap \lambda^{1/2+\alpha} C^\infty[0,\infty)_\lambda,
	\end{equation} 
	unique up to multiplicative constants. If $Q\in \calQ$ is such that $\lim_{\lambda\to 0^+}  \lambda^{-1/2+\alpha} Q = 0$, then $Q \in \operatorname{span}_\bbC Q_0$, and $Q = (\lim_{\lambda\to 0^+} \lambda^{-1/2-\alpha} Q(\lambda) ) Q_0$. 
	\label{prop:quasimode_small_argument}
\end{proposition}
\begin{proof}
	Recall\footnote{This is a special case of the theory of regular singular PDE developed in \cite[Chp.\ 5]{MelroseAPS}. The case of regular singular ODE is classical, especially for analytic coefficients --- see e.g.\ \cite[Chp.\ 10.3]{WhittakerWatson}\cite[\S5.4-5]{OlverBook}. The proof consists of two steps. First, a formal series solution is constructed, algebraically. Then, it is asymptotically summed, which produces a solution to the ODE modulo a forcing which decays faster than any power of $\lambda$ as $\lambda\to 0^+$, i.e.\ an $O(\lambda^\infty)$-quasimode. Then, standard methods which rely only on very weak properties of solutions (for instance an explicit Schwartz kernel) produce an $O(\lambda^\infty)$-solution to the forced ODE, which, when added to the original summed formal series, produces a true solution.} that a second-order regular singular ordinary differential operator $(\lambda \partial_\lambda)^2 + a(\lambda) \lambda \partial_\lambda + b(\lambda)$  has kernel contained in
	\begin{equation}
	\begin{cases}
	\lambda^{\gamma_-} C^\infty [0,\infty)_\lambda + \lambda^{\gamma_+} C^\infty [0,\infty)_\lambda = \calA^{(\gamma_-,0)\cup (\gamma_+,0)}[0,\infty)_\lambda & (\gamma_+-\gamma_-\notin \bbZ), \\
	\lambda^{\gamma_-} C^\infty [0,\infty)_\lambda + \lambda^{\gamma_+} \log(\lambda) C^\infty [0,\infty)_\lambda = \calA^{(\gamma_-,0)\cup (\gamma_+,1)}[0,\infty)_\lambda & (\text{otherwise}),
	\end{cases}
	\end{equation}
	where $\gamma_\pm$ are the indicial roots of the equation, i.e.\ the roots of the polynomial $\gamma^2 + a(0)\gamma + b(0)$, and we have chosen them such that $\Re \gamma_- \leq \Re \gamma_+$. Moreover, one of the members of a basis of solutions can be chosen to be in $\lambda^{\gamma_+} C^\infty [0,\infty)_\lambda$. 
	
	The operator $N$ defined by \cref{eq:barN} is regular singular at $\lambda=0$ (up to a factor of $\lambda^2$), which becomes clearer upon writing 
	\begin{equation}
	\lambda^2 N = - \Big( \lambda \frac{\partial}{\partial \lambda} \Big)^2 + \lambda \frac{\partial}{\partial \lambda} + \varsigma \lambda^{\kappa+2} + \alpha^2 - \frac{1}{4} + \Psi(\lambda).
	\label{eq:misc_048}
	\end{equation}
	Since $\Psi(0)=0$, the indicial polynomial is $\gamma^2 - \gamma - \alpha^2+1/4$, which has roots $\gamma_\pm = \pm \alpha + 1/2$. The difference $\gamma_+ - \gamma_- = 2\alpha$ is integral precisely when $\alpha\in 2^{-1} \bbZ$. Thus, the proposition is a corollary of this general theory. 
\end{proof}

So, every $Q(\lambda)\in \calQ$ is polyhomogeneous at $\lambda=0$, with index set $\calF(\alpha) = (1/2-\alpha,0) \cup (1/2+\alpha,0)$ if $\alpha\notin 2^{-1}\bbZ$ and $\calF(\alpha) = (1/2-\alpha,0) \cup (1/2+\alpha,1)$ otherwise.

Any nonzero element of $\operatorname{span}_\bbC\{Q_0\}$ is called the \emph{recessive} solution of the ODE. 

\begin{remark} 
As seen in the examples below, if $\Psi(\lambda) \in C^\infty[0,\infty)_{\lambda^{\kappa+2}}$, then the argument above can be sharpened, via the coordinate change $\Lambda=\lambda^{\kappa+2}$, to yield the absence of log terms as long as $\alpha \notin 2^{-1} (\kappa+2) \bbZ$.
\label{rem:logsharp}
\end{remark} 

Asymptotics for large argument:
\begin{proposition}
	If $Q\in \calQ$ and $\varsigma >0$, then $Q(\rho^{-2/(\kappa+2)})\in \exp(2(\kappa+2)^{-1} \rho^{-1}) \rho^{\kappa/(2\kappa+4)} C^\infty[0,\infty)_{\rho}$, and there exists a nonzero $Q_\infty \in \calQ$ such that 
	\begin{equation} 
	Q_\infty(\rho^{-2/(\kappa+2)})\in \exp(-2  (\kappa+2)^{-1} \rho^{-1}) \rho^{\kappa/(2\kappa+4)}C^\infty[0,\infty)_{\rho}.
	\label{eq:misc_kz1}
	\end{equation}
	If $\varsigma <0$, then there exist some $Q_\pm (\rho^{-1/(\kappa+2)})\in \exp(\pm 2i(\kappa+2)^{-1}\rho^{-1}) \rho^{\kappa/(2\kappa+4)}C^\infty[0,\infty)_{\rho}$
	such that $\calQ = \operatorname{span}_\bbC\{Q_-,Q_+\}$. For all three of $Q_-,Q_+,Q_\infty$, the leading order terms in the $C^\infty [0,\infty)_\rho$ factors at $\rho=0$ (referring to \cref{eq:misc_kz1} for $Q_\infty$) are all nonzero. If $\varsigma>0$ and 
	\begin{equation}
		 \lim_{\rho\to 0^+}  \rho^{-\kappa/(2\kappa+4)}\exp(-2(\kappa+2)^{-1} \rho^{-1})Q(\rho^{-2/(\kappa+2)})=0, 
	\end{equation}
	then $Q\in \operatorname{span}_\bbC Q_\infty$. 
	\label{prop:quasimode_large_argument}
\end{proposition}
\begin{proof}
	Rewriting $N$ in terms of $\rho=1/\lambda^{(\kappa+2)/2}$, the result, which is most easily derived from substituting $\lambda \partial_\lambda =  -2^{-1}(\kappa+2) \rho \partial_\rho$ into \cref{eq:misc_048}, is 
	\begin{equation}
	-\frac{4\lambda^2 N}{(\kappa+2)^2} = \Big( \rho \frac{\partial}{\partial \rho} \Big)^2 +\frac{2\rho}{\kappa+2} \frac{\partial}{\partial \rho} - \frac{4}{(\kappa+2)^2} \Big[ \frac{\varsigma}{\rho^2} + \alpha^2 - \frac{1}{4} + \Phi(\rho^2) \Big], 
	\end{equation} 
	where $\Phi(\varrho) = \Psi(\varrho^{-1/(\kappa+2)}) \in \smash{C^\infty[0,\infty)_{\varrho}}$. Note that this is \emph{not} a regular singular ODE at $\rho=0$, because of the $\varsigma/\rho^2$ term in the brackets. Removing this term, the remainder of the operator is regular singular. Consequently, we can appeal to Liouville--Green theory in the form it is presented in \cite[Chp. 7]{OlverBook} to conclude the proposition. 
	
	To wit, to convert the operator above into the form considered by Olver, let $\Lambda=1/\rho = \lambda^{(\kappa+2)/2}$, in terms of which 
	\begin{equation}
	-\frac{4 \lambda^2 N}{\Lambda^2 (\kappa+2)^2} = \frac{\partial^2}{\partial \Lambda^2} +\frac{\kappa }{\kappa+2}\frac{1}{\Lambda} \frac{\partial}{\partial \Lambda} - \frac{4}{(\kappa+2)^2} \Big[ \varsigma  + \frac{1}{\Lambda^2}\Big(\alpha^2 - \frac{1}{4} + \Phi\Big(\frac{1}{\Lambda^2}\Big)\Big) \Big],
	\end{equation}
	and the kernel of this operator is also $\calQ$. 
	So, the Liouville--Green expansion applies. Note that the coefficient of the first-order term $\kappa(\kappa+2)^{-1} \Lambda^{-1} \partial_\Lambda$ is $O(1/\Lambda)$ and so does not contribute to the phase. The phase $\varphi$ (chosen to be positive in the $\varsigma>0$ case) is therefore $\varphi = 2 (\kappa+2)^{-1} \Lambda$ (up to an arbitrary additive constant) and the decay rate is $\sim \Lambda^{-\nu}=\rho^\nu$ with $\nu = \kappa / (2\kappa+4)$. The conclusion of \cite[Chp. 7- Thm. 2.1]{OlverBook} applies with these parameters, and the statement of the proposition may be read off of it.
	A minor bibliographic note is that, in \cite[Chp. 7- Thm. 2.1]{OlverBook}, Olver assumes what in our context is the analyticity of $\Phi$, whereas we only assume smoothness. This assumption is absent in \cite[Chp. 7- \S1]{OlverBook} so is unnecessary for producing solutions to the ODE modulo Schwartz errors (relative to the desired exponential growth or decay in the $\varsigma>0$ case), and such errors may be solved away via the method of variation of parameters. A more involved variant of the same standard argument appears below, so we do not belabor the details. 
\end{proof}

If the coefficients of $N$ are real, then, in the $\varsigma<0$ case, $Q_0\neq Q_\pm$. However, in the $\varsigma>0$ case, $Q_0 = Q_\infty$ (in which case $N$ could be said to have a ``bound state'') is possible.

Combining the preceding two propositions, and identifying $\mathrm{fe}= [0,\infty)_\lambda \cup [0,\infty)_{1/\lambda^{(\kappa+2)/2}}$:
\begin{corollary}
	Let $Q\in \calQ$. 
	Letting $\chi \in C_{\mathrm{c}}^\infty(\bbR;\bbR)$ be identically $1$ near the origin, 
	\begin{itemize}
		\item if $\varsigma>0$, then $Q(\lambda) \in \exp( 2 (\kappa+2)^{-1}  \chi(1/\lambda) \lambda^{(\kappa+2)/2}) \calA^{(-\kappa/(2\kappa+4),0),\calF(\alpha)}(\mathrm{fe})$, where $(-\kappa/(2\kappa+4),0)$ is the index set at $\lambda=\infty$ and $\calF(\alpha)$ is the index set at $\lambda=0$. Also, 
		\begin{equation} 
		Q_\infty(\lambda) \in \exp( -2(\kappa+2)^{-1} \chi(1/\lambda)  \lambda^{(\kappa+2)/2}) \calA^{(-\kappa/(2\kappa+4),0),\calF(\alpha)}(\mathrm{fe}).
		\end{equation}
		\item If $\varsigma<0$, then instead $Q_\pm(\lambda) \in \exp( \pm 2 i (\kappa+2)^{-1}   \chi(1/\lambda) \lambda^{(\kappa+2)/2}) \calA^{(-\kappa/(2\kappa+4),0),\calF(\alpha)}(\mathrm{fe})$. 
	\end{itemize}
	So, $Q$ is of exponential-polyhomogeneous type on $\mathrm{fe}$.
\end{corollary}

As a reminder, $[0,\infty]_\lambda = [0,\infty)_\lambda \cup (0,\infty]_{1/\lambda}$.

\begin{proposition}
	Suppose that $f\in \calA^{\calE,\calF}[0,\infty]_\lambda$, where $\calE$ is the index set at $\lambda=\infty$ and $\calF$ is the index set at $\lambda=0$.
	Then, $f(\zeta/h^{2/(\kappa+2)}) \in \calA^{(0,0), (\kappa+2)^{-1}\calE ,(0,0),\calF}(M)$.
\end{proposition}
Here, $(\kappa+2)^{-1}\calE$ denotes the smallest index set containing $\{((\kappa+2)^{-1}j,k) : (j,k) \in \calE\}$. If $\kappa$ is an integer, then this is $\{((\kappa+2)^{-1}j,k) : (j,k) \in \calE\}$ itself.
\begin{proof}
	In terms of $\varrho_{\mathrm{ze}} = h^2 / (\zeta+h^{2/(\kappa+2)})^{\kappa+2}$, $\varrho_{\mathrm{fe}} = \zeta+h^{2/(\kappa+2)}$, and $\varrho_{\mathrm{be}} = \zeta /(\zeta+h^{2/(\kappa+2)})$, we have $\zeta/h^{2/(\kappa+2)} = \varrho_{\mathrm{be}} \varrho_{\mathrm{ze}}^{-1/(\kappa+2)}$, so  
	\begin{equation} 
	f(\zeta/h^{2/(\kappa+2)}) = f(\varrho_{\mathrm{be}} \varrho_{\mathrm{ze}}^{-1/(\kappa+2)}). 
	\end{equation} 
	As $\varrho_{\mathrm{ze}}$ is nonvanishing near $\mathrm{be}$, this suffices to imply that $f$ is polyhomogeneous with the claimed index sets $\calF$ at $\mathrm{be}$ and $(0,0)$ at $\mathrm{fe}$ away from $\mathrm{ie}\cup \mathrm{ze}$. To study the situation near $\mathrm{ie}\cup \mathrm{ze}$, consider $g(t) = f(1/t)$, which lies in $\calA^\calE[0,\infty)_t$. Thus,
	\begin{equation}
	f(\zeta/h^{2/(\kappa+2)}) = g(\varrho_{\mathrm{ze}}^{1/(\kappa+2)} \varrho_{\mathrm{be}}^{-1}).
	\end{equation}
	Since $\varrho_{\mathrm{be}}$ is nonvanishing near $\mathrm{ze}$, this suffices to imply that $f$ is polyhomogeneous with the claimed index sets $(\kappa+2)^{-1} \calE$ at $\mathrm{ze}$ and $(0,0)$ at $\mathrm{ie}$ and $\mathrm{fe}$ away from $\mathrm{be}$. 
\end{proof}

Cf.\ Melrose's pullback theorem.

There is a geometric interpretation of the previous proposition, which generalizes the discussion at the end of \S\ref{subsec:model}. We start with $M_1 = [0,\infty]_\lambda \times [0,\infty)_{h^{2/(\kappa+2)}}$ on which $f$, viewed as a function independent of $h$, satisfies $f\in \calA^{\calE,(0,0),\calF}(M_1)$, where the middle index set is that at $\{h=0\}$. Performing a polar blowup of the upper corner $\{\lambda=\infty,h=0\}$, we denote the result, after modifying the smooth structure at the front face of the blowup so that $h^2$ becomes a boundary-defining-function of its interior, 
\begin{equation}
M_2 = [M_1 ;\{\lambda=\infty,h=0\}  ; \kappa+2].
\end{equation}
We have $f\in \calA^{\calE,(\kappa+2)^{-1}\calE,(0,0),\calF}(M_2)$, where the index sets are at $\mathrm{cl}_{M_2} \{\lambda=\infty,h>0\}$, then the front face of the blowup, $\mathrm{cl}_{M_2} \{\lambda<\infty,h=0\}$, and $\{\lambda=0\}$, respectively. Let 
\begin{equation} 
M' = \mathrm{cl}_{M_2} \{\lambda h^{2/(\kappa+2)} \leq Z\} \subseteq M_2,
\end{equation} 
as in \S\ref{subsec:model}.
From $f\in \calA^{\calE,(\kappa+2)^{-1}\calE,(0,0),\calF}(M_2)$ follows immediately $f\in \calA^{(0,0),(\kappa+2)^{-1}\calE,(0,0),\calF}(M')$, where the first index set is at the curve $\mathrm{cl}_{M_2} \{\lambda h^{2/(\kappa+2)} = Z\}$.
Since the map 
\begin{equation} 
(0,Z)_\zeta\times (0,\infty)_h \ni (\zeta,h) \mapsto (\zeta/h^{2/(\kappa+2)},h )
\end{equation} 
extends to a diffeomorphism $M\to M'$ identifying $\{\lambda=0\} \subset M'$ with $\mathrm{be}$, $\mathrm{cl}_{M'}\{\lambda<\infty,h=0\}= \mathrm{cl}_{M_2}\{\lambda<\infty,h=0\}$ with $\mathrm{fe}$, the front face of the blowup with $\mathrm{ze}$, and the curve 
\begin{equation} 
\mathrm{cl}_{M_2} \{\lambda h^{2/(\kappa+2)} = Z\}
\end{equation} 
with $\mathrm{ie}$ --- refer to \Cref{fig_alt} --- the already deduced $f\in \calA^{(0,0),(\kappa+2)^{-1}\calE,(0,0),\calF}(M')$ is equivalent to the desired result.

Combining the propositions above:
\begin{corollary}
	If $Q\in \calQ$, then $Q(\zeta / h^{2/(\kappa+2)})$ is of nontrivial exponential-polyhomogeneous type on $M$.
\end{corollary}

\subsection{Examples}
\label{subsec:examples}

\begin{example}[Liouville--Green] 
	If $\kappa=0$, $\alpha=1/2$, and $\Psi=0$, then $N = - \partial_\lambda^2 + \varsigma$, so  $\calQ = \operatorname{span}_\bbC\{ \calQ_-=e^{- i \lambda}, \calQ_+=e^{+ i \lambda} \}$ if $\varsigma<0$ and $\calQ = \operatorname{span}_\bbC\{ \calQ_\infty= e^{- \lambda}, e^{ \lambda} \}$ if $\varsigma>0$. 
\end{example}
The next most classical case is:
\begin{example}[JWKB]
	If $\kappa=1$, $\alpha=1/2$, and $\Psi=0$, then $N = - \partial_\lambda^2 + \varsigma \lambda$, so $NQ=0$ is Airy's ODE (or its reflection across the origin, depending on the value of $\varsigma$). So, if $\varsigma>0$, then 
	\begin{equation} 
		\calQ = \{c_1 \operatorname{Ai}(\lambda)+c_2 \operatorname{Bi}(\lambda):c_1,c_2\in \bbC\}
	\end{equation} 
	is precisely the set of Airy functions on the positive real axis. If $\varsigma<0$, then similarly $\calQ = \{c_1 \operatorname{Ai}(-\lambda)+c_2 \operatorname{Bi}(-\lambda):c_1,c_2\in \bbC\}$. \Cref{prop:quasimode_large_argument}, applied to these cases, states the qualitative form of the large-argument asymptotic expansions  \cite[\href{http://dlmf.nist.gov/9.7}{\S9.7}]{NIST} of the Airy functions and their derivatives. In particular, the large-argument expansions of $\operatorname{Ai}(\lambda),\operatorname{Bi}(\lambda)$ are in integral powers of $\rho=\smash{1/\lambda^{3/2}}$.
\end{example}
Generalizing the previous two examples:
\begin{example}[Cf.\ \S\ref{subsec:model}]
	If $\Psi=0$, then
	\begin{equation}
	N =
	- \frac{\partial^2}{\partial \lambda^2} + \varsigma \lambda^\kappa + \frac{1}{\lambda^2} \Big( \alpha^2 - \frac{1}{4} \Big).
	\end{equation}
	As observed by Langer, the elements of $\calQ = \ker N$ can be written in terms of Bessel functions. (For comparison with the previous two examples, recall that, up to a polynomial weight, the trigonometric functions are Bessel functions of order $1/2$, and the Airy functions are Bessel functions of order $1/3$.) More precisely,
	\begin{equation}
	\calQ = \bigg\{\lambda^{1/2} I\bigg(\frac{2\lambda^{(\kappa+2)/2}}{\kappa+2} \bigg) : I(t)\text{ a solution to } t^2 \frac{\mathrm{d}^2 I}{\mathrm{d} t^2} + t \frac{\mathrm{d} I}{\mathrm{d} t} - ( \varsigma t^2 + \nu^2) I(t) = 0\bigg\}.
	\end{equation}
	The ODE satisfied by $I(t)$ is Bessel's ODE of order $\nu = 2\alpha/(\kappa+2)$, so $I$ is a (modified, if $\varsigma>0$, and unmodified otherwise) Bessel function.
	For this special the case, the conclusions of the previous propositions can be checked explicitly using the small- or large- argument expansions of the Bessel functions. The small argument expansions are in \cite[ \href{http://dlmf.nist.gov/10.8}{\S10.8}]{NIST}. The large argument expansions are \emph{Hankel's expansions} \cite[\href{http://dlmf.nist.gov/10.17.i}{\S10.17(i)}]{NIST}\cite[\href{http://dlmf.nist.gov/10.40.i}{\S10.40(i)}]{NIST}, as mentioned already in \S\ref{subsec:Bessel}. The expansions given in \cite{NIST} are Poincar\'e-type expansions of the Bessel functions and their first derivatives. Taken together, these suffice to imply smoothness in terms of $\rho=1/\lambda^{(\kappa+2)/2}$ at $\rho=0$, with control of higher derivatives coming from the ODE. Besides the argument used above, Hankel's expansions can be proven in many ways, e.g. extracted from the integral representations of the Bessel functions via the method of stationary phase.
\end{example} 
	
\subsection{The inhomogeneous model problem}
\label{subsec:inh}
	
We now study the forced ODE $Nu = f Q + g Q'$ for $f,g \in \calS(\bbR)\cap C_{\mathrm{c}}^\infty(0,\infty]$. That is, $f,g$ are Schwartz functions vanishing near (and past) the origin.

The solutions can be produced using the standard Schwartz kernel construction, which we now recall. Let $Q_1,Q_2 \in \calQ$ denote linearly independent elements of $\calQ$, in which case their Wronskian $\frakW \in \bbC$, which we define with the sign convention 
\begin{equation} 
\frakW = Q_1Q_2' - Q_1'Q_2,
\end{equation} 
is nonzero. Then, for each $\lambda'>0$, the function $K(\lambda,\lambda')=K[Q_1,Q_2](\lambda,\lambda') \in C^\infty(\bbR^+_\lambda\setminus\{\lambda'\})$ defined by 
\begin{equation} 
K(\lambda,\lambda') = \frac{1}{\frakW} 
\begin{cases}
Q_1(\lambda)Q_2(\lambda') & (\lambda>\lambda'), \\ 
Q_2(\lambda)Q_1(\lambda') & (\lambda<\lambda'), 
\end{cases}
\end{equation} 
solves $N K(\lambda,\lambda') = \delta(\lambda-\lambda')$, where $\delta \in \calD'(\bbR)$ denotes a Dirac $\delta$-function. Thus, for nice functions $F \in C^\infty[0,\infty)_\lambda$, including all $F\in C_{\mathrm{c}}^\infty(0,\infty)_\lambda$, the function $N^{-1}[Q_1,Q_2] F$ defined by 
\begin{equation} 
N^{-1}[Q_1,Q_2] F=K(\lambda,F) = \int_0^\infty K(\lambda,\lambda') F(\lambda') \dd \lambda'
\label{eq:misc_070}
\end{equation} 
solves $N(N^{-1}[Q_1,Q_2] F)=F$.
In this way, each choice of $Q_1,Q_2$ leads to a right inverse
$N^{-1}[Q_1,Q_2]: C_{\mathrm{c}}^\infty(0,\infty)_\lambda\to C^\infty(0,\infty)_\lambda$ of $N$. 

Note that $N^{-1}[Q_1,Q_2]$ depends on $Q_1,Q_2$ only through $\bbC^\times Q_1,\bbC^\times Q_2$.

\begin{proposition}
Let $Q\in \calQ$, and fix independent $\hat{Q},\tilde{Q}\in \calQ$; if $\varsigma>0$, then choose $\tilde{Q}=\calQ_\infty$ to be the exponentially decaying mode. Then, $N^{-1}[\tilde{Q},\hat{Q}]$ extends to a map
\begin{equation} 
\{f Q + gQ'+h  :f,g,h  \in \calS(\bbR)\cap C_{\mathrm{c}}^\infty(0,\infty]_\lambda \} \to C^\infty(0,\infty)_\lambda
\end{equation} 
such that \cref{eq:misc_070} holds for all $F$ in the codomain and such that the extension is a right inverse of $N$.
\end{proposition}
\begin{proof}
	Given the hypotheses, the function $K(\lambda,F)$ given by
	\begin{equation}
	K(\lambda,F) = \frac{\tilde{Q}(\lambda)}{\frakW}\int_0^\lambda \hat{Q}(s) F(s) \dd s + \frac{\hat{Q}(\lambda)}{\frakW}\int_\lambda^\infty \tilde{Q}(s) F(s) \dd s
	\label{eq:misc_079}
	\end{equation}
	is well-defined
	for $F = f Q + gQ' +h$ with $f,g,h$ as above, the integrals on the right-hand side being absolutely convergent. Let \begin{equation} 
		N^{-1}[\tilde{Q},\hat{Q}] F(\lambda)= K(\lambda,F).
	\end{equation} 
	This is a linear extension of $N^{-1}[\tilde{Q},\hat{Q}]$ to $\{f Q + gQ' +h :f,g,h  \in \calS(\bbR)\cap C_{\mathrm{c}}^\infty(0,\infty]_\lambda \}$. Differentiating $K(\lambda,F)$ in $\lambda$ works as before, so $N K(\lambda,F) = F(\lambda)$. 
\end{proof}

Denote the extension by $N^{-1}[\tilde{Q},\hat{Q}]$ as well.

One annoyance is that, even if $F\in C_{\mathrm{c}}^\infty(0,\infty)_\lambda$, then $v=N^{-1}[\tilde{Q},\hat{Q}]F$ need not be $O( \langle \lambda \rangle^{-\infty} \tilde{Q})$ as $\lambda\to\infty$, as, in this case,
\begin{equation}
\lim_{\lambda\to\infty} \tilde{Q} (\lambda)^{-1} v(\lambda) = \frac{1}{\frakW} \int_0^\infty  \hat{Q}(\lambda) F(\lambda) \dd \lambda, 
\label{eq:obs}
\end{equation}
the right-hand side having no good reason to vanish in general. But this is the only obstruction, so:

\begin{proposition}
	Fix $\Lambda_0>\lambda_0>0$. 
	Given $F$ of the form $F = fQ + g Q'$ for $f,g \in \calS(\bbR)\cap C_{\mathrm{c}}^\infty (\lambda_0,\infty]$, there exist $\beta, \gamma \in \calS(\bbR)\cap C_{\mathrm{c}}^\infty (0,\infty]$ such that the function $v \in C^\infty(0,\infty)_\lambda$ defined by 
	\begin{equation} 
	v(\lambda) = \beta(\lambda) Q(\lambda) + \gamma(\lambda) Q'(\lambda)
	\end{equation} 
	solves $Nv =  F + R$ for some $R\in C_{\mathrm{c}}^\infty(0,\infty)_\lambda$ with $\operatorname{supp} R \subseteq [\lambda_0,\Lambda_0]$.   
	\label{prop:normal_mapping}
\end{proposition}
\begin{proof}  
	Take $\hat{Q},\tilde{Q}$ as in the previous proposition. Since these are linearly independent (which implies linear independence when restricted to any nonempty interval), there exists a $R\in C_{\mathrm{c}}^\infty(0,\infty)_\lambda$ with $\operatorname{supp} R \subseteq [\lambda_0,\Lambda_0]$ such that: 
	\begin{align}
	\int_0^\infty  \tilde{Q}(\lambda) R(\lambda) \dd \lambda &= -\int_0^\infty  \tilde{Q}(\lambda) F(\lambda) \dd \lambda, 
	\intertext{and, if $Q$ is not exponentially growing, then} 
	\int_0^\infty  \hat{Q}(\lambda) R(\lambda) \dd \lambda &= -\int_0^\infty  \hat{Q}(\lambda) F(\lambda) \dd \lambda,
	\label{eq:misc_124}
	\end{align}
	the integrals on the right-hand side being absolutely convergent.

	Now define $v = N^{-1}[\tilde{Q},\hat{Q}] (F+R)$. Then, $Nv = F+R$. Observe that $v$ vanishes identically near $\lambda=0$. Indeed, for $\lambda<\lambda_0$,  we have
	\begin{equation}
	v(\lambda) = \frac{\hat{Q}(\lambda)}{\frakW} \int_0^\infty \tilde{Q}(s) (F(s)+R(s)) \dd s = 0. 
	\end{equation}
	
	By \Cref{prop:Mod1}, 
	\begin{multline}
		\Big\{\tilde{Q}(\lambda) \int_0^\lambda \hat{Q}(s) (F(s)+R(s)) \dd s,\;
		\hat{Q}(\lambda) \int_\lambda^\infty \tilde{Q}(s) (F(s)+R(s)) \dd s \Big\}\\ \subset (|Q(\lambda)|^2 + |Q'(\lambda)|^2)^{1/2}\calS(\bbR).
		\label{eq:misc_126}
	\end{multline}
	Indeed, 
	\begin{equation}
		\hat{Q}(\lambda) \int_\lambda^\infty \tilde{Q}(s) (F(s)+R(s)) \dd s \in (|Q(\lambda)|^2 + |Q'(\lambda)|^2)^{1/2}\calS(\bbR), 
	\end{equation}
	and, if $Q$ is not exponentially growing, then, if $\lambda$ is large enough,
	\begin{equation}
		\tilde{Q}(\lambda) \int_0^\lambda \hat{Q}(s) (F(s)+R(s)) \dd s = -\tilde{Q}(\lambda) \int_\lambda^\infty \hat{Q}(s) F(s) \dd s \in (|Q(\lambda)|^2 + |Q'(\lambda)|^2)^{1/2}\calS(\bbR),
	\end{equation}
	owing to \cref{eq:misc_124},  
	while if $Q$ \emph{is} exponentially growing, then, owing to the exponential decay of $\tilde{Q}$, we still have 
	\begin{equation}
		\tilde{Q}(\lambda) \int_I \hat{Q}(s) (F(s)+R(s)) \dd s \in (|Q(\lambda)|^2 + |Q'(\lambda)|^2)^{1/2}\calS(\bbR) 
		\label{eq:misc_129}
	\end{equation}
	for 
	\begin{itemize}
		\item $I=[0,\lambda/2]$, using that $\tilde{Q}(\lambda)|\hat{Q}(\lambda/2)|^2 \in (|Q(\lambda)|^2 + |Q'(\lambda)|^2)^{1/2}\calS(\bbR)$ (using the exponential decay of $\tilde{Q}$, the exponential growth of $Q$, and not using anything about $f,g$ other than the fact that they lie in $L^\infty$),
		\item $I=[\lambda/2,\lambda]$, using the Schwartz decay of $f,g$ (and a polynomial bound on $\tilde{Q} \hat{Q}$). 
	\end{itemize}
	Combining these cases, we get \cref{eq:misc_129} for $I=[0,\lambda]$. So, \cref{eq:misc_126} holds.
	
	From \cref{eq:misc_126}, we conclude that the functions $\beta = v Q^*(\lambda)/(|Q(\lambda)|^2 + |Q'(\lambda)|^2)$ and $\gamma = v Q'(\lambda)^*/(|Q(\lambda)|^2 + |Q'(\lambda)|^2)$ lie in $\calS(\bbR)\cap C_{\mathrm{c}}^\infty (0,\infty]$.  
	These are defined such that $v = \beta Q + \gamma Q'$. 
\end{proof}

\section{Quasimodes at $\mathrm{ze}$}
\label{sec:ze}

The basic strategy of Langer--Olver is to look for a solution $u$ to the ODE 
\begin{equation} 
	Pu = - h^2 \frac{\partial^2 u}{\partial \zeta^2} + \varsigma \zeta^\kappa u +  h^2 \psi u=0,
\end{equation} 
of the form 
\begin{equation}
	u = \beta Q\Big( \frac{\zeta}{h^{2/(\kappa+2)}}\Big) + h^{\kappa/(\kappa+2)} \gamma Q'\Big( \frac{\zeta}{h^{2/(\kappa+2)}}\Big) 
	\label{eq:uform}
\end{equation}
for $\beta,\gamma \in C^\infty((0,Z]_\zeta\times [0,\infty)_{h^2};\bbC)$ that are well-behaved as $\zeta\to 0^+$, where $Q$ is an arbitrary member of $\calQ$. In Langer's original work, well-behaved meant smooth down to $\zeta=0$, but something weaker, which ends up being polyhomogeneity on $M$, is required here. 

\begin{remark*} 
The power of $h$ in front of the $ Q'$ term in \cref{eq:uform} has been chosen so that, for  each fixed $\zeta>0$, both terms on the right-hand side  have the same order of magnitude as $h\to 0^+$. This makes the semiclassical structure of the argument below more apparent.
\end{remark*} 

Specifically, we will prove that, for some specific index set $\calE$ (not too large --- see \cref{eq:calEdef_good}):
\begin{proposition}
 	There exist functions $\beta_0,\gamma_0\in\calA^{\calE-1}(M)$ supported away from $\mathrm{be}$ such that, defining $u$ by 
 	\begin{equation}
 		u = (1+ \varrho_{\mathrm{ze}} \varrho_{\mathrm{fe}} \beta_0) Q\Big( \frac{\zeta}{h^{2/(\kappa+2)}}\Big)+ \varrho_{\mathrm{ze}}^{(\kappa+1)/(\kappa+2)} \varrho_{\mathrm{fe}} \gamma_0 Q'\Big( \frac{\zeta}{h^{2/(\kappa+2)}}\Big).
 	\end{equation} 
 	we have $Pu = f Q + g Q'$ for $f,g \in \varrho_{\mathrm{be}}^{-1} \varrho_{\mathrm{ze}}^\infty \varrho_{\mathrm{fe}}^\kappa \calA^{\calE}(M)$ and $g$ supported away from $\mathrm{be}$. 
 	\label{prop:ze_upshot}
\end{proposition}
The proof will take up the entirety of this section.

In the rest of the body of the paper, we will abbreviate $Q=Q(\zeta/h^{2/(\kappa+2)})$ and $Q'=Q'(\zeta/h^{2/(\kappa+2)})$, as the $\zeta/h^{2/(\kappa+2)}$ argument will be clear from context.

Applying $P$ to $u$ of the form in \cref{eq:uform}, the result is again a function of a similar form: 
\begin{equation}
P u = \begin{bmatrix}
Q \\ h^{\kappa/(\kappa+2)} Q'
\end{bmatrix}^\intercal
\bigg( - h^2 \frac{\partial^2}{\partial \zeta^2}  - 2
\begin{bmatrix}
0 & \varsigma \zeta^\kappa  \\ 
1 & 0 
\end{bmatrix} h \frac{\partial}{\partial \zeta}
-
\begin{bmatrix}
0 & 1 \\ 
0 & 0
\end{bmatrix}
\Big( 2\phi h^3 \frac{\partial}{\partial \zeta} + h \varsigma \kappa \zeta^{\kappa-1} +  h^3 \phi'  \Big)
+ h^2 E \bigg) \begin{bmatrix}
\beta \\ \gamma 
\end{bmatrix},
\label{eq:misc_80}
\end{equation}
where, as earlier, $\phi(z,h) = \psi(z,h) - E(z,h)$ (for $E$ as in \cref{eq:psi_form}), and where $\phi'(z,h)=\partial_z \phi(z,h)$.

Because $\phi \in \varrho_{\mathrm{be}}^{-2}\varrho_{\mathrm{fe}}^{-2} C^\infty(M)$, we have $\phi' \in \varrho_{\mathrm{be}}^{-3}\varrho_{\mathrm{fe}}^{-3} C^\infty(M)$.

This suggests attempting to choose $\beta$ and $\gamma$ so as to satisfy the system of ODEs $LU =0$, where $U = (\beta,\gamma)$ and $L\in\operatorname{Diff}_\hbar^2((0,Z]_\zeta;\bbC^2)$ is the matrix-valued semiclassical operator 
\begin{equation} 
L=- h^2 \frac{\partial^2}{\partial \zeta^2}  - 2
\begin{bmatrix}
0 & \varsigma \zeta^\kappa  \\ 
1 & 0 
\end{bmatrix} h \frac{\partial}{\partial \zeta}
-
\begin{bmatrix}
0 & 1 \\ 
0 & 0
\end{bmatrix}
\Big( 2\phi h^3 \frac{\partial}{\partial \zeta} + h \varsigma \kappa \zeta^{\kappa-1} +  h^3 \phi'  \Big)
+ h^2 E
\label{eq:misc_021}
\end{equation}
appearing in \cref{eq:misc_80}.

\begin{remark} 
Given $\beta,\gamma\in C^\infty((0,Z)_z\times (0,\infty)_h)$ such that \cref{eq:uform} holds, it is not true that $Pu=0 \Rightarrow LU =0$ for $U=(\beta,\gamma)$, at least without specifying more about $\beta,\gamma$. 

For each $h>0$, the kernel of $L(h)$ is $4$-dimensional. The map $U\mapsto u$, where $u$ is defined in terms of $\beta,\gamma$ by \cref{eq:uform}, sends $\operatorname{ker} L(h)$ to $\operatorname{ker} P(h)$, and it is easily seen that this map has full rank. So, for every individual $h>0$, any $u(-,h)\in \operatorname{ker} P(h)$
can be decomposed as in \cref{eq:uform} for some functions $\beta(-,h),\gamma(-,h) \in C^\infty (0,Z]$ such that the function $U(-,h)$ defined by $U(-,h)=(\beta(-,h),\gamma(-,h))$
satisfies $L(h)U(-,h)=0$.
Thus, passage to the vector-valued ODE $LU=0$ is without loss of generality as far as constructing elements of $\ker P$ is concerned. 

Curiously, $L$ is independent of $Q$; it is only the map $U\mapsto u$ that is $Q$-dependent.  
This is related to the semiclassical structure of $L$, which we examine in \S\ref{subsec:semiclassical}. This subsection can be skipped on first reading. 
\end{remark}

If $U$ satisfies the ODE $LU=0$ on the nose, then the function $u$ defined by \cref{eq:uform} satisfies $Pu=0$ on the nose, but this is too much to ask for right away, so instead we try to solve the equation up to small errors.
We might like an $O(h^\infty)$ error, which means that the error should be Schwartz at both edges $\mathrm{ze}$ and $\mathrm{fe}$ comprising the $h\to 0^+$ regime, in which case it is simply Schwartz at $\{h=0\}$ in $[0,Z]_\zeta\times[0,\infty)_{h^2}$, but this is still too much to ask.
In this section, we aim only for an $O(\varrho_{\mathrm{ze}}^\infty)$ error, meaning in 
\begin{equation}
\varrho_{\mathrm{ze}}^\infty \varrho_{\mathrm{fe}}^\kappa \calA^\calE(M)\oplus \varrho_{\mathrm{ze}}^\infty \varrho_{\mathrm{fe}}^{\kappa/2}  \calA^\calF(M) = \bigcap_{k\in \bbN }(\varrho_{\mathrm{ze}}^k \varrho_{\mathrm{fe}}^\kappa \calA^\calE(M) \oplus \varrho_{\mathrm{ze}}^k \varrho_{\mathrm{fe}}^{\kappa/2} \calA^\calF(M))
\label{eq:misc_540}
\end{equation}
for some to-be-determined index sets $\calE,\calF$. (The powers of $\varrho_{\mathrm{fe}}$ in \cref{eq:misc_540} can be absorbed into a redefinition of $\calE,\calF$. The convention above is used to match later notation.)
That is, the error in our quasimode construction will be Schwartz only at $\mathrm{ze}$, but uniformly so all the way up to $\mathrm{fe}$, in a precise sense.

First, we work on the ray $\mathrm{ze}\backslash \mathrm{fe}$, meaning that we will not worry about uniformity up to $\mathrm{fe}$ (but uniformity at the other endpoint $\mathrm{ze}\cap\mathrm{ie}$ is handled trivially). This ``formal'' part of the argument is essentially found in \cite[\S12.14]{OlverBook}, but we present a full exposition here, in the relevant level of generality.  The formal part is contained in \S\ref{subsec:quasimode_ze/fe}. The situation at the other endpoint, $\mathrm{ze}\cap \mathrm{fe}$, will be analyzed second, with uniformity up to $\mathrm{fe}$ a consequence. The argument is outlined in \S\ref{subsec:quasimodes_zefe_sketch}, and the details are in \S\ref{subsec:quasimodes_zefe}.

\subsection{Semiclassical structure}
\label{subsec:semiclassical}
We refer to \cite{Zworski} for unfamiliar terminology.

Though $L(h)$ is elliptic for each individual $h>0$, $L$ is not elliptic \emph{as a semiclassical operator}. This is true even in the ``classically forbidden case'' $\varsigma>0$, in which case $P$ itself \emph{is} semiclassically elliptic. The semiclassically principal part $L_0$ of $L$ consists of the first two terms in \cref{eq:misc_021}, 
\begin{equation}
L_0=- h^2 \frac{\partial^2}{\partial \zeta^2}  - 2
\begin{bmatrix}
0 & \varsigma \zeta^\kappa  \\ 
1 & 0 
\end{bmatrix} h \frac{\partial}{\partial \zeta}.
\end{equation}
(Here we are ignoring behavior as $\zeta\to 0^+$. If $\phi=0$, then $L\in S \operatorname{Diff}^2_\hbar([0,Z]_\zeta;\bbC^2)$, and the principal part is principal in this stronger sense.)
Let ${}^\hbar T^* [0,Z]$ denote the semiclassical cotangent bundle over $[0,Z]$, which we parameterize by the map $(h,\zeta,\mu)\mapsto h^{-1} \mu \dd \zeta$, so $\mu$ is the frequency coordinate dual to the sole variable $\zeta$. Then, 
\begin{equation} 
{}^\hbar T^* [0,Z]\cong [0,\infty)_h\times [0,Z]_\zeta\times \bbR_\mu.
\end{equation}
In terms of these coordinates, the semiclassical principal symbol $\sigma_\hbar^2(L) : {}^\hbar T^* [0,Z] \to \bbC$ is 
\begin{equation}
\sigma_\hbar^2(L) =  \mu^2 \operatorname{Id} - 2 i \mu  \begin{bmatrix}
0 & \varsigma \zeta^\kappa \\ 
1 & 0 
\end{bmatrix},
\end{equation}
gotten by replacing $h\partial_\zeta$ with $i\mu$ in $L_0$. 
The semiclassical characteristic set $\operatorname{char}^2_\hbar(L) \subset {}^\hbar T^* [0,Z]$ is by definition the set of points $(h=0,\zeta,\mu)$ at which $\sigma_\hbar^2(L)(0,\zeta,\mu)\in \bbC^{2\times 2}$ fails to be invertible. This occurs when the determinant 
\begin{equation} 
\det \sigma_\hbar^2(L) = \mu^4 + 4 \varsigma \mu^2 \zeta^\kappa = \mu^2(\mu^2 + 4 \varsigma \zeta^\kappa)
\end{equation}
vanishes. 
This includes the zero section $\{\mu=0\}$, and, if $\varsigma<0$, then it also includes the sets $\{\mu = \pm 2 \zeta^{\kappa/2}\}$. 
The interpretation of $\sigma_\hbar^2(L)$ is that it describes the possible oscillations $\exp(i h^{-1} \varphi(\zeta) )$ present in the $h\to 0^+$ asymptotics of solutions to $LU = 0$, with the allowed $\varphi$ (only defined locally) being such that $(h=0,\zeta, \varphi'(\zeta)) \in \sigma^2_\hbar(L)$ for each $\zeta>0$.  
The key feature of $L$ that allows it to have polyhomogeneous solutions is therefore that its characteristic set includes the zero section $\{\mu=0\}$, since this corresponds to \emph{non}-oscillatory (non-exponential) asymptotics.

See \Cref{fig:characteristic} for an illustration of the characteristic set in the $\kappa=1,3$ cases, and see \Cref{fig:trapping} for the $\kappa=-1,2$ cases.

The additional characteristic set in the classically allowed case $\varsigma<0$ signals the existence of highly oscillatory solutions to $LU=0$ with an initial value that is non-oscillatory as $h\to 0$.
Consider solutions $u[Q_\pm]$ to $Pu[Q_\pm]=0$ with the  form \Cref{thm2} for $Q_\pm$ in place of $Q$, so that, for $h$ sufficiently small (so as to avoid any zeroes of $Q_\pm(Z/h^{2/(\kappa+2)})$),  
\begin{equation}
u[Q_\pm](Z,h) =(1+o(1)) Q_\pm\Big(\frac{Z}{h^{2/(\kappa+2)}}\Big) 
\end{equation}
as $h\to 0^+$.
The surjectivity of the $Q$-dependent map $\ker L\ni U \mapsto u\in \operatorname{ker} P$ means that we can write $u[Q_\mp]$ in the form \cref{eq:uform} for $Q=Q_\pm$ (not $Q=Q_\mp$!) and nonzero $U_\mp\in \operatorname{ker} L$:
\begin{equation}
u[Q_\mp] = 
\begin{bmatrix}
Q_\pm \\ h^{\kappa/(\kappa+2)} Q'_\pm 
\end{bmatrix}^\intercal U_\mp.
\end{equation}
By assumption, $u[Q_\mp]$ is highly oscillatory as $h\to 0^+$, with roughly the same phase as $Q_\mp$ and therefore with the \emph{opposite} phase as $Q_\pm$. It follows that $U_\mp$ must be oscillating with twice the phase with which $Q_\mp$ is oscillating. 
Thus, there exist highly oscillatory solutions $U$ to $LU=0$, namely $U_\pm$, but these are not the solutions constructed via asymptotic series below, which are smooth at $h=0$, away from the transition point.

The situation is similar to the \emph{conjugated} perspective of Vasy \cite{VasyLA,VasyN0L} in microlocal scattering theory (where the setting was the sc-calculus of Parenti--Shubin--Melrose rather than the semiclassical calculus, but these are similar in many respects), but there the characteristic set has only one nonzero ``branch'' present, whereas $\smash{\sigma_\hbar^2(L)}$ has both branches $\{\mu =\pm 2 |\zeta|^{\kappa/2}\}$ present, as depicted in the figures above. This difference is unsurprising, as $L$ (unlike Vasy's conjugated operator) does not depend on the choice of $Q\in \calQ$, so no branch can be singled out.

\begin{figure}[t]
	\begin{tikzpicture}
	\begin{axis}[xlabel=$\zeta$, ylabel=$\mu$, grid=both, grid style = {dotted},
	minor tick num=2,
	width = .5\textwidth,
	height = .35\textwidth,
	xmin = -1, 
	xmax=1,
	scaled ticks=true,
	clip=true,
	]
	\addplot[color=darkred, mark=none, thick, domain=-2:2, samples=100] ({-x^2/4},{x});
	\addplot[color=darkred, mark=none, thick] ({x},{0});
	\addlegendentry{\raisebox{5pt}{$\Big\{ \operatorname{det}  
			\Big[ 
			\begin{smallmatrix}
			\mu^2 & -2 i \zeta \mu \\
			-2i \mu & \mu^2
			\end{smallmatrix}\Big] = 0
			\Big\}$}};
	\node[right] at (axis cs: -.975,.75) {$\color{darkred}\{\mu = 2 \sqrt{-\zeta}\}$};
	\node[right] at (axis cs: -.3,-1.34) {$\color{darkred}\{\mu = -2 \sqrt{-\zeta}\}$};
	\node[below] at (axis cs: .5,0) {$\color{darkred}\{\mu=0\}$};
	\filldraw[darkred] (axis cs: 0,0) circle (1.5pt);
	\end{axis} 
	\end{tikzpicture}
	\begin{tikzpicture}
	\begin{axis}[xlabel=$\zeta$, grid=both, grid style = {dotted},
	minor tick num=2,
	width = .5\textwidth,
	height = .35\textwidth,
	xmin = -1, 
	xmax=1,
	scaled ticks=true,
	clip=true,
	]
	\addplot[color=darkred, mark=none, thick, domain=-2:2, samples=100] ({-x^2/4},{x^3});
	\addplot[color=darkred, mark=none, thick] ({x},{0});
	\addlegendentry{\raisebox{5pt}{$\Big\{ \operatorname{det}  
			\Big[ 
			\begin{smallmatrix}
			\mu^2 & -2 i \zeta^3 \mu \\
			-2i \mu & \mu^2
			\end{smallmatrix}\Big] = 0
			\Big\}$}};
	\node[right] at (axis cs: -.4,3) {$\color{darkred}\{\mu = 2 |\zeta|^{3/2}\}$};
	\node[right] at (axis cs: -.75,-6) {$\color{darkred}\{\mu = -2 |\zeta|^{3/2}\}$};
	\node[below] at (axis cs: .5,0) {$\color{darkred}\{\mu=0\}$};
	\filldraw[darkred] (axis cs: 0,0) circle (1.5pt);
	\end{axis} 
	\end{tikzpicture}
	\caption{
		The semiclassical characteristic set ${\color{darkred}\operatorname{char}_h^2(L) }= \{\mu=0\text{ or }\mu^2 + 4\varsigma\zeta^\kappa=0\} \subset \{h=0\} \cap {}^\hbar T^* (-1,+1)$ of $L$, in the $\kappa=1,3$ cases, with $\varsigma=1$; we have extended $L$ to $[-Z,+Z]$ so as to include both cases of $\varsigma$ in one figure. Various subsets of the characteristic set have been labeled. When $\kappa\geq 3$, the characteristic set has a cusp singularity. When $\kappa=1$, it has a fold singularity (i.e.\ \emph{caustic}) with respect to the projection onto the base.
	}
	\label{fig:characteristic}
\end{figure}

\subsection{Construction away from $\mathrm{fe}$: formalities}
\label{subsec:quasimode_ze/fe}
Since $\phi$ is smooth at $\mathrm{ze}$, and since $h^2$ serves as a defining function for $\mathrm{ze}\backslash \mathrm{fe}$ in $M\backslash \mathrm{fe}$, we can Taylor expand 
\begin{align}
\phi &\sim \sum_{k=0}^\infty h^{2k} \phi_k(\zeta) \in C^\infty(\mathrm{ze}\backslash \mathrm{fe})[[h^2]]
\intertext{at $\mathrm{ze}\backslash \mathrm{fe}$. For later use, note that, because $\phi \in \varrho_{\mathrm{be}}^{-2}\varrho_{\mathrm{fe}}^{-2} C^\infty(M)$, we have $\zeta^{2+k(\kappa+2)} \phi_k \in C^\infty(\mathrm{ze})$. Since $E$ is smooth at $\mathrm{ze}$, we can similarly expand}
E &\sim \sum_{k=0}^\infty h^{2k} E_k(\zeta) \in C^\infty(\mathrm{ze}\backslash \mathrm{fe})[[h^2]].
\end{align}
Since $E\in \varrho_{\mathrm{be}}^{-1}\varrho_{\mathrm{fe}}^{-1} C^\infty(M)$, we have $\zeta^{1+k(\kappa+2)} E_k \in C^\infty(\mathrm{ze})$.

Consider the formal version of $L$,
\begin{multline}
\mathsf{L} = - h^2 \frac{\partial^2}{\partial \zeta^2}  - 2 
\begin{bmatrix}
0 & \varsigma \zeta^\kappa \\ 
1 & 0 
\end{bmatrix} h \frac{\partial}{\partial \zeta} 
- h 
\begin{bmatrix}
0 & \varsigma \kappa  \zeta^{\kappa-1} \\
0 & 0
\end{bmatrix}
- h^3 
\begin{bmatrix}
0 & 1 \\ 0 & 0
\end{bmatrix}\sum_{k=0}^\infty h^{2k} \Big( 2 \phi_k \frac{\partial}{\partial \zeta} + \phi_k'  \Big) 
\\  + h^2 \sum_{k=0}^\infty h^{2k}  E_k(\zeta). 
\end{multline} 
This is an element of $\operatorname{Diff}^2(\mathrm{ze}\backslash \mathrm{fe};\bbC^2)[[h]]$. 
The task before us is to construct $\mathsf{U} \in C^\infty(\mathrm{ze}\backslash \mathrm{fe};\bbC^2)[[h]]$ satisfying $\mathsf{L}\mathsf{U} = 0$.

Consider the following ansatz for $\mathsf{U}$,:
\begin{equation}
\mathsf{U} = 
\sum_{k=0}^\infty h^{2k} \begin{bmatrix}
\beta_{k} \\ 0
\end{bmatrix}
+
\sum_{k=0}^\infty h^{2k+1} \begin{bmatrix}
0 \\ \gamma_{k}
\end{bmatrix}, \qquad \beta_{k},\gamma_{k} \in C^\infty(\mathrm{ze}^\circ).
\label{eq:Uform}
\end{equation}
This is the Langer--Olver ansatz. 

Given $\mathsf{U}$ of this form, the formal ODE $\mathsf{L}\mathsf{U}$ is equivalent to the conjunction of 
\begin{equation}
\frac{\mathrm{d}\beta_0}{\dd \zeta} = 0
\end{equation}
and
\begin{equation}
2 \frac{\mathrm{d}\beta_{k+1}}{\mathrm{d}\zeta} = - \frac{\mathrm{d}^2\gamma_{k}}{\mathrm{d}\zeta^2} + \sum_{j=0}^k  E_j \gamma_{k-j}
\end{equation}
and 
\begin{equation}
2 \varsigma \zeta^\kappa\frac{\mathrm{d}\gamma_{k}}{\mathrm{d} \zeta} +\varsigma \kappa \zeta^{\kappa-1} \gamma_{k}    = - \frac{\mathrm{d}^2\beta_{k}}{\mathrm{d}\zeta^2} + \sum_{j=0}^k  E_j \beta_{k-j} -  \sum_{j=0}^{k-1}  \Big( 2 \phi_j \frac{\mathrm{d} \gamma_{k-j-1}}{\mathrm{d} \zeta} + \phi'_j \gamma_{k-j - 1} \Big)
\end{equation}
holding for all $k\in \bbN$. 
Integrating these recursion relations yields the equivalent integral formulas: 
\begin{equation} 
\beta_0=1+c_0,
\end{equation} 
\begin{equation}
	\beta_{k+1}(\zeta) =  \frac{1}{2} \int_\zeta^Z \Big(\frac{\dd^2 \gamma_{k}(\omega) }{\dd \omega^2} - \sum_{j=0}^kE_j(\omega) \gamma_{k-j}(\omega) \Big) \dd \omega + c_{k},
	\label{eq:rec1} 
\end{equation}
\begin{multline}
\gamma_{k}(\zeta) = \frac{\varsigma}{2 \zeta^{\kappa/2}}\bigg[\int_\zeta^Z \bigg(  \frac{\mathrm{d}^2\beta_{k}(\omega)  }{\mathrm{d} \omega^2}- \sum_{j=0}^k E_j(\omega) \beta_{k-j}(\omega)  \\ 
+ \sum_{j=0}^{k-1} \Big( 2 \phi_j(\omega) \frac{\mathrm{d} \gamma_{k-j-1}(\omega)}{\mathrm{d} \omega} + \phi'_j(\omega) \gamma_{k-j - 1}(\omega) \Big) \bigg) \frac{\dd \omega}{\omega^{\kappa/2}} + C_{k}\bigg], 
\label{eq:rec2}
\end{multline}
where $c_{k},C_{k}\in \bbC$ are arbitrary constants of integration. These constants of integration have to do with the fact that if  $U=\{U(-,h)\}_{h>0}$ solves the semiclassical ODE $LU=0$, then so does $(1+c(h)) U$ for any function $c:[0,\infty)_h\to \bbC$. Analogously, if $\mathsf{U}$ solves $\mathsf{L}\mathsf{U}=0$, then so does $\mathsf{c}\mathsf{U}$ for any $\mathsf{c}\in \bbC[[h^2]]$. If $\mathsf{U}$ is constructed as above, then it can be checked that $\mathsf{c}\mathsf{U}$ arises from a different choice of $c_k$. Below, we take $c_k=0$ for all $k$.

The interpretation of $C_k$ is similar but more complicated. We will choose $C_k$ so as to minimize the index sets appearing in the $\zeta\to 0^+$ expansions of the $\gamma_k(\zeta)$. In other words, $C_k$ is to be chosen such that 
\begin{equation}
\int_\zeta^Z \bigg(  \frac{\mathrm{d}^2\beta_{k}(\omega)  }{\mathrm{d} \omega^2}- \sum_{j=0}^k E_j(\omega) \beta_{k-j}(\omega)  + \sum_{j=0}^{k-1} \Big( 2 \phi_j \frac{\mathrm{d} \gamma_{k-j-1}}{\mathrm{d} \omega} + \phi'_j \gamma_{k-j - 1} \Big) \bigg) \frac{\dd \omega}{\omega^{\kappa/2}} +C_k
\end{equation} 
has no constant term in its polyhomogeneous expansion in $\zeta$ at $\zeta=0$.

The equations above give a recursive definition for $\mathsf{U} \in C^\infty(\mathrm{ze}^\circ)[[h]]$. By construction, $\mathsf{L}\mathsf{U}=0$.

\begin{figure}[t]
	\begin{tikzpicture}
	\begin{axis}[xlabel=$\zeta$, ylabel=$\mu$, grid=both, grid style = {dotted},
	minor tick num=2,
	width = .5\textwidth,
	height = .35\textwidth,
	xmin = -3, 
	xmax=3,
	ymin = -10,
	ymax = 10,
	scaled ticks=true,
	clip=true,
	]
	\addplot[color=darkred, mark=none, thick, domain=1:10, samples=100] ({-4/(\x*\x)},{\x});
	\addplot[color=darkred, mark=none, thick, domain=1:10, samples=100] ({-4/(\x*\x)},{-\x});
	\addplot[color=darkred, mark=none, thick] ({x},{0});
	\addlegendentry{\raisebox{5pt}{$\Big\{ \operatorname{det}  
			\Big[ 
			\begin{smallmatrix}
			\mu^2 & -2 i \zeta^{-1} \mu \\
			-2i \mu & \mu^2
			\end{smallmatrix}\Big] = 0
			\Big\}$}};
	\node[right] at (axis cs: -.2,3.4) {$\color{darkred}\{\mu = 2 |\zeta|^{-1}\}$};
	\node[below] at (axis cs: .75,-.1) {$\color{darkred}\{\mu=0\}$};
	\node[below] at (axis cs: -1.5,-3) {$\color{darkred}\{\mu=-2|\zeta|^{-1}\}$};
	\end{axis} 
	\end{tikzpicture}
	\begin{tikzpicture}
	\begin{axis}[xlabel=$\zeta$, grid=both, grid style = {dotted},
	minor tick num=2,
	width = .5\textwidth,
	height = .35\textwidth,
	xmin = -1, 
	xmax=1,
	scaled ticks=true,
	clip=true,
	]
	\addplot[color=darkred, mark=none, thick, domain=-2:2, samples=100] ({x},{2*x});
	\addplot[color=darkred, mark=none, thick, domain=-2:2, samples=100] ({x},{-2*x});
	\addplot[color=darkred, mark=none, thick] ({x},{0});
	\addlegendentry{\raisebox{5pt}{$\Big\{ \operatorname{det}  
			\Big[ 
			\begin{smallmatrix}
			\mu^2 & -2 i \zeta^2 \mu \\
			-2i \mu & \mu^2
			\end{smallmatrix}\Big] = 0
			\Big\}$}};
	\node[right] at (axis cs: -.975,.7) {$\color{darkred}\{\mu = 2 \zeta\}$};
	\node[below] at (axis cs: .75,-.1) {$\color{darkred}\{\mu=0\}$};
	\node[below] at (axis cs: .22,-1) {$\color{darkred}\{\mu=-2\zeta\}$};
	\filldraw[darkred] (axis cs: 0,0) circle (1.5pt);
	\end{axis} 
	\end{tikzpicture}
	\caption{
		The semiclassical characteristic set ${\color{darkred}\operatorname{char}_h^2(L) }= \{\mu=0\text{ or }\mu^2 - 4\varsigma\zeta^\kappa=0\} \subset {}^\hbar T^* (-1,+1)$ of $L$, in the $\kappa=-1$ and $\kappa=2$ cases, the latter with $\varsigma<0$. In the former, the characteristic set hits fiber infinity at $\zeta=0$.
	}
	\label{fig:trapping}
\end{figure}

\subsection{Behavior at $\mathrm{fe}$: sketch}
\label{subsec:quasimodes_zefe_sketch}

We now analyze the behavior of the coefficients $\beta_{k}(\zeta),\gamma_{k}(\zeta)$ as $\zeta \to 0^+$. Since these Taylor coefficients are thought of as functions on $\mathrm{ze}$, the $\zeta \to 0^+$ limit corresponds to the corner $\mathrm{ze}\cap \mathrm{fe}$. Before addressing the situation in full generality, let us suppose that $\phi=0$ and 
\begin{equation} 
E \in C^\infty([0,Z]_z\times [0,\infty)_{h^2};\bbC).
\label{eq:little_e_ass}
\end{equation}  
The obstruction encountered by Olver in \cite[Chp. 12- \S14]{OlverBook} already comes into view, namely the fact that, if $\kappa>1$, the right-hand side of \cref{eq:rec2} will typically be singular as $\zeta \to 0^+$ (except for special $E$), and this cannot be fixed by judiciously choosing the $C_k$'s (unlike in the $\kappa\leq 1$ case). Since $\beta_{0}=1$, $\gamma_{0}$ is already singular, with 
\begin{equation}
	\gamma_{0}(\zeta) = \frac{\varsigma E_0(0)(1+o(1))}{(\kappa-2) \zeta^{\kappa-1}}  
	\label{eq:misc_r31}
\end{equation} 
if $\kappa \neq 2$, and with an extra $\log \zeta$ if $\kappa=2$.

At first glance, the existence of such a singularity may seem paradoxical ---
since the ODE is regular singular at $\zeta=0$ with indicial roots $0,1$, solutions cannot be singular at $\{\zeta = 0, h>0\}$ except for possible logarithmic terms. So, given that they are singular, how can $\beta_{k},\gamma_{k}$ be regarded as Taylor coefficients in $h$?
With the appropriate goemetric perspective, there is no problem; the formal series $\mathsf{U}$ should only be taken seriously as a possible asymptotic expansion near $\mathrm{ze}$, which is disjoint from the closure $\mathrm{be} = \mathrm{cl}_M\{\zeta = 0, h>0\}$
of $\{\zeta = 0, h>0\} = \mathrm{be}^\circ$ in $M$. In short, irregularity at the edge $\mathrm{fe}$ is consistent with regularity at the edge $\mathrm{be}$.

But this is merely a justification for performing \emph{some} blowup of the relevant corner of $[0,Z]_z\times [0,\infty)_{h^2}$. Why is the blowup defining $M$ the correct one for the problem at hand? 
Certainly, it is natural to hope that the blowup is the same one on which the quasimodes  
\begin{equation} 
Q(\zeta/h^{2/(\kappa+2)}),Q'(\zeta/h^{2/(\kappa+2)})
\end{equation} 
are exponential-polyhomogeneous, but it is not obvious that this should be the case. These are quasimodes for $P$, and here we are constructing quasimodes for $L$.
So, this fails to answer the question definitively.

In order to answer it definitively, we need to work in local coordinates. Let $\rho = \smash{h / \zeta^{(\kappa+2)/2}}$
and $x=\zeta$. Then $(x,\rho^2)$ is a valid coordinate chart near $\mathrm{ze}$. (We give $\zeta$ a new name to avoid conflating the partial derivative $\partial/\partial \zeta$ in the original local coordinate system $(\zeta,h^2)$ with the partial derivative $\partial/\partial x$ in the new one.) More precisely, the map 
\begin{equation} 
	(x,\rho^2)\mapsto (\zeta,h^2)=\big(x, \rho^2 x^{\kappa+2} \big)
	\label{eq:cornercoords}
\end{equation} 
defines a diffeomorphism $[0,Z]_x\times [0,\infty)_{\rho^2} \to M\backslash \mathrm{be} \supset \mathrm{ze}$. In terms of these new coordinates, the vector field $\partial/\partial \zeta$ can be written
\begin{equation}
	\frac{\partial}{\partial \zeta} = \frac{\partial}{\partial x}  - \frac{\kappa+2}{2} \frac{\rho}{x} \frac{\partial}{\partial \rho}. 
	\label{eq:misc_025}
\end{equation} 

The map
\begin{equation}
 C^\infty(0,Z)_\zeta[[h]]\ni \mathsf{s}=\sum_{k=0}^\infty h^k s_k(\zeta) \mapsto  \sum_{k=0}^\infty \rho^k x^{k(\kappa+2)/2} s_k(x) \overset{\mathrm{def}}= \{\mathsf{s}\}
\end{equation} 
defines an isomorphism 
\begin{equation} 
\{\cdots\}:C^\infty(0,Z]_\zeta[[h]] \to  C^\infty(0,Z]_x[[\rho]] \subset C^\infty(\mathrm{ze}^\circ)[[\rho]]
\end{equation} 
of $\operatorname{Diff}[0,Z]_\zeta$-modules; the result of applying the right-hand side of \cref{eq:misc_025} to $\rho x^{(\kappa+2)/2}$ gives $0$. (This is just the statement that $\partial h/\partial \zeta=0$, written in the new coordinate system.) 
Thus, 
\begin{equation}
	\frac{\partial \mathsf{s}}{\partial \zeta}  = \sum_{k=0}^\infty \rho^k x^{k(\kappa+2)/2} \frac{\partial s_k(x)}{\partial x} =   \Big\{\sum_{k=0}^\infty h^k \frac{\partial s_k(\zeta)}{\partial \zeta}\Big\}. 
	\label{eq:misc_ny3}
\end{equation}
So, $\{\cdots\}$ commutes with differential operators, replacing derivatives in $\zeta$ with derivatives in $x$. We may therefore drop the ``$\{\cdots\}$'' in \cref{eq:misc_ny3} without risk of confusion, identifying elements of $C^\infty(0,Z)_\zeta[[h]]$ with the corresponding elements of $C^\infty(\mathrm{ze}^\circ)[[\rho]]$.

\begin{figure}[t!]
	\begin{tikzpicture}[decoration={
			markings,
			mark=at position .5 with {\arrow[scale=1.5,>=latex]{>}}}]
		\fill[lightgray!20] (0,0) -- (4,0) -- (4,3) -- (0,3) -- cycle;
		\draw (4,0) -- (0,0) -- (0,3) -- (4,3);
		\draw[dashed] (4,0) -- (4,3);
		\draw[gray, ->] (.1,.1) -- (.7,.1) node[above] {$h^2$};
		\draw[gray, ->] (.1,.1) -- (.1,.7) node[right] {$\zeta$};
		\draw[darkred, postaction={decorate}] (0,3) -- (0,0); 
		\draw[darkred, postaction={decorate}] (1,3) -- (1,0); 
	\end{tikzpicture}
	\qquad
	\begin{tikzpicture}[decoration={
			markings,
			mark=at position .5 with {\arrow[scale=1.5,>=latex]{>}}}]
		\fill[lightgray!20] (1.5,0) -- (4,0) -- (4,3) -- (0,3) -- (0,1.5) arc(90:0:1.5);
		\draw (4,0) -- (1.5,0) arc(0:90:1.5) (0,1.5) -- (0,3) -- (4,3);
		\draw[dashed] (4,0) -- (4,3);
		\draw[darkred, postaction={decorate}] (0,3) -- (0,1.5); 
		\draw[darkred, postaction={decorate}] (1,3) to[out=-90, in=90] (2,0); 
		\node[gray] () at (2.5,2) {$M$};
		\node () at (.8,.8) {fe};
		\node () at (2.75,.3) {be};
	\end{tikzpicture}
	\caption{\textit{Left:} two paths along level sets of $h$, $\{h=\varepsilon\}$ and $\{h=0\}$, in the rectangle $[0,Z]_\zeta\times [0,\infty)_{h^2}$. Following the solutions of the ODE along these paths, different behavior is seen as $\zeta\to 0^+$: smoothness for $h=\varepsilon$ (assuming $\phi=0$, $E$ is smooth), singularity for $h=0$ (for $\kappa\geq 2$). This is despite the fact that $\varepsilon$ can be taken arbitrarily small. \textit{Right:} the resolution of the paradox via a blowup, which separates the paths along which different behavior is seen. One path ends at fe, one ends at be. }
\end{figure}

At a heuristic level, we can answer the question asked above by examining the orders of the singularities of $\beta_{k}(\zeta),\gamma_{k}(\zeta)$ at $\zeta=0$. We continue to assume $\phi = 0$ and \cref{eq:little_e_ass} holds.
Since differentiation in $\zeta$ can worsen the singularity by at most one order, and since integration mollifies the singularity by at least one order, the recurrence relations \cref{eq:rec1}, \cref{eq:rec2} tell us that, when comparing $\beta_{k+1},\gamma_{k+1}$ to $\beta_k,\gamma_{k}$, the order of the singularity has increased by no more than $\kappa+2$ orders. 
But, the former each appear in \cref{eq:Uform} with an extra power of $\smash{h^2 = \smash{\rho^2 x^{\kappa+2}}}$ relative to the previous. So, in the formal series $\mathsf{U}$, the worsening singularity of the coefficients is exactly canceled by the extra factors of $x$ present in the multiplier $h^2$. This is what allows it to be interpreted as an asymptotic expansion at $\mathrm{ze} = \{\rho=0\}$, including the endpoint $\mathrm{ze}\cap\mathrm{fe}$, in powers of $\rho$. 
The fortuitous numerology needed to make this cancellation happen would not work with a different blowup.

One notable exception to the above heuristic is in going from $k=0$ to $k=1$, because 
\begin{equation}
\frac{\partial^2 \beta_{0}}{\partial^2 \zeta} = 0.
\end{equation}
This means that going from $\beta_{0}$ to $\beta_{1}$ increases the order of the singularity by no more than $\kappa$ orders (this becomes $\kappa+1$ if $E$ is allowed to have a pole), which is an improvement relative to the other $k$.
(Depending on the behavior of $E$ at $\zeta=0$, there may be similar improvements for more $k$, or greater improvements.) Consequently, every subleading term in the formal series for $\beta$ in $\rho$ is $O(\zeta^2)$ (or $O(\zeta)$ if $E$ has a pole) relative to the leading term as $\zeta \to 0^+$, and so are subleading in that sense too. 
This does not imply that the expansion is a \emph{joint} expansion, which would require that subsequent terms have further increasing decay at $\zeta=0$; the improvement noted typically stops at $O(\zeta^2)$.  
Nevertheless, even this improvement aids in understanding leading behavior at $\mathrm{fe}$.

\subsection{Behavior at $\mathrm{fe}$: details}
\label{subsec:quasimodes_zefe}

A precise version of the argument in the previous section requires keeping track of logarithmic terms (which we avoided in \cref{eq:misc_r31}, in the main term, by assuming $\kappa\geq 3$), but these complicate matters only slightly. 
Allowing general $\phi\neq 0$ and $E$ does not modify the conclusions, but only because the assumptions placed on these functions in the introduction were precisely those needed for the argument here to go through.

We lay out some notation. If $\calF \subset \bbC\times \bbN$ is an index set, then let $\partial \calF$, $\calF_+$ denote the index sets
\begin{equation}
	\partial\calF = \{ (j-1+n,k) : (j,k)\in \calF \text{ and } j\neq 0, n\in \bbN \} \cup \{ (j-1,k-1) : (j,k)\in \calF \text{ and } k\geq 1\} 
\end{equation}
and
\begin{equation}
	\calF_+ =   \{(j+1,k)\in \bbC\times \bbN: (j,k) \in \calF\} \cup \{(n,k+1) \in \bbN\times \bbN : (-1,k)\in \calF\}.
\end{equation}
Let $\calF_{+\mathrm{c}} = \{(n,0) : n\in \bbN\} \cup \calF_+$, which is also an index set. 
For any $F\in \calA^{\calF,(0,0)}[0,Z]_\zeta$, where $\calF$ denotes the index set at $\zeta=0$, we have $\partial F \in \calA^{\partial \calF,(0,0)} [0,Z]_\zeta$, hence the notation. 
Also, 
\begin{equation}
	\int_\zeta^Z F(\omega) \dd \omega \in  \calA^{\calF_{+\mathrm{c}},(1,0) }[0,Z]_\zeta, 
\end{equation}
and, for some choice of constant of integration $C$, 
\begin{equation} 
C+\int_\zeta^Z F(\omega) \dd \omega \in \calA^{\calF_+,(0,0)}[0,Z]_\zeta.
\end{equation} 
These claims can be verified by expanding $F(\zeta)$ around $\zeta = 0$, differentiating or integrating each term in the expansion by hand, and then locating the error in the appropriate conormal space. 

Also, recall, for each $\lambda \in \bbC$, the notation $\calF+\lambda  = \{(j+\lambda,k):(j,k)\in \calF\} = \lambda + \calF$. 

\begin{proposition}
	Except for $\beta_{0}$, the functions $\beta_{k},\gamma_{k}\in C^\infty(\mathrm{ze}^\circ)$ recursively defined above satisfy $x^{k(\kappa+2)}\beta_{k}\in \calA^{\calE,(0,0)}(\mathrm{ze})$ and $x^{k(\kappa+2)}\gamma_{k}\in \calA^{\calF,(0,0)}(\mathrm{ze})$ for index sets $\calE\subseteq \bbN\times \bbN$ and $\calF\subseteq \bbZ\times \bbN$ defined by 
	\begin{equation} 
	\calE =  (1,0)\cup
	\begin{cases}
	\bigcup_{j=1}^\infty ((\kappa+2)j,j) & (\kappa\notin 2\bbN), \\ 
	(\kappa/2+1,1)\cup \bigcup_{j=1}^\infty ((\kappa+2)j,2j) \cup ((\kappa+2)j+\kappa/2+1,2j+1) & (\kappa \in 2 \bbN),
	\end{cases} 
	\label{eq:calEdef_good}
	\end{equation}
	and $\calF = \calE - \kappa - 1$. 
\end{proposition}
As above, in $\calA^{\calI,(0,0)}(\mathrm{ze})$, the index set at $\mathrm{ze}\cap \mathrm{fe}$ is $\calI$ and the index set at $\mathrm{ze}\cap \mathrm{ie}$ is $(0,0)$. 
\begin{proof}
	Via a straightforward inductive argument using \cref{eq:rec1}, \cref{eq:rec2}, we have, for all $k\in \bbN$, $x^{k(\kappa+2)}\smash{\beta_{k}}\in \calA^{\calE_{k},(0,0)}(\mathrm{ze})$ and $x^{k(\kappa+2)}\smash{\gamma_{k}}\in \calA^{\calF_{k},(0,0)}(\mathrm{ze})$ for index sets $\calE_k,\calF_k$ defined recursively as follows. 
	The index sets $\calE_0,\calF_0$, are defined by $\calE_0 = (0,0)$ and $\calF_0 =(\calE_0 -1 - \kappa/2)_+  - \kappa/2$. That is, 
	\begin{equation}
	\calF_0 = 
	\begin{cases}
	(-\kappa,0) & (\kappa\notin 2\bbN), \\
	(-\kappa,0) \cup (-\kappa/2,1) & (\kappa \in 2\bbN) . 
	\end{cases}
	\label{eq:F0_index_set}
	\end{equation}
	Then, $\calE_1$ is defined by
	\begin{equation}
	\calE_1 = 	((\partial^2\calF_0) \cup (\calF_0-1 )  )_{+\mathrm{c}}+ \kappa+2 = 
	\begin{cases}
	(1,0) & (\kappa=-1), \\
	(1,0) \cup (2,2) & (\kappa=0), \\
	(1,0) \cup (\kappa+2,1) & (\kappa \in 2\bbN+1), \\
	(1,0) \cup (\kappa/2+1,1) \cup (\kappa+2,2) & (\kappa \in 2\bbN+2).
	\end{cases}
	\end{equation}
	(The $+\kappa+2$ on the left-hand side comes from the fact that we want $x^{\kappa+2}\beta_1\in \calA^{\calE_1,(0,0)}$, not $\beta_1\in \calA^{\calE_1,(0,0)}$. It therefore comes ``for free.'') 
	Note that $\calE_1$ starts at $(1,0)$, in accordance with the discussion at the end of the previous subsection. (If $\phi$ were $0$ and $E$ were smooth, then this would be $(2,k)$ for some $k$, because the definition of $\calF_0$ would read $\calF_0=(\calE_0-\kappa/2)_+-\kappa/2$.)
	
	In order to state the definitions of $\calF_1$ and $\calE_k,\calF_k$ for $k\in \bbN^{\geq 2}$, let $\calE_k' = \calE_k - k(\kappa+2)$ and $\calF_k' = \calF_k - k(\kappa+2)$. (These are supposed to be the index sets describing $\beta_k,\gamma_k$ directly, not the weighted versions $x^{k(\kappa+2)}\beta_k,x^{k(\kappa+2)}\gamma_k$.) Then, for $k\in \bbN^{\geq 2}$,  
	\begin{equation} 
	\calE_k = (((\partial^2\calF_{k-1}') \cup (\calF_{k-1}'-1 )  )_{+\mathrm{c}}+ k(\kappa+2))\cup \calE_{k-1}, 
	\label{eq:misc_170}
	\end{equation} 
	and, for all $k\in \bbN^+$,
	\begin{equation}
	\calF_k = (((\partial^2 \calE_k') \cup (\calE_k'- 1) \cup (\calE_0-1 - k(\kappa+2)) \cup (\calF_{k-1}'-3 ) - \kappa/2)_+ -\kappa/2 + k(\kappa+2))\cup \calF_{k-1}.
	\label{eq:misc_171}
	\end{equation}
	Then, induction using \cref{eq:rec1}, \cref{eq:rec2} suffices to prove that $x^{k(\kappa+2)}\smash{\beta_{k}}\in \calA^{\calE_{k},(0,0)}(\mathrm{ze})$ and $x^{k(\kappa+2)}\smash{\gamma_{k}}\in \calA^{\calF_{k},(0,0)}(\mathrm{ze})$ for all $k$.\footnote{The only non-obvious part is that the contributions to \cref{eq:rec1}, \cref{eq:rec2} involving $E_1,E_2,\dots$, $\phi_1,\phi_2,\dots$, and $\phi'_1,\phi'_2,\dots$ are already included, but this is done using (i) that the singularities at the origin of $E_j(x),\phi_j(x),\phi_j'(x)$ are worsening by at most $1/x^{\kappa+2}$ every time $j$ increments and (ii) the inclusions $\calF_{j-1}\subseteq \calF_{j}$ and $\calE_j\subseteq \calE_{j+1}$, which hold for all $j\in \bbN^+$, by construction. (The reason why we have the $\calE_0-1 - k(\kappa+2)$ term in \cref{eq:misc_171} is because $\calE_0\not\subset \calE_k$ for $k\geq 1$. The $-1-k(\kappa+2)$ is coming from the singularity of that order in $E_{k}$.)}

	The rest of the argument involves upper-bounding $\calE_k,\calF_k$.

	If $\kappa\notin 2\bbN$, let, for each $K\in \bbN$, $\bar{\calE}_K = \{(j,k)\in \calE: k\leq K\} \subseteq \calE$ and $\bar{\calF}_K = \{(j,k)\in \calF: k\leq K\} \subseteq \calF$. Otherwise, if $\kappa\in 2\bbN$, let $\bar{\calE}_K = \{(j,k)\in \calE: k\leq 2K\} \subseteq \calE$ and $\bar{\calF}_K = \{(j,k)\in \calF: k\leq 2K+1\} \subseteq \calF$.

	We now check via induction that $\calE_k\subseteq \bar{\calE}_k$ for all $k\in \bbN^+$ and $\calF_k\subseteq \bar{\calF}_k$ for all $k\in \bbN$.  Once demonstrated, this proves the proposition.

	As the base case, we first observe that $\calF_0 = \bar{\calF}_0$, which follows immediately from the definition.

	For the inductive step, let $k\in \bbN^+$, and suppose that $\calE_{j}\subseteq\bar{\calE}_{k-1}$ and $\calF_{j}\subseteq \bar{\calF}_{k-1}$ for all $j\leq k-1$. Then the recursive formula for $\calE_k$ yields 
	\begin{equation}
		\calE_k \subseteq  ( (\bar{\calF}_{k-1} - (k-1) (\kappa+2)-2 )_{+\mathrm{c}} + k(\kappa+2) )\cup \bar{\calE}_{k-1}.
	\end{equation}
	The set $ (\bar{\calF}_{k-1} - (k-1) (\kappa+2)-2 )_{+\mathrm{c}} + k(\kappa+2)$ is, by definition, a union $\calI\cup \calJ\cup \calK$ of three terms, $\calI = (0,0) + k(\kappa+2) = (k(\kappa+2),0)\subseteq \bar{\calE}_0\subseteq \bar{\calE}_k$, 
	\begin{equation}
		\calJ=   (\bar{\calF}_{k-1} - (k-1) (\kappa+2)-2 ) + 1 + k(\kappa+2) = \bar{\calF}_{k-1} + \kappa +1 \subseteq \bar{\calE}_k, 
	\end{equation}
	and 
	\begin{align}
		\begin{split}
		\calK 
		&= \{(k(\kappa+2)+n,j+1) :n\in \bbN, (-1,j)\in \bar{\calF}_{k-1} -(k-1)(\kappa+2)-2 \}  \\ 
		&= \{(k(\kappa+2)+n,j+1) : n\in \bbN,  (k(\kappa+2)-\kappa-1,j) \in \bar{\calF}_{k-1} \}.
		\end{split}
	\end{align}
	By the definition of $\bar{\calF}_{k-1}$, the condition $(k(\kappa+2)-\kappa-1,j) \in \bar{\calF}_{k-1}$ holds if and only if both of
	\begin{itemize}
		\item $j\leq k-1$ if $\kappa\notin 2\bbN$ and $j\leq 2k-1$ if $\kappa \in 2\bbN$,
		\item $(k(\kappa+2),j) \in \calE$
	\end{itemize}
	hold. The second condition is actually redundant, because $(k(\kappa+2),k-1) \in \calE$ if $\kappa \notin 2\bbN$ and $(k(\kappa+2),2k-1) \in \calE$ otherwise.
	So,  
	\begin{equation}
	\calK = 
	\begin{cases}
	(k(\kappa+2),k) & (\kappa\notin 2\bbN), \\ 
	(k(\kappa+2),2k) & (\kappa\in 2\bbN).
	\end{cases}
	\end{equation} 
	So, $\calK\subseteq \bar{\calE}_k$. 
	So, altogether, $\calE_k \subseteq \bar{\calE}_k$. 
	
	For the remainder of the inductive step, suppose that $k\in \bbN^+$, and suppose that $\calE_k \subseteq \bar{\calE}_k$ and $\calF_{k-1}\subseteq \bar{\calF}_{k-1}$. Then, the recursive formula for $\calF_k$ yields 
	\begin{equation}
		\calF_k \subseteq (((\bar{\calE}_k-k(\kappa+2)-2)\cup ( \bar{\calF}_{k-1} - 3 - (k-1)(\kappa+2)) -\kappa/2)_+-\kappa/2+k(\kappa+2))\cup \bar{\calF}_{k-1}  . 
	\end{equation}
	The right-hand side is the union of $\bar{\calF}_{k-1}$, $(\bar{\calE}_{k} - k(\kappa+2) -2 - \kappa/2)_+-\kappa/2+k(\kappa+2)$,
	and $(\bar{\calF}_{k-1} - 3 - (k-1)(\kappa+2) -\kappa/2)_+-\kappa/2+k(\kappa+2)$. 
	Because $\bar{\calF}_{k-1} \subseteq \bar{\calE}_k - \kappa -1$, the last of these satisfies
	\begin{equation}
		(\bar{\calF}_{k-1} - 3 - (k-1)(\kappa+2) -\kappa/2)_+-\kappa/2+k(\kappa+2) \subseteq (\bar{\calE}_{k}  - k(\kappa+2) -2 -\kappa/2)_+-\kappa/2+k(\kappa+2),
	\end{equation}
	which is the other set we listed. So, in order to show that $\calF_k \subseteq \bar{\calF}_k$, it suffices to show that $(\bar{\calE}_{k}  - k(\kappa+2)-2-\kappa/2 )_+-\kappa/2+k(\kappa+2)\subseteq \bar{\calF}_k$.
	This set is a union $\calI\cup \calJ$ of two terms, 
	\begin{equation}
	\calI=(\bar{\calE}_{k}  - k(\kappa+2)-2-\kappa/2 +1)-\kappa/2+k(\kappa+2) = \bar{\calE}_k -\kappa-1 \subset \bar{\calF}_k
	\end{equation}
	and 
	\begin{align}
	\begin{split}
	\calJ &= \{(-\kappa/2 +k(\kappa+2) +n,j+1) : n\in \bbN, (-1,j) \in \bar{\calE}_{k} - 2 - k(\kappa+2)-\kappa/2\} \\
	&= \{(-\kappa/2 +k(\kappa+2) +n,j+1) : n\in \bbN, (1+\kappa/2+k(\kappa+2),j) \in \bar{\calE}_{k} \}.
	\end{split}
	\end{align}
	If $\kappa\notin 2\bbN$, then $\calJ = \varnothing$, so $\calJ\subseteq \bar{\calF}_k$ trivially. Otherwise, the condition $(1+\kappa/2+k(\kappa+2),j) \in \bar{\calE}_{k}$ is satisfied if and only if both of 
	\begin{itemize}
		\item $j \leq 2k$,
		\item  
		\begin{equation}
			1+\kappa/2+k(\kappa+2) \geq 
			\begin{cases}
				1 & (j=0), \\
				\kappa/2+1 & (j=1), \\
				(\kappa+2)(j/2) & (j \in 2\bbN^+), \\ 
				(\kappa+2)2^{-1}(j-1)+\kappa/2+1 & (\text{otherwise}),
			\end{cases}
		\end{equation}
	\end{itemize}
	hold. The second condition is actually redundant, as, if we plug in $j=2k$, then $1+\kappa/2+k(\kappa+2)\geq k(\kappa+2)$ holds. So, in the $\kappa\in 2\bbN$ case, 
	\begin{equation}
		\calJ = (-\kappa/2+k(\kappa+2),2k+1) \subseteq \bar{\calF}_k.
	\end{equation}
\end{proof}

\begin{figure}[t]
\begin{tikzpicture}
	\begin{axis}[xlabel=$j$, ylabel=$k$, grid=both, grid style = {dotted},
		minor tick num=1,
		width = .3\textwidth,
		height = .3\textwidth,
		xmin = -.5, 
		xmax= 6,
		ymin = -.5,
		ymax = 5,
		scaled ticks=true,
		clip=true,
		title = {$\calE$\text{ for }$\kappa\in\{-1,1\}$},
 		]
		\filldraw[gray] (axis cs: 1,0) circle (3pt);
		\filldraw[darkred, pattern color = darkred, pattern=crosshatch] (axis cs: 1,0) circle (3pt);
		\filldraw[gray] (axis cs: 2,0) circle (3pt);
		\filldraw[darkred, pattern color = darkred, pattern=crosshatch] (axis cs: 2,0) circle (3pt);
		\filldraw[gray] (axis cs: 1,1) circle (3pt);
		\filldraw[gray] (axis cs: 3,0) circle (3pt);
		\filldraw[darkred, pattern color = darkred, pattern=crosshatch] (axis cs: 3,0) circle (3pt);
		\filldraw[gray] (axis cs: 2,1) circle (3pt);
		\filldraw[gray] (axis cs: 2,2) circle (3pt);
		\filldraw[gray] (axis cs: 4,0) circle (3pt);
		\filldraw[gray] (axis cs: 3,1) circle (3pt);
		\filldraw[darkred, pattern color = darkred, pattern=crosshatch] (axis cs: 4,0) circle (3pt);
		\filldraw[darkred, pattern color = darkred, pattern=crosshatch] (axis cs: 4,1) circle (3pt);
		\filldraw[gray] (axis cs: 3,2) circle (3pt);
		\filldraw[gray] (axis cs: 3,3) circle (3pt);
		\filldraw[gray] (axis cs: 5,0) circle (3pt);
		\filldraw[gray] (axis cs: 4,1) circle (3pt);
		\filldraw[darkred, pattern color = darkred, pattern=crosshatch] (axis cs: 5,0) circle (3pt);
		\filldraw[darkred, pattern color = darkred, pattern=crosshatch] (axis cs: 4,1) circle (3pt);
		\filldraw[gray] (axis cs: 4,2) circle (3pt);
		\filldraw[gray] (axis cs: 4,3) circle (3pt);
		\filldraw[gray] (axis cs: 4,4) circle (3pt);
		\filldraw[gray] (axis cs: 6,0) circle (3pt);
		\filldraw[gray] (axis cs: 5,1) circle (3pt);
		\filldraw[darkred, pattern color = darkred, pattern=crosshatch] (axis cs: 6,0) circle (3pt);
		\filldraw[darkred, pattern color = darkred, pattern=crosshatch] (axis cs: 5,1) circle (3pt);
		\filldraw[gray] (axis cs: 5,2) circle (3pt);
		\filldraw[gray] (axis cs: 5,3) circle (3pt);
		\filldraw[gray] (axis cs: 5,4) circle (3pt);
		\filldraw[gray] (axis cs: 5,5) circle (3pt);
		\filldraw[gray] (axis cs: 7,0) circle (3pt);
		\filldraw[gray] (axis cs: 6,1) circle (3pt);
		\filldraw[gray] (axis cs: 6,2) circle (3pt);
		\filldraw[darkred, pattern color = darkred, pattern=crosshatch] (axis cs: 7,0) circle (3pt);
		\filldraw[darkred, pattern color = darkred, pattern=crosshatch] (axis cs: 6,1) circle (3pt);
		\filldraw[darkred, pattern color = darkred, pattern=crosshatch] (axis cs: 6,2) circle (3pt);
		\filldraw[darkred, pattern color = darkred, pattern=crosshatch] (axis cs: 3,1) circle (3pt);
		\filldraw[gray] (axis cs: 6,3) circle (3pt);
		\filldraw[gray] (axis cs: 6,4) circle (3pt);
		\filldraw[gray] (axis cs: 6,5) circle (3pt);
	\end{axis} 
\end{tikzpicture}
\begin{tikzpicture}
	\begin{axis}[xlabel=$j$, grid=both, grid style = {dotted},
		minor tick num=1,
		width = .6\textwidth,
		height = .3\textwidth,
		xmin = -.5, 
		xmax= 9,
		ymin = -.5,
		ymax = 5,
		scaled ticks=true,
		clip=true,
		title = {$\calE$\text{ for }$\kappa\in\{0,2\}$},
		]
		\filldraw[gray] (axis cs: 1,0) circle (3pt);
		\filldraw[gray] (axis cs: 1,1) circle (3pt);
		\filldraw[darkred, pattern color = darkred, pattern=crosshatch] (axis cs: 1,0) circle (3pt);
		\filldraw[gray] (axis cs: 2,0) circle (3pt);
		\filldraw[darkred, pattern color = darkred, pattern=crosshatch] (axis cs: 2,0) circle (3pt);
		\filldraw[gray] (axis cs: 2,1) circle (3pt);
		\filldraw[darkred, pattern color = darkred, pattern=crosshatch] (axis cs: 2,1) circle (3pt);
		\filldraw[gray] (axis cs: 3,0) circle (3pt);
		\filldraw[darkred, pattern color = darkred, pattern=crosshatch] (axis cs: 3,0) circle (3pt);
		\filldraw[gray] (axis cs: 3,1) circle (3pt);
		\filldraw[darkred, pattern color = darkred, pattern=crosshatch] (axis cs: 3,1) circle (3pt);
		\filldraw[gray] (axis cs: 2,2) circle (3pt);
		\filldraw[gray] (axis cs: 3,2) circle (3pt);
		\filldraw[gray] (axis cs: 4,0) circle (3pt);
		\filldraw[gray] (axis cs: 3,1) circle (3pt);
		\filldraw[darkred, pattern color = darkred, pattern=crosshatch] (axis cs: 4,0) circle (3pt);
		\filldraw[darkred, pattern color = darkred, pattern=crosshatch] (axis cs: 4,1) circle (3pt);
		\filldraw[gray] (axis cs: 3,2) circle (3pt);
		\filldraw[gray] (axis cs: 3,3) circle (3pt);
		\filldraw[gray] (axis cs: 5,0) circle (3pt);
		\filldraw[gray] (axis cs: 4,1) circle (3pt);
		\filldraw[darkred, pattern color = darkred, pattern=crosshatch] (axis cs: 5,0) circle (3pt);
		\filldraw[darkred, pattern color = darkred, pattern=crosshatch] (axis cs: 5,1) circle (3pt);
		\filldraw[gray] (axis cs: 4,2) circle (3pt);
		\filldraw[darkred, pattern color = darkred, pattern=crosshatch] (axis cs: 4,2) circle (3pt);
		\filldraw[gray] (axis cs: 4,3) circle (3pt);
		\filldraw[gray] (axis cs: 4,4) circle (3pt);
		\filldraw[gray] (axis cs: 6,0) circle (3pt);
		\filldraw[gray] (axis cs: 6,1) circle (3pt);
		\filldraw[darkred, pattern color = darkred, pattern=crosshatch] (axis cs: 6,0) circle (3pt);
		\filldraw[darkred, pattern color = darkred, pattern=crosshatch] (axis cs: 6,1) circle (3pt);
		\filldraw[gray] (axis cs: 5,2) circle (3pt);
		\filldraw[gray] (axis cs: 5,3) circle (3pt);
		\filldraw[darkred, pattern color = darkred, pattern=crosshatch] (axis cs: 5,2) circle (3pt);
		\filldraw[gray] (axis cs: 5,4) circle (3pt);
		\filldraw[gray] (axis cs: 5,5) circle (3pt);
		\filldraw[gray] (axis cs: 7,0) circle (3pt);
		\filldraw[gray] (axis cs: 7,1) circle (3pt);
		\filldraw[gray] (axis cs: 6,2) circle (3pt);
		\filldraw[darkred, pattern color = darkred, pattern=crosshatch] (axis cs: 7,0) circle (3pt);
		\filldraw[darkred, pattern color = darkred, pattern=crosshatch] (axis cs: 7,1) circle (3pt);
		\filldraw[darkred, pattern color = darkred, pattern=crosshatch] (axis cs: 6,2) circle (3pt);
		\filldraw[gray] (axis cs: 6,3) circle (3pt);
		\filldraw[darkred, pattern color = darkred, pattern=crosshatch] (axis cs: 7,3) circle (3pt);
		\filldraw[gray] (axis cs: 6,4) circle (3pt);
		\filldraw[gray] (axis cs: 6,5) circle (3pt);
		\filldraw[gray] (axis cs: 8,0) circle (3pt);
		\filldraw[gray] (axis cs: 8,1) circle (3pt);
		\filldraw[gray] (axis cs: 7,2) circle (3pt);
		\filldraw[gray] (axis cs: 7,3) circle (3pt);
		\filldraw[gray] (axis cs: 7,4) circle (3pt);
		\filldraw[gray] (axis cs: 7,5) circle (3pt);
		\filldraw[gray] (axis cs: 9,0) circle (3pt);
		\filldraw[gray] (axis cs: 9,1) circle (3pt);
		\filldraw[gray] (axis cs: 8,2) circle (3pt);
		\filldraw[gray] (axis cs: 8,3) circle (3pt);
		\filldraw[gray] (axis cs: 8,4) circle (3pt);
		\filldraw[gray] (axis cs: 8,5) circle (3pt);
		\filldraw[gray] (axis cs: 9,2) circle (3pt);
		\filldraw[gray] (axis cs: 9,3) circle (3pt);
		\filldraw[gray] (axis cs: 9,4) circle (3pt);
		\filldraw[gray] (axis cs: 9,5) circle (3pt);
		\filldraw[darkred, pattern color = darkred, pattern=crosshatch] (axis cs: 8,0) circle (3pt);
		\filldraw[darkred, pattern color = darkred, pattern=crosshatch] (axis cs: 8,1) circle (3pt);
		\filldraw[darkred, pattern color = darkred, pattern=crosshatch] (axis cs: 7,2) circle (3pt);
		\filldraw[darkred, pattern color = darkred, pattern=crosshatch] (axis cs: 7,3) circle (3pt);
		\filldraw[darkred, pattern color = darkred, pattern=crosshatch] (axis cs: 9,0) circle (3pt);
		\filldraw[darkred, pattern color = darkred, pattern=crosshatch] (axis cs: 9,1) circle (3pt);
		\filldraw[darkred, pattern color = darkred, pattern=crosshatch] (axis cs: 8,2) circle (3pt);
		\filldraw[darkred, pattern color = darkred, pattern=crosshatch] (axis cs: 8,3) circle (3pt);
		\filldraw[darkred, pattern color = darkred, pattern=crosshatch] (axis cs: 8,4) circle (3pt);
		\filldraw[darkred, pattern color = darkred, pattern=crosshatch] (axis cs: 3,1) circle (3pt);
		\filldraw[darkred, pattern color = darkred, pattern=crosshatch] (axis cs: 4,1) circle (3pt);
		\filldraw[gray] (axis cs: 5,1) circle (3pt);
		\filldraw[darkred, pattern color = darkred, pattern=crosshatch] (axis cs: 5,1) circle (3pt);
		\filldraw[darkred, pattern color = darkred, pattern=crosshatch] (axis cs: 6,3) circle (3pt);
		\filldraw[darkred, pattern color = darkred, pattern=crosshatch] (axis cs: 9,2) circle (3pt);
		\filldraw[darkred, pattern color = darkred, pattern=crosshatch] (axis cs: 9,3) circle (3pt);
		\filldraw[darkred, pattern color = darkred, pattern=crosshatch] (axis cs: 9,4) circle (3pt);
	\end{axis} 
\end{tikzpicture}
\caption{The first few points $(j,k)$ in $\calE$, for $\kappa=-1,0,1,2$. The cases $\kappa=-1,0$ are in dark gray, and $\kappa=1,2$ in hatched {\color{darkred}red}.}
\end{figure}

Via the standard asymptotic summation construction, there exist functions $\beta,\gamma$ such that, for any $K\in \bbN$, 
\begin{equation}
\beta-1 - \sum_{k=1}^{K-1} h^{2k} \beta_{k} \in \varrho^{K}_{\mathrm{ze}}\calA^{\calE}(M), \qquad 
\gamma - \sum_{k=0}^{K-1} h^{2k+1} \gamma_{k} \in \varrho^{K+1/2}_{\mathrm{ze}} \varrho^{1+\kappa/2}_{\mathrm{fe}} \calA^{\calF}(M).
\end{equation}
We may choose $\beta,\gamma$ such that $\beta-1,\gamma$ vanish identically near $\mathrm{be}$.

Let $U=(\beta,\gamma)$. This might not solve $LU=0$, but:
\begin{proposition}
	$LU\in \varrho_{\mathrm{be}}^{-1} \varrho_{\mathrm{ze}}^\infty \varrho_{\mathrm{fe}}^\kappa \calA^{\calE}(M) \oplus  \varrho_{\mathrm{be}}^{\infty} \varrho_{\mathrm{ze}}^\infty \varrho_{\mathrm{fe}}^{\kappa/2} \calA^{\calE}(M)$, with the second component supported away from $\mathrm{be}$.
	\label{prop:LU}
\end{proposition}
\begin{proof}
	 Since $\beta-1,\gamma$ are supported away from $\mathrm{be}$, 
	\begin{equation}
	LU= 
	L \begin{bmatrix}
	1 \\ 0
	\end{bmatrix}
	= h^2 E \begin{bmatrix}
	1 \\ 0
	\end{bmatrix}
	\end{equation}
	near $\mathrm{be}$. The right-hand side is in $\varrho_{\mathrm{be}}^{-1}\varrho_{\mathrm{fe}}^{\kappa+1} C^\infty(M)\oplus \{0\}\subset \varrho_{\mathrm{be}}^{-1} \varrho_{\mathrm{fe}}^\kappa \calA^{\calE}(M)\oplus \{0\}$ near $\mathrm{be}$. 
	
	It remains to understand the situation away from $\mathrm{be}$, for which we can work in the local coordinate system $(x,\rho^2)$ discussed in \cref{eq:cornercoords}. 
	
	Consider the components of $L$ as weighted elements of $\operatorname{Diff}_{\mathrm{b}}=\operatorname{Diff}_{\mathrm{b}}([0,Z]_x\times [0,\infty)_{\rho^2})$, i.e.\ the unital algebra of differential operators generated over $C^\infty([0,Z]_x\times [0,\infty)_{\rho^2};\bbC)$ by the vector fields $\rho^2\partial_{\rho^2} \propto \rho \partial_\rho$ and $x \partial_x$.
	Let $\smash{\operatorname{Diff}^k_{\mathrm{b}}}$ denote the subset of $k$th order elements of $\operatorname{Diff}_{\mathrm{b}}$. 
	\Cref{eq:misc_025} shows that 
	\begin{equation}
		h \frac{\partial}{\partial \zeta} \in \rho x^{\kappa/2} \operatorname{Diff}^1_{\mathrm{b}}. 
	\end{equation}
	Consequently, 
	\begin{equation}
		L \in 
		\begin{pmatrix}
			\rho^2 x^{\kappa}  \operatorname{Diff}^2_{\mathrm{b}} & \rho x^{3\kappa/2} \operatorname{Diff}^1_{\mathrm{b}} \\
			\rho x^{\kappa/2} \operatorname{Diff}^1_{\mathrm{b}}  & \rho^2  x^{\kappa}  \operatorname{Diff}^2_{\mathrm{b}} 
		\end{pmatrix},
	\end{equation} 
	meaning that the entries $L_{00}$, $L_{01}$, $L_{10}$, $L_{11}$ of $L$ are in the corresponding sets on the right-hand side.

	It follows that, since $U\in (1,0)+ (\varrho_{\mathrm{ze}} \calA^{\calE}(M))\oplus \varrho_{\mathrm{ze}}^{1/2} \varrho_{\mathrm{fe}}^{(\kappa+2)/2} \calA^{\calF}(M)$, $LU$ is polyhomogeneous on $M$ away from $\mathrm{be}$, and more specifically 
	\begin{equation} 
		LU \in \varrho_{\mathrm{ze}} \varrho_{\mathrm{fe}}^\kappa \calA^{\calE}(M\backslash \mathrm{be}) \oplus \varrho_{\mathrm{ze}}^{3/2} \varrho_{\mathrm{fe}}^{\kappa/2} \calA^{\calE}(M \backslash \mathrm{be}).
	\end{equation}  
	Thus, in order to conclude the proposition, it suffices to show that the terms in the Taylor expansion of $LU$ at $\mathrm{ze}$ vanish identically. The coefficients of this asymptotic expansion are polyhomogeneous functions on $\mathrm{ze}$ and can therefore be identified with their restrictions to $\mathrm{ze}^\circ$. 
	
	Consequently, it suffices to check that, for each $\zeta>0$, the Taylor expansion of $U(\zeta,h)$ at $h=0$ vanishes. By construction, this expansion is precisely $\mathsf{L}\mathsf{U}=0$. 
\end{proof}

Finally, we can prove the main result of this section:

\begin{proof}[Proof of  \Cref{prop:ze_upshot}]
	Consider the function $u$ defined by \cref{eq:uform}, $u=\beta Q + h^{\kappa/(\kappa+2)} \gamma Q'$,
	with $\beta,\gamma$ as above. 
	By \Cref{prop:LU}, this satisfies 
	\begin{equation}
	Pu = \begin{bmatrix}
	Q \\ h^{\kappa/(\kappa+2)} Q'
	\end{bmatrix}^\intercal L U \in \varrho_{\mathrm{be}}^{-1} \varrho_{\mathrm{ze}}^\infty \varrho_{\mathrm{fe}}^{\kappa/2} \begin{bmatrix}
	Q \\ h^{\kappa/(\kappa+2)} Q'
	\end{bmatrix}^\intercal
	\begin{bmatrix}\varrho_{\mathrm{fe}}^{\kappa/2} \calA^{\calE}(M)\;\;\;\;\\  \calA^{\calE}(M)
	\end{bmatrix}. 
	\end{equation}
	Writing out the inner product,
	\begin{equation}
	\begin{bmatrix}
	Q \\ h^{\kappa/(\kappa+2)} Q'
	\end{bmatrix}^\intercal
	\begin{bmatrix}\varrho_{\mathrm{fe}}^{\kappa/2} \calA^{\calE}(M)\;\;\;\;\\  \calA^{\calE}(M)
	\end{bmatrix} = Q \varrho_{\mathrm{fe}}^{\kappa/2} \calA^{\calE}(M) + h^{\kappa/(\kappa+2)} Q' \calA^{\calE}(M),
	\end{equation}
	and $h^{\kappa/(\kappa+2)} \calA^{\calE}(M) = \varrho_{\mathrm{ze}}^{\kappa/(2\kappa+4)}\varrho_{\mathrm{fe}}^{\kappa/2} \calA^{\calE}(M)$. 
	So $Pu = f Q +  g Q'$ for some $f,g \in \varrho_{\mathrm{be}}^{-1} \varrho_{\mathrm{ze}}^\infty \varrho_{\mathrm{fe}}^\kappa \calA^{\calE}(M)$. 
	
	We just need to verify that $\beta,\gamma$ can be written in terms of $\beta_0,\gamma_0$ of the desired form.
	Indeed, the function $\beta_0$ defined by $\beta_0 = \varrho_{\mathrm{ze}}^{-1}\varrho_{\mathrm{fe}}^{-1}( \beta-1)$ is in $\calA^{\calE-1}(M)$ by construction, and a short calculation gives 
	\begin{equation} 
	h^{\kappa/(\kappa+2)} \gamma \in h^{(2\kappa+2)/(\kappa+2) } \calA^{\calF}(M) = \varrho_{\mathrm{ze}}^{(\kappa+1)/(\kappa+2)} \varrho_{\mathrm{fe}} \calA^{\calE-1}(M).
	\end{equation}  
	
	Since \Cref{prop:LU} says that the second component of $LU$ is supported away from $\mathrm{be}$, the same holds for $g$. Also, since $\beta-1$ and $\gamma$ are supported away from $\mathrm{be}$, the same applies to $\beta_0,\gamma_0$. 
\end{proof}

\section{Analysis at the high corner of $\mathrm{fe}$}
\label{sec:fe}

We now turn to solving, modulo errors which are small relative to $Q$, the forced ODE $Pu=f Q + g Q'$ near the high corner $\mathrm{ze} \cap\mathrm{fe}$. 
Our assumption on $f,g$, coming from the previous section, is that 
\begin{equation} 
f,g \in \varrho_{\mathrm{ze}}^\infty \varrho_{\mathrm{fe}}^\kappa \calA^{\calE}(M) = 
\bigcap_{k=0}^\infty \varrho_{\mathrm{ze}}^k \varrho_{\mathrm{fe}}^\kappa \calA^{\calE}(M)  .
\end{equation} 
(Behavior near $\mathrm{be}$ is irrelevant.)
The key idea is that $P$ is well approximated by $N(P)$ near $\mathrm{fe}$, so an approximate solution can be constructed by asymptotically summing a formal polyhomogeneous series in the bdf of $\mathrm{fe}$. This yields $O(h^\infty)$ errors relative to $Q$. In this section, we will only do this near $\mathrm{ze}$.

It will be most convenient to use the ``coordinates'' $\lambda = \zeta/h^{2/(\kappa+2)}$ and $\varrho = \smash{h^{2/(\kappa+2)}}$. 
For any $h_0>0,\lambda_0>0$, let 
\begin{align} 
V_{h_0,\lambda_0} &= \{(\zeta,h) \in (0,Z)\times (0,h_0) : \zeta >\lambda_0 h^{2/(\kappa+2)}\}, \\ 
U_{h_0,\lambda_0} &= \{(\zeta,h) \in (0,Z)\times (0,h_0) : \zeta < \lambda_0 h^{2/(\kappa+2)}\} .
\end{align} 
These are open subsets of $[0,Z]_\zeta\times [0,\infty)_h$. The subsets $\smash{\bar{V}_{h_0,\lambda_0}}=(\mathrm{cl}_M V_{h_0,\lambda_0})^\circ\subset M$ and $\smash{\bar{U}_{h_0,\lambda_0}}=(\mathrm{cl}_M U_{h_0,\lambda_0})^\circ \subset M$ are relatively open subsets of $M$ containing some points of the boundary. (Here, $\circ$ denotes the topological interior as taken in $M$.)

\begin{remark}
	The reader should note that $(\lambda,\varrho)$ is not a valid coordinate chart near the corner $\mathrm{ze}\cap\mathrm{fe}$. Rather, $(\zeta,\lambda^{-(\kappa+2)})$ are; see \Cref{fig}, \Cref{fig:low}. However, because the inhomogeneities we work with in this section are all Schwartz at ze, this will not cause a problem. 
\end{remark}
The set $\bar{U}_{h_0,\lambda_0}$ is a neighborhood of the ``low corner'' $\mathrm{fe}\cap \mathrm{be}$, and $\bar{V}_{h_0,\lambda_0}$ is a neighborhood of the ``high corner'' $\mathrm{ze}\cap \mathrm{fe}$. If $\lambda_1>\lambda_0$, then $V_{h_0,\lambda_0} \cup U_{h_0,\lambda_1}\supset \mathrm{fe}$.

In this section our concern is the situation in $V_{h_0,\lambda_0}$, and $U_{h_0,\lambda_0}$ is considered later. 

\begin{figure}[ht]
	\begin{tikzpicture}
	\fill[gray!5] (5,1.5) -- (0,1.5) -- (0,0)  arc(90:0:1.5) -- (5,-1.5) -- cycle;
	\draw[dashed] (5,1.5) -- (5,-1.5);
	\node (ff) at (.75,-.75) {fe};
	\node (zfp) at (-.35,.75) {$\mathrm{ze}$};
	\node (pf) at (2.5,1.8) {$\mathrm{ie}$};
	\node (mf) at (3.25,-1.8) {$\mathrm{be}$};
	\filldraw[darkblue, opacity =.25] (0,0) arc(90:40:1.5) to[out=40, in=190] (4.15, .8) -- (4.15,1.5) -- (0,1.5) -- cycle;
	\filldraw[darkred, opacity =.25] (1.5,-1.5) arc(0:50:1.5) to[out=40, in=190] (4.15, 1) -- (4.15,-1.5) -- cycle;
	\filldraw[darkred, opacity =.35, pattern=dots, pattern color=darkred] (1.5,-1.5) arc(0:50:1.5) to[out=40, in=190] (4.15, 1) -- (4.15,-1.5) -- cycle;
	\node[darkred] (?) at (3,-.25) {$U$};
	\node[darkblue] (?) at (2,1) {$V$};
	\draw[->, darkred] (1.6,-1.4) -- (2.2,-1.4) node[above] {$\qquad\qquad\varrho=h^{2/(\kappa+2)}$};
	\draw[->, darkred] (1.6,-1.4) to[out=90, in=-60] (1.375,-.65) node[right] {$\lambda$};
	\draw[->, darkblue] (.1,.1) -- (.1,.75) node[right] {$\zeta$};
	\draw[->, darkblue] (.1,.1) to[out=0, in=150] (.6,-.01) node[above right] {$\!\!\!\!\!\lambda^{-(\kappa+2)}$};
	\draw[->, darkred] (1.6,-1.4) to[out=90, in=-60] (1.375,-.65) node[right] {$\lambda$};
	\draw[dashed] (0,0) arc(90:50:1.5) to[out=40, in=187] (5, 1.1) node[right]{$\Gamma_{\lambda_1}$};
	\draw[dashed] (0,0) arc(90:40:1.5) to[out=40, in=187] (5, .9) node[below right]{$\Gamma_{\lambda_0}$};
	\draw (5,1.5) -- (0,1.5) -- (0,0)  arc(90:0:1.5)  -- (5,-1.5);
	\end{tikzpicture}
	\caption{The overlapping neighborhoods $\color{darkred}U=U_{h_0,\lambda_1}$ and $\color{darkblue}V=V_{h_0,\lambda_0}$ for some $h_0>0$ and $\lambda_1>\lambda_0>0$. Note that $U,V$ contain points of $\partial M$.}
	\label{fig:low}
\end{figure}

In the coordinate system $(\lambda,\varrho)$, $P$ can be considered as a family of ordinary differential operators on $\mathrm{fe}^\circ = \bbR^+_\lambda$, and $\varrho^{-\kappa} P, \varrho^{-\kappa} N(P)$ have smooth coefficients all the way down to $\varrho=0$. 
Let 
\begin{equation} 
\hat{N}(P) = \varrho^{-\kappa} N(P) \in \operatorname{Diff}^2(\mathrm{fe}^\circ).
\end{equation} 
To leading order at $\mathrm{fe}$, $P$ is indeed given by $N(P)$, as $\varrho^{-\kappa} P - \hat{N}(P) = \varrho^2 E$
is lower order, both with respect to the usual differential sense and with respect to decay as $\varrho \to 0^+$. The latter of these is not immediately obvious, as $E\in \zeta^{-1} C^\infty(M;\bbC)$ can be singular at $\mathrm{fe}$ and $\mathrm{be}$, but this just means that $E = \varrho^{-1} \smash{\hat{E}}$ for some
\begin{equation} 
\hat{E} \in \lambda^{-1} C^\infty( [0,\infty)_\lambda \times [0,\infty)_\varrho ;\bbC ).
\end{equation}    
So, $\varrho^2 E = \varrho \hat{E} =  O(\varrho)$ is still decaying as $\varrho \to 0^+$, at least for $0 \ll \lambda \ll \infty$. 

Consider the Taylor series 
\begin{equation} 
E \sim \sum_{k=-1}^\infty \varrho^k E_k(\lambda) = \sum_{k=-1}^\infty \zeta^k \lambda^{-k} E_k(\lambda)
\label{eq:misc_120}
\end{equation} 
of $E$ at $\mathrm{fe}$. 
Accordingly, we have the following formal version of $\varrho^{-\kappa} P$:
\begin{equation}
\mathsf{P} = \hat{N}(P) + \sum_{k=-1}^\infty \varrho^{k+2} E_k(\lambda) \in \operatorname{Diff}^2(\mathrm{fe}^\circ)[[\varrho]],
\end{equation}
the coefficients of which are elements of $\operatorname{Diff}^2(\mathrm{fe}^\circ)$.

From $E\in \zeta^{-1} C^\infty(M;\bbC)$ and the fact that the second series in \cref{eq:misc_120} is the Taylor series of $E$ at $\mathrm{fe}$ (near $\mathrm{ze}\cap \mathrm{fe}$) using the smooth coordinates $\zeta,\lambda^{-(\kappa+2)}$ valid near $\mathrm{ze}\cap \mathrm{fe}$, it follows that 
\begin{equation} 
E_k(\rho^{-2/(\kappa+2)}) \in \rho^{-2k/(\kappa+2)} C^\infty[0,\infty)_{\rho^2}.
\end{equation}
That is, each $\lambda^{-k} E_k(\lambda)$ can be interpreted as a smooth function on $\mathrm{fe} \backslash \mathrm{be}$.

The main result of this subsection is:
\begin{proposition}
	Let $\calE$ denote an index set, and let $h_0,\lambda_0>0$, thus supported away from $\mathrm{be}$, $\lambda_1>\lambda_0$, and $\Lambda>\lambda_1$. 
	Given $Q\in \calQ$ and $f,g \in \varrho_{\mathrm{ze}}^\infty \varrho_{\mathrm{fe}}^\kappa \calA^{\calE}(M)$ satisfying $\operatorname{supp} f,\operatorname{supp}g \Subset \bar{V}_{h_0,\lambda_1}$, there exist 
	\begin{itemize}
		\item $\beta,\gamma\in \varrho_{\mathrm{ze}}^\infty \calA^{\calE}(M)$ satisfying $\operatorname{supp} \beta,\operatorname{supp} \gamma \Subset \bar{V}_{h_0,\lambda_0} \backslash \mathrm{ie}$, thus supported away from $\mathrm{ie}\cup \mathrm{be}$,
		\item $R\in \varrho_{\mathrm{fe}}^\kappa \calA^{\calE}(M)$ satisfying 
		\begin{equation} 
		\operatorname{supp} R \Subset \bar{V}_{h_0,\lambda_0} \cap \bar{U}_{h_0,\Lambda}
		\label{eq:Rsuppcond},
		\end{equation}  
		thus supported away from $\mathrm{ie}\cup \mathrm{ze}\cup \mathrm{be}$, and 
		\item $f_0,g_0 \in \varrho_{\mathrm{ze}}^\infty \varrho_{\mathrm{fe}}^\infty C^\infty(M;\bbC)= h^\infty C^\infty(M;\bbC)$ satisfying $\operatorname{supp} f_0,\operatorname{supp} g_0 \Subset \bar{V}_{h_0,\lambda_1}$, thus supported away from $\mathrm{be}$, 
	\end{itemize}
	such that the function $u= \beta Q + \gamma Q'$ solves $Pu =(f+f_0) Q + (g+g_0) Q'+R$. 
	\label{prop:high_corner}
\end{proposition}
Thus, given $O(\varrho_{\mathrm{ze}}^\infty)$ forcing relative to $Q$, we can solve the ODE modulo $O(h^\infty)$ errors (relative to $Q$) near $\mathrm{ze}\cap \mathrm{fe}$. 
\begin{proof}
	If $Q=0$, then the claim holds trivially, so assume otherwise.  
	We construct $u$ order-by-order. 
	By linearity, we may assume without loss of generality that one of $f,g$ is identically $0$. Both cases are similar, so we focus on the case where $g= 0$.

	Expand $f$ at $\mathrm{fe}$ as 
	\begin{equation} 
		\varrho^{-\kappa}f \sim \sum_{(j,k)\in \calE}^\infty f_{j,k} \varrho^j \log^k \varrho 
		\label{eq:misc_146}
	\end{equation}
	where $f_{j,k} \in \varrho_{ \mathrm{ze}}^\infty C^\infty_{\mathrm{c}}(\mathrm{fe}\backslash \mathrm{be}) \subset \calS(\bbR_\lambda)$. Necessarily, $\operatorname{supp} f_{j,k} \cap [0,\lambda_1)=\varnothing$. Let $\mathsf{f}$ denote the formal series on the right-hand side of \cref{eq:misc_146}. 
	
	Fix $\lambda_{1/2} \in (\lambda_0,\lambda_1)$, $\lambda_2 \in (\lambda_1,\Lambda)$, and $\lambda_3 \in (\lambda_2,\Lambda)$. 
	Suppose that we have formal polyhomogeneous series 
	\begin{equation} 
	\mathsf{b} = \sum_{(j,k)\in \calE} b_{j,k} \varrho^j \log^k \varrho, \qquad \mathsf{c} = \sum_{(j,k)\in \calE} c_{j,k} \varrho^j \log^k \varrho
	\end{equation} 
	with coefficients $b_{j,k},c_{j,k} \in \calS(\bbR_\lambda)\cap C_{\mathrm{c}}^\infty (\lambda_2,\infty]$ such that, setting $\mathsf{u} = \mathsf{b} Q(\lambda) + \mathsf{c}Q'(\lambda)$, there exists some formal series
	\begin{equation}
	\mathsf{R} = \sum_{(j,k)\in \calE} R_{j,k}(\lambda) \varrho^j \log^k \varrho
	\end{equation}
	with coefficients $R_{j,k}\in C_{\mathrm{c}}^\infty(\lambda_{1/2},\lambda_3)$, such that
	\begin{equation} 
	\mathsf{P}(\mathsf{b} Q(\lambda) + \mathsf{c} Q'(\lambda)) = \mathsf{R}+ \mathsf{f} Q(\lambda)
	\end{equation}
	holds as an identity of formal polyhomogeneous series. Then, $\mathsf{b},\mathsf{c}$ may be asymptotically summed to yield $\beta\in \varrho_{\mathrm{ze}}^\infty \calA^{\calE}(M)$ and $\gamma\in\varrho_{\mathrm{ze}}^\infty \calA^{\calE}(M)$ whose polyhomogeneous expansions at $\mathrm{fe}$ are given by $\mathsf{b},\mathsf{c}$ respectively, and 
	these can be chosen to be supported in $\bar{V}_{h_0,\lambda_0}\backslash \mathrm{ie}$. Indeed, we necessarily have that $\beta,\gamma$ are supported there modulo an $O(h^\infty)$ error, which can be subtracted off.
	Likewise, $\mathsf{R}$ may be asymptotically summed to yield $R\in \varrho_{\mathrm{fe}}^\kappa \calA^{\calE}(M)$ whose expansion at $\mathrm{fe}$ is given by $\varrho^{\kappa}\mathsf{R}$, and $R$ can be chosen so as to satisfy the support condition \cref{eq:Rsuppcond}.

	Given $\beta,\gamma$ as in the previous paragraph, let $u= \beta Q+\gamma Q'$. By construction, the polyhomogeneous expansion of $\varrho^{-\kappa} Pu$ at $\mathrm{fe}^\circ$ is given by $\mathsf{R}+ \mathsf{f}Q(\lambda)$, and the error $R_0 = Pu - f Q - R$ satisfies 
	\begin{equation} 
	R_0\in \varrho_{\mathrm{ze}}^\infty\varrho_{\mathrm{fe}}^\infty \operatorname{Mod}_Q^{1/2} C^\infty_{\mathrm{c}}(\bar{V}_{h_0,\lambda_0};\bbC) \subset \varrho_{\mathrm{ze}}^\infty\varrho_{\mathrm{fe}}^\infty\varrho_{\mathrm{be}}^\infty  \operatorname{Mod}_Q^{1/2} C^\infty(M;\bbC),
	\end{equation}
	where $\operatorname{Mod}_Q(\lambda)$ is the \emph{modulus function} defined in the appendix, by \cref{eq:modulus_def}.
	
	Letting $f_0 = R_0 Q^*/ (|Q|^2+|Q'|^2)$ and $g_0 = R_0 (Q')^*/ (|Q|^2+|Q'|^2)$, $f_0,g_0 \in \varrho_{\mathrm{ze}}^\infty \varrho_{\mathrm{fe}}^\infty C^\infty_{\mathrm{c}}(\bar{V}_{h_0,\lambda_0};\bbC)$
	holds, as does the desired identity $Pu =(f+f_0) Q + (g+g_0) Q'+R$.

	To summarize, we have shown that it suffices to construct formal series $\mathsf{b},\mathsf{c},\mathsf{R}$ with the desired properties. Such formal series can be constructed recursively. Indeed, 
	\begin{multline}
	\mathsf{P}(\mathsf{b} Q(\lambda) + \mathsf{c} Q'(\lambda)) = \sum_{(j,k) \in\calE} \varrho^j \log \varrho^k \Big[ \hat{N}(P) (b_{j,k} Q + c_{j,k} Q') \\ + \sum_{(j',k) \in \calE\text{ and } j' < j} E_{j-j'-2}(\lambda) (b_{j',k} Q + c_{j',k} Q') \Big]. 
	\end{multline}
	Having already defined $b_{j',k},c_{j',k}$ for $j'<j$, set
	\begin{equation}
	v_{j,k} = \hat{N}(P)^{-1}[\tilde{Q},\hat{Q}]\Big[ f_{j,k} Q -\sum_{(j',k) \in \calE\text{ and } j' < j} E_{j-j'-2}(\lambda) (b_{j',k} Q + c_{j',k} Q')\Big], 
	\end{equation}
	where, $\tilde{Q},\hat{Q}$, $\hat{N}(P)^{-1}[\tilde{Q},\hat{Q}]$ are as in \S\ref{subsec:inh}. By \Cref{prop:normal_mapping}, $v_{j,k}$ can be modified in a compact subset so that it can be written as
	$b_{j,k} Q + c_{j,k} Q'$ for $b_{j,k},c_{j,k} \in \calS(\bbR_\lambda)\cap C_{\mathrm{c}}^\infty (\lambda_2,\infty]$ satisfying  
	\begin{equation} 
	\hat{N}(P) v_{j,k} = \hat{N}(P) (b_{j,k} Q + c_{j,k} Q')  =  R_{j,k}+f_{j,k} Q -\sum_{(j',k) \in \calE\text{ and } j' < j} E_{j-j'-2}(\lambda) (b_{j',k} Q + c_{j',k} Q')
	\end{equation}
	for $R_{j,k}\in C_{\mathrm{c}}^\infty(\lambda_{1/2},\lambda_3)$ (whose support actually has to be contained in $[\lambda_1,\lambda_3)$). 
	By construction, $\mathsf{P}(\mathsf{b} Q(\lambda) + \mathsf{c} Q'(\lambda)) = \mathsf{R} + \mathsf{f} Q(\lambda)$ is satisfied.
\end{proof}

\begin{figure}
	\begin{tikzpicture}
		\draw[->] (-1,0) -- (5,0) node[right] {$\lambda$};
		\fill[black] (-.5,0) circle (1.5pt) node[above] {$\lambda_0$};
		\fill[black] (.25,0) circle (1.5pt) node[below] {$\lambda_{1/2}$};
		\fill[black] (1,0) circle (1.5pt) node[below] {$\lambda_1$};
		\fill[black] (2,0) circle (1.5pt) node[below] {$\lambda_2$};
		\fill[black] (3,0) circle (1.5pt) node[below] {$\lambda_3$};
		\fill[black] (4,0) circle (1.5pt) node[below] {$\Lambda$};
		\draw[->] (1,0) to[out=0,in=190] (2.7,1) to[out=10, in=160] (5,.5) node[above] {$f_{j,k}$};
		\draw[->] (2.4,0) to[out=20, in=250] (3.5,2) to[out=70,in=155] (5,1.5) node[above right] {$b_{j,k},c_{j,k}$};
		\draw[darkred] (1,0) to[out=0, in=180] (1.92,1.5) node[above left] {$R_{j,k}$} to[out=0,in=180]  (2.8,0);
	\end{tikzpicture}
	\caption{The supports of the functions appearing in the proof of \Cref{prop:high_corner}. (Functions may not be to scale.)}
\end{figure}

\section{$O(h^\infty)$ error analysis}
\label{sec:error}

In this section, we study, away from $\mathrm{be}$, the forced ODE $Pu = f Q + g Q'$ for $
f,g \in  \varrho_{\mathrm{fe}}^\infty\varrho_{\mathrm{ze}}^\infty C^\infty(M) =  h^\infty C^\infty([0,Z]_\zeta\times [0,\infty)_{h};\bbC)$, which we take to be
supported away from $\mathrm{be}$. 
The goal is to construct an \emph{exact} solution which is $O(h^\infty)$ relative to the (possibly exponentially growing) quasimode $Q$.
Since we are concerned with the situation away from $\mathrm{be}$, it is acceptable to produce errors away from $\mathrm{ie}\cup \mathrm{ze}$. 
We restrict attention to $h\in [0,h_0)$ for some $h_0>0$ which can be taken as small as desired.

If $\varsigma<0$, i.e.\ we are in the classically allowed case, then the forcing $F=fQ + g Q'$ satisfies $F\in h^\infty C^\infty([0,Z]_\zeta\times [0,\infty)_{h};\bbC)$ as well. 
However, when $\varsigma>0$, i.e.\ we are in the classically forbidden case, then it is critical to keep track of the exponential weights at $\mathrm{ze}$ present in the quasimodes and their derivatives, as these can outweigh any $O(h^\infty)$ multiplier. One way of accomplishing this is to work with the modulus function
\begin{equation}
\operatorname{Mod}_Q(\zeta,h) = \operatorname{Mod}_Q(\zeta/h^{2/(\kappa+2)}), 
\end{equation}  
where $\operatorname{Mod}_Q(\lambda)$ is defined by \cref{eq:modulus_def}, except, in order to avoid conflicting notation, replace $\chi$ in \cref{eq:modulus_def} by $\chi_0$, where $\chi_0 \in C_{\mathrm{c}}^\infty(\bbR;[0,1])$ is identically $1$ in some neighborhood of the origin:
\begin{equation}
	\operatorname{Mod}_Q(\lambda) = |Q(\lambda)|^2 + \lambda^2  \langle \chi_0(1/\lambda)\lambda^{(\kappa+2)/2} \rangle^{-2} |Q'(\lambda)|^2.
\end{equation} 
We utilize the weighted Banach spaces
\begin{equation} 
\operatorname{Mod}_Q^\nu L^\infty =
\operatorname{Mod}_Q(\zeta,h)^\nu L^\infty (V_{h_0,\lambda_0};\bbC),
\end{equation}
$\nu\in \bbR$, 
In addition, let 
\begin{equation}
h^\infty \operatorname{Mod}_Q^\nu L^\infty= \bigcap_{k\in \bbN} h^k \operatorname{Mod}_Q^\nu L^\infty.
\end{equation}
When $\nu=0$, the ``$\smash{\operatorname{Mod}_Q^0}$'' is omitted from all notation.    

The two main properties of $\operatorname{Mod}_Q$ used (mostly without further comment) here are \Cref{prop:Mod1} and \Cref{prop:Mod2}. 

The main result of this section is the following:
\begin{proposition} 
Fix $Q\in \calQ$. Given $f,g \in h^\infty C^\infty([0,Z]_\zeta\times [0,\infty)_{h};\bbC)$ satisfying $(\operatorname{supp} f \cup \operatorname{supp} g)\cap \mathrm{be}=\varnothing$, there exist $\beta,\gamma$ with the following four properties:
\begin{enumerate}
	\item $\beta,\gamma \in h^\infty C^\infty([0,Z]_\zeta\times [0,\infty)_{h};\bbC)$ 
	\item $\operatorname{supp}_M \beta, \operatorname{supp}_M \gamma$ are disjoint from $\mathrm{be}$, and
	\item the function $u= \beta Q +  \gamma Q'$ solves $Pu = fQ + gQ' +  w$ for $w\in h^\infty C^\infty([0,Z]_\zeta\times [0,\infty)_{h};\bbC)$ such that $\operatorname{supp}_M w$ is disjoint from $\mathrm{ze} \cup \mathrm{be}$.
\end{enumerate}
Here, $\operatorname{supp}_M F$ is the closure in $M$ of $\{p\in M^\circ : F(p)\neq 0\}$.
\label{prop:error}
\end{proposition}
\begin{proof}
	The claim is trivially true if $Q=0$, so assume otherwise.

	We choose a basis $A,B$ for $\calQ$. If $\varsigma>0$, choose $A = Q_{\infty}$ to be the exponentially decaying mode (one should imagine $A$ as Airy-like) and $B$ to be any other element of $\calQ$ such that $\operatorname{span}_\bbC\{A,B\}=\calQ$. If $\varsigma<0$, choose instead $A=Q_-$ and $B=Q_+$.

	Fix $\chi \in C^\infty_{\mathrm{c}}(M;[0,1])$ identically $1$ near $\mathrm{ze}$ and vanishing near $\mathrm{be}$, and suppose that $\bar{\chi}\in C_{\mathrm{c}}^\infty(M;[0,1])$ is identically $1$ near $\mathrm{ze}$ and identically vanishing on $\operatorname{supp}(1-\chi)$.

	As an ansatz for $u$, take 
	\begin{equation}
	u = \aleph \bar{\chi} A\Big( \frac{\zeta}{h^{2/(\kappa+2)}}\Big) + \beth \bar{\chi}  B\Big( \frac{\zeta}{h^{2/(\kappa+2)}}\Big)
	\end{equation}
	for to-be-constructed
	$\aleph,\beth \in  C^\infty((0,Z]_\zeta\times (0,\infty)_{h};\bbC)$. 
	Below, we will abbreviate $A(\zeta/h^{2/(\kappa+2)})$ and 
	$B(\zeta/h^{2/(\kappa+2)})$ as $A,B$. 
	
	We first address sufficient conditions on $(\aleph,\beth)$ for such a $u$ to be expressable as $\beta Q + \gamma Q'$ for $\beta,\gamma$ satisfying properties (1) and (2). Define
	\begin{align}
	\beta &= \operatorname{Mod}_Q(\zeta,h)^{-1} Q(\zeta/h^{2/(\kappa+2)})^* u, \\
	\gamma &= \operatorname{Mod}_Q(\zeta,h)^{-1} \zeta^2 h^{-(\kappa+2)}  \langle h^{-1}\chi_0 (h^{2/(\kappa+2)}/\zeta)\zeta^{(\kappa+2)/2} \rangle^{-2} Q'(\zeta/h^{2/(\kappa+2)})^*  u,
	\end{align}
	where $\chi_0$ is as above.
	Then, $\beta,\gamma \in C^\infty((0,Z)_\zeta\times (0,\infty)_h;\bbC )$ and $u = \beta Q + \gamma Q'$. 
	The presence of the cutoff $\bar{\chi}$ in $u$ means that property (2) is satisfied automatically.
	
	In order for property (1) to be satisfied, the $\aleph,\beth$ are to be constructed satisfying the following estimates:
	\begin{itemize}
		\item if $\varsigma<0$ or $Q\notin \operatorname{span}_\bbC Q_\infty$, then, for every $j,k \in \bbN$, 
		\begin{equation}
		\frac{\partial^j\partial^k (\chi \aleph) }{\partial h^j \partial \zeta^k} \in h^\infty \operatorname{Mod}_B L^\infty =  h^\infty \operatorname{Mod}_Q L^\infty, \qquad 
		\frac{\partial^j \partial^k (\chi \beth)}{\partial h^j \partial \zeta^k} \in h^\infty L^\infty.
		\label{eq:misc_135}
		\end{equation}
		\item If $\varsigma>0$ and  $Q\in \operatorname{span}_\bbC Q_\infty \backslash \{0\}$, then 
		\begin{equation}
		\frac{\partial^j\partial^k (\chi \aleph) }{\partial h^j \partial \zeta^k} \in h^\infty  L^\infty, 
		\qquad 
		\frac{\partial^j \partial^k (\chi \beth)}{\partial h^j \partial \zeta^k} \in h^\infty \operatorname{Mod}_{A} L^\infty = h^\infty \operatorname{Mod}_{Q_\infty} L^\infty.
		\label{eq:misc_136}
		\end{equation}
	\end{itemize}
	We now check that these suffice to conclude property (1). If $\varsigma<0$ or $Q\notin \operatorname{span}_\bbC Q_\infty$, then this computation is recorded in \Cref{lemma:tech1} below. Indeed, $\beta,\gamma$ are a linear combination of functions of the form discussed in the lemma with polyhomogeneous coefficients on $M$ that are smooth at $\mathrm{ie}$. Since the product of a polyhomogeneous function on $M$ smooth at $\mathrm{ie}$ with an element of 
	\begin{equation} 
	h^\infty C^\infty([0,Z]_\zeta\times [0,\infty)_{h};\bbC) = \varrho_{\mathrm{fe}}^\infty \varrho_{\mathrm{ze}}^\infty C^\infty(M)
	\end{equation} 
	supported away from $\mathrm{be}$ results in an element of $h^\infty C^\infty([0,Z]_\zeta\times [0,\infty)_{h};\bbC)$, it can be concluded that (1) holds. 
	If $\varsigma>0$ and $Q\in \operatorname{span}_\bbC Q_\infty$, then we instead appeal to \Cref{lemma:tech2}, and the desired conclusion follows exactly as before. 
	
	In order to satisfy property (3), it suffices to arrange that the function $u_0 = \aleph A + \beth B$ satisfies $Pu_0 = f Q + g Q'$. The property then holds with 
	\begin{equation} 
	w = -(1-\bar{\chi}) (fQ +gQ') + [P,\bar{\chi}] u_0.
	\end{equation} 
	As follows from writing $P$ in the coordinate system $(\lambda,\varrho)$, the result of applying $[P,\bar{\chi}]$ to to $u_0$ lies in $h^\infty  C^\infty([0,Z]_\zeta\times [0,\infty)_{h};\bbC)$, and since $[P,\bar{\chi}]$ is vanishing near $\mathrm{ze}\cup \mathrm{be}$, the support of $w$ is disjoint from $\mathrm{ze}\cup\mathrm{be}$ as well. 
	
	To summarize, in order to conclude the proposition it suffices to construct $\aleph,\beth \in  C^\infty((0,Z]_\zeta\times (0,\infty)_{h};\bbC)$ with the following two properties:
	\begin{itemize}
		\item $\aleph,\beth$ satisfy the estimates above (\cref{eq:misc_135}, \cref{eq:misc_136}) and
		\item  $u_0 = \aleph A + \beth B$ satisfies $Pu = f Q + g Q'$.
	\end{itemize}
	The $\aleph,\beth$ are constructed by \Cref{prop:var} below. The proposition states that the second item holds. So, we only need to confirm that $\aleph,\beth$ satisfy the desired estimates.
	This is given by the conjunction of the lemmas below. Indeed, \Cref{prop:reduction} shows that it suffices to prove the $k=0$ cases of the desired estimates (with a slightly enlarged cutoff), deducing the others.  The propositions \Cref{prop:est1} and \Cref{prop:est2} provide the required $k=0$ result, at least if $\chi$ is supported in $\mathrm{cl}_M\{ \zeta \geq h^{2/(\kappa+2)} \lambda_0 \}$ for sufficiently large $\lambda_0$, which can be arranged.
\end{proof}

We now fill in the required lemmas.

\begin{lemma}
	Suppose that $\aleph,\beth \in C^\infty((0,Z)_\zeta\times (0,\infty)_h;\bbC)$ satisfy \cref{eq:misc_135} for each $j,k\in \bbN$. 
	Then, unless $\varsigma>0$ and $Q\in \operatorname{span}_\bbC Q_\infty$, each
	\begin{equation}
	\calR \in \Big\{ \frac{ \aleph \bar{\chi} Q^* A}{\operatorname{Mod}_Q(\zeta,h)}, \frac{ \beth \bar{\chi} Q^* B}{\operatorname{Mod}_Q(\zeta,h)}, \frac{ \aleph \bar{\chi} (Q')^* A}{\operatorname{Mod}_Q(\zeta,h)}, \frac{ \beth \bar{\chi} (Q')^* B}{\operatorname{Mod}_Q(\zeta,h)} \Big\}
	\label{eq:misc_13a}
	\end{equation}
	satisfies $\calR \in h^\infty  C^\infty([0,Z]_\zeta\times [0,\infty)_{h};\bbC)$. 
	\label{lemma:tech1}
\end{lemma}
\begin{proof}
	We first demonstrate the proof for $\calR = \operatorname{Mod}_Q(\zeta,h)^{-1} \aleph \bar{\chi} Q^* A$. For any $j,k\in \bbN$, $\smash{\partial_h^j\partial_\zeta^k \calR}$ is a linear combination of 
	\begin{equation}
	\frac{\partial^{j_0}\partial^{k_0}(\chi\aleph)}{\partial h^{j_0}\partial \zeta^{k_0}}
	\frac{\partial^{j_1}\partial^{k_1}\bar{\chi}}{\partial h^{j_1}\partial \zeta^{k_1}}
	\frac{\partial^{j_2}\partial^{k_2}Q^*}{\partial h^{j_2}\partial \zeta^{k_2}}
	\frac{\partial^{j_3}\partial^{k_3}A}{\partial h^{j_3}\partial \zeta^{k_3}} 
	\frac{\partial^{j_4}\partial^{k_4}}{\partial h^{j_4}\partial \zeta^{k_4}} \frac{1}{\operatorname{Mod}_Q}
	\label{eq:misc_147}
	\end{equation}
	for $j_0,j_1,j_2,j_3,j_4,k_0,k_1,k_2,k_3,k_4\in \bbN$ with $j_0+j_1+j_2+j_3+j_4 = j$ and $k_0+k_1+k_2+k_3+k_4=k$. By assumption, the first term in \cref{eq:misc_147} is in $h^\infty \operatorname{Mod}_B L^\infty$. Multiplying the second term by some large power of $h$, the result is in $\varrho_{\mathrm{be}}^\infty L^\infty$. (The reason is that $\partial_h,\partial_\zeta$ are powers of boundary-defining-functions times smooth vector fields on $M$.) Multiplying the third term by some large power of $\varrho_{\mathrm{be}}h$, the third term is in $\smash{\operatorname{Mod}_B^{1/2} L^\infty}$, per \Cref{prop:Mod1}. Likewise, multiplying the fourth term by some large power of $\varrho_{\mathrm{be}}h$, it is in $\smash{\operatorname{Mod}_A^{1/2} L^\infty}$. Finally, using \Cref{prop:Mod1} again, computing the higher derivatives of $\operatorname{Mod}_Q(\lambda)^{-1}$ yields
	\begin{equation}
	\varrho_{\mathrm{be}}^Kh^K\frac{\partial^{j_4}\partial^{k_4}}{\partial h^{j_4}\partial \zeta^{k_4}} \frac{1}{\operatorname{Mod}_Q} \in \operatorname{Mod}_B^{-1} L^\infty
	\end{equation}
	for some $K\in \bbN$. 
	So, all in all, $\partial_h^j\partial_\zeta^k \calR \in h^\infty L^\infty$. We can then conclude  $\calR \in h^\infty  C^\infty([0,Z]_\zeta\times [0,\infty)_{h};\bbC)$.
	
	Consider now the case $\calR = \operatorname{Mod}_Q(\zeta,h)^{-1} \beth \bar{\chi} Q^* B$. For any $j,k\in \bbN$, $\smash{\partial_h^j\partial_\zeta^k \calR}$ is a linear combination of 
	\begin{equation}
	\frac{\partial^{j_0}\partial^{k_0}(\chi\beth)}{\partial h^{j_0}\partial \zeta^{k_0}}
	\frac{\partial^{j_1}\partial^{k_1}\bar{\chi}}{\partial h^{j_1}\partial \zeta^{k_1}}
	\frac{\partial^{j_2}\partial^{k_2}Q^*}{\partial h^{j_2}\partial \zeta^{k_2}}
	\frac{\partial^{j_3}\partial^{k_3}B}{\partial h^{j_3}\partial \zeta^{k_3}} 
	\frac{\partial^{j_4}\partial^{k_4}}{\partial h^{j_4}\partial \zeta^{k_4}} \frac{1}{\operatorname{Mod}_Q}
	\label{eq:misc_149}
	\end{equation}
	for $j_0,j_1,j_2,j_3,j_4,k_0,k_1,k_2,k_3,k_4\in \bbN$ as above. By assumption, the first term in \cref{eq:misc_149} is in $h^\infty L^\infty$. The second and third were understood above. Multiplying the fourth term by some large power of $\varrho_{\mathrm{be}}h$, the fourth term is in $\smash{\operatorname{Mod}_B^{1/2} L^\infty}$. The fifth term was also understood above. 
	So, all in all, $\smash{\partial_h^j\partial_\zeta^k \calR }\in h^\infty L^\infty$. We can then conclude  $\calR \in h^\infty  C^\infty([0,Z]_\zeta\times [0,\infty)_{h};\bbC)$.

	Each of the remaining two cases are similar to one of the previous two. 
\end{proof}

\begin{lemmap}
	Suppose that $\varsigma>0$ and $Q\in \operatorname{span}_\bbC Q_\infty \backslash \{0\}$, and suppose that $\aleph,\beth \in C^\infty((0,Z)_\zeta\times (0,\infty)_h;\bbC)$ satisfy \cref{eq:misc_136} for each $j,k\in \bbN$. Then, the $\calR$ as in \cref{eq:misc_13a} 
	all satisfy $\calR \in h^\infty  C^\infty([0,Z]_\zeta\times [0,\infty)_{h};\bbC)$. 
	\label{lemma:tech2}
\end{lemmap}
The proof is completely analogous to that of \Cref{lemma:tech1}, so it is omitted. 

\subsection{Construction of $\aleph,\beth$}

\begin{proposition}[Variation of parameters]
	Let $F\in C^\infty( (0,Z]_\zeta \times (0,\infty)_h;\bbC)$, $\zeta_0 \in C^\infty( \bbR^+_h ; (0,Z])$. 
	There exist unique $\aleph,\beth \in C^\infty( (0,Z]_\zeta \times (0,\infty)_h;\bbC)$ vanishing at the graph $\Gamma(\zeta_0)=\{(\zeta_0(h),h) : h\in \bbR^+\} \subset (0,Z]_\zeta \times (0,\infty)_h$ such that, firstly,
	\begin{equation}
	\frac{\partial \aleph}{\partial \zeta} A 
	+\frac{\partial \beth}{\partial \zeta} B  = 0
	\label{eq:var1}
	\end{equation}
	and, secondly, the function $u$ defined by $u= \aleph A + \beth B$ satisfies the forced ODE $Pu=h^2 F$. These $\aleph,\beth$ solve 
	\begin{equation}
	\frac{\partial}{\partial \zeta}
	\begin{bmatrix}
	\aleph \\ \beth 
	\end{bmatrix}
	= \frac{h^{2/(\kappa+2)}}{\frakW} \bigg( E
	\begin{bmatrix}
	- A B & -B^2 \\ A^2 & AB
	\end{bmatrix}\begin{bmatrix}
	\aleph \\ \beth 
	\end{bmatrix}
	+ F \begin{bmatrix}
	B \\ - A
	\end{bmatrix}
	\bigg), 
	\label{eq:var}
	\end{equation}
	where $\frakW$ is the Wronskian $\frakW = A(\lambda) B'(\lambda) - A'(\lambda) B(\lambda)$, which is independent of $\lambda$ and therefore just a constant. 
	Moreover, $\aleph,\beth$ depend linearly on $F$. 
	\label{prop:var}
\end{proposition}
\begin{proof}
	Assuming that $\aleph,\beth \in C^\infty( (0,Z]_\zeta \times (0,\infty)_h;\bbC)$ satisfy \cref{eq:var1}, the function $u$ defined by $u = \aleph A + \beth B$ satisfies $Pu=h^2 F$ if and only if 
	\begin{equation}
	\frac{1}{h^{2/(\kappa+2)}} \Big( \frac{\partial \aleph}{\partial \zeta} A' + \frac{\partial \beth}{\partial \zeta} B'\Big) =  (\aleph A + \beth B) E  - F .
	\label{eq:var2}
	\end{equation} 
	Combining \cref{eq:var1}, \cref{eq:var2} into a single system of ODEs, the result is
	\begin{equation}
	\begin{bmatrix}
	A & B \\ 
	A' & B' 
	\end{bmatrix}
	\frac{\partial}{\partial \zeta}
	\begin{bmatrix}
	\aleph \\ \beth 
	\end{bmatrix} 
	= h^{2/(\kappa+2)} \Big( E \begin{bmatrix}
	0 & 0 \\ A & B
	\end{bmatrix}\begin{bmatrix}
	\aleph \\ \beth 
	\end{bmatrix} 
	- F \begin{bmatrix}
	0 \\ 1 
	\end{bmatrix} \Big).
	\label{eq:var_comb}
	\end{equation}
	Inverting the matrix on the left-hand side, 
	\begin{equation}
	\begin{bmatrix}
	A & B \\ A' & B'
	\end{bmatrix}^{-1} 
	= \frac{1}{\frakW} \begin{bmatrix}
	B' & - B \\ - A' & A
	\end{bmatrix}.
	\end{equation}
	Thus, \cref{eq:var_comb} is equivalent to \cref{eq:var}. For each $h>0$, this forced ODE has a unique solution $(\aleph(-,h),\beth(-,h)) \in C^\infty( (0,Z]_\zeta;\bbC^2)$ vanishing at $\zeta_0(h)$, as follows from the theory of linear ODE with smooth dependence on parameters. This solution depends linearly on $F$. 
\end{proof}

\begin{proposition}
	Suppose that $\aleph,\beth \in C^\infty( (0,Z]_\zeta \times (0,\infty)_h;\bbC)$ satisfy \cref{eq:var}. Let $\chi_0 \in C^\infty(M)$ be identically $1$ on $\operatorname{supp} \chi$ and vanishing near $\mathrm{be}$. If $\varsigma<0$ or $Q\notin \operatorname{span}_\bbC Q_\infty$ and
	\begin{equation}
	 \frac{\partial^j (\chi_0 \aleph) }{\partial h^j} \in h^\infty \operatorname{Mod}_Q L^\infty,\qquad
	\frac{\partial^j (\chi_0 \beth)}{\partial h^j} \in h^\infty L^\infty.
	\end{equation} 
	holds for all $j\in \bbN$, then \cref{eq:misc_135} holds for each $j,k\in \bbN$. Similarly, if $\varsigma>0$ and $Q\in \operatorname{span}_\bbC Q_\infty \backslash \{0\}$ and
	\begin{equation}
	\frac{\partial^j(\chi_0 \aleph) }{\partial h^j} \in h^\infty  L^\infty, \qquad
	\frac{\partial^j (\chi_0 \beth)}{\partial h^j} \in h^\infty \operatorname{Mod}_{Q_\infty} L^\infty.
	\end{equation}
	holds for all $j\in \bbN$, then \cref{eq:misc_136} holds for each $j,k\in \bbN$. 
	\label{prop:reduction}
\end{proposition}
\begin{proof} 
Using the ODE \cref{eq:var}, one concludes that, for each $k\in \bbN$,
	\begin{multline}
	\frac{\partial^j}{\partial h^j} \frac{\partial^{k}}{\partial \zeta^{k}}
	\begin{bmatrix}
	\chi \aleph \\ \chi \beth 
	\end{bmatrix} =   \sum_{j_0=0}^j\binom{j}{j_0}\Bigg[ \frac{\partial^{j-j_0}\partial^{k} \chi}{\partial h^{j-j_0} \partial \zeta^{k}} \frac{\partial^{j_0} }{\partial h^{j_0}} 
	\begin{bmatrix} \chi_0 \aleph \\ \chi_0 \beth \end{bmatrix} 
	+  \frac{1}{\frakW}\sum_{k_0=1}^{k}\binom{k}{k_0} \frac{\partial^{j-j_0}\partial^{k-k_0} \chi}{\partial h^{j-j_0} \partial \zeta^{k-k_0}} \\
	\times\sum_{j_1=0}^{j_0} \binom{j_0}{j_1} \Big( \frac{\partial^{j_0-j_1} h^{2/(\kappa+2)} }{\partial h^{j_0-j_1}} \Big) \Bigg[  \sum_{j_2=0}^{j_1} \sum_{k_1=0}^{k_0-1} \binom{j_1}{j_2}\binom{k_0-1}{k_1} \frac{ \partial^{j_1-j_2}\partial^{k_0-1-k_1} F}{  \partial h^{j_1-j_2} \partial \zeta^{k_0-1-k_1}}  \frac{\partial^{j_2} \partial^{k_1}}{\partial h^{j_2} \partial \zeta^{k_1}} 
	\begin{bmatrix}
	B \\ - A
	\end{bmatrix}    \\ 
	+\sum_{j_2=0}^{j_1} \sum_{k_1=0}^{k_0-1}\sum_{j_3=0}^{j_2} \sum_{k_2=0}^{k_1} \binom{j_1}{j_2}\binom{j_2}{j_3} \binom{k_0-1}{k_1}\binom{k_1}{k_2} \Big( \frac{\partial^{j_1-j_2} \partial^{k_0-1-k_1} E }{\partial h^{j_1-j_2} \partial \zeta^{k_0-1-k_1} } \Big) \calM_{j_3,k_2} \frac{\partial^{j_2-j_3} \partial^{k_1-k_2} }{\partial h^{j_2-j_3} \partial \zeta^{k_1-k_2} } \begin{bmatrix}
	\chi_0 \aleph \\ \chi_0 \beth
	\end{bmatrix}  
	\Bigg]\Bigg]
	\label{eq:misc_139}
	\end{multline}
	for any $j\in \bbN$, where 
	\begin{equation}
	\calM_{j,k} =  \frac{\partial^j \partial^k}{\partial h^j \partial \zeta^k} \begin{bmatrix}
	- AB & - B^2 \\ A^2 & A B
	\end{bmatrix}.
	\end{equation}
	(The factors of $\chi_0$ in \cref{eq:misc_139} are not required --- they can be dropped without changing the right-hand side, owing to the support conditions on $\chi,\chi_0$ --- but we include them for use below.)

	Each derivative of $\chi$ is (one-step) polyhomogeneous on $M$ and supported away from $\mathrm{be}$, so, using the assumptions,
	\begin{equation}
	 \frac{\partial^{j-j_0}\partial^{k} \chi}{\partial h^{j-j_0} \partial \zeta^{k}} \frac{\partial^{j_0} }{\partial h^{j_0}} 
	\begin{bmatrix} \chi_0 \aleph \\ \chi_0 \beth \end{bmatrix} 
	\in 
	\begin{cases}
	h^\infty \operatorname{Mod}_Q L^\infty\oplus h^\infty L^\infty  & (Q\notin \operatorname{span}_\bbC Q_\infty),\\ 
	h^\infty L^\infty \oplus h^\infty \operatorname{Mod}_{Q_\infty} L^\infty  & (\text{otherwise}),
	\end{cases}
	\end{equation}
	so this contribution lies in the desired space. 
	(Here, by ``$Q\notin \operatorname{span}_\bbC Q_\infty$'' we mean that $\varsigma<0$ or $Q\notin \operatorname{span}_\bbC Q_\infty$.)
	If $k=0$, then this term is the only term in \cref{eq:misc_139}, so the claim has been proven in this case.

	Applying the ODE $NQ=0$ to rewrite second and higher derivatives of $Q$ in terms of zeroth and first derivatives, the terms
	\begin{equation}
	\frac{\partial^{j-j_0}\partial^{k-k_0} \chi}{\partial h^{j-j_0} \partial \zeta^{k-k_0}} 
	\Big( \frac{\partial^{j_0-j_1} h^{2/(\kappa+2)} }{\partial h^{j_0-j_1}} \Big)  \frac{ \partial^{j_1-j_2}\partial^{k_0-1-k_1} F}{ \partial h^{k_0-1-k_2} \partial \zeta^{j-j_1}}  \frac{\partial^{j_2} \partial^{k_1}}{\partial h^{j_2} \partial \zeta^{k_1}} 
	\begin{bmatrix}
	B \\ - A
	\end{bmatrix} 
	\end{equation}
	in \cref{eq:misc_139}
	lie in the same space. If $\varsigma<0$, this is just a matter of multiplying elements of $h^\infty L^\infty$. If $\varsigma>0$, then we need to keep track of the exponential factors, with the result being that, if  $Q\notin \operatorname{span}_\bbC Q_\infty$, then the entries lie in the entries of 
	\begin{equation}
	h^\infty \operatorname{Mod}_Q^{1/2}  \begin{bmatrix}
	h^\infty B L^\infty + h^\infty B' L^\infty\\
	h^\infty A L^\infty+ h^\infty A'L^\infty
	\end{bmatrix}  \subseteq 
	\begin{cases}
	(h^\infty \operatorname{Mod}_Q L^\infty) \oplus 
	h^\infty L^\infty & (Q\notin \operatorname{span}_\bbC Q_\infty), \\ 
	h^\infty L^\infty \oplus h^\infty \operatorname{Mod}_{Q_\infty} L^\infty & (\text{otherwise}).
	\end{cases}
	\end{equation}

	The remaining terms in \cref{eq:misc_139} only involve fewer $\partial_\zeta$ falling on $\aleph,\beth$ than the left-hand side and so can be controlled inductively. Computing $\calM_{j,k}$, the result is that 
	\begin{equation}
	(\zeta^\infty h^\infty L^\infty) \calM_{j,k} 
	\subseteq \zeta^\infty h^\infty 
	\begin{bmatrix}
	L^\infty & \operatorname{Mod}_B L^\infty \\
	\operatorname{Mod}_{A} L^\infty & L^\infty 
	\end{bmatrix}. 
	\end{equation}
	So, if $\varsigma<0$ or $Q\notin \operatorname{span}_\bbC Q_\infty$, the corresponding terms in \cref{eq:misc_139} above are, once the result has been proven for all smaller $k$, known to be in 
	\begin{equation}
	 \zeta^\infty h^\infty 
	\begin{bmatrix}
	L^\infty & \operatorname{Mod}_B L^\infty \\
	\operatorname{Mod}_{A} L^\infty & L^\infty 
	\end{bmatrix} \begin{bmatrix}
	\operatorname{Mod}_Q L^\infty \\ 
	L^\infty
	\end{bmatrix}
	\subseteq  \zeta^\infty h^\infty \begin{bmatrix}
	\operatorname{Mod}_Q L^\infty \\ 
	L^\infty
	\end{bmatrix}.
	\end{equation}
	If $Q\in \operatorname{span}_\bbC Q_\infty \backslash \{0\}$, then the computation is instead 
	\begin{equation}
	\zeta^\infty h^\infty 
	\begin{bmatrix}
	L^\infty & \operatorname{Mod}_B L^\infty \\
	\operatorname{Mod}_{A} L^\infty & L^\infty 
	\end{bmatrix} \begin{bmatrix}
	L^\infty \\ 
	\operatorname{Mod}_{Q_\infty}  L^\infty
	\end{bmatrix}
	\subseteq  \zeta^\infty h^\infty \begin{bmatrix}
	L^\infty \\ 
	\operatorname{Mod}_{Q_\infty}  L^\infty
	\end{bmatrix}.
\end{equation} 
\end{proof}

The ODE \cref{eq:var}, together with vanishing initial conditions along the graph of a curve $\zeta_0:\smash{\bbR^+_h}\to [0,Z]$, can be combined into an integral equation. Let $G = -h^{2/(\kappa+2)} \frakW^{-1} F$  and $\smash{\hat{E}} = - \frakW^{-1} \zeta E \in C^\infty(M)$. 
Integrating \cref{eq:var}, the solution of the initial value problem is
\begin{equation}
\begin{bmatrix}
\aleph(\zeta,h) \\ \beth(\zeta,h) 
\end{bmatrix} 
=  \!\int_{\zeta}^{\zeta_0}\! \bigg( \frac{\hat{E}(\omega,h)}{\lambda(\omega,h)}
\begin{bmatrix}
-A(\omega,h)B(\omega,h)\!\!\! & -B(\omega,h)^2 \\ 
A(\omega,h)^2 & A(\omega,h)B(\omega,h)
\end{bmatrix} 
\begin{bmatrix} 
\aleph(\omega,h) \\ \beth(\omega,h) 
\end{bmatrix}+ G(\omega,h)
\begin{bmatrix}
B(\omega,h) \\ - A(\omega,h)
\end{bmatrix}
\bigg)\dd \omega,
\label{eq:var_integrated}
\end{equation}
where $\lambda(\zeta,h) = \zeta/h^{2/(\kappa+2)}$ as usual. Thus, $(\aleph,\beth)$ is a fixed point of an affine map whose linear part is given by
\begin{equation}
\begin{bmatrix}
\gimel(\zeta,h) \\ \daleth(\zeta,h) 
\end{bmatrix}
\mapsto \int_\zeta^{\zeta_0} \frac{\hat{E}(\omega,h)}{\lambda(\omega,h)} \begin{bmatrix}
-A(\omega,h)B(\omega,h) & -B(\omega,h)^2 \\ 
A(\omega,h)^2 & A(\omega,h)B(\omega,h)
\end{bmatrix}
\begin{bmatrix} 
\gimel(\omega,h) \\ \daleth(\omega,h) 
\end{bmatrix} \dd \omega. 
\end{equation}
We call this linear map $\Phi$ if $\zeta_0=Z$ and $\Xi$ if $\zeta_0(h) = \lambda_0 h^{2/(\kappa+2)}$.
We can consider $\Phi,\Xi$ as functions $L^\infty_{\mathrm{loc}}(V_{h_0,\lambda_0};\bbC)^2\to L^\infty_{\mathrm{loc}}(V_{h_0,\lambda_0};\bbC)^2$.

In the classically forbidden case, $\varsigma>0$, in order to get a solution with the desired properties, the initial data cannot be prescribed at an arbitrary location.  
At the level of technical details, the difficulty is that, for each $\zeta_0 \in [0,Z]\backslash \{\zeta\}$, the bound 
\begin{equation}
\Big|\int_{\zeta_0}^\zeta \exp\bigg( \pm \frac{4}{\kappa+2} \frac{\omega^{(\kappa+2)/2}}{h} \bigg) F(\omega,h) \dd \omega  \Big| \leq |\zeta-\zeta_0| \exp\bigg(\pm \frac{4}{\kappa+2} \frac{\zeta^{(\kappa+2)/2}}{h} \bigg) \lVert F \rVert_{L^\infty} 
\end{equation}
only holds for one choice of sign $\pm$, namely that matching the sign of $\zeta-\zeta_0$. It is only in this case that the exponential has the correct monotone behavior. Estimates of this form are key in the analysis below. 

The difficulty is not an artifact of the method. 
The fundamental issue is that, according to the Liouville--Green expansion, if one specifies initial conditions somewhere and supplies a forcing $F$ with $\operatorname{supp} F \Subset (0,Z)$, then it should be expected that a solution to $Pu=F$ will, for some $\zeta \in (0,Z)$, grow exponentially fast as $h\to 0^+$ relative to $F$.
If the forcing $F$ is $F=f Q+ g Q'$ for $O(h^\infty)$ coefficients $f,g$, then $u$ is growing exponentially faster than $Q,Q'$. Since the overshoot is exponentially bad, even the $O(h^\infty)$ terms in $f,g$ are not sufficient to restore decay. This is why, unless $Q\in \operatorname{span}_\bbC Q_\infty$, we instead supply initial conditions for \emph{small} $\zeta$, say at $\zeta=0$, or, in what ends up being necessary if $\alpha\neq 1/2$, along $\Gamma_{\lambda_0} = \{\zeta = \lambda_0 h^{(\kappa+2)/2}\}$ for some $\lambda_0>0$.
However, if $Q\in \operatorname{span}_\bbC Q_\infty \backslash \{0\}$, then instead one has to supply the initial conditions at $\zeta = Z$. In the classically allowed case, all choices work equally well.  

So, we split the remaining lemmas among two subsections:
\begin{itemize}
	\item In \S\ref{subsec:good}, we handle the case where $\varsigma<0$ or $Q\in \operatorname{span}_\bbC Q_\infty$, i.e.\ the case where $Q$ is not exponentially growing.
	\item In \S\ref{subsec:bad}, we handle the remaining cases, i.e.\ the classically forbidden case with $Q$ generic. 
\end{itemize}
The classically allowed case, $\varsigma<0$, is actually handled twice, as the discussion in \S\ref{subsec:bad} applies to it.

See \cite[\S15.5]{SimonComplex} for an exposition of the general method in a simpler setting.

Let $\Theta = (1+\varsigma)/2$. 
Let $\calX_{h_0,\lambda_0}\subset L^\infty_{\mathrm{loc}}(V_{h_0,\lambda_0};\bbC)^2$ denote the Banach space
\begin{equation}
\calX_{h_0,\lambda_0} =
\begin{cases}
( \exp(4\Theta (\kappa+2)^{-1} \zeta^{(\kappa+2)/2}/h ) L^\infty (V_{h_0,\lambda_0};\bbC) ) \oplus  L^\infty(V_{h_0,\lambda_0} ;\bbC) & (Q\notin \operatorname{span}_\bbC Q_\infty), \\ 
L^\infty (V_{h_0,\lambda_0};\bbC) ) \oplus  ( \exp(-4 \Theta(\kappa+2)^{-1} \zeta^{(\kappa+2)/2}/h ) L^\infty (V_{h_0,\lambda_0};\bbC) ) & (\text{otherwise}).
\end{cases} 
\label{eq:misc_189}
\end{equation}

\subsection{Case without exponential growth}
\label{subsec:good}

\begin{lemma} 
	Suppose that either $\varsigma<0$ or that $Q\in \operatorname{span}_\bbC Q_\infty\backslash \{0\}$, so that the second case of \cref{eq:misc_189} holds. 
	For each $h_0,\lambda_0>0$, $\Phi$ is a bounded endomorphism of $\calX_{h_0,\lambda_0}$. Moreover, for fixed $h_0>0$, the operator norm $\lVert \Phi \rVert_{\calX_{h_0,\lambda_0}\to \calX_{h_0,\lambda_0}}$ satisfies 
	\begin{equation} 
	\lVert \Phi \rVert_{\calX_{h_0,\lambda_0}\to \calX_{h_0,\lambda_0}} =
	O(\lambda_0^{-(\kappa+2)/2})
	\end{equation} 
	as $\lambda_0\to\infty$ (meaning that nothing is implied as $\lambda_0\to 0^+$).
	\label{lem:Phi_boundedness}
\end{lemma}
\begin{proof}
	Let $\lVert - \rVert_{L^\infty}$ stand for the $L^\infty(V_{h_0,\lambda_0} ;\bbC)$ norm. Then, 
	\begin{equation} 
	\lVert \hat{E} \rVert_{L^\infty} = \lVert \hat{E} \rVert_{C^0(\mathrm{cl}_M V_{h_0,\lambda_0})}
	\end{equation} 
	is finite, since $\mathrm{cl}_M V_{h_0,\lambda_0}$ is a compact subset of $M$, on which $\hat{E}$ is continuous. 
	
	Furthermore,
	\begin{equation} 
		\lVert \lambda^{-1} AB \rVert_{L^\infty} = \lVert \lambda^{-1} A(\lambda) B(\lambda) \rVert_{L^\infty[\lambda_0,\infty)_\lambda} = O(\lambda_0^{-(\kappa+2)/2}),
	\end{equation} 
	and, letting $\lVert  \exp(\cdots) \lambda^{-1} B^2 \rVert_{L^\infty} = 
	\lVert  \exp(-4\Theta (\kappa+2)^{-1} \omega^{(\kappa+2)/2}/h ) \lambda(\omega,h)^{-1} B(\omega/h^{2/(\kappa+2)}) ^2 \rVert_{L^\infty}$, 
	\begin{equation} 
	\lVert  \exp(\cdots) \lambda^{-1} B^2 \rVert_{L^\infty} = \lVert \exp(-4 \Theta (\kappa+2)^{-1} \lambda^{(\kappa+2)/2} ) \lambda^{-1}B(\lambda)^2 \rVert_{L^\infty [\lambda_0,\infty)_\lambda} = O(\lambda_0^{-(\kappa+2)/2}) 
	\end{equation}
	as $\lambda_0\to \infty$. These estimates are immediate corollaries of \Cref{prop:quasimode_large_argument}. 
	
	Similarly, the quantity 
	\begin{equation} 
		\lVert  \exp(\cdots) \lambda^{-1} A^2 \rVert_{L^\infty} = 
	\lVert  \exp(4\Theta (\kappa+2)^{-1} \omega^{(\kappa+2)/2}/h ) \lambda(\omega,h)^{-1} A(\omega/h^{2/(\kappa+2)})^2 \rVert_{L^\infty}
	\end{equation} 
	satisfies 
	\begin{equation} 
	\lVert  \exp(\cdots) \lambda^{-1} A^2 \rVert_{L^\infty} = \lVert \exp(4 \Theta (\kappa+2)^{-1} \lambda^{(\kappa+2)/2} ) \lambda^{-1} A(\lambda)^2 \rVert_{L^\infty [\lambda_0,\infty)_\lambda} = O(\lambda_0^{-(\kappa+2)/2}).
	\end{equation}

	So, in order to prove the proposition, it suffices to bound $\smash{\lVert \Phi(\gimel,\daleth) \rVert_{\calX_{h_0,\lambda_0}}}$ by a product of $\lVert \hat{E} \rVert_{L^\infty}$, $\lVert (\gimel,\daleth) \rVert_{\calX_{h_0,\lambda_0}}$, and some linear combination of the three norms 
	\begin{equation} 
	\lVert \lambda^{-1} AB\rVert_{L^\infty}, \lVert \exp(\cdots)\lambda^{-1} B^2 \rVert_{L^\infty}, \lVert \exp(\cdots)\lambda^{-1} A^2 \rVert_{L^\infty} = O(\lambda_0^{(\kappa+2)/2}).
	\end{equation}  
	
	Let $\gimel \Phi$ denote the first component of $\Phi (\gimel,\daleth)$, and let $\daleth \Phi$ denote the second. 
	We want to bound 
	\begin{equation}
	\lVert \Phi (\gimel,\daleth)
	\rVert_{\calX_{h_0,\lambda_0}} = \lVert \gimel \Phi(\omega,h) \rVert_{L^\infty} + \lVert \exp(4\Theta (\kappa+2)^{-1} \omega^{(\kappa+2)/2}/h ) \daleth \Phi(\omega,h) \rVert_{L^\infty}
	\label{eq:misc_16f}
	\end{equation}
	for $\gimel \in  L^\infty(V_{h_0,\lambda_0};\bbC)$ and $\daleth(\zeta,h) \in \exp(-4\Theta (\kappa+2)^{-1} \zeta^{(\kappa+2)/2}/h ) L^\infty(V_{h_0,\lambda_0};\bbC)$. 
	The first term can be bounded as follows:
	\begin{multline}
	| \gimel \Phi(\zeta,h) | \leq Z \lVert \hat{E} \rVert_{L^\infty} ( \lVert\lambda^{-1} AB \rVert_{L^\infty} \lVert \gimel \rVert_{L^\infty} \\ +
	\lVert \exp(\cdots) \lambda^{-1} B^2 \rVert_{L^\infty} \lVert \exp(4\Theta (\kappa+2)^{-1} \omega^{(\kappa+2)/2}/h ) \daleth(\omega,h) \rVert_{L^\infty}  ) .
	\end{multline}
	In order to bound $\daleth\Phi$, the inequality 
	\begin{multline}
	|\daleth \Phi (\zeta,h)| \leq Z \lVert \hat{E} \rVert_{L^\infty}( \sup\{ \lambda^{-1}|A(\omega/h^{(\kappa+2)/2} )|^2 : \omega\geq \zeta \} \lVert \gimel \rVert_{L^\infty}  \\ +  \lVert \lambda^{-1} AB \rVert_{L^\infty}  \sup\{ |\daleth(\omega,h)| : \omega\geq \zeta\})
	\label{eq:misc_199}
	\end{multline}
	can be used. Because $\exp(-4\Theta(\kappa+2)^{-1} \zeta^{(\kappa+2)/2}/h)$ is decreasing on $(0,Z)$ for each individual $h$,
	\begin{align}
	&\operatorname{sup}\{\lambda^{-1} |A(\omega/h^{2/(\kappa+2)})|^2: \omega\geq \zeta\} \leq \exp(-4\Theta (\kappa+2)^{-1} \zeta^{(\kappa+2)/2}/h )  \lVert \exp(\cdots) \lambda^{-1} A^2 \rVert_{L^\infty}, 
	\intertext{and}
	&\operatorname{sup}\{|\daleth(\omega,h)| :\omega\geq \zeta \}  \leq \exp(-4\Theta (\kappa+2)^{-1} \zeta^{(\kappa+2)/2}/h ) \lVert  \exp(4\Theta (\kappa+2)^{-1} \omega^{(\kappa+2)/2}/h ) \daleth(\omega,h) \rVert_{L^\infty}.
	\end{align}
	So, \cref{eq:misc_199} yields
	\begin{multline} 
	\lVert \exp(4\Theta (\kappa+2)^{-1} \zeta^{(\kappa+2)/2}/h ) \daleth \Phi(\zeta,h) \rVert_{L^\infty} \\ \leq Z \lVert \hat{E} \rVert_{L^\infty} (\lVert \exp(\cdots) \lambda^{-1} A^2 \rVert_{L^\infty}+ \lVert \lambda^{-1} A B \rVert_{L^\infty} ) \lVert (\gimel,\daleth) \rVert_{\calX_{h_0,\lambda_0}}.
	\end{multline}
	Altogether, $\lVert \Phi (\gimel,\daleth) \rVert_{\calX_{h_0,\lambda_0}} \leq \Sigma Z \lVert \hat{E} \rVert_{L^\infty} \lVert (\gimel,\daleth) \rVert_{\calX_{h_0,\lambda_0}}$ for $\Sigma = 	\lVert \lambda^{-1} AB\rVert_{L^\infty}+ \lVert \exp(\cdots)\lambda^{-1} B^2 \rVert_{L^\infty}+ \lVert \exp(\cdots)\lambda^{-1} A^2 \rVert_{L^\infty}$. 
\end{proof}

\begin{proposition}
	There exists some $\lambda_{00}>0$ such that, if $\lambda_0\geq \lambda_{00}$, then for any $\aleph,\beth \in C^\infty((0,Z]_\zeta \times (0,\infty)_h;\bbC) $ satisfying \cref{eq:var_integrated} with $\zeta_0 = Z$, then, if $\varsigma<0$ or $Q\in \operatorname{span}_\bbC Q_\infty \backslash \{0\}$, then
	\begin{equation}
	\frac{\partial^j \aleph }{\partial h^j} \in h^\infty  L^\infty (V_{h_0,\lambda_0};\bbC), \qquad 
	\frac{\partial^j \beth }{\partial h^j} \in h^\infty  \operatorname{Mod}_{A} L^\infty (V_{h_0,\lambda_0};\bbC) 
	\end{equation} 
	hold for all $j\in \bbN$.
	\label{prop:est1}
\end{proposition}
\begin{proof}
	Since $\aleph,\beth$ depend linearly on $F$, which is $O(h^\infty)$ relative to $Q$, it actually suffices to prove only that  
	\begin{equation} 
		\partial_h^j  \aleph \in L^\infty (V_{h_0,\lambda_0};\bbC),\qquad \partial_h^j \beth \in \operatorname{Mod}_A L^\infty (V_{h_0,\lambda_0};\bbC).
		\label{eq:misc_255}
	\end{equation} 
	Indeed, once this is known, we can use the $O(h^\infty)$ assumption on $F$ (and the linear dependence of $\aleph,\beth$ on $F$) to 	 
	conclude that $\partial_h^j  \aleph \in h^k L^\infty (V_{h_0,\lambda_0};\bbC)$ and $\partial_h^j \beth \in h^k \operatorname{Mod}_A L^\infty (V_{h_0,\lambda_0};\bbC)$ for all $j,k\in \bbN$.
	
	Let 
	\begin{align}
	\calI &= \int_{\zeta}^{Z} \bigg( \frac{\hat{E}(\omega,h)}{\lambda(\omega,h)}
	\begin{bmatrix}
	-A(\omega,h)B(\omega,h)\!\!\! & -B(\omega,h)^2 \\ 
	A(\omega,h)^2 & A(\omega,h)B(\omega,h)
	\end{bmatrix} 
	\begin{bmatrix} 
	\aleph(\omega,h) \\ \beth(\omega,h) 
	\end{bmatrix} \dd \omega, \\ 
	\calC &= \int_\zeta^{Z} G(\omega,h)  \begin{bmatrix}
	B(\omega,h) \\ - A(\omega,h)
	\end{bmatrix} \dd \omega. 
	\end{align} 
	Then, $(\aleph,\beth)^\intercal = \calI+\calC$. Differentiating, 
	\begin{equation}
	\frac{\partial^j}{\partial h^j} 
	\begin{bmatrix}
	\aleph \\ \beth 
	\end{bmatrix}
	= \frac{\partial^j \calI}{\partial h^j} + \frac{\partial^j \calC}{\partial h^j}.
	\end{equation}

	First consider 
	\begin{equation}
	\frac{\partial^k \calC}{\partial h^k} = 
	\sum_{k_1=0}^{k} \binom{k}{k_1}  \int_\zeta^Z \frac{\partial^{k-k_1} G}{\partial h^{k-k_1}} \frac{\partial^{k_1} }{\partial h^{k_1}}  \begin{bmatrix}
	B(\omega,h) \\ - A(\omega,h)
	\end{bmatrix} \dd \omega.
	\label{eq:misc_188}
	\end{equation}
	Each of the integrands above is in $h^\infty \zeta^{-K} \operatorname{Mod}_Q^{1/2} (( \operatorname{Mod}_B^{1/2} L^\infty) \oplus ( \operatorname{Mod}_A^{1/2} L^\infty)) = h^\infty \zeta^{-K} (L^\infty \oplus (\operatorname{Mod}_A L^\infty) )$ for some large $K\in \bbN$. Thus, by the to-be-proven \Cref{lem:modAint}, each term in the sum in \cref{eq:misc_188} is in $h^\infty (L^\infty \oplus \operatorname{Mod}_{A} L^\infty)$. 
	
	On the other hand, consider
	\begin{multline}
	\frac{\partial^k \calI}{\partial h^k} = \sum_{k_1=0}^{k}\sum_{k_2=0}^{k_1}\sum_{k_3=0}^{k_2} \binom{k}{k_1,k_2,k_3} \int_\zeta^Z \frac{\partial^{k_0-k_1} \hat{E}(\omega,h)}{ \partial h^{k_0-k_1} } \Big( \frac{\partial^{k_1-k_2}}{\partial h^{k_1-k_2}} \frac{1}{\lambda(\omega,h)} \Big)  \\ 
	\times \calM_{k_2-k_3,0}(\omega,h) \frac{\partial^{k_3}}{\partial h^{k_3}} 
	\begin{bmatrix}
	\aleph(\omega,h) \\ \beth(\omega,h)
	\end{bmatrix} \dd \omega.
	\label{eq:misc_200}
	\end{multline}
	Suppose that we have proven \cref{eq:misc_255} for $j<k$.
	Inserting a cutoff $\chi_0 \in C^\infty(M)$ which is identically $1$ on $V_{h_0,\lambda_0}$ and therefore on $\mathrm{cl}_MV_{h_0,\lambda_0} $ and identically vanishing near $\mathrm{be}$, the inductive hypothesis shows, unless $k_3=k$, that the integrand in \cref{eq:misc_200} satisfies 
	\begin{equation}
	\frac{\partial^{k_0-k_1} \hat{E}(\omega,h)}{ \partial h^{k_0-k_1} } \Big( \frac{\partial^{k_1-k_2}}{\partial h^{k_1-k_2}} \frac{1}{\lambda(\omega,h)} \Big)  \calM_{k_2-k_3,0}(\omega,h) \frac{\partial^{k_3}}{\partial h^{k_3}} 
	\begin{bmatrix}
	\aleph(\omega,h) \\ \beth(\omega,h)
	\end{bmatrix} 
	\in h^\infty 
	\begin{bmatrix}
	L^\infty \\ 
	\operatorname{Mod}_A L^\infty 
	\end{bmatrix}
	\end{equation}
	on $\bar{V}_{h_0,\lambda_0}$. Thus, by \Cref{lem:modAint}, each term in the sum in \cref{eq:misc_200}, except possibly the single term with $k_3 = k$, lies in $h^\infty (L^\infty \oplus \operatorname{Mod}_A L^\infty)$.

	To summarize, for some $\calC_j \in h^\infty (L^\infty (V_{h_0,\lambda_0};\bbC) \oplus \operatorname{Mod}_{A} L^\infty (V_{h_0,\lambda_0};\bbC))$, 
	\begin{equation}
	\frac{\partial^j}{\partial h^j} 
	\begin{bmatrix}
	\aleph \\ \beth 
	\end{bmatrix}
	= \Phi \frac{\partial^j}{\partial h^j}
	\begin{bmatrix}
	\aleph \\ 
	\beth 
	\end{bmatrix}
	+ \calC_j. 
	\end{equation}
	Consider the map $L^\infty (V_{h_0,\lambda_0};\bbC) \oplus \operatorname{Mod}_{A} L^\infty (V_{h_0,\lambda_0};\bbC)\to L^\infty (V_{h_0,\lambda_0};\bbC) \oplus \operatorname{Mod}_{A} L^\infty (V_{h_0,\lambda_0};\bbC)$ given by
	\begin{equation}
	\begin{bmatrix}
	\gimel \\ \daleth
	\end{bmatrix}
	\mapsto \Phi\begin{bmatrix}
	\gimel \\ \daleth
	\end{bmatrix}
	+ \calC_j. 
	\label{eq:misc_191}
	\end{equation}
	So, $(\partial_h^j \aleph ,\partial_h^j \beth)^\intercal$ is a fixed point of the affine map \cref{eq:misc_191}.

	By \Cref{lem:Phi_boundedness}, if $\lambda_0$ is sufficiently large, then $\Phi$ is a contraction map. Thus, the map \cref{eq:misc_191} has a unique fixed point $(\gimel_j,\daleth_j)$, and it is given by a convergent series $\sum_{n=0}^\infty \Phi^n \calC_j$. For any $h_1>0$, the map \cref{eq:misc_191} is also a contraction map on $L^\infty (V_{h_0,\lambda_0} \cap \{h\geq h_1\};\bbC) \oplus \operatorname{Mod}_{A} L^\infty (V_{h_0,\lambda_0}\cap \{h\geq h_1\};\bbC)$, and $(\gimel_j|_{\{h\geq h_1\}},\daleth_j|_{\{h\geq h_1\}})$ is a fixed point for it. But, so is 
	\begin{equation} 
	\partial_h^j(\aleph,\beth)|_{V_{h_0,\lambda_0}\cap \{h\geq h_1\}} \in C^\infty(V_{h_0,\lambda_0}\cap \{h\geq h_1\})^2. 
	\end{equation}
	Thus, by the uniqueness of the fixed point, $(\gimel_j,\daleth_j)$ agrees with $\smash{\partial_h^j}(\aleph,\beth)$ in $\{h\geq h_1\}$. Since $h_1$ can be taken arbitrarily small, the agreement holds everywhere in $V_{h_0,\lambda_0}$. We can therefore conclude that 
	\begin{equation}
	\partial_h^j (\aleph,\beth) \in L^\infty (V_{h_0,\lambda_0};\bbC) \oplus \operatorname{Mod}_{A} L^\infty (V_{h_0,\lambda_0};\bbC). 
	\end{equation}
\end{proof}

\begin{lemma}
	If $H\in h^\infty  \operatorname{Mod}_A L^\infty(V_{h_0,\lambda_0};\bbC)$, then $\int_\zeta^Z H(\omega,h) \dd \omega \in h^\infty \operatorname{Mod}_A L^\infty(V_{h_0,\lambda_0};\bbC)$. 
	\label{lem:modAint}
\end{lemma}
\begin{proof}
	In the $\varsigma<0$ case, the claim is true by the polynomial growth/decay of the elements of $\calQ$, so assume $\varsigma>0$, in which case $A\in \operatorname{span}_\bbC Q_\infty \backslash \{0\}$. 
	We can write $H(\zeta,h) = h^k H_k(\zeta,h) \operatorname{Mod}_A(\zeta,h)$ for some $H_k \in h^k L^\infty (V_{h_0,\lambda_0};\bbC)$. Then, 
	\begin{equation}
	\Big|\int_\zeta^Z H(\omega,h)\dd \omega \Big| \leq h^k Z \lVert H_k \rVert_{L^\infty} \sup\{\operatorname{Mod}_A (\omega,h) : \omega\in [\zeta,Z] \} . 
	\end{equation}
	The supremum on the right-hand side is finite and can be bounded as 
	\begin{equation}
	\sup\{\operatorname{Mod}_A (\omega,h) : \omega\in [\zeta,Z] \} \leq \operatorname{sup}\{ \operatorname{Mod}_A(\lambda) : \lambda\geq \lambda(\zeta,h) \} < \infty
	\end{equation}
	on $V_{h_0,\lambda_0}$. 
	By \Cref{lem:mod_exp_lower}, there exists some $\lambda_1>\lambda_0$ such that for all $\lambda,\lambda'\geq \lambda_1$ with $\lambda>\lambda'$,  $\operatorname{Mod}_A(\lambda) \leq 2 \operatorname{Mod}_A(\lambda')$. Consequently:
	\begin{equation}
	\operatorname{sup}\{ \operatorname{Mod}_A(\lambda) : \lambda\geq \lambda(\zeta,h) \} \leq 
	\begin{cases}
	2 \operatorname{Mod}_A(\lambda(\zeta,h)) & (\lambda(\zeta,h) \geq \lambda_1 ), \\ 
	\operatorname{max}\{2 \operatorname{Mod}_A(\lambda_1) , \operatorname{sup}_{\lambda \in [\lambda(\zeta,h) , \lambda_1]}\operatorname{Mod}_A(\lambda)  \} & (\text{otherwise} ).
	\end{cases}
	\end{equation} 
	If the second case holds, then $\lambda(\zeta,h) \in [\lambda_0,\lambda_1]$. Because $\operatorname{Mod}_A(\lambda)$ is nonvanishing, there exists some $C>1$ such that $\operatorname{Mod}_A(\lambda') \leq C \operatorname{Mod}_A(\lambda)$ for all $\lambda,\lambda' \in [\lambda_0,\lambda_1]$. The inequality above therefore yields 
	\begin{equation}
	\operatorname{sup}\{ \operatorname{Mod}_A(\lambda) : \lambda\geq \lambda(\zeta,h) \} \leq 2C \operatorname{Mod}_A(\lambda(\zeta,h)) = 2 C \operatorname{Mod}_A(\zeta,h).
	\end{equation}
	So, all in all, 
	\begin{equation}
	\Big|\int_\zeta^Z H(\omega,h)\dd \omega \Big| \leq 2 C h^k Z \lVert H_k \rVert_{L^\infty}  \operatorname{Mod}_A(\zeta,h) \in h^k \operatorname{Mod}_A L^\infty .
	\end{equation}
	Since $k$ was arbitrary, this completes the proof. 
\end{proof}

\subsection{Case without exponential decay}
\label{subsec:bad}

\begin{lemma}
	If $\varsigma<0$ or $Q\notin \operatorname{span}_\bbC Q_\infty$, so that the \emph{first} case of \cref{eq:misc_189} holds, then, for each $h_0,\lambda_0>0$, $\Xi$ is a bounded endomorphism of $\calX_{h_0,\lambda_0}$. Moreover, for fixed $h_0>0$, the operator norm $\lVert \Xi \rVert_{\calX_{h_0,\lambda_0}\to \calX_{h_0,\lambda_0}}$ satisfies 
	\begin{equation} 
	\lVert \Xi \rVert_{\calX_{h_0,\lambda_0}\to \calX_{h_0,\lambda_0}} =
	O(\lambda_0^{-(\kappa+2)/2})
	\end{equation} 
	as $\lambda_0\to\infty$. 
	\label{lem:Xi_boundedness}
\end{lemma}
\begin{proof}
	Let $\gimel\Xi$ denote the first component of $\Xi$, and let $\daleth \Xi$ denote the second. 
	We want to bound 
	\begin{equation}
	\lVert \Xi (\gimel,\daleth)
	\rVert_{\calX_{h_0,\lambda_0}} = \lVert \exp(-4\Theta (\kappa+2)^{-1} \omega^{(\kappa+2)/2}/h ) \gimel \Xi(\omega,h) \rVert_{L^\infty} + \lVert \daleth \Xi \rVert_{L^\infty}
	\label{eq:misc_205}
	\end{equation}
	for $\gimel \in \exp(4\Theta (\kappa+2)^{-1} \zeta^{(\kappa+2)/2}/h ) L^\infty(V_{h_0,\lambda_0};\bbC)$ and $\daleth \in L^\infty(V_{h_0,\lambda_0};\bbC)$.  
	The final term above can be bounded by
	\begin{multline}
	| \daleth \Xi(\zeta,h) | \leq Z \lVert \hat{E} \rVert_{L^\infty} ( \lVert \exp(\cdots) \lambda^{-1} A^2 \rVert_{L^\infty} \lVert \exp(-4\Theta (\kappa+2)^{-1} \omega^{(\kappa+2)/2}/h ) \gimel(\omega,h) \rVert_{L^\infty} \\ +
	\lVert \lambda^{-1} A B \rVert_{L^\infty} \lVert \daleth \rVert_{L^\infty}  ) .
	\end{multline}
	In order to bound the first term in \cref{eq:misc_205}, the inequality
	\begin{multline}
	|\gimel \Xi (\zeta,h)| \leq Z \lVert \hat{E} \rVert_{L^\infty}( \lVert \lambda^{-1} AB \rVert_{L^\infty} \operatorname{sup}\{|\gimel(\omega,h)| :\omega\leq \zeta \}  \\ +  \operatorname{sup}\{\lambda^{-1} |B(\omega/h^{2/(\kappa+2)})|^2: \omega\leq \zeta\} \lVert \daleth \rVert_{L^\infty} )
	\end{multline}
	can be used. Because $\exp(4\Theta(\kappa+2)^{-1} \zeta^{(\kappa+2)/2}/h)$ is increasing on $(0,Z)$ for each individual $h$, 
	\begin{align}
	&\operatorname{sup}\{|\gimel(\omega,h)| :\omega\leq \zeta \}  \leq \exp(4\Theta (\kappa+2)^{-1} \zeta^{(\kappa+2)/2}/h ) \lVert  \exp(-4\Theta (\kappa+2)^{-1} \omega^{(\kappa+2)/2}/h ) \gimel(\omega,h) \rVert_{L^\infty}, 
	\intertext{and}
	&\operatorname{sup}\{\lambda^{-1} |B(\omega/h^{2/(\kappa+2)})|^2: \omega\leq \zeta\} \leq \exp(4\Theta (\kappa+2)^{-1} \zeta^{(\kappa+2)/2}/h )  \lVert \exp(\cdots) \lambda^{-1} B^2 \rVert_{L^\infty}.  
	\end{align}
	So, 
	\begin{multline} 
	\lVert \exp(-4\Theta (\kappa+2)^{-1} \omega^{(\kappa+2)/2}/h ) \gimel \Xi(\omega,h) \rVert_{L^\infty} \\ \leq Z \lVert \hat{E} \rVert_{L^\infty} ( \lVert \lambda^{-1} A B \rVert_{L^\infty} +  \lVert \exp(\cdots) \lambda^{-1} B^2 \rVert_{L^\infty}) \lVert (\gimel,\daleth) \rVert_{\calX_{h_0,\lambda_0}}.
	\end{multline} 
	Altogether, $\lVert \Xi (\gimel,\daleth) \rVert_{\calX_{h_0,\lambda_0}} \leq \Sigma Z \lVert \hat{E} \rVert_{L^\infty} \lVert (\gimel,\daleth) \rVert_{\calX_{h_0,\lambda_0}}$ for $\Sigma = 	\lVert \lambda^{-1} AB\rVert_{L^\infty}+ \lVert \exp(\cdots)\lambda^{-1} B^2 \rVert_{L^\infty}+ \lVert \exp(\cdots)\lambda^{-1} A^2 \rVert_{L^\infty}$. 
\end{proof}

\begin{propositionp}
	There exists some $\lambda_{00}>0$ such that, if $\lambda_0\geq \lambda_{00}$, then for any $\aleph,\beth \in C^\infty((0,Z]_\zeta \times (0,\infty)_h;\bbC) $ satisfying \cref{eq:var_integrated} with $\zeta_0(h) = \lambda_0 h^{2/(\kappa+2)}$, then, if $\varsigma<0$ or $Q\notin \operatorname{span}_\bbC Q_\infty$, then
	\begin{equation}
	\frac{\partial^j \aleph }{\partial h^j} \in h^\infty \operatorname{Mod}_{B}  L^\infty (V_{h_0,\lambda_0};\bbC), \qquad 
	\frac{\partial^j \beth }{\partial h^j} \in h^\infty   L^\infty (V_{h_0,\lambda_0};\bbC)
	\end{equation} 
	hold for all $j\in \bbN$.
	\label{prop:est2}
\end{propositionp}
The proof is analogous to the previous, for instance using \Cref{lem:Xi_boundedness} in place of \Cref{lem:Phi_boundedness}, and using \Cref{lem:modBint} in place of \Cref{lem:modAint}, so we omit the details.

\begin{lemma}
	If $H\in h^\infty  \operatorname{Mod}_B L^\infty(V_{h_0,\lambda_0};\bbC)$, then, letting $\zeta_0 = \lambda_0 h^{2/(\kappa+2)}$, $\int_{\zeta_0}^\zeta H(\omega,h) \dd \omega \in h^\infty \operatorname{Mod}_B L^\infty(V_{h_0,\lambda_0};\bbC)$. 
	\label{lem:modBint}
\end{lemma}
\begin{proof}
	In the $\varsigma<0$ case, the claim is true by the polynomial growth/decay of the elements of $\calQ$, so assume $\varsigma>0$. 
	We can write $H(\zeta,h) = h^k H_k(\zeta,h) \operatorname{Mod}_B(\zeta,h)$ for some $H_k \in L^\infty (V_{h_0,\lambda_0};\bbC)$. Then, 
	\begin{equation}
	\Big|\int_{\zeta_0}^\zeta H(\omega,h)\dd \omega \Big| \leq h^k Z \lVert H_k \rVert_{L^\infty} \sup\{\operatorname{Mod}_B (\omega,h) : \omega\in [\zeta_0,\zeta] \} . 
	\end{equation}
	The supremum on the right-hand side is finite and can be written as 
	\begin{equation}
	\sup\{\operatorname{Mod}_B (\omega,h) : \omega\in [\zeta_0,\zeta] \} \leq \operatorname{sup}\{ \operatorname{Mod}_B(\lambda) : \lambda_0\leq \lambda  \leq \lambda(\zeta,h) \} < \infty. 
	\end{equation}
	By \Cref{lem:mod_exp_lower}, there exists some $\lambda_1>0$ such that for all $\lambda,\lambda'\geq \lambda_1$ with $\lambda>\lambda'$,  $2\operatorname{Mod}_B(\lambda) \geq  \operatorname{Mod}_B(\lambda')$. Consequently: 
	\begin{equation}
	\operatorname{sup}\{ \operatorname{Mod}_B(\lambda) : \lambda_0\leq \lambda\leq \lambda(\zeta,h) \} \leq \max\{2 \operatorname{Mod}_B(\lambda(\zeta,h)) , \operatorname{sup}_{\lambda \in [\lambda_0, \lambda_1]}\operatorname{Mod}_B(\lambda) \}.
	\end{equation}
	In $\lambda>\lambda_0$, there exists some $C>2$ such that $\operatorname{Mod}_B(\lambda')\leq C\operatorname{Mod}_B(\lambda)$ whenever $\lambda'\in [\lambda_0,\lambda_1]$. Thus, 
	$\operatorname{sup}\{ \operatorname{Mod}_B(\lambda) : \lambda_0\leq \lambda\leq \lambda(\zeta,h) \} \leq C \operatorname{Mod}_B(\lambda(\zeta,h))$. So, 
	\begin{equation}
	\Big|\int_{\zeta_0}^\zeta H(\omega,h)\dd \omega \Big| \leq C h^k Z \lVert H_k \rVert_{L^\infty}  \operatorname{Mod}_B(\zeta,h) \in h^k \operatorname{Mod}_B L^\infty .
	\end{equation}
	Since $k$ was arbitrary, this completes the proof.
\end{proof}

\section{Analysis at the low corner of $\mathrm{fe}$}
\label{sec:low}

We now analyze the forced ODE $Pu=F$ near $\mathrm{fe} \cap \mathrm{be}$. The forcing $F \in C^\infty(M^\circ)$ is now assumed to satisfy 
\begin{equation} 
\operatorname{supp} F \Subset \bar{U}_{h_0, \Lambda }  = \{0\leq \lambda <\Lambda \text{ and } \varrho < \bar{\varrho}\} \subset M
\end{equation} 
for some $\Lambda>0$ and $\bar{\varrho} \in (0,Z/\Lambda)$, related to $h_0$ by $h_0 = \bar{\varrho}^{(\kappa+2)/2}$. (Note that $\bar{U}_{h_0,\Lambda}\cap \mathrm{ie} = \varnothing$.) 
Refer to \Cref{fig:low} for the definition of $\lambda,\varrho$. 
The behavior of $F$ at $\mathrm{be} = \{\lambda=0\}$ will be specified. 
We draw attention to the fact that $\Lambda$ may be arbitrarily large (but still finite). 
Thus, it is a bit misleading to refer to the analysis below as occurring ``near'' $\mathrm{fe}\cap \mathrm{be}$,  what matters is excluding $\mathrm{ie}\cup \mathrm{ze}$.

We can identify $\bar{U}_{h_0,\Lambda}$ with $[0,\Lambda)_\lambda\times [0,\bar{\varrho})_\varrho$. 
In the coordinates $(\lambda,\varrho)$, $P$ is a family $\{\hat{P}(\varrho)\}_{\varrho\geq 0}$ of operators on $[0,\Lambda]_\lambda$ converging, as $\varrho \to 0^+$, to $N(P)$ in an appropriate sense, e.g.\ in the topology of $\smash{\operatorname{Diff}^2_{\mathrm{b,E}}[0,\Lambda]_\lambda }$. The indicial roots, $\gamma_\pm = 1/2\pm \alpha$,
are independent of $\varrho$, per the assumption that $\alpha$ be a constant. Recall that we are assuming that $\Re \alpha >0$.

If $\calE,\calF$ are two index sets, we use 
\begin{equation} 
\calA_{\mathrm{c}}^{\calE,\calF}( [0,\Lambda)_\lambda\times [0,\bar{\varrho})_\varrho) \subseteq \calA^{-\infty,-\infty, \calE,\calF}(M)
\end{equation} 
to denote the set of polyhomogeneous functions $u$ with index set $\calE$ at $\{\varrho=0\}$ and $\calF$ at $\mathrm{be}=\{\lambda=0\}$ satisfying $\operatorname{supp} u \Subset \bar{U}_{h_0, \Lambda }= [0,\Lambda)_\lambda\times [0,\bar{\varrho})_\varrho$, so in particular $u$ is compactly supported.

\begin{proposition}
	Suppose that, for some index sets $\calE,\calF\subset \bbC\times \bbN$ with $\Re j>-3/2-\Re \alpha$ for every $(j,k)\in \calF$, we are given $F\in \calA_{\mathrm{c}}^{\calE,\calF}( [0,\Lambda)_\lambda\times [0,\bar{\varrho})_\varrho)$.
	Then, there exists a solution 
	\begin{equation} 
	u\in \calA_{\mathrm{c}}^{\calE,\calG}( [0,\Lambda)_\lambda\times [0,\bar{\varrho})_\varrho)
	\label{eq:misc_0jk}
	\end{equation} 
	to $Pu=F$ for some index set $\calG$. 
	\label{prop:low_corner}
\end{proposition}
\begin{proof}
	The map $[0,\bar{\varrho})_{\varrho} \ni \varrho \mapsto \hat{P}(\varrho)\in \lambda^{-2}\operatorname{Diff}_{\mathrm{b,E}}^2[0,\Lambda]_\lambda$
	is smooth. Recall that $\operatorname{Diff}_{\mathrm{b,E}}[0,\Lambda]$ consists of those differential operators in the algebra generated over $C^\infty[0,\Lambda]$ by $\lambda \partial_\lambda$. 
	Moreover, the coefficient of $\partial_\lambda$ in $P(\varrho)$ is identically zero.
	From the theory of regular singular ODE, there exist independent solutions 
	\begin{equation} 
	\{v_\pm(\lambda,\varrho)\}_{\varrho \in [0,\bar{\varrho})} \subseteq \calA^{\calF(\alpha),(0,0)}[0,\Lambda]_\lambda
	\end{equation}
	to $\hat{P}(\varrho) v_{\pm}(\lambda,\varrho) = 0$ depending smoothly on $\varrho$, all the way down to $\varrho=0$. Here, $\calF(\alpha)$ is the index set at $\lambda=0$. Choose $v_+$ to be recessive at $\lambda=0$.

	Now let $W(\varrho) =  v_-(\lambda,\varrho)v_+'(\lambda,\varrho)-v_-'(\lambda,\varrho) v_+(\lambda,\varrho) $ denote their Wronskian, where the primes denote differentiation in the first slot. The Wronskian is a function of $\varrho$ alone, smooth down to $\varrho=0$. 
	Since $v_\pm$ are independent, $W$ is nonvanishing. 
	Now let 
	\begin{equation}
	K(\lambda,\lambda',\varrho) = \frac{1}{W(\varrho)} 
	\begin{cases}
	v_- (\lambda,\varrho) v_+ (\lambda',\varrho) & (\lambda > \lambda'), \\ 
	v_+(\lambda,\varrho) v_-(\lambda',\varrho) & (\lambda<\lambda').
	\end{cases}
	\end{equation}
	Consider the function $v(\lambda,\varrho) = K(\lambda,F(-,\varrho),\varrho)$, i.e.\ 
	\begin{equation}
	v(\lambda,\varrho) = \int_0^\infty K(\lambda,\lambda',\varrho) F(\lambda',\varrho) \dd \lambda ' .
	\end{equation}
	The integral converges because $v_+(\lambda) \in \lambda^{1/2 + \Re \alpha} L^\infty_{\mathrm{loc}}[0,\infty)_\lambda$, so 
	\begin{equation} \Big| 
	\int_0^\lambda v_+(\lambda',\varrho) F(\lambda',\varrho) \dd \lambda'\Big| \preceq \int_0^\lambda s^{1/2 + \Re \alpha + \min \calF} \dd s \preceq \int_0^\lambda s^{-1+\epsilon} \dd s < \infty
	\end{equation}
	for some $\epsilon>0$, 
	where $\min \calF = \min\{\Re j : (j,k)\in \calF\}$.

	The function $v$ satisfies $\hat{P}(\varrho)v = F$, and it is polyhomogeneous: $v \in \calA^{\calE,\calG_0}( [0,\Lambda)_\lambda\times [0,\bar{\varrho})_\varrho)$
	for some index set $\calG_0$, which can be computed in terms of $\calF(\alpha)$ and $\calF$. 
	
	We have $v(\Lambda,\varrho),\partial_\lambda v(\lambda,\varrho)|_{\lambda=\Lambda} \in \calA^{\calE}[0,\bar{\varrho})_{\varrho}$. 
	
	Note that $v$ may not satisfy $\operatorname{supp} v \Subset [0,\Lambda)_\lambda \times [0,\bar{\varrho})_\varrho$, so $v$ is not the desired solution to $Pu=F$. The next step is to add to $v$ a linear combination of $v_\pm$ so that the resulting solution is compactly supported in $\lambda$. 
	
	Let $w(-,\varrho)$ denote the solution to $\hat{P}(\varrho)w(-,\varrho)=0$ with initial conditions $w(\Lambda,\varrho) =- v(\Lambda,\varrho)$ and $\partial_\lambda w(\lambda,\varrho)|_{\lambda=\Lambda}=-\partial_\lambda v(\lambda,\varrho)|_{\lambda=\Lambda}$. This can be written in terms of $v_\pm$ as 
	\begin{equation}
	w(\lambda,\varrho) =- \frac{1}{W(\varrho)} 
	\begin{bmatrix}
	v_-(\lambda,\varrho) \\ 
	v_+(\lambda,\varrho)
	\end{bmatrix}^\intercal 
	\begin{bmatrix}
	v_+'(\Lambda,\varrho) & -v_+(\Lambda,\varrho) \\ 
	-v_-'(\Lambda,\varrho) & v_-(\Lambda,\varrho)
	\end{bmatrix}
	\begin{bmatrix}
	v(\Lambda,\varrho) \\ 
	v'(\Lambda,\varrho)
	\end{bmatrix}.
	\end{equation}
	Thus, $w$ is also in $\calA^{\calE,\calF(\alpha)}( [0,\Lambda)_\lambda\times [0,\bar{\varrho})_\varrho)$. Letting $u=v+w$, we have $u\in \calA^{\calE,\calG}( [0,\Lambda)_\lambda\times [0,\bar{\varrho})_\varrho)$ for $\calG$ the smallest index set containing $\calG_0,\calF(\alpha)$. This solves the desired ODE with vanishing initial data at $\Lambda$. Since the forcing satisfies $\operatorname{supp}F \Subset [0,\Lambda)_\lambda\times [0,\bar{\varrho} )_\varrho$, this implies the compact  support condition in \cref{eq:misc_0jk}. 
\end{proof}

\begin{remark}[The $\Re\alpha=0$ case]
	In the main part of the paper, we assume that $\Re \alpha>0$. The reason (besides the fact that the ODE only depends on $\alpha^2$, so we may assume without loss of generality that $\Re \alpha\geq 0$) is that when $\Re \alpha=0$, then the two indicial roots $\gamma_\pm$ of the normal operator $N$ have the same real part. Possibly even $\gamma_-=\gamma_+$. 
	
	Most of the discussion above goes through, \textit{mutatis mutandis}, except the explicit index sets involved in expansions at $\mathrm{be}$ may change. Since the analysis at $\mathrm{be}$ is confined to this section, \S\ref{sec:low}, it suffices to examine how this section changes upon allowing $\Re \alpha=0$. (Some of the discussion in \S\ref{sec:quasimodes}, \S\ref{sec:modulus} also involves behavior at $\mathrm{be}$, but the only place this is used in the main argument is in the present section.) There, the main difference between the $\Re \alpha=0$ and $\Re \alpha\neq 0$ cases is that in the former case it no longer makes sense to speak of a ``recessive'' solution to the ODE. The only place we used this was in choosing a nonzero solution $v_+$ of the ODE such that 
	\begin{equation} 
		v_+(\lambda) \in \lambda^{1/2 + \Re \alpha} L^\infty_{\mathrm{loc}}[0,\infty)_\lambda.
	\end{equation}  
	If $\Re \alpha>0$, then a non-recessive solution of the ODE will not have this property, only lying in $\lambda^{1/2 - \Re \alpha} L^\infty_{\mathrm{loc}}[0,\infty)_\lambda$. However, if $\Re \alpha=0$, then \emph{every} solution of the ODE lies in $\lambda^{1/2} L^\infty_{\mathrm{loc}}[0,\infty)_\lambda$, so the argument goes through.
\end{remark}

\section{Main proof}
\label{sec:main}

We turn now to the proof of \Cref{thm2} in the $W=1$ case which we observed in \S\ref{sec:Langer} suffices. Our goal, given $Q\in \calQ$, is to construct
\begin{itemize}
	\item $\beta, \gamma \in \calA^{\calE_0,\infty}(M)$ with $\operatorname{supp} \beta,\operatorname{supp} \gamma$ disjoint from $\mathrm{be}$, and
	\item  $\delta \in \calA^{\calE_0,\calG}(M)$ with $\operatorname{supp} \delta \cap( \mathrm{ie}\cup \mathrm{ze})= \varnothing$, 
\end{itemize} 
where $\calE_0$ is the index sets referenced in the theorem, such that the function $u$ defined by
\begin{equation}
u =  (1+\varrho_{\mathrm{ze}}\varrho_{\mathrm{fe}} \beta) Q\Big( \frac{\zeta}{h^{2/(\kappa+2)}} \Big) +  h^{(2\kappa+2)/(\kappa+2)} \gamma Q'\Big( \frac{\zeta}{h^{2/(\kappa+2)}} \Big)+ \varrho_{\mathrm{be}}^{1/2-\alpha} \varrho_{\mathrm{fe}} \delta 
\label{eq:misc_218}
\end{equation}
solves $Pu=0$ in $\{h<h_0\}$ for some $h_0>0$.

The upshot of \S\ref{sec:ze}, as recorded in \Cref{prop:ze_upshot}, was that, letting $\calE_ 0 = \calE-1\subset \bbN\times \bbN$, with $\calE$ as in the proposition, there exist $\beta_0,\gamma_0\in \calA^{\calE_0}(M)$ with support disjoint from $\mathrm{be}$ such that the function 
\begin{equation} 
u_0 = (1+ \varrho_{\mathrm{ze}} \varrho_{\mathrm{fe}} \beta_0) Q + \varrho_{\mathrm{ze}}^{(\kappa+1)/(\kappa+2)} \varrho_{\mathrm{fe}} \gamma_0 Q' 
\end{equation} 
satisfies $Pu_0 = f_0 Q + g_0 Q'$ for $f_0,g_0 \in \smash{\varrho_{\mathrm{be}}^{-1}\varrho_{\mathrm{ze}}^\infty \varrho_{\mathrm{fe}}^\kappa \calA^{\calE}(M)}$, with $g_0$ supported away from $\mathrm{be}$. Write $f_0 = f_{00} + f_{01}$ for
\begin{equation}
	f_{00},f_{01} \in \smash{\varrho_{\mathrm{ze}}^\infty \varrho_{\mathrm{fe}}^\kappa \calA^{\calE}(M)}
\end{equation}
with $f_{00}$ supported away from $\mathrm{be}$ and $f_{01}$ supported away from $\mathrm{ze}$. This decomposition can be arranged because $\mathrm{ze}\cap \mathrm{be} =\varnothing$.

We now apply \Cref{prop:high_corner} with $f =-f_{00}$ and $g=-g_{0}$ to produce $\beta_1,\gamma_1 \in\smash{ \varrho_{\mathrm{ze}}^\infty \calA^{\calE_0,\infty}(M)}$ supported away from $\mathrm{ie}\cup \mathrm{be}$ and $R\in \smash{\varrho_{\mathrm{fe}}^\kappa \calA^{\calE,\infty}(M)}$ supported away from $\mathrm{ie}\cup\mathrm{ze}\cup\mathrm{be}$ such that the function $u_1 = \varrho_{\mathrm{fe}}(\beta_1 Q + \gamma_1 Q')$ solves $Pu_1 = (-f_{00}+f_1) Q + (-g_{0}+g_1)Q'+R$
for some $f_1,g_1 \in h^\infty C^\infty(M;\bbC)$ supported away from $\mathrm{be}$.

Now apply \Cref{prop:error} with $f = - f_1$ and $g = - g_1$ to produce $\beta_2,\gamma_2 \in h^\infty C^\infty(M;\bbC)$ supported away from $\mathrm{be}$ such that the function $u_2 = \beta_2 Q + \gamma_2 Q'$ solves $P u_2 = - f_1 Q - g_1 Q' + H$
for some $H\in h^\infty C^\infty(M;\bbC)$ supported away from $\mathrm{ie}\cup \mathrm{ze}\cup \mathrm{be}$. 
There exist $\Lambda,h_1>0$ such that the quantity $\bar{\varrho}$ defined by
\begin{equation} 
\bar{\varrho} = h_1^{2/(\kappa+2)}
\end{equation} 
satisfies $\bar{\varrho} < Z/\Lambda$ (so that the set $\bar{U}_{h_1,\Lambda}$ is defined) and  such that $(\operatorname{supp} R \cup \operatorname{supp} H) \cap \{h<h_1\} \Subset \bar{U}_{h_1,\Lambda}$.
Let $\chi \in C_{\mathrm{c}}^\infty[0,\infty)_h$ be identically $1$ near $h=0$ and supported in $\{h<h_1\}$.  

Finally, we apply \Cref{prop:low_corner} 
with $F = - \chi f_{01} Q  - \chi R - \chi H$
to give, for some index set $\calG$, $\delta_0\in \calA^{\calE,\calG}(M)$ with $\operatorname{supp} \delta_0 \cap (\mathrm{ie}\cup \mathrm{ze}) = \varnothing$ satisfying 
\begin{equation} 
	P \delta_0 = - \chi f_{01} Q  - \chi R - \chi  H.
\end{equation} 
The hypothesis of that proposition regarding the index set of the forcing is satisfied, because the only contribution to the index set is $f_{01} Q$, which has index set $\calF(\alpha)-1$ at zero. So, $\min \{\calF(\alpha) - 1\} = -1/2- \Re \alpha > -3/2 - \Re \alpha$.

\begin{figure}[h!]
	\begin{tikzpicture}
	\fill[gray!5] (5,1.5) -- (0,1.5) -- (0,0)  arc(90:0:1.5) -- (5,-1.5) -- cycle;
	\draw[dashed] (5,1.5) -- (5,-1.5);
	\node (ff) at (.75,-.75) {fe};
	\node (zfp) at (-.35,.75) {$\mathrm{ze}$};
	\node (pf) at (2.5,1.8) {$\mathrm{ie}$};
	\node (mf) at (3.25,-1.8) {$\mathrm{be}$};
	\draw (5,1.5) -- (0,1.5) -- (0,0)  arc(90:0:1.5)  -- (5,-1.5);
	\filldraw[darkblue, opacity =.25] (0,0) arc(90:30:1.5) to[out=30, in=185] (5, 0) -- (5,1.5) -- (0,1.5) -- cycle;
	\filldraw[gray, opacity =.35, pattern = north east lines] (0,0) arc(90:50:1.5) to[out=40, in=180] (5, .7) -- (5,1.5) -- (0,1.5) -- cycle; 
	\node (?) at (1.2,.9) {$\operatorname{supp} \beta_0,\gamma_0$};
	\node[darkblue] (?) at (3.5,.25) {$\operatorname{supp} f_0,g_0$};
	\end{tikzpicture}
	\qquad
	\begin{tikzpicture}
	\fill[gray!5] (5,1.5) -- (0,1.5) -- (0,0)  arc(90:0:1.5) -- (5,-1.5) -- cycle;
	\draw[dashed] (5,1.5) -- (5,-1.5);
	\node[white] (ff) at (.75,-.75) {fe};
	\node[white] (zfp) at (-.35,.75) {$\mathrm{ze}$};
	\node[white] (pf) at (2.5,1.8) {$\mathrm{ie}$};
	\node[white] (mf) at (3.25,-1.8) {$\mathrm{be}$};
	\filldraw[darkblue, opacity =.25] (0,0) arc(90:10:1.5) to[out=20, in=182] (5, -.7) -- (5,1.5) -- (0,1.5) -- cycle;	
	\node[darkblue, opacity=.9] (?) at (4,1) {$\operatorname{supp} f_1,g_1$};
	\filldraw[darkgray, opacity =.25, pattern = north east lines] (0,0) arc(90:45:1.5) to[out=40, in=185] (2.7, .25) -- (2.7,1) -- (0,1) -- cycle;
	\node (?) at (1.1,.65) {$\operatorname{supp} \beta_1,\gamma_1$};
	\filldraw[darkred, opacity=.25] (3.75,.5) to[out=185,in=40] (.92,-.32) arc(52:28:1.5) to[out=30,in=185] (3.75,-.22) -- cycle; 
	\draw (5,1.5) -- (0,1.5) -- (0,0)  arc(90:0:1.5)  -- (5,-1.5);
	\node[darkred] (?) at (2.5,-.05) {$\operatorname{supp} R$};
	\end{tikzpicture}
	\begin{tikzpicture}
	\fill[gray!5] (5,1.5) -- (0,1.5) -- (0,0)  arc(90:0:1.5) -- (5,-1.5) -- cycle;
	\draw[dashed] (5,1.5) -- (5,-1.5);
	\node[white] (ff) at (.75,-.75) {fe};
	\node[white] (zfp) at (-.35,.75) {$\mathrm{ze}$};
	\node[white] (pf) at (2.5,1.8) {$\mathrm{ie}$};
	\node[white] (mf) at (3.25,-1.8) {$\mathrm{be}$};
	\filldraw[darkred, opacity=.25]  (5,1.5) -- (2.5,1.5) -- (2.5,.55) to[out=195, in=40] (.88,-.29)  arc(54:30:1.5) to[out=30, in=180] (5,.1)  -- cycle ;
	\filldraw[darkgray, opacity =.25, pattern = north east lines] (0,0) arc(90:50:1.5) to[out=40, in=188] (2.7, .45) -- (2.7,1.5) -- (0,1.5) -- cycle;
	\node (?) at (1.3,1) {$\operatorname{supp} \beta_2,\gamma_2$}; 
	\draw (5,1.5) -- (0,1.5) -- (0,0)  arc(90:0:1.5)  -- (5,-1.5);
	\node[darkred] (?) at (3.5,1) {$\operatorname{supp} H$};
	\end{tikzpicture}
	\qquad 
	\begin{tikzpicture}
	\fill[gray!5] (5,1.5) -- (0,1.5) -- (0,0)  arc(90:0:1.5) -- (5,-1.5) -- cycle;
	\draw[dashed] (5,1.5) -- (5,-1.5);
	\node[white] (ff) at (.75,-.75) {fe};
	\node[white] (zfp) at (-.35,.75) {$\mathrm{ze}$};
	\node[white] (pf) at (2.5,1.8) {$\mathrm{ie}$};
	\node[white] (mf) at (3.25,-1.8) {$\mathrm{be}$}; 
	\filldraw[darkred, opacity =.25] (1.5,-1.5) arc(0:52:1.5) to[out=40, in=185] (3.75,.5) -- (3.75,-1.5) -- cycle;
	\node[darkred] (?) at (2.75,-.2) {$\operatorname{supp} F$};
	\node () at (2.75,-.65) {$=$};
	\node (?) at (2.75,-1) {$\operatorname{supp} \delta_0$};
	\draw (5,1.5) -- (0,1.5) -- (0,0)  arc(90:0:1.5)  -- (5,-1.5);
	\filldraw[darkgray, opacity =.25, pattern = north west lines] (1.5,-1.5) arc(0:52:1.5) to[out=40, in=185] (3.75,.5) -- (3.75,-1.5) -- cycle;
	\end{tikzpicture}
	\caption{\textit{Top left}: the supports of the functions $\beta_0,\gamma_0,f_0,g_0$ appearing in the first step of the argument, using \S\ref{sec:ze}. \textit{Top right}: the supports of the functions $\beta_1,\gamma_1, f_1,g_1, R$ appearing in the second step, using \S\ref{sec:fe}. \textit{Bottom left}: the supports of the functions $\beta_2,\gamma_2,H$ appearing in the third step, using \S\ref{sec:error}. \textit{Bottom right}: the supports of the functions appearing in the final step of the argument, using \S\ref{sec:low}.}
	\label{fig:step1}
\end{figure}

We can write $\delta_0 = \varrho_{\mathrm{be}}^{1/2-\alpha} \varrho_{\mathrm{fe}} \delta$
for $\delta \in \calA^{\calE_0,\calG}$, for some other $\calG$.
Set $u=u_0+u_1+u_2+\delta_0$. 
This solves 
\begin{equation} 
Pu=(1-\chi(h)) (f_{01}+H+R)
\end{equation} 
and has the form \cref{eq:misc_218}
for 
\begin{align}
	\begin{split}
		\beta &= \beta_0 + \varrho_{\mathrm{ze}}^{-1}\beta_1 + \varrho_{\mathrm{ze}}^{-1} \varrho_{\mathrm{fe}}^{-1}\beta_2, \\
		\gamma &= \gamma_0 +\varrho_{\mathrm{ze}}^{-(\kappa+1)/(\kappa+2)} \gamma_1+\varrho_{\mathrm{ze}}^{-(\kappa+1)/(\kappa+2)}\varrho_{\mathrm{fe}}^{-1}\gamma_2
	\end{split}
\end{align}
By construction, $\beta,\gamma$ lie in the desired function spaces, and they have the required support properties. 

If a new $h_0>0$ is chosen such that $\chi=1$ identically on $[0,h_0]$, then $Pu=0$ for $h<h_0$.

\appendix

\section{Bounds on the modulus function}
\label{sec:modulus}

In the body of the paper, we saw that, when upgrading the quasimode $Q\in \calQ\backslash \{0\}$ to a full solution to the ODE, it is convenient to work with the \emph{modulus function} $\operatorname{Mod}_Q(\lambda)$ defined by
\begin{equation}
\operatorname{Mod}_Q(\lambda) = |Q(\lambda)|^2 + \lambda^2  \langle \chi(1/\lambda)\lambda^{(\kappa+2)/2} \rangle^{-2} |Q'(\lambda)|^2,
\label{eq:modulus_def}
\end{equation} 
where $\langle \rho \rangle = (1+\rho^2)^{1/2}$ and $\chi\in C_{\mathrm{c}}^\infty(\bbR;[0,1])$ is identically $1$ in some neighborhood of the origin. In this section, $L^\infty = L^\infty (\bbR^+_\lambda)$. 

\begin{lemma}
	$\operatorname{Mod}_Q(\lambda)\neq 0$ for any $\lambda>0$.
	\label{lem:nonvanishing} 
\end{lemma}
\begin{proof}
	This is the statement that $Q(\lambda)$ and $Q'(\lambda)$ cannot vanish simultaneously. Because $Q$ satisfies the second-order ODE $NQ=0$, it is the case that, for each $\lambda_0>0$, $Q$ is uniquely determined among elements of $\calQ$ by the pair $(Q(\lambda_0),Q'(\lambda_0))$. If we were to have $Q(\lambda_0),Q'(\lambda_0)=0$, then this would force $Q$ to vanish identically, contrary to assumption.
\end{proof}

The key property of the modulus functions is that they control the \emph{other} elements of $\calQ$ and their derivatives:
\begin{proposition}
	Suppose that either $\varsigma<0$ or $Q\notin \operatorname{span}_\bbC Q_\infty$. Let $A\in \calQ$. Then, for any $k\in \bbN$, there exists some $K\in \bbN$ depending on $k$ such that
	\begin{equation} 
	\frac{\partial^k A(\lambda)}{\partial \lambda^k} \in 
	\lambda^{-K} \langle \lambda \rangle^{2K}
	\operatorname{Mod}_Q(\lambda)^{1/2}L^\infty.
	\label{eq:misc_mlm}
	\end{equation}
	If $\varsigma>0$, then $\partial_\lambda^k Q_\infty \in \lambda^{-K} \langle \lambda \rangle^{2K} \operatorname{Mod}_{Q_\infty}(\lambda)^{1/2}L^\infty$ for some $K\in \bbN$ depending on $k$.
	\label{prop:Mod1} 
\end{proposition}
\begin{proof}
	By \Cref{lem:nonvanishing}, the claim holds within any bounded interval $I\Subset \bbR^+_\lambda$ worth of $\lambda$. 
	
	To prove the claim in an interval $I\Subset [0,\infty)_\lambda$, which may contain $0$, first use that, because $A$ is polyhomogeneous at $\lambda=0$ by \Cref{prop:quasimode_small_argument}, it and its derivatives are growing at most polynomially as $\lambda\to 0^+$. Then use that, because $Q$ is nonzero, the functions $\operatorname{Mod}_Q(\lambda),\operatorname{Mod}_{Q_\infty}(\lambda)$ are decaying at most polynomially as $\lambda\to 0^+$. So the proposition also holds here. 
	
	It therefore remains only to check that \cref{eq:misc_mlm} holds in $\{\lambda \geq \lambda_0\}$ for some $\lambda_{0} = \lambda_{0}(k,A,Q)>0$, and to check the corresponding result for $Q_\infty$.

	If $\varsigma<0$, then the elements of $\calQ$ and their derivatives all grow at worst polynomially as $\lambda\to \infty$, as shown by \Cref{prop:quasimode_large_argument}. 
	Similarly, \Cref{lem:mod_poly_lower} below shows that $\operatorname{Mod}_Q$ obeys a polynomial \emph{lower bound} in the same asymptotic regime. 
	The proposition therefore holds in this case.

	If $\varsigma>0$ and $Q\notin \operatorname{span}_\bbC Q_\infty$, then, by \Cref{prop:quasimode_large_argument}, it is the case that $|A(\lambda)|\leq C |Q(\lambda)|$ for some $C>0$ if $\lambda$ is sufficiently large. Thus, since $|Q(\lambda)| \leq \operatorname{Mod}_Q(\lambda)^{1/2}$,
	the proposition holds in the $k=0$ case. Consider now the $k=1$ case. Differentiating the conclusion of \Cref{prop:quasimode_large_argument} once, $A'(\lambda)$ is at worst growing polynomially as $\lambda\to\infty$ relative to $Q$, so the bound \cref{eq:misc_mlm} also holds in this case.
	Having verified the $k=0,1$ cases of the proposition, we deduce the remaining cases from the ODE $NA = 0$ which $A$ is defined to satisfy. If $k\in \bbN^{\geq 2}$, differentiating this ODE $k-2$ times yields 
	\begin{equation}
	 \frac{\partial^{k} A}{\partial \lambda^k} = \sum_{j=0}^{k-2}  \binom{k-2-j}{j}\frac{\partial^j A}{\partial \lambda^j}   \frac{\partial^{k-2-j}}{\partial \lambda^{k-2-j}} \Big( \varsigma \lambda^\kappa + \frac{1}{\lambda^2} \Big(\alpha^2-\frac{1}{4} \Big) + \frac{\Psi(\lambda)}{\lambda^2} \Big).
	\end{equation}
	Taking $K_0\in \bbN$ sufficiently large, each of the derivatives of $\varsigma \lambda^\kappa + \lambda^{-2} (\alpha^2-1/4) + \lambda^{-2} \Psi(\lambda)$ lies in $\lambda^{-K_0} \langle \lambda \rangle^{2K_0} L^\infty(\bbR^+_\lambda)$. 
	We can conclude that 
	\begin{equation} 
	\partial_\lambda^k A \in \lambda^{-K-K_0} \langle \lambda \rangle^{2K+2K_0} \operatorname{Mod}_Q(\lambda)^{1/2} L^\infty,
	\end{equation} assuming we have proven already that $\partial_\lambda^j A \in \lambda^{-K} \langle \lambda \rangle^{2K} \operatorname{Mod}_Q(\lambda)^{1/2} L^\infty$
	for each $j\in \{0,\ldots,k-2\}$.
	This inductive argument completes the proof in the $Q\notin \operatorname{span}_\bbC Q_\infty$ case. 
	
	In the $Q\in \operatorname{span}_\bbC Q_\infty$ case, then the result follows from the same inductive argument, except the base cases $k=0,1$ are now trivial because we are only comparing $|Q_\infty|^2,|Q_\infty'|^2$ with $\smash{\operatorname{Mod}_{Q_\infty}}$. 
\end{proof}

Let $\lambda = \rho^{-2/(\kappa+2)}$. 
Recall that $1/\lambda^{\kappa+2}$ is a boundary-defining-function of $\{\lambda=\infty\}= \mathrm{ze}\cap \mathrm{fe}$ in $\mathrm{fe}$.

\begin{lemma} 
	If $\varsigma>0$, then $\operatorname{Mod}_Q(\rho^{-2/(\kappa+2)}) \in \exp(4 (\kappa+2)^{-1} \rho^{-1} )\rho^{\kappa/(\kappa+2)}C^\infty [0,\infty)_\rho$. Unless $Q \in \operatorname{span}_\bbC Q_\infty$, 
	\begin{equation} 
	\operatorname{Mod}_Q(\rho^{-2/(\kappa+2)})^{-1} \in \exp(-4 (\kappa+2)^{-1}  \rho^{-1} ) \rho^{-\kappa/(\kappa+2)}C^\infty [0,\infty)_\rho.
	\end{equation} 
	Also, $\operatorname{Mod}_{Q_\infty}(\rho^{-2/(\kappa+2)}) \in \exp(-4 (\kappa+2)^{-1}\rho^{-1} ) \rho^{\kappa/(\kappa+2)}C^\infty [0,\infty)_\rho$,
	and $\operatorname{Mod}_{Q_\infty}(\rho^{-2/(\kappa+2)})^{-1} \in \exp(4 (\kappa+2)^{-1}\rho^{-1}) \rho^{-\kappa/(\kappa+2)}C^\infty [0,\infty)_\rho$.
	\label{lem:mod_exp_lower}
\end{lemma}
\begin{proof}
	Certainly $Q(\rho^{-2/(\kappa+2)})\in \exp(2 (\kappa+2)^{-1} \rho^{-1}) \rho^{\kappa/(2\kappa+4)} C^\infty [0,\infty)_\rho$.

	Writing $Q = \exp(2 (\kappa+2)^{-1} \rho^{-1} ) q$ for $q(\rho)\in \rho^{\kappa/(2\kappa+4)} C^\infty [0,\infty)_\rho$, taking derivatives yields 
	\begin{equation}
	\lambda \partial_\lambda Q(\rho^{-2/(\kappa+2)}) = \exp(2 (\kappa+2)^{-1} \rho^{-1})( \rho^{-1}  q +  \lambda \partial_\lambda q), 
	\label{eq:misc_258}
	\end{equation}
	and $\lambda \partial_\lambda q (\rho^{-2/(\kappa+2)}) \in \rho^{\kappa/(2\kappa+4)} C^\infty[0,\infty)_\rho$ by the dilation invariance of the operator $\lambda \partial_\lambda$. \Cref{eq:misc_258} yields 
	\begin{equation}
	\lambda \langle \lambda^{(\kappa+2)/2} \rangle^{-1} \partial_\lambda Q(\lambda) = \exp(2 (\kappa+2)^{-1}\rho^{-1} ) \rho^{\kappa/(2\kappa+4)} C^\infty[0,\infty)_\rho. 
	\end{equation}
	Thus, $\operatorname{Mod}_Q(\lambda)$ lies in the claimed space. Except for the $Q\in \operatorname{span}_\bbC Q_\infty$ case, the function $q$ has nonvanishing leading order term at $\lambda=\infty$, since otherwise $Q$ would be in $\operatorname{span}_\bbC Q_\infty$ by \Cref{prop:quasimode_large_argument}. From this, it follows that the leading order term of $\operatorname{Mod}_Q(\lambda)\exp(-4 (\kappa+2)^{-1} \rho^{-1} ) \in \rho^{\kappa/(\kappa+2)} C^\infty[0,\infty)_\rho$
	is nonvanishing.  This implies that $\operatorname{Mod}_Q(\lambda)^{-1}$ lies in the desired space. 
	
	The remaining clause of the theorem is proven similarly, switching the signs of the exponentials around. 
\end{proof}

\begin{proposition}
	Suppose that $\varsigma>0$ and $Q\notin \operatorname{span}_\bbC Q_\infty$. 
	For some $K\in \bbN$, we have $\operatorname{Mod}_{Q_\infty} \in \lambda^{-K} \langle \lambda \rangle^{2K} \operatorname{Mod}_Q^{-1} L^\infty$ and $\operatorname{Mod}_{Q_\infty}^{-1} \in \lambda^{-K} \langle \lambda \rangle^{2K} \operatorname{Mod}_{Q} L^\infty$. Consequently, for any $n,n',m,m'\in \bbZ$ satisfying $n'-m'=n-m$,
	it is the case that
	\begin{equation}
	\operatorname{Mod}_Q^n \operatorname{Mod}_{Q_\infty}^m L^\infty \subseteq \lambda^{-K} \langle \lambda \rangle^{2K} \operatorname{Mod}_Q^{n'} \operatorname{Mod}_{Q_\infty}^{m'} L^\infty 
	\end{equation}
	for some $K=K(n,m,n',m')\in \bbN$.
	\label{prop:Mod2}
\end{proposition}
\begin{proof}
	This is an immediate corollary of \Cref{lem:mod_exp_lower} (and \Cref{prop:quasimode_small_argument}, to control $\lambda\to 0^+$ behavior).
\end{proof}

If $\varsigma<0$, then, unless $Q \in \operatorname{span}_\bbC Q_\pm$, then the real-valued functions $|Q|^2$ and $\operatorname{Mod}_Q - |Q|^2$ have oscillatory terms as their leading asymptotics in the $\lambda\to \infty$ limit, and thus vanish infinitely often. We already know that the modulus function itself is nonvanishing, but it has to be ruled out that it does not take on any superpolynomially small values. Indeed:
\begin{lemma}
	If $\varsigma<0$, then, for some $C=C(Q)>0$, $\operatorname{Mod}_Q(\lambda) - C\lambda^{-\kappa/2} \in  \lambda^{-(\kappa+1)} L^\infty_{\mathrm{loc}} (0,\infty]_{\lambda}$.
	Thus, $\operatorname{Mod}_Q(\lambda)^{-1} -  C^{-1} \lambda^{\kappa/2} \in  \lambda^{-1} L^\infty_{\mathrm{loc}} (0,\infty]_{\lambda}$. 
	\label{lem:mod_poly_lower}
\end{lemma}
\begin{proof}
	We write $Q= c_- Q_- + c_+ Q_+$ for $c_-,c_+\in \bbC$ not both zero. As above, we can write $Q_\pm = \exp( \pm 2 i(\kappa+2)^{-1}  \lambda^{(\kappa+2)/2} ) q_\pm$ for $q_\pm \in \rho^{\kappa/(2\kappa+4)} C^\infty [0,\infty)_\rho$ having nonvanishing leading parts at $\lambda=\infty$, so that 
	\begin{equation}
		\lim_{\rho \to 0^+} \rho^{-\kappa/(2\kappa+4)} q_\pm \neq 0.
	\end{equation}
	Thus, for some $C_\pm \in \bbC\backslash \{0\}$, we have $q_\pm - \rho^{\kappa/(2\kappa+4)} C_\pm \in \rho^{(3\kappa+4)/(2\kappa+4)} C^\infty [ 0,\infty)_\rho$.
	Observe that 
	\begin{align}
	|Q|^2 &= |c_-|^2 |q_-|^2 + |c_+|^2 |q_+|^2+  2 \Re [  \exp( 4 i(\kappa+2)^{-1} \lambda^{(\kappa+2)/2} ) c_+ c_-^* q_+ q_-^*] \\ 
	&=  \rho^{\kappa/(\kappa+2)} (|c_- C_-|^2 + |c_+ C_+ |^2 +  2 \Re [  \cdots ]) + \rho^{(2\kappa+2)/(\kappa+2)} L^\infty,
	\end{align}
	where $\Re [\cdots] = \Re [\exp( 4 i(\kappa+2)^{-1} \lambda^{(\kappa+2)/2} ) c_+ c_-^* C_+ C_-^*]$.

	By the same computation as above, $\lambda \langle \lambda^{(\kappa+2)/2} \rangle^{-1} \partial_\lambda Q_\pm(\lambda)  = \exp(\pm 2i  (\kappa+2)^{-1} \lambda^{(\kappa+2)/2}) \tilde{q}_\pm$ for $\tilde{q}_\pm$ satisfying 
	\begin{equation}
	\tilde{q}_\pm \mp i \rho^{\kappa/(2\kappa+4)} C_\pm \in \rho^{(3\kappa+4)/(2\kappa+4)} C^\infty [0,\infty)_\rho.   
	\end{equation}
	Thus, 
	\begin{align}
	|\lambda \langle \lambda^{(\kappa+2)/2} \rangle^{-1}Q_\pm'|^2 &= |c_-|^2 |\tilde{q}_-|^2 + |c_+|^2 |\tilde{q}_+|^2+  2 \Re [  \exp( 4 i(\kappa+2)^{-1} \lambda^{(\kappa+2)/2} ) c_+ c_-^* \tilde{q}_+ \tilde{q}_-^*] \\ 
	&=  \rho^{\kappa/(\kappa+2)} (|c_- C_-|^2 + |c_+ C_+ |^2 -  2 \Re [  \cdots ]) + \rho^{(2\kappa+2)/(\kappa+2)}L^\infty,
	\end{align}
	So, $\operatorname{Mod}_Q(\rho^{-2/(\kappa+2)}) = \rho^{\kappa/(\kappa+2)} C + \rho^{(2\kappa+2)/(\kappa+2)} L^\infty$ for $C=2|c_-C_-|^2+2|c_+C_+|^2$. Rewriting this in terms of $\lambda$, the result is the claim.
\end{proof}

\section{The initial value problem}
\label{sec:ivp}

We now consider the derivation of \Cref{thm1} from \Cref{thm2}.

Let $Q_1,Q_2 \in \calQ$ denote a basis of $\calQ$. If $\varsigma<0$, then choose $Q_1 = Q_-$ and $Q_2 = Q_+$. If $\varsigma>0$, choose $Q_1 = Q_\infty$ and $Q_2$ to be an independent element of $\calQ$. 
Let $u_1,u_2$ denote the corresponding solutions to $Pu=0$ produced by \Cref{thm2}, for $Q=Q_1$ and $Q=Q_2$, respectively.

First, a lemma:
\begin{lemma}
	Suppose that $W=1$, so that $\zeta=z$. 
	The Wronskian $W\{u_1,u_2\}(h) = u_1 u_2 ' - u_1' u_2$ satisfies 
	\begin{equation}
	W\{u_1,u_2\}(h) - h^{-2/(\kappa+2)} W\{Q_1,Q_2\} \in h^{\kappa/(\kappa+2)} C^\infty [0,\infty)_h ,
	\label{eq:misc_231}
	\end{equation}
	where $W\{Q_1,Q_2\} = Q_1(\lambda)Q_2'(\lambda) - Q_1'(\lambda)Q_2(\lambda)$, evaluated at arbitrary $\lambda$. 
	\label{lem:Wr}
\end{lemma}
Here, $Q_j'(\lambda)=\partial_\lambda Q_j(\lambda)$, whereas $u_j'=\partial_\zeta u_j(\zeta,h)$. 
\begin{proof}

	Because $u_1u_2' - u_1'u_2$ is only a function of $h$, we can evaluate it at any $\zeta$. Near $\mathrm{ie}$, we can write 
	$u_1 = (1+h a_1) Q_1(\lambda)$ and $u_2 = (1+h a_2) Q_2(\lambda)$ for $a_1,a_2 \in C^\infty((0,Z]_\zeta\times [0,\infty)_{h};\bbC)$
	and $\lambda = \zeta/h^{2/(\kappa+2)}$, in part because, for the chosen $Q_\bullet$, the functions $Q_1(\lambda),Q_2(\lambda)$ are not vanishing for $\lambda$ sufficiently large. (This is just a way of saying that the Liouville--Green ansatz only requires a single term, unlike Langer's ansatz which involves both $Q,Q'$.)

	Then, explicitly, 
	\begin{multline}
	  W\{u_1,u_2\}(h) = h^{-2/(\kappa+2)} W\{Q_1,Q_2\} + hW\{Q_1(\lambda), a_2 Q_2(\lambda) \}(h) \\
	   + hW\{a_1 Q_1(\lambda), Q_2(\lambda)\}(h) + h^2 W\{a_1 Q_1(\lambda), a_2 Q_2(\lambda)\}(h).
	\end{multline}
	(Because $a_2 Q_2,a_1 Q_1$ typically do not solve the ODE, the ``Wronskians'' on the right-hand side depend on the $\lambda$ at which they are evaluated, but we will not denote this explicitly.) 
	
	The particular choice of $Q_1,Q_2$ means that 
	\begin{equation}
	h^{-\kappa/(\kappa+2)} Q_1(\lambda)Q_2(\lambda), Q_1'(\lambda) Q_2(\lambda), Q_1(\lambda) Q_2'(\lambda)   \in C^\infty((0,Z]_\zeta\times [0,\infty)_{h};\bbC),
	\end{equation}
	using \Cref{prop:quasimode_large_argument}, where $\lambda$ is still $\zeta/h^{2/(\kappa+2)}$.

	Then, we explicitly evaluate
	\begin{multline}
	W\{Q_1(\lambda), a_2 Q_2(\lambda)\}(h) = a_2' Q_1(\lambda) Q_2(\lambda)  + h^{-2/(\kappa+2)} a_2 Q_1(\lambda)  Q_2'(\lambda) \\ - h^{-2/(\kappa+2)} a_2 Q_1'(\lambda) Q_2(\lambda) \in h^{-2/(\kappa+2)} C^\infty,
	\end{multline} 
	\begin{multline}
	W\{a_1 Q_1(\lambda), Q_2(\lambda)\}(h) = -a_1' Q_2(\lambda) Q_1 (\lambda) - h^{-2/(\kappa+2)} a_1 Q_2(\lambda)  Q_1'(\lambda) \\ 
	+ h^{-2/(\kappa+2)} a_1 Q_2'(\lambda) Q_1(\lambda)  \in h^{-2/(\kappa+2)} C^\infty,
	\end{multline}
	and 
	\begin{multline}
	W\{a_1 Q_1, a_2 Q_2\} = a_1  a_2' Q_1 Q_2 + h^{-2/(\kappa+2)} a_1 a_2 Q_1 Q_2' - a_1' a_2 Q_1 Q_2 - h^{-2/(\kappa+2)} a_1 a_2 Q_1 Q_2' \\ 
	\in h^{-2/(\kappa+2)} C^\infty,
	\end{multline}
	where $C^\infty = C^\infty((0,Z]_\zeta\times [0,\infty)_{h};\bbC)$. Adding everything together, the claim follows.
\end{proof}

We return to the proof of \Cref{thm1}.  
By the Langer diffeomorphism, it suffices to consider the $W=1$ case.
Thus, if $h$ is sufficiently small, \Cref{lem:Wr} implies that  $u_1,u_2$ are linearly independent, and a solution $u$ to $Pu=0$ with prescribed initial data $u|_{\mathrm{ie}},u'|_{\mathrm{ie}}$ can be written as 
\begin{align}
\begin{split}
u(\zeta,h) &= 
\begin{bmatrix}
u_1(\zeta,h) \\ u_2(\zeta,h)
\end{bmatrix}^\intercal
\begin{bmatrix}
u_1|_{\mathrm{ie}}(h) & u_2|_{\mathrm{ie}}(h) \\ 
u_1'|_{\mathrm{ie}}(h) & u_2'|_{\mathrm{ie}}(h)
\end{bmatrix}^{-1} 
\begin{bmatrix}
u|_{\mathrm{ie}}(h) \\ u'|_{\mathrm{ie}}(h)
\end{bmatrix} \\ 
&= \frac{1}{W\{u_1,u_2\}(h) } 
\begin{bmatrix}
u_1(\zeta,h) \\ u_2(\zeta,h)
\end{bmatrix}^\intercal
\begin{bmatrix}
u_2'|_{\mathrm{ie}}(h) & -u_2|_{\mathrm{ie}}(h) \\ 
-u_1'|_{\mathrm{ie}}(h) & u_1|_{\mathrm{ie}}(h)
\end{bmatrix}
\begin{bmatrix}
u|_{\mathrm{ie}}(h) \\ u'|_{\mathrm{ie}}(h)
\end{bmatrix}.
\end{split}
\end{align}
From \Cref{lem:Wr} (and the fact that $W\{Q_1,Q_2\}\neq0$), it follows that the
reciprocal $1/W\{u_1,u_2\}(h)$ is polyhomogeneous.
Thus, $u$ is a linear combination of products of functions of exponential-polyhomogeneous type on $M$, so is of exponential-polyhomogeneous type itself. 

\section{Collapsing $\mathrm{fe}$}
\label{sec:collapsing}

We now consider some simplifications to the main result that apply when $\kappa \in \{-1,0,1\}$ and $P$ is of the form 
\begin{equation}
P = - h^2 \frac{\partial^2}{\partial z^2} + \varsigma z^\kappa W(z)  + \frac{h^2 c}{z^2} + h^2 E, 
\label{eq:misc_291}
\end{equation}
where 
\begin{itemize}
	\item if $\kappa \in \{0,1\}$, then $c=0$ and $E\in C^\infty( [0,Z]_z\times [0,\infty)_{h^2};\bbC)$,
	\item if $\kappa=-1$, then $c\in \bbC$ is arbitrary and $E\in z^{-1} C^\infty( [0,Z]_z\times [0,\infty)_{h^2};\bbC)$.
\end{itemize}
These are exactly the three main cases that Olver considers in \cite{OlverOriginal, OlverTransition, OlverBook}. Our goal is to explain why, in this setting, we can improve the description of the coefficients in \Cref{thm2} \emph{by collapsing} (i.e.\ blowing down) $\mathrm{fe}$, so as to get a result akin to the Langer--Olver/uniform JWKB expansion which Langer \cite{Langer31, Langer32, Langer49} and Olver base their treatment of transition point problems on. In other words, we want to explain why, in \Cref{thm2}, polyhomogeneity on $M$ can be replaced by ordinary smoothness on the rectangle.

It is convenient in applications to allow the analysis to extend beyond $z=0$. Let $Z_- \in (-\infty,0]$, $W \in C^\infty([Z_-,Z]_z;\bbR^+)$, and suppose that
\begin{equation}
E \in 
\begin{cases}
C^\infty([Z_-,Z]_z\times [0,\infty)_{h^2};\bbC ) & (\kappa=0,1), \\ 
z^{-1}C^\infty([Z_-,Z]_z\times [0,\infty)_{h^2};\bbC ) & (\kappa=-1).
\end{cases}
\end{equation}

With the stated assumptions, we have the following:
\begin{theorem}
	For any $Q\in \calQ$, there exist $\beta,\gamma \in C^\infty( [0,Z]_z\times [0,\infty)_{h^2};\bbC)$ such that the function $u$ defined by 
	\begin{equation}
	u = 
	\sqrt[4]{ \frac{\xi^\kappa}{W} }
	\bigg[ (1+h^2 \beta) Q \Big(  \frac{\zeta}{h^{2/(\kappa+2)}} \Big) + h^{(2\kappa+2)/(\kappa+2)} \gamma Q'\Big(  \frac{\zeta}{h^{2/(\kappa+2)}} \Big)   \bigg]
	\label{eq:misc_228}
	\end{equation}
	solves $Pu=0$, where 
	\begin{equation} 
	\zeta(z) = \operatorname{sign}(z) \Big(\frac{\kappa+2}{2} \int_0^{|z|} \omega^{\kappa/2} \sqrt{W( \operatorname{sign}(z) \omega)} \dd \omega \Big)^{2/(\kappa+2)} \in C^\infty(\bbR_z),
	\end{equation}
	$\xi(z)=z^{-1}\zeta(z)$. 
	In fact, if $\kappa \in \{-1,0\}$, we can arrange $\gamma \in zC^\infty( [0,Z]_z\times [0,\infty)_{h^2};\bbC)$.
	\label{thm:collapsed}
\end{theorem}

Under the stated conditions, one can simplify \Cref{thm1} in an analogous way. 

The $\kappa=0$ case is just the ordinary Liouville--Green approximation, and the $\kappa=1,-1$ cases are the expansions in \cite[Chp.\ 11, 12]{OlverBook}, respectively. The only difference is that differentiability of the expansion in the semiclassical parameter is included.

\begin{proof}
	By the Langer diffeomorphism, it suffices to consider the $W=1$ case. 
	The proof of the theorem in this case follows that of \Cref{thm2}, with a few simplifications.

	The first key simplification is that the coefficients functions $\beta_k,\gamma_k \in C^\infty(0,Z]_\zeta$ defined by \cref{eq:rec1}, \cref{eq:rec2}, with $C_k$ as in the body and an appropriate choice of $c_k$, are all extendable to elements of $C^\infty[Z_-,Z]_\zeta$. (And if $\kappa=0,-1$, then the same applies to $\zeta^{-1} \gamma_\bullet$.) 
	If, for some $k\in \bbN$, we know that $\beta_1(\zeta),\ldots,\beta_k(\zeta) \in  C^\infty [Z_-,Z]_\zeta$, and if, in the $\kappa=-1$ case, we also know that $\gamma_{k-1} \in \zeta C^\infty[Z_-,Z]_\zeta$ (taking this to be $0$ if $k=0$), then \cref{eq:rec2}, which now reads 
	\begin{equation}
	\gamma_k(\zeta) = -\frac{\varsigma}{2\zeta^{\kappa/2}} \int_0^\zeta \Big( \frac{\mathrm{d}^2 \beta_k(\omega)}{\dd \omega^2} - \sum_{j=0}^k E_j(\omega) \beta_{k-j}(\omega) + \frac{2c}{\omega} \frac{\mathrm{d}}{\mathrm{d} \omega} \Big( \frac{\gamma_{k-1}(\omega)}{\omega} \Big) \Big) \frac{\dd \omega}{\omega^{\kappa/2}},
	\label{eq:misc_279}
	\end{equation}
	tells us that $\gamma_k(\zeta) \in C^\infty[Z_-,Z]_\zeta$ (using the assumptions on $c,E$). In fact, if $\kappa=-1,0$, then $\gamma_{k}(\zeta) \in \zeta C^\infty[Z_-,Z]_\zeta$ (using that, if $\kappa=0$, we are assuming that $c=0$). 
	
	Suppose we also know that $\gamma_0,\dots,\gamma_{k-2}\in \zeta C^\infty[Z_-,Z]_\zeta$. \Cref{eq:rec1}, for some choice of $c_k$, reads 
	\begin{equation}
	\beta_{k+1} =  - \frac{1}{2}\int_0^\zeta \Big(\frac{\mathrm{d}^2\gamma_k(\omega)}{\dd \omega^2} - \sum_{j=0}^k E_j(\omega) \gamma_{k-j}(\omega) \Big) \dd \omega. 
	\end{equation} 
	So, $\beta_{k+1}$ is also smooth. Indeed, the only terms in the integrand which might not be smooth are the $E_\bullet$'s, but this is only allowed when $\kappa=-1$, and in this case $\gamma_{k-j}(\omega)\in \omega C^\infty[Z_-,Z]_\omega$, so the product $E_j \gamma_{k-j}$ is still smooth.
	
	Since $\beta_0 = 1$, one proceeds inductively to conclude that all $\beta_k,\gamma_k$ are smooth down to $\zeta=0$ (and through, if $Z_-<0$), and that, if $\kappa \in \{-1,0\}$, then $\zeta^{-1}\gamma_k(\zeta)$ is also smooth. 
	Then, the $\beta_k,\gamma_k$ can be asymptotically summed, not in polyhomogeneous spaces on $M$ as in the body of the paper, but in  $C^\infty=C^\infty([Z_-,Z]_\zeta\times [0,\infty)_{h^2};\bbC)$. If $\kappa \in \{-1,0\}$, then  we can instead asymptotically sum $\gamma$ in $z C^\infty$. 
	Then, one gets a $v$ with the form of the right-hand side of \cref{eq:misc_228} such that $Pv  \in h^\infty C^\infty$. 
	
	The analysis in \S\ref{sec:error} applies mutatis mutandis, and there is no longer any need to restrict attention to $\lambda\geq \lambda_0$ for some $\lambda_0>0$. Instead, the fixed point argument takes $h_0$ sufficiently small so as to guarantee the smallness of the operator norms of the maps $\Phi,\Xi$. So the argument in that section directly produces a $w$ of the form $w = \delta Q + \varepsilon Q'$ for $\delta,\varepsilon \in h^\infty C^\infty$ such that $Pw = - Pv$. 
	
	Then, $u = w+v$ has the desired form, after absorbing the $\delta,\varepsilon$ into a redefinition of $\beta,\gamma$, and satisfies $Pu=0$.  
\end{proof}

\begin{remark}
	In \cite[\S12.14.4]{OlverBook}\cite{OlverTransition}, Olver lists a few exceptional cases that, while not of the form above (either because $\kappa\geq 2$ or else because $\kappa\in \{0,1\}$ but $c\neq 0$), can be transformed into the form above with $\kappa=-1$ via a change of variables. 
	One of these exceptional cases appears often enough to be worth mentioning: if $W,E$ are both even in $z$, then we can allow $c\neq 0$ if $\kappa=0$. 
	We saw an example of this in \Cref{rem:SHO_Coulomb_relation}.
	\label{rem:exsimp}
\end{remark}
 
\section{Details of the hydrogen example}
\label{sec:hydrogen_details}
We now present the details omitted from \S\ref{subsec:Coulomb}, which studied the hydrogen operator
\begin{equation}
	P=- \frac{\partial^2}{\partial r^2} - E - \frac{\mathsf{Z}}{r} + \frac{\ell(\ell+1) }{r^2} ,
	\label{eq:misc_285}
\end{equation} 
for $E\in \bbR, \mathsf{Z}\in \bbR^\times$, and $\ell\in \bbN$. 
This appendix consists of two subsections:
\begin{itemize}
	\item The first, \S\ref{subsubsec:coulombic_high_energy}, discusses the simplest example we are aware of where, to the best of our knowledge, our expansion is novel, in this case at the first subleading order. 
	This is the problem of high-energy behavior near the Coulombic singularity, i.e.\ the regime where $|E|\to\infty$ and $r\to 0$ together, when the initial data is supplied at fixed $r>0$ rather than $r=0$. 
	In this application, $\kappa=0$ (this is the Liouville--Green case), but the Coulombic singularity of the $\psi$ term renders traditional methods and expansions inapplicable, as we will explain.
	\item  \S\ref{subsubsec:Rydberg} describes the opposite asymptotic regime, the \emph{Rydberg limit}, where $E\to 0$ and $r\to\infty$ together. The $\kappa=-1,1$ cases arise in this example. Here, the interesting behavior is not due to the Coulombic singularity, but rather the long-range nature of the Coulomb potential, which, despite decaying as $r\to\infty$, can compete with the energy term $-E$ in \cref{eq:misc_285} when $E$ is small.
	
	The results in this subsection are more classical than those in \S\ref{subsubsec:coulombic_high_energy}. One of our purposes in presenting the details is to illustrate how semiclassical methods can sometimes be applied in asymptotic regimes that are not obviously semiclassical. (It is well-known how semiclassical analysis arises in the high-energy, $E\to\infty$ limit. See e.g.\ \S\ref{subsubsec:coulombic_high_energy}. This is why Weyl's law is regarded as a semiclassical result. That semiclassical analysis is relevant to the \emph{low}-energy, $E\to 0$, limit of the hydrogen atom is less well-known.) Here, our expansions are not new, but their term-by-term differentiability in all directions seems to be (cf.\ \cite{Toprak}).
\end{itemize}

Standard references on the Coulomb wavefunctions include \cite{Curtis}\cite[\href{http://dlmf.nist.gov/33}{\S33}]{NIST}.

\subsection{High-energy asymptotics near the singularity}
\label{subsubsec:coulombic_high_energy}

The operator $P$ in \cref{eq:misc_285} is not of the form \cref{eq:1}, but $Pu=0$ is equivalent to $P_0u=0$ for $P_0 = |E|^{-1} P$, which has the form 
\begin{equation}
	P_0 =  - h^2 \frac{\partial^2}{\partial z^2} \mp 1 + h^2 \psi 
\end{equation}
for $h=|E|^{-1/2}$, where we are now writing `$z$' in place of `$r$,' the sign $\mp$ is positive if $E$ is negative and negative if $E$ is positive, and 
\begin{equation}
	\psi = \frac{\ell(\ell+1)}{z^2} - \frac{\mathsf{Z}}{z} .
\end{equation}
Thus, $P$ has the form \cref{eq:1} for $\kappa=0$, and $\psi$ is of precisely the form \cref{eq:psi_init}, with $\nu = \ell(\ell+1)$, $\varphi(h) =-\mathsf{Z}$, and $G=0$. 
Thus, we immediately get from \Cref{thm1} the existence of asymptotic expansions as $|E|\to\infty$ and/or $z\to 0^+$. Besides the ``pure'' large-energy regime ($|E|\to\infty$ for $z$ fixed) and the ``pure'' small-argument regime ($z\to 0^+$ for $|E|$ fixed), we also have a transitional asymptotic regime, fe, probed by taking $|E|\to \infty$ and $z\to 0^+$ together while fixing the ratio $\lambda = r |E|^{1/2}$. See \Cref{fig:X}.

It is the last regime that is of interest to us here.
Using \cref{eq:misc_ak3}, the asymptotics in question can be rephrased in terms of the asymptotic behavior of the Whittaker functions $\operatorname{WhittW}_{k,\mu},\operatorname{WhittM}_{k,\mu}$ in the limit where $k\to 0$. The Whittaker-M function is well-behaved in this limit, since it is defined via its Taylor series, and the coefficients in that Taylor series are smooth down to $k=0$ \cite[\href{http://dlmf.nist.gov/13.14.E6}{Eq.\ 13.14.6}]{NIST} (and derivatives in $k$ are under sufficient control, as can be shown using the ODE). On the other hand, it is difficult to see whether the Whittaker-W function should be well-behaved in this limit, since it is singled out among the solutions of Whittaker's ODE by its large-argument behavior, and large-argument asymptotics are strongly affected by the presence/absence of a Coulomb potential.

Since the Whittaker functions can be related to the confluent hypergeometric functions, we can ask the corresponding questions about the confluent hypergeometric functions.

Does a complete solution exist for any of these equivalent problems? If by ``complete solution'' we mean a full asymptotic expansion, the answer appears to be no. I am unable to locate in standard references such as \cite{Bateman}\cite{Curtis}\cite{OlverBook}\cite{NIST}, much about the relevant limits.
If $k$ is fixed to $0$, and then $z\to 0^+$, one has an explicit expansion for the Whittaker-W function \cite[\href{http://dlmf.nist.gov/13.14.E8}{Eq.\ 13.14.8}]{NIST}, but it is not clear what happens when one differentiates that expansion in $k$. Indeed, in \cite{Curtis}, Curtis states, citing \cite{OlverTransition}:
\begin{quote}
	\textit{No expansions can be found in terms of functions of a single variable, however, which are uniformly valid in an interval containing the origin.} \cite[\S7.2]{Curtis}
\end{quote}
It therefore appears that the asymptotic expansion that \Cref{thm1} states exists is novel in this context.

Why has this case not been understood previously? 
The answer has to do with the Coulombic $1/z$ singularity in $\psi$. This simple pole breaks standard arguments. Indeed, Olver delimits the cases the Langer--Olver/uniform WKB method can handle in \cite[\S12.14]{OlverBook}, and the case we are considering here is not among them: it is not one of Olver's three main cases (which only allow $\psi$ to be singular if $\kappa=-1$, whereas $\kappa=0$ here), and it is not one of the exceptional forms enumerated in \cite[\S12.14.4]{OlverBook}, because \cite[Eq.\ 12.14.17]{OlverBook} allows only a pure double pole -- no single pole. (This is why the Coulomb term $1/z$ is the problem, not $\ell(\ell+1)/z^2$,  even though the latter is more singular at $z=0$.)

Rather than being a technical glitch, this is a fundamental obstacle, one that exists for conceptual reasons (that Olver understood): no expansion of the sort Olver wanted holds. As discussed in \S\ref{subsec:JWKB} (in the $\kappa=1$ case), one sees in such an asymptotic expansion a special function and its derivative, but the \emph{coefficients} of those special functions are
formal series in $h$ whose coefficients are smooth functions of $z$, all the way down to $z=0$. That is, the coefficients (what we call $\beta$ and $\gamma$ in \Cref{thm2}, \Cref{thm:collapsed}) are smooth on $[0,Z]_z\times [0,\infty)_{h^2}$. As Olver explains, this is not possible in the case we are currently considering. He proves a general no-go theorem to this effect in \cite{OlverTransition}, which is the source of Curtis's claim above; see \cite[\S12.14]{OlverBook} for a summary.
This is where Olver's discussion ends and where ours begins. 
The coefficients $\beta,\gamma$ are (as \Cref{thm2} guarantees) instead polyhomogeneous on a compactification with \emph{two} $h\to 0^+$ edges, and each has its own asymptotic expansion associated with it. More specifically, $\beta,\gamma$ are polyhomogeneous on $M$. This is how we evade Olver's no-go theorem.

To see why $\beta,\gamma$ cannot both be smooth on $[0,Z]_z\times [0,\infty)_{h^2}$, let us recall Olver's argument in \cite[\S12.14]{OlverBook}. This boils down to attempting the proof of \Cref{thm:collapsed} and seeing what fails.  
In the present setting, what \Cref{thm:collapsed} would state is the following Langer--Olver-type result: for any $Q\in \calQ$ (where $\calQ$ is the set of products $t^{1/2} I(t)$ for $I(t)$ a Bessel function -- unmodified if $E>0$, modified otherwise -- of order $\nu=2^{-1}+\ell$), there exist $\beta,\gamma \in C^\infty([0,Z]_z\times [0,\infty)_{h^2};\bbC)$ such that 
\begin{equation}
	u = (1+h^2 \beta) Q\Big( \frac{z}{h} \Big) + h \gamma Q'\Big( \frac{z}{h} \Big) 
	\label{eq:misc_03x}
\end{equation}
solves $Pu=0$. (Cf.\ \cref{eq:JWKB_uniform}.) But, the proof of \Cref{thm:collapsed} tells us what $\gamma_0=\gamma|_{h=0}$ would have to be: up to a constant, and taking absolute values for simplicity, 
\begin{equation}
	|\gamma_0(z)| = \frac{1}{2} \Big|\int^1_z    \frac{\mathsf{Z}}{\omega}   \dd \omega\Big| 
	\label{eq:misc_0y9}
\end{equation} 
(since the $\beta_0$ in \cref{eq:misc_279} is identically $1$). But unless $\mathsf{Z}=0$, this integral diverges logarithmically as $z\to 0^+$; $\gamma_0$ cannot be smooth at $z=0$.

We can see this in some basic numerics by examining the growth rate of $\gamma$ as $z,h\to 0^+$ together. 
To simplify matters as much as possible, consider only the case $E>0$ and $\ell=0$. Then, $\calQ = \operatorname{span}_\bbC\{e^{it},e^{-it}\}$, as discussed in \S\ref{subsec:model}. Then, 
the two terms in \cref{eq:misc_03x} can be combined into one, so what we want to disprove is the existence of solutions $u_\pm \in \ker P$ of the form 
\begin{equation}
	u_\pm = (1+h \Gamma_\pm) e^{\pm iz/h} 
\end{equation}
for $\Gamma_\pm \in C^\infty([0,Z]_z\times [0,\infty)_{h};\bbC)$. Again, we know that there exist polyhomogeneous $\Gamma_\pm$ (with a specified index set, so that $\Gamma_\pm$ is $O(1)$ at $\mathrm{ze}$ away from $\mathrm{fe}$) such that $u_\pm \in \ker P$. The question is whether the $\Gamma_\pm$ can be smooth before blowing up the corner of $[0,Z]_z\times [0,\infty)_{h}$. What we will check numerically is that we cannot have
\begin{equation}
	\Gamma_\pm = O(1).
	\label{eq:misc_042}
\end{equation}
For each individual $h>0$, the ODE $P(h)u=0$ is regular singular at $z=0$, with indicial roots $0,1$, so the elements of $\{\Gamma_\pm(h)\}_{h>0}$ must extend continuously to $z=0$. (Alternatively, this can be read off \cite[\href{http://dlmf.nist.gov/13.14}{\S13.14}]{NIST}.) Consequently, if \cref{eq:misc_042} fails, it must be because $\Gamma_\pm$ blows up at the corner $\{z=0,h=0\}$, in accordance with the discussion above. For the sake of contradiction, let us assume \cref{eq:misc_042} and see what follows.

Let $u_{1,0}$ denote the solution to $Pu=0$ such that $u(1)=0$ and $u'(0)=0$. Then, it follows from \cref{eq:misc_042} that 
\begin{equation}
	u_{1,0} = \cos(h^{-1}(z-1)) + O(h).
	\label{eq:misc_043}
\end{equation}
Let $\calE = h^{-1} (u_{1,0} - \cos(h^{-1}(z-1))) = |E|^{1/2}(u_{1,0} - \cos( |E|^{1/2}(r-1)))$. If \cref{eq:misc_043} were true, then $\calE$ would be $O(1)$.
What we will do is follow  $\calE$ along the three curves depicted in \Cref{fig:cone_test}. What we will see is that, when following $\calE$ along the curve $\{r=1/|E|^{1/2}\}$ (which hits fe), $\calE$ diverges logarithmically. So, $\calE \neq O(1)$ and \cref{eq:misc_042} is numerically falsified.

\Cref{fig:cone_test} shows plots of $\calE$ followed along these three curves.
Different behaviors are seen along each: 
\begin{enumerate}
	\item In the first, where $E$ is fixed and $z\to 0^+$, the function $\calE$ converges, just because in the $\ell=0$ case the elements of \cref{eq:misc_ak3} are continuous functions of $z$, as discussed above.
	\item In the second, where $z$ is fixed and $E\to\infty$, $\calE$ oscillates boundedly, in accordance with the Liouville--Green theory.
	\item In the third, $\calE$ oscillates with amplitude $\propto \log (E)$. This is in accordance with the logarithmic divergence of \cref{eq:misc_0y9} as $z\to 0^+$. 
\end{enumerate}

\begin{figure}[h]
	\begin{subfigure}[b]{.32\textwidth}
		\includegraphics[scale=.17]{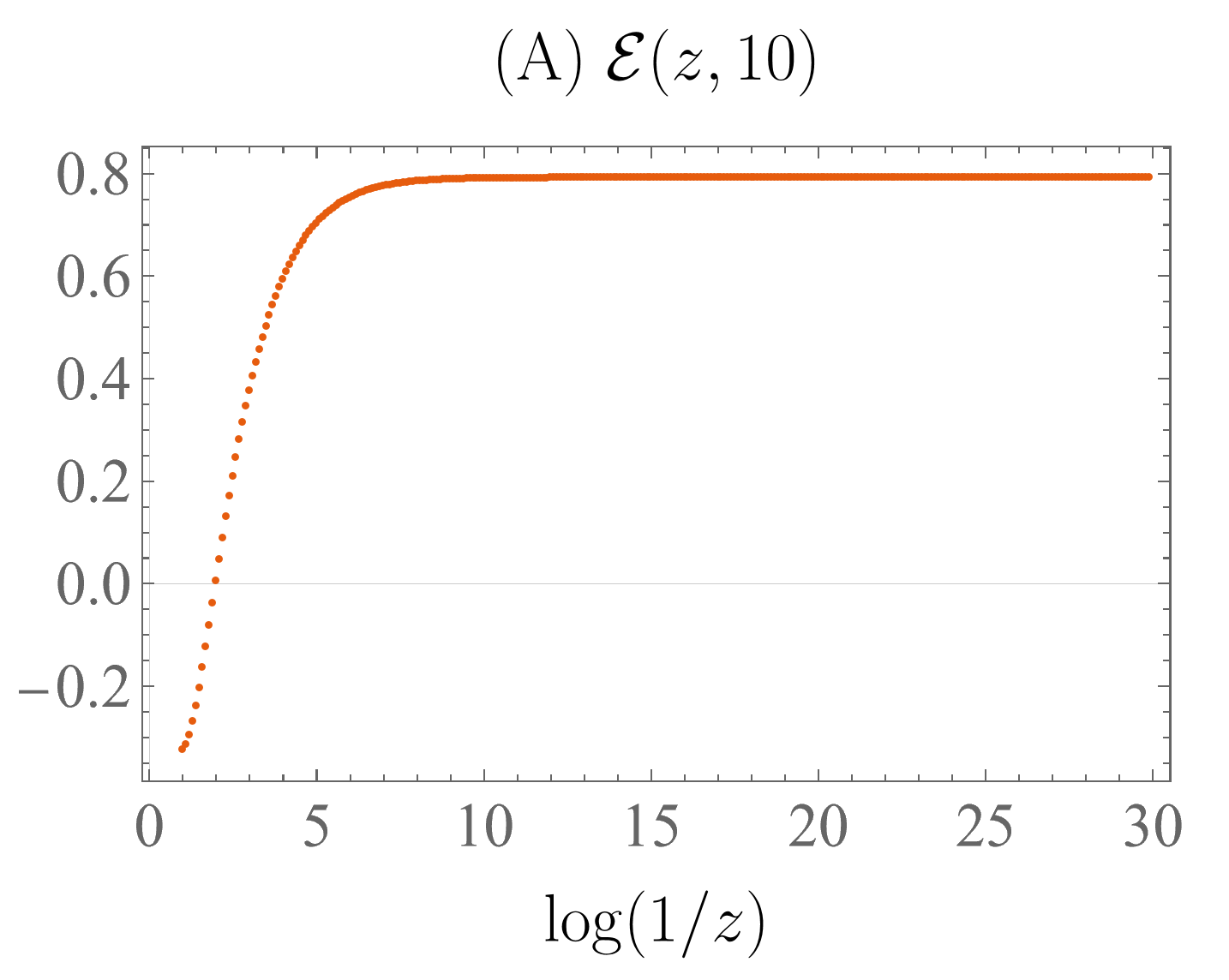}
	\end{subfigure}
	\begin{subfigure}[b]{.32\textwidth}
		\includegraphics[scale=.17]{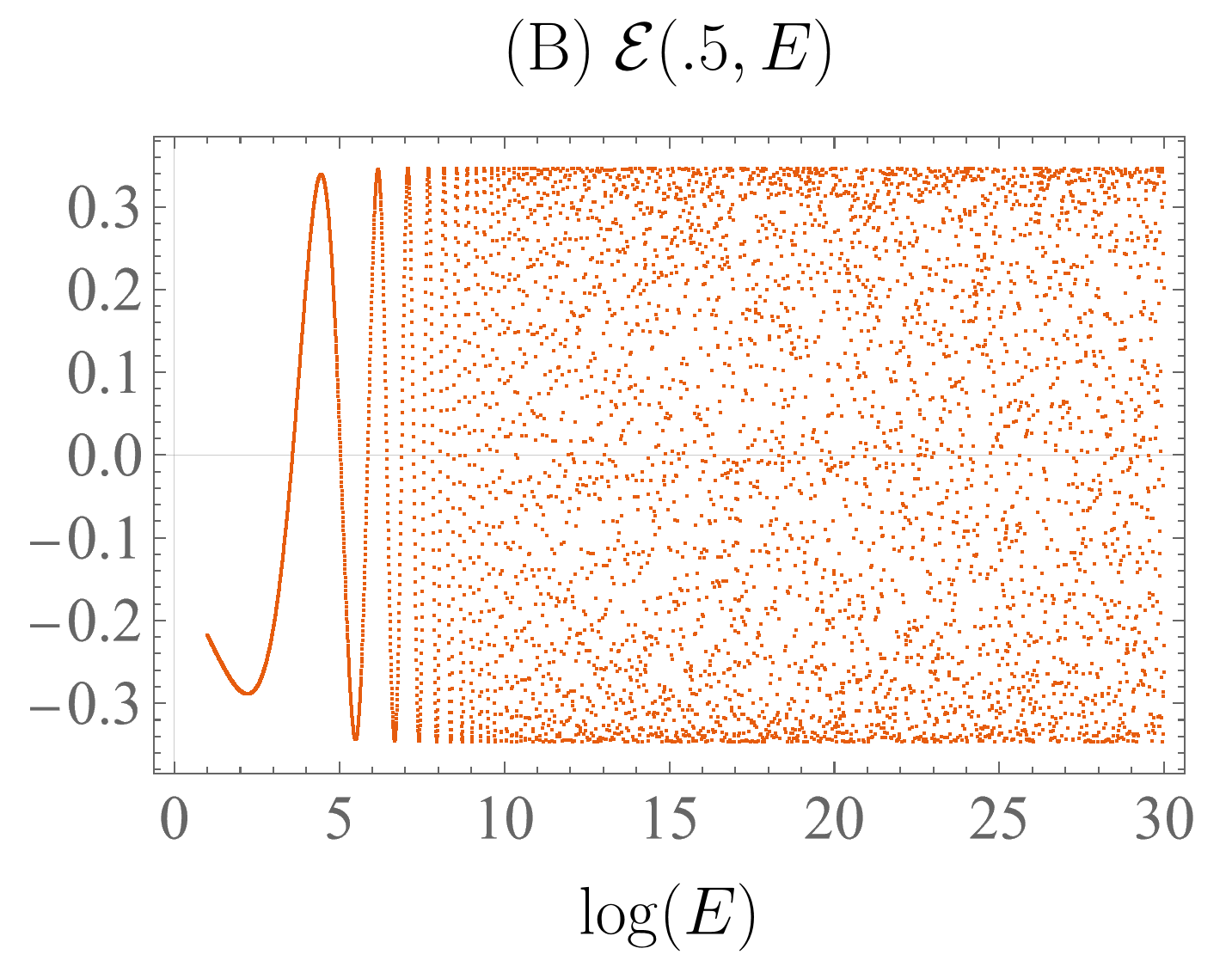}
	\end{subfigure}\;
	\begin{subfigure}[b]{.32\textwidth}
		\includegraphics[scale=.17]{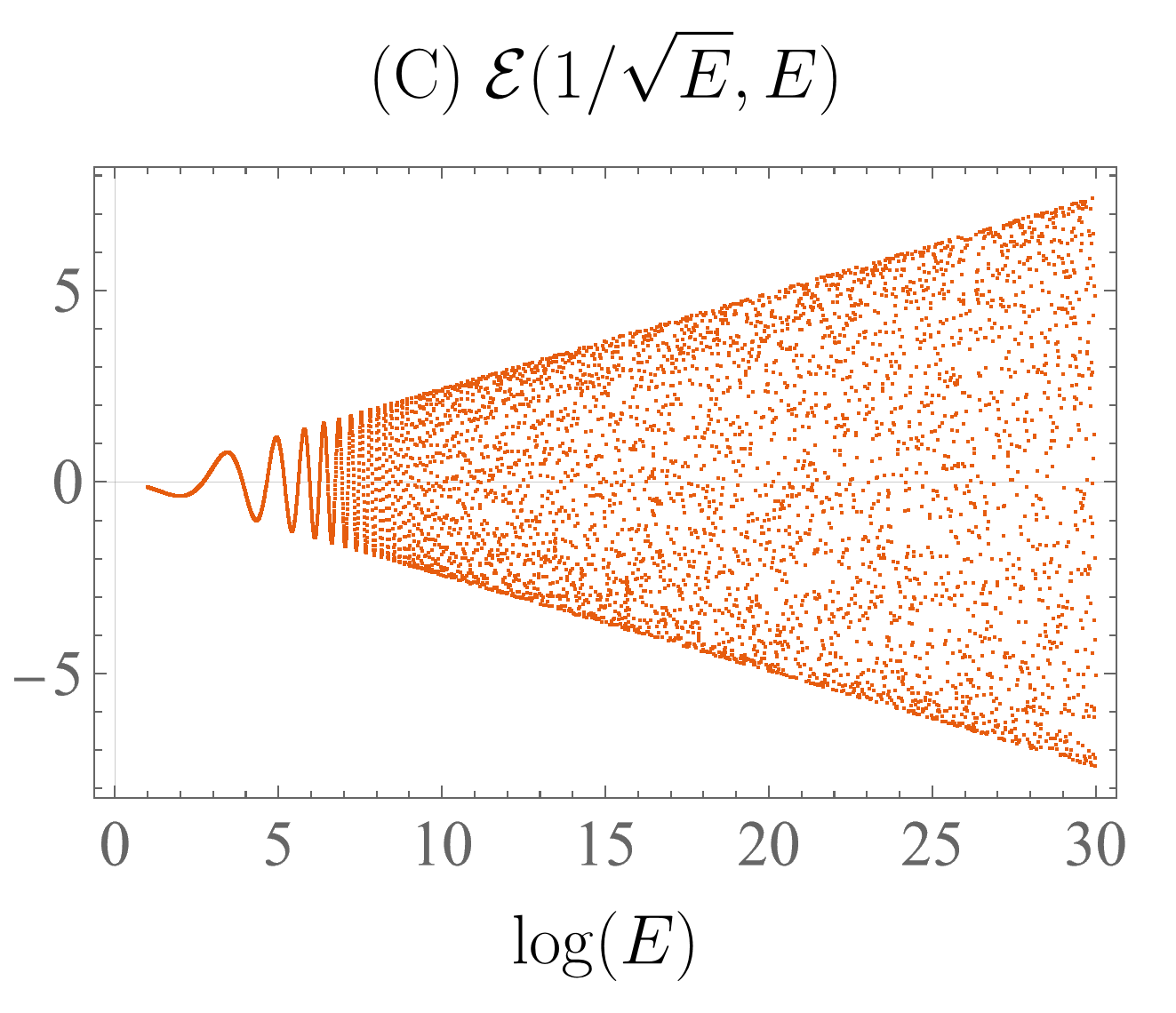}
	\end{subfigure}
	\begin{tikzpicture}[scale=.9, decoration={
		markings,
		mark=at position .5 with {\arrow[scale=1.5,>=latex]{>}}}]
	\draw[white] (-1.5,0) -- (0,0);
	\fill[fill=lightgray!20] (0,-2) -- (0,1.5) -- (5,1.5) -- (5,-.5) to[out=180,in=90] (3.5,-2) -- cycle;
	\begin{scope}
	\clip (0,-2) -- (0,1.5) -- (5,1.5) -- (5,-.5) to[out=180,in=90] (3.5,-2) -- cycle;
	\draw[darkred, postaction={decorate}] (1.5,1.5)--(5,-2);
	\node[below left, darkred] () at (3.2,-.1) {$\Gamma_1=$(C)};
	\end{scope}
	\draw[darkred, postaction={decorate}] (0,1)  to[out=0,in=180] (5,.7) node[above left] {(B)};
	\draw[darkred, postaction={decorate}] (.7,1.5) node[above] {(A)} to[out=-90,in=90] (.5,-2);
	\draw  (5,1.5) -- (5,-.5) to[out=180,in=90] (3.5,-2) -- (0,-2);
	\draw (0,-2) -- (0,1.5) -- (5,1.5);
	\node () at (2,-2.25) {be};
	\node[white] () at (2,-2.7) {be};
	\node () at (5.35,.5) {ze};
	\node () at (4.25,-1.25) {fe};
	\node () at (2.5,1.8) {$M^{(\mathrm{I})}$ };
	\draw[darkgray, ->] (3.4,-1.9) to[out=90, in=231] (3.7,-1.05) node[left] {$rE^{1/2}$};
	\draw[darkgray, ->] (4.9,-.4) to[out=180,in=25] (4.2,-.6) node[above] {$\frac{1}{r\sqrt{E}}$};
	\draw[darkgray, ->] (4.9,-.4) --(4.9,.3) node[left] {$r$};
	\draw[darkgray, ->] (3.4,-1.9) --(2.7,-1.9) node[above] {$E^{-1}$};
	\end{tikzpicture}
	\includegraphics[scale=.49]{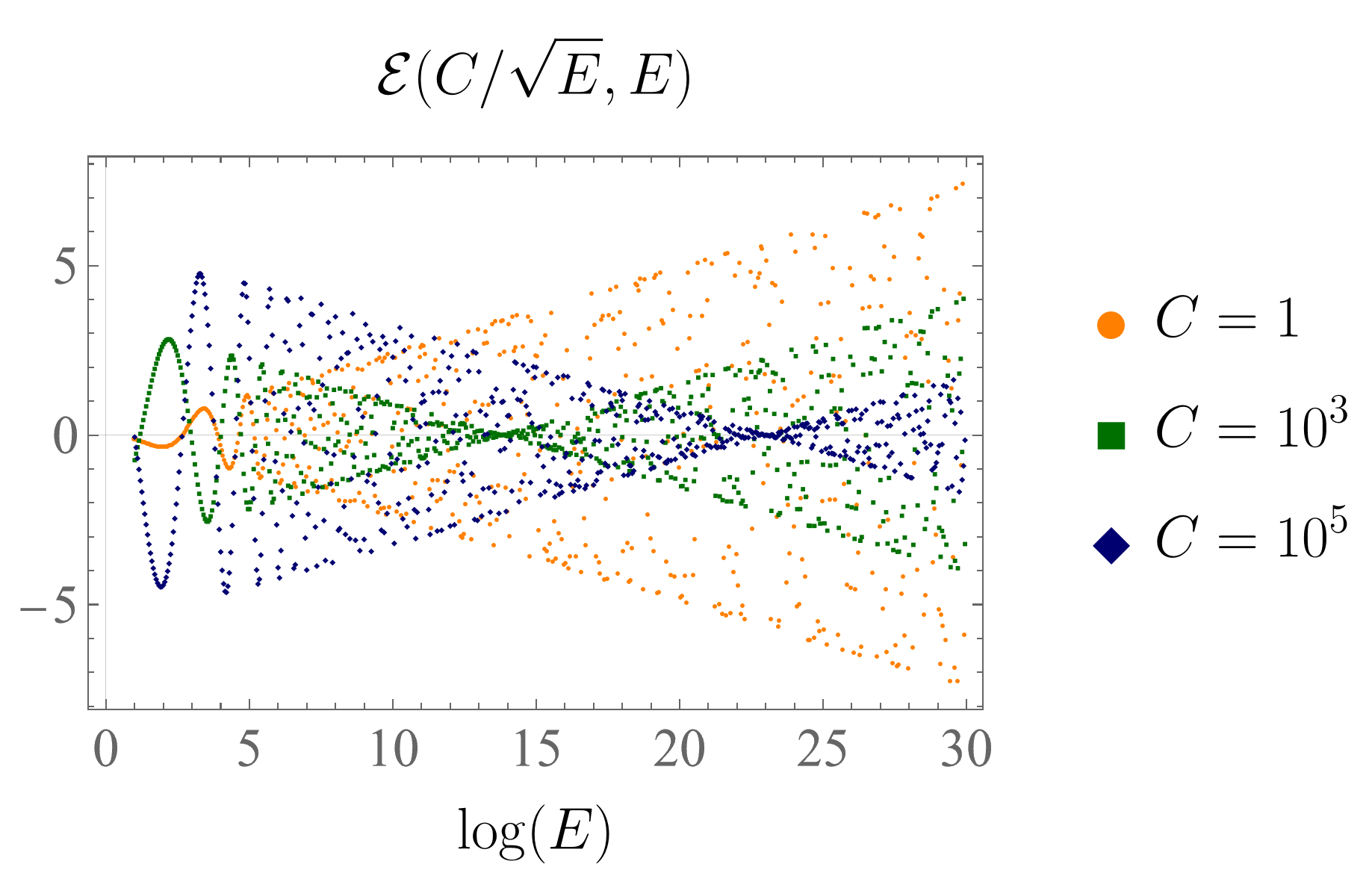}
	\caption{The scaled error $\calE = E^{1/2} (u_{1,0} - \cos(E^{1/2}(z-1)))$ between $u_{1,0}$ and its leading ansatz $\cos(h^{-1}(z-1))$, sampled periodically along the three depicted curves in $(0,1)_r\times \bbR^+_E$, (\textit{top left}) a level set of $E=h^{-2}$ (\textit{top middle}) a level set of $z=r$, and (\textit{top right}) along $\{z=1/E^{1/2}\}$. \textit{Bottom left:} The three curves hit the interiors of $\mathrm{be},\mathrm{ze},\mathrm{fe}\subset M$ respectively. The fact that we see $\calE$ diverging like $\log E$ along $\{z=1/E^{1/2}\}$ as $E\to\infty$ shows why we need the regime fe. \textit{Bottom right:} the same plot in (C), except showing $\calE$ sampled on $\Gamma_C$ for three $C$.}
	\label{fig:cone_test}
\end{figure}

This demonstrates why polyhomogeneity on $M$ is the notion required to understand the behavior of $\calE$: the two boundary edges $\{z=0\},\{h=0\}\subset [0,Z]_z\times [0,\infty)_{h^2}$ are not enough to capture the diversity of asymptotic behavior seen. Fundamentally different behavior, in this case a logarithmic divergence which is not seen at either adjacent edge, is seen at the corner $\{z=0,h=0\}$.

\begin{remark}[The difficult point is the high corner $\mathrm{ze}\cap\mathrm{fe}$]
	So far, we have been supplying initial data away from the transition point, but an alternative is supplying a boundary condition at the transition point. For example, this is how the Whittaker-M function is defined (whereas the Whittaker-W function is best characterized by its large-argument behavior). 
	
	It turns out (see the argument in \S\ref{sec:low}) that it is straightforward to study the asymptotics of solutions of the ODE in $\{r|E|^{1/2}<C\}$, for any $C>0$, if we supply the boundary condition at $r=0$. In other words, if we supply the boundary condition at the transition point, then it becomes straightforward to study asymptotics on $M\backslash \mathrm{ze}$. 
	But then it becomes difficult to understand asymptotics at $\mathrm{ze}$, that is as $E\to\infty$ for $r>0$ fixed. In \cite{Toprak}, Black--Toprak--Vergara--Zou required uniform estimates for the Whittaker-M function, so they needed to understand behavior at $\mathrm{ze}$, as well as at $\mathrm{fe}$. That they were supplying the boundary condition at the transition point did not trivialize the problem.

	What we have been doing here is supplying the initial data at $r=1$, and then the difficulty is understanding asymptotics at $\mathrm{fe}$. 
	What this makes clear is that, regardless of where the initial data is supplied, the crux of the problem is getting past the ``high corner'' $\mathrm{ze}\cap \mathrm{fe}$'' of $M$. In other words, the problem is one of \emph{connecting} the asymptotics in $\{r|E|^{1/2}<C\}$ with the asymptotics in $\{r>\varepsilon\}$. 
	The real novelty in our analysis here is the joint expansion near $\mathrm{ze}\cap\mathrm{fe}$, which provides a full asymptotic expansion for the connection formula.
\end{remark}

\begin{figure}[t!]
	\begin{tikzpicture}[scale=.9, decoration={
		markings,
		mark=at position .5 with {\arrow[scale=1.5,>=latex]{>}}}]
	\draw[white] (-1.5,0) -- (0,0);
	\fill[fill=lightgray!20] (0,-2) -- (0,1.5) -- (5,1.5) -- (5,-.5) to[out=180,in=90] (3.5,-2) -- cycle;
	\begin{scope}
	\clip (0,-2) -- (0,1.5) -- (5,1.5) -- (5,-.5) to[out=180,in=90] (3.5,-2) -- cycle;
	\fill[darkgray!20] (.2,-2) -- (5,-2) -- (.2,0) -- cycle;
	\fill[darkgray!20] (.2,1) -- (5,.1) -- (5,1.5) -- (.2,1.5) -- cycle;
	\end{scope}
	\node () at (1.6,-1.6) {\cite[\href{http://dlmf.nist.gov/33.5}{\S33.5}]{NIST} };
	\node () at (3,1.1) {\cite[\S6.1]{Curtis} };
	\draw  (5,1.5) -- (5,-.5) to[out=180,in=90] (3.5,-2) -- (0,-2);
	\draw (0,-2) -- (0,1.5) -- (5,1.5);
	\node () at (2,-2.25) {be};
	\node[white] () at (2,-2.7) {be};
	\node () at (5.35,.5) {ze};
	\node () at (4.25,-1.25) {fe};
	\node () at (2.5,1.8) {$M^{(\mathrm{I})}$ };
	\fill[black] (5,-.5) circle (1.5pt) node[right] {$\mathrm{ze}\cap\mathrm{fe}$};
	\draw[darkred, postaction={decorate}] (4.85,1.5) to[out=-90, in=30] (4.2,-.3) to[out=210, in=90] (3.2,-2); 
	\end{tikzpicture}
	\caption{The crux of \S\ref{subsubsec:coulombic_high_energy} is connecting the asymptotics at ze away from fe with the asymptotics at fe away from ze. Each region has been labeled with a reference where the relevant asymptotics can be found. Our challenge comes from navigating around the corner $\mathrm{ze}\cap\mathrm{fe}$, the ``high corner'' of $M$.}
\end{figure}

\subsection{The Rydberg regime}
\label{subsubsec:Rydberg}

Next, we discuss the low-energy regime. The material in this subsection is more classical than the material in the last, and the approximations are not new. One can find them, for example, in \cite[\S11.4.3, \S12]{OlverBook}\cite[\href{http://dlmf.nist.gov/33.20}{\S33.20}]{NIST}. That the expansions can be differentiated term-by-term in each direction seems to be, but only in the case where $\operatorname{sign}(E)\neq \operatorname{sign}(\mathsf{Z})$.\footnote{The $E,\mathsf{Z}>0$ case, i.e.\ hydrogen ``ionization,'' is contained in \cite[Appendix A]{SussmanACL}. The $E,\mathsf{Z}<0$ case is strictly easier, as it is entirely classically forbidden --- a charged particle in a purely repulsive potential cannot have negative energy, and the corresponding Schr\"odinger operator admits no bound states. This is why we said above that the differentiability-in-all-directions of the Rydberg expansions is novel only when $\operatorname{sign}(E)\neq \mathsf{Z}$. However, $\operatorname{sign}(E)\neq \mathsf{Z}$ is the most interesting case, because of the turning point at $\hat{r}=\mathsf{Z}$.} Given the amount of work required in \cite[\S2]{Toprak} to prove the first few symbolic estimates, this technical improvement is worth proving in general, once and for all, which is what we do in this paper. However, the purpose of this subsection is entirely expository; our goal is to explain how the $\kappa=1,-1$ cases of the semiclassical operator \cref{eq:1} appear in the present, non-semiclassical problem, and to illustrate the utility of geometric singular analysis as an organizational tool in multi-scale analysis. Consequently, we will only discuss the basic asymptotics, which can be considered standard material.

In order to rewrite \cref{eq:Whittaker} in semiclassical form, let $\hat{r} = r h^2$ for $h= |E|^{1/2}$. Then, \cref{eq:Whittaker} becomes 
\begin{equation}
	- h^2 \frac{\mathrm{d}^2 u}{\mathrm{d} \hat{r}^2} - \Big( \pm 1 + \frac{\mathsf{Z}}{\hat{r}} - \frac{h^2\ell(\ell+1)}{\hat{r}^2}\Big) u = 0.
\end{equation}
This is now manifestly semiclassical. This idea of rescaling the the independent variable to transform a non-semiclassical problem into a semiclassical one is common in applications. We will see it again below. Let us now study this ODE on the rectangle  
\begin{equation} 
	[0,\infty)_{\hat{r}}\times [0,\infty)_{h^2} = [0,\infty)_{\hat{r}}\times [0,\infty)_E
\end{equation} 
and identify the transition points.

If $\operatorname{sign}(E)=\operatorname{sign}(\mathsf{Z})$, there is only a single transition point, namely at $\hat{r}=0$. The potential $\mp 1 - \mathsf{Z}/\hat{r}$ has a simple pole there, so this is the $\kappa=-1$ case, as depicted in \Cref{fig:covering}. 

This transition point is also present if $\operatorname{sign}(E)\neq \operatorname{sign}(\mathsf{Z})$, but in this case we have another transition point at $\hat{r}= \mathsf{Z}$, where the potential 
\begin{equation} 
	V(\hat{r})=\mp 1 - \mathsf{Z}/\hat{r}
\end{equation} 
vanishes simply. Thus, this falls into the $\kappa=1$ case; we have a turning point.
The physics contained in this statement is the following:
\begin{itemize}
	\item $\mathsf{Z}>0$ case: if the potential is attractive and $E<0$, then the corresponding classical dynamics is the Kepler problem for a particle bound in a Newtonian gravitational well. A ball thrown up comes down. That is the turning point.
	\item $\mathsf{Z}<0$ case: if the potential is repulsive and $E>0$, then the corresponding classical dynamics is Rutherford scattering. A proton with fixed energy can only come so close to another proton before it is repelled. 
\end{itemize}
Either way, it is an almost immediate corollary of \Cref{thm1} that solutions of the ODE with exponential-polyhomogeneous data at, say, $\hat{r}=1000$ are exponential-polyhomogeneous on the mwc $X$ in \Cref{fig:X}, within the regions that are called 
\begin{equation} 
	M^{(\mathrm{II})}, M_0^{(\mathrm{III})}, M_0^{(\mathrm{IV})} \subset X_1
\end{equation} 
in \Cref{fig:covering}.

The only reason for the word ``almost'' in the preceding paragraph is that it may not be immediately obvious that performing the $\kappa=-1$ quasihomogeneous blowup of $\{\hat{r}=0,E=0\}$ in $[0,\mathsf{Z}/2)_{\hat{r}}\times [0,\infty)_{|E|}$ gives what is depicted as $M^{(\mathrm{II})}$ in \Cref{fig:covering}, the edge 
\begin{equation} 
	\mathrm{cl}_X\{E=0,r<\infty\} \subset M^{(\mathrm{II})}
\end{equation} 
of which is labeled as being parameterized by $r$, whereas the quasihomogeneous blowup resolves the ratio $\lambda=\hat{r}/h^{2/(\kappa+2)}$. But indeed, because $\kappa=-1$, 
\begin{equation} 
	\lambda = \hat{r} / h^{2/(\kappa+2)} = \hat{r}/|E|=r.
\end{equation}
So, the front face of the blowup is in fact parameterized by $r$. (Moreover, the edge $\{\hat{r}=0,|E|>0\}$ of $[0,\infty)_{\hat{r}}\times [0,\infty)_{|E|}$ is naturally identified with the edge $\{r=0,|E|>0\}$ of $X$.)

The reader may note that this is the same computation at the end of \S\ref{subsec:model} in order to show that $M$ can be identified with a subset $M'\subseteq M_2$ of the mwc that we called `$M_2$' there; see \Cref{fig_alt}. What we called $M'$ is what we are calling $M^{(\mathrm{II})}$ now.

\begin{figure}[h!]
	\begin{center}
		\hspace{3em}
		\begin{tikzpicture}[scale=.9]
			\fill[fill=lightgray!20] (0,-1.5) -- (0,1.9) -- (4,1.9) -- (4,-1.5) -- cycle;
			\draw (4,-1.5) -- (0,-1.5) -- (0,1.9);
			\draw[dashed] (0,1.9) -- (4,1.9) -- (4,-1.5);
			\node () at (2,-1.8) {$\{\hat{r}=0\}$};
			\node () at (2,1.3) {$\{\hat{r}=\mathsf{Z}\}$};
			\node[white] () at (2,2.2) {$\{\hat{r}=0\}$};
			\draw[darkred, ->] (.1,-1.4) -- (.1,-.5) node[right] {$\hat{r}$};
			\draw[darkred, ->] (.1,-1.4) -- (1,-1.4) node[above] {$|E|$};
			\fill[black] (0,-1.5) circle (2pt);
			\fill[black] (0,1) circle (2pt);
			\draw[dashed] (0,1) -- (4,1);
		\end{tikzpicture}
		\qquad 
		\begin{tikzpicture}[scale=.9]
			\fill[fill=lightgray!20] (1.5,-1.5) to[out=90, in=0] (0,0) -- (0,.7) arc(-90:90:.3) -- (0,1.9) -- (4,1.9) -- (4,-1.5) -- cycle;
			\begin{scope}
				\clip (1.5,-1.5) to[out=90, in=0] (0,0) -- (0,.7) arc(-90:90:.3) -- (0,1.9) -- (4,1.9) -- (4,-1.5) -- cycle;
				\fill[fill=darkgray!20] (0,-1.5) rectangle (4,.3);
				\fill[fill=darkgray!20] (4,.4) rectangle (0,1);
				\fill[fill=darkgray!20] (0,1) rectangle (4,2);
			\end{scope}
			\draw (4,-1.5) -- (1.5,-1.5) to[out=90, in=0] (0,0) -- (0,.7) arc(-90:90:.3) -- (0,1.9);
			\draw[dashed] (0,1.9) -- (4,1.9) -- (4,-1.5);
			\node[white] () at (2,-1.8) {$\{\hat{r}=0\}$};
			\node () at (3.25,1.3) {$\{\hat{r}=\mathsf{Z}\}$};
			\node[white] () at (2,2.2) {$\{\hat{r}=0\}$};
			\draw[dashed] (.3,1) -- (4,1);
			\draw[darkred, ->] (1.6,-1.4) -- (2.4,-1.4) node[above] {$-E$};
			\draw[darkred, ->] (1.6,-1.4) to[out=90,in=-65] (1.42,-.7) node[right] {$-\hat{r}/E$};
			\draw[darkred, ->] (.1,.1) -- (.1,.5) node[right] {$\hat{r}$};
			\draw[darkred, ->] (.1,.1) to[out=0,in=170] (.65,0) node[right] {$-E/\hat{r}$};
			\node () at (3.4,-.5) {$M^{(\mathrm{II})}$};
			\node () at (3.4,.7) {$M^{(\mathrm{III})}$};
			\node () at (1,1.5) {$M^{(\mathrm{IV})}$};
		\end{tikzpicture}
		\quad 
		\begin{tikzpicture}[xscale=-.9,yscale=.9] 
			\fill[fill=lightgray!20] (-.2,-2) -- (-.2,0) to[out=180,in=270] (-1.7,1.5) -- (-5,1.5) -- (-5,-.5) -- (-5,-2) -- cycle;
			\draw[dashed] (-1.7,1.5) -- (-5,1.5) -- (-5,-2);
			\begin{scope}
				\clip (-.2,-2) -- (-.2,0) to[out=180,in=270] (-1.7,1.5) -- (-5,1.5) -- (-5,-2) -- cycle;
				\fill[fill=darkgray!20] (0,-2) -- (0,1) to[out=225, in=30] (-3.3,-2) -- cycle;
				\fill[fill=darkgray!20] (-.1,1) to[out=225, in=30] (-3.5,-2) -- (-5,-2) -- (-5,-1.2) to[out=20, in=225] (-.8,1.5);
				\fill[fill=darkgray!20] (0,2) to[out=228, in=25] (-5,-1.375) -- (-5,2) -- cycle;
				\draw[dashed] (0,1.5) to[out=214, in=20] (-5,-1.5);
				\filldraw[fill=white] (-1.3,.6) circle (.4);
			\end{scope}
			\draw (-5,-2) -- (-.2,-2) -- (-.2,0) to[out=180,in=270] (-1.7,1.5);
			\node () at (-5.8,-1.4) {$\{\hat{r}=\mathsf{Z}\}$};
			\fill[white] (-1.3,.6) circle (.39);
			\draw[darkred, ->] (-.3,-1.9) --(-1,-1.9) node[above] {$-E$};
			\draw[darkred, ->] (-.3,-1.9) --(-.3,-1.2) node[right] {$r$};
			\node[black] () at (-2.5,-2.3) {$\mathrm{cl}\{\hat{r}=0,E<0\} = \{r=0\}$};
			\draw[darkred, ->] (-.3, -.1) -- (-.3,-.75) node[right] {$r^{-1}$};
			\draw[darkred, ->] (-.3, -.1) to[out=180, in=-30] (-.8,.01) node[below] {$\hat{r}$};
			\node () at (-1.9,-1.7) {$M^{(\mathrm{II})}$};
			\node () at (-4.1,-1.7) {$M^{(\mathrm{III})}$};
			\node () at (-3.5,.5) {$M^{(\mathrm{IV})}$};
			\draw[dashed] (-1.7,1.5) -- (-5,1.5);
			\draw[dashed] (-1.7,1.5) -- (-5,1.5) -- (-5,-2);
		\end{tikzpicture}
	\end{center}
	\caption{\textit{Left:} the rectangle $[0,\infty)_{\hat{r}}\times [0,\infty)_E$ considered in \S\ref{subsubsec:Rydberg}. Here, $\hat{r} = r|E|$. The transition points which need to be blown up are indicated. (The $\hat{r}=\mathsf{Z}$ transition point is absent if $\operatorname{sign}(E)=\operatorname{sign}(\mathsf{Z})$.) \textit{Middle:} the result of performing the two blowups. This corresponds to a neighborhood of $\{E=0\}$ in $X$, in \Cref{fig:X}. \textit{Right:} a diffeomorphic change of perspective, based on $\hat{r}/|E|=r$, showing the identification of the middle figure with the low-energy region (excluding $\{\hat{r}=\infty\}$) in \Cref{fig:X}.}
	\label{fig:cone_test_paths}
\end{figure}

Both transition points here are best dealt with using \Cref{thm:collapsed}, which applies in this setting. 

It is worth taking a moment to understand what the space $\calQ$ of quasimodes appearing in this theorem  (and \Cref{thm2}) are: they are exactly the solutions of the ODE \cref{eq:Whittaker} frozen at zero energy:
\begin{equation}
	\calQ = \Big\{Q(r) : \frac{\mathrm{d}^2 Q}{\mathrm{d} r^2} = \Big(- \frac{\mathsf{Z}}{r} +\frac{\ell(\ell+1)}{r^2} \Big) Q \Big\}.
\end{equation}
This is a form of Bessel's ODE \emph{in terms of $r^{1/2}$}: 
\begin{equation}
	\calQ= \operatorname{span}_\bbC \{ r^{-1/2}J_{1+2\ell}(2 \sqrt{\mathsf{Z} r}),  r^{-1/2}K_{1+2\ell}(2 \sqrt{\mathsf{Z} r})\}.
	\label{eq:calQ_Coulomb_model}
\end{equation}
Elementarily, it is true that within compact subsets of $\bbR^+_r$ (and this can be extended uniformly up to $r=0$), we can find a basis of solutions of the original ODE $Pu=0$, \cref{eq:Whittaker}, that are approximated by elements of the kernel $\calQ$  of the ``zero-energy'' ODE as $E\to 0$. Indeed, any 
basis with constant initial data on $r=1$ has this property. But this does not tell us how to approximate our solutions as $E\to 0$ \emph{while} $r\to\infty$. 
This is what the Langer--Olver approximation/\Cref{thm2} gives, an improved approximation that holds uniformly within $\{\hat{r}<Z\}$ for any $Z<\mathsf{Z}$. (Near or past the turning point $\hat{r}=\mathsf{Z}$, which again is only present if $\operatorname{sign}(E)\neq \operatorname{sign}(\mathsf{Z})$, we need to instead use \cref{eq:JWKB_uniform}.) 

The hydrogen bound states are, up to a factor of $r$, the solutions of \cref{eq:Whittaker} that are normalized in $L^2(\bbR^+_r,r^2 \dd r)$ (which exist only if $E<0<\mathsf{Z}$). These exist only for a countable set of $E$ because normalizability as $r\to 0^+$ picks out, for each $E<0$, a one-dimensional subspace of solutions and normalizability as $r\to\infty$ picks out another one-dimensional subspace of solutions; only for a countable set of $E$, converging  up to zero energy from below, will these subspaces coincide, in which case we have a bound state. The fact that the hydrogen bound states are normalized based on their $L^2$-norm rather than their initial data along some curve means that we cannot directly apply \Cref{thm1} to them, but this obstacle is quickly overcome by renormalizing the wavefunctions according to their large- or small-argument asymptotics. 

The hydrogen wavefunction (or really their radial profile) with principal quantum number $n\in \bbN$ and azimuthal quantum number $\ell\in \bbN$ is given by
\begin{equation}
	\psi_{n,\ell}(r) = \sqrt{ \Big( \frac{\mathsf{Z}}{n} \Big)^3 \frac{(n-\ell-1)!}{2n(n+\ell)!}} e^{-\mathsf{Z}r/(2n)} \Big( \frac{\mathsf{Z}r}{n} \Big)^\ell L^{2\ell+1}_{n-\ell-1} \Big( \frac{ \mathsf{Z}r}{n} \Big),
\end{equation}
where $L^\alpha_k$ is the $\alpha$th generalized Laguerre polynomial of degree $k$. This has eigenvalue $E_n = -\mathsf{Z}^2/(4n^2)$.
(Note: $\psi_{n,\ell}$ does not satisfy 
\begin{equation} 
	\psi''_{n,\ell}+(E_n + \mathsf{Z}/r-\ell(\ell+1)/r^2)\psi_{n,\ell} = 0, 
\end{equation} 
Rather, $r \psi_{n,\ell}$ does. This is because when separating variables in the 3D hydrogen atom to get \cref{eq:Whittaker}, one conjugates by $r^{-1}$, in accordance with \Cref{rem:reduction}, to remove the $\partial_r$ term in the ODE coming from the first-order term present in the 3D Laplacian when written in spherical coordinates.) The normalization of $\psi_{n,\ell}$ is such that 
\begin{equation}
	\int_0^\infty |\psi_{n,\ell}(r)|^2 r^2 \dd r =1,
\end{equation}
so $\psi_{n,\ell}$ is a unit vector in $L^2(\bbR^+,r^2 \dd r)$.

The discussion above suggests that $\psi_{n,\ell}$ should have interesting behavior on four different scales:
\begin{itemize}
	\item The $r = o(1/E)$ scale, where the bound states should be approximable by a multiple of $J_{1+2\ell}(2\sqrt{\mathsf{Z} r})$, the recessive element of $\calQ$ (\cref{eq:calQ_Coulomb_model}). In order to study this scale, note that 
	\begin{align}
		r^{-1/2}J_{1+2\ell}(2\sqrt{\mathsf{Z}r}) &= (1+o(1)) \mathsf{Z}^{2^{-1}+\ell} r^\ell/(2+2\ell-1)!\\ 
		\psi_{n,\ell}(r) &= (1+o(1)) 2^{-1/2} \mathsf{Z}^{\ell+3/2} n^{-(\ell+2)} r^\ell   L_{n-\ell-1}^{2\ell+1}(0) \sqrt{\frac{(n-\ell-1)!}{(n+\ell)!}}
	\end{align}
	as $r\to 0^+$. 
	So we should have $\psi_{n,\ell}(r)\approx C_{n,\ell}  r^{-1/2}J_{1+2\ell}(2\sqrt{\mathsf{Z}r})$ for 
	\begin{equation}
		C_{n,\ell} = \frac{ \mathsf{Z} L_{n-\ell-1}^{2\ell+1}(0) }{\sqrt{2}n^{\ell+2} } (2+2\ell-1)! \sqrt{\frac{(n-\ell-1)!}{(n+\ell)!}} .
	\end{equation}
	\item The $r\sim C/E$ scale for $C\in (0,\mathsf{Z})$ is the semiclassical ``classically allowed'' region, where Liouville--Green applies. We can group this with the $r\sim C/E$ scale for $C>\mathsf{Z}$, the semiclassical ``classically forbidden'' region where Liouville--Green applies.
	
	An interesting way to study the ODE in this regime is to consider the probability measure $\mu_{n,\ell}$ on $\bbR^+_{\hat{r}}$ defined by 
	\begin{equation}
		\dd \mu_{n,\ell} (\hat{r}) = |\psi_{n,\ell}(r)|^2 r^2 \dd r =  |E_n|^{-3} |\psi_{n,\ell}(\hat{r}/|E_n|) |^2 \hat{r}^2\dd \hat{r} .
	\end{equation}
	According to Born's rule, this is the probability density of finding the electron at radius $r=\hat{r}/E_n$ away from the origin. As $n\to\infty$, the curve $\dd \mu_{n,\ell}/ \dd \hat{r}$ forms a distinct envelope, 
	\begin{equation}
		 \frac{4}{\mathsf{Z}\pi} \Big(\frac{\mathsf{Z}}{\hat{r}}-1\Big)^{-1/2}1_{\hat{r}\in [0,\mathsf{Z}]} \dd \hat{r}, 
		 \label{eq:misc_349}
	\end{equation}
	in accordance with Liouville--Green.
	A closely related statement is that $\mu_{n,\ell}$ converges in law to the probability measure $\mu_\infty$ given by
	\begin{equation}
		\dd \mu_\infty(\hat{r}) = \frac{2}{\mathsf{Z}\pi} \Big(\frac{\mathsf{Z}}{\hat{r}}-1\Big)^{-1/2}1_{\hat{r}\in [0,\mathsf{Z}]} \dd \hat{r}. 
		\label{eq:misc_350}
	\end{equation}
	This admits a semiclassical interpretation in terms of the amount of time a ball thrown in the air spends at each height. 
	
	\begin{figure}[h!]
		\includegraphics[scale=.7]{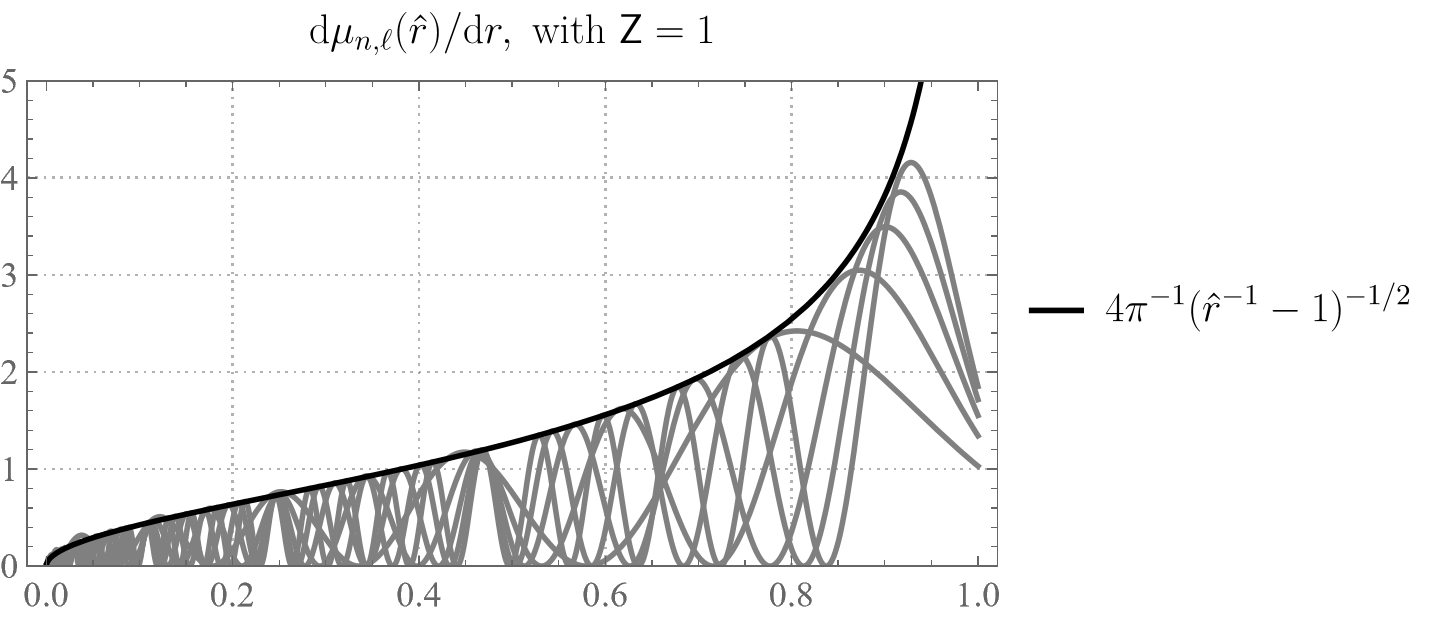}
		\caption{The densities $\dd \mu_{n,\ell}/ \mathrm{d} r$, plotted vs. $\hat{r}$.}
	\end{figure}

	(The different factor of $2$ in \cref{eq:misc_349}, \cref{eq:misc_350}, is due to the oscillatory nature of $\psi_{n,\ell}$; the average value of $|\psi_{n,\ell}|^2$ over a small interval's worth of $\hat{r}$ is very close to half of the value of the envelope on that interval.)
	
	The envelope and $\dd \mu_\infty / \dd \hat{r}$ become singular at two points: the origin and the turning point, i.e.\ at the two transition points of the underlying semiclassical ODE.
	\item  The final scale is $r=\mathsf{Z}/E + o(E^{-4/3})$ the region near the turning point, where we expect the behavior of the Airy funtions to be relevant. Indeed, a large Airy-like spike can be seen in this regime --- see \Cref{fig:hydro}.
\end{itemize}

\begin{figure}[h!]
	\begin{tikzpicture}[scale=1.2]
		\fill[lightgray!20] (-3,0) -- (0,0) -- (0,1.5) arc(-90:-180:1.5) -- (-3,3) -- cycle;
		\draw (-3,0) -- (0,0) -- (0,1.5) arc(-90:-180:1.5) -- (-3,3);
		\draw[dashed] (-3,0) -- (-3,3);
		\draw[white] (0,0) -- (0,-1);
		\draw[darkred,->] (-.1,.1) -- (-.1,.6) node[left] {$r$};
		\draw[darkred,->] (-.1,.1) -- (-.6,.1) node[above] {$-E$};
		\draw[darkred,->] (-.1,1.4) -- (-.1,.9) node[left] {$r^{-1}$};
		\draw[darkred,->] (-.1,1.4) to[out=180, in=-20] (-.6,1.5) node[left] {$-rE$};
		\draw[darkred,->] (-1.6,2.9) -- (-2.1,2.9) node[below] {$-E$};
		\draw[darkred,->] (-1.6,2.9) to[out=-90,in=110] (-1.5,2.4) node[below left] {$(-rE)^{-1}$};
		\node () at (-1.9,.6) {$Y$};
		\fill[black] (-1.26,2.2) circle (1.5pt);
		\draw[dashed] (-1.26,2.2) node[right] {\;$\{-rE=\mathsf{Z}\}$} -- (-3,1);
	\end{tikzpicture}\qquad
	\includegraphics[scale=.5]{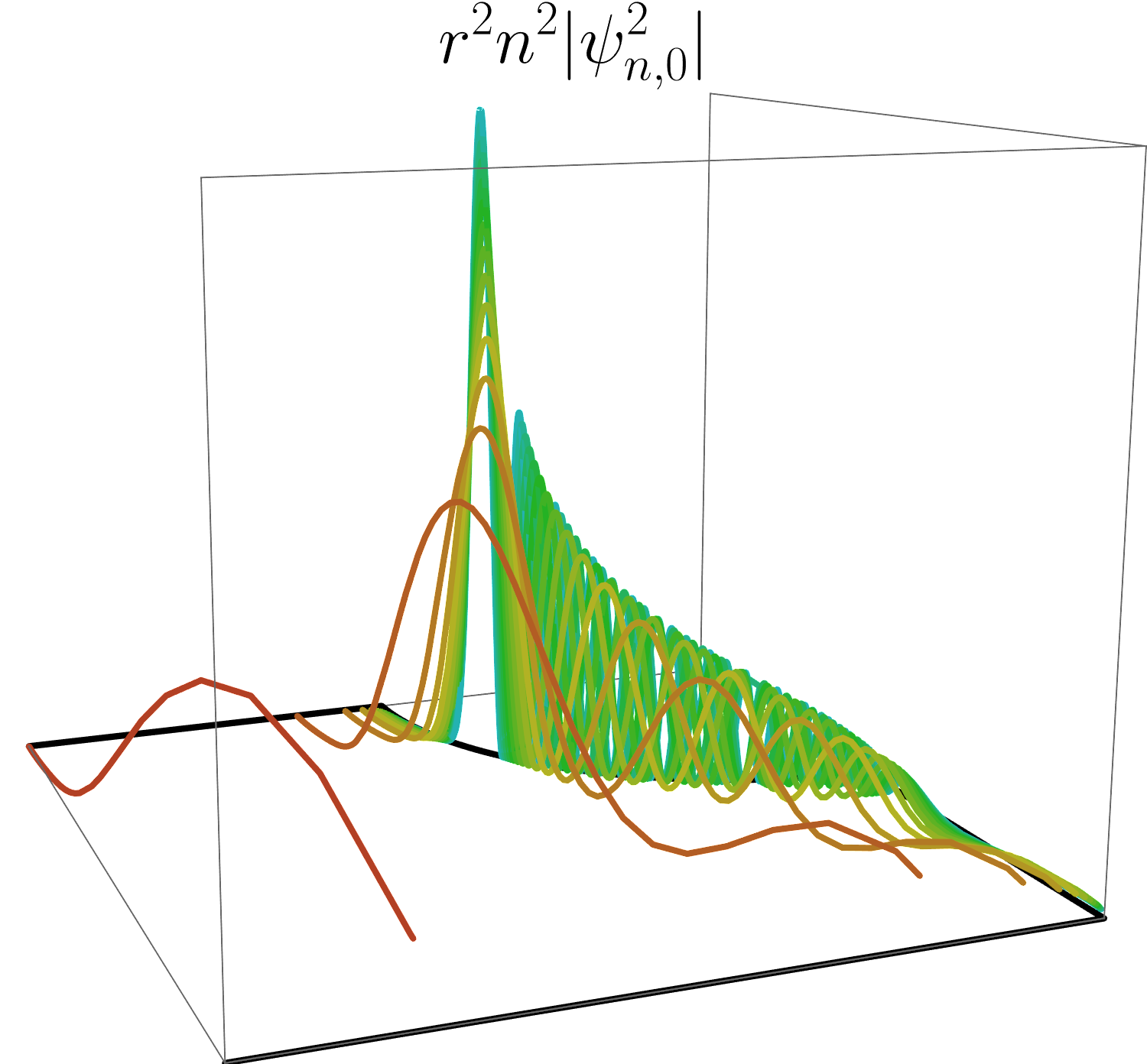}
	\caption{A plot of the radial probability density of the first 15 hydrogen s-orbital radial functions (rescaled by $n^2$), plotted over the compactification $\iota:\bbR^-_E\times \bbR^+_r\hookrightarrow Y$ for $Y$ depicted on the left. Each curve over which $r^2 n^2 |\psi_{n,0}^2|$ is plotted is the image under $\iota$ of $(E_n,r)$, where $E_n<0$ is the $n$th hydrogen eigenvalue. The hydrogen bound states display different behavior on each of three scales, as can be seen from the figure on the right. Note the Airy-like spike near $\{-rE =\mathsf{Z}\}$.}
	\label{fig:hydro}
\end{figure}

\section{An example with logarithms}
\label{sec:second}
Let us return to the simple harmonic oscillator, but this time, rather than consider the ``usual'' semiclassical scaling in \S\ref{subsec:parabolic}, consider 
\begin{equation} 
	P = - h^2 \frac{\partial^2}{\partial z^2} + z^2-h^2 .
	\label{eq:misc_zy2}
\end{equation}  
We will investigate the $\kappa=2$ transition point (a double turning point/equilibrium point) of \cref{eq:misc_zy2} at $z=0$. 
Our goal is to calculate the appearance of log terms in semiclassical expansions, showing again why the notion of polyhomogeneity, which generalizes smoothness in part by allowing log terms, is the appropriate notion to work with, not smoothness.

\begin{remark}[Relationship to \S\ref{subsec:parabolic}]
	If we begin with \cref{eq:misc_049} (with $\mathsf{k}=1$, $\ell=0$), for $E>0$, and define $z = \sqrt{E}r$ and $h = E$, then the $P$ there becomes, up to an unimportant normalization, the right-hand side of \cref{eq:misc_zy2}. Fixing $z$ and sending $h\to 0^+$ corresponds to sending $E\to 0^+$ and sending $r\to\infty$ at some rate; see \Cref{fig:parabolic_X_2}. Thus, in the mwc $X_1$ depicted in \Cref{fig:parabolic_X}, the regime we are exploring at present is the result of blowing up the midpoint of the edge $\{\hat{r}=\infty\}$.
	
	We saw in \S\ref{subsec:parabolic} that it was not necessary to blow up this point if initial data is prescribed at $r=\hat{r} \langle E\rangle^{1/2}$ for $\hat{r}$ fixed. What we are doing in the present section is instead prescribing initial data at $r = z |E|^{-1/2}$ for $z$ fixed. We will see interesting asymptotics in the corresponding regime not because it is a fundamental regime for the operator but because it is where we are supplying the boundary condition. 
\end{remark}

\begin{figure}[h!]
	\begin{tikzpicture}
		\filldraw[fill=lightgray!20] (-1,.5) arc(90:0:.5)  -- (2.5,0) arc(180:90:.5) -- (3,1.5) arc(-90:-270:.5) -- (3,3) -- (-1,3) -- cycle;
		\draw[dashed, darkred] (1,3) to[out=-90, in=135] (2.6,.355) ;
		\fill[darkred] (1,3) circle (1.5pt);
		\node () at (.15,2) {$X_1$};
		\draw[dashed, darkred] (1,3) -- (1,0);
		\node[darkred] () at (3.8,.2) {$\{E^{1/2} r=z_0\}$};
		\node[darkred] () at (.25,.5) {$\{E = 0\}$};
	\end{tikzpicture}
	\begin{tikzpicture}
		\filldraw[fill=lightgray!20] (-1,.5) arc(90:0:.5)  -- (2.5,0) arc(180:90:.5) -- (3,1.5) arc(-90:-270:.5) -- (3,3) -- (1.25,3) arc(0:-180:.5) -- (-1,3) -- cycle;
		\fill[black] (.75,2.5) circle (1.5pt);
		\draw[dashed] (.75,2.5) -- (.75,0);
		\draw[darkred, ->] (.85,2.4) to[out=10, in=235] (1.2,2.6) node[right] {$\sqrt{E} r$};
		\node[white] () at (3.8,.2) {$\{$};
		\draw[darkred,->] (.85,.1) -- (1.5,.1) node[above] {$E$};
		\draw[darkred,->] (.85,.1) -- (.85,.75) node[right] {$r$};
		\draw[darkred, ->] (.85,2.4) -- (.85,2) node[right] {$1/r$};
		\node () at (.15,2) {$X_2$};
		\node[gray] () at (.4,1.25) {fe};
		\node[gray] () at (2.1,.25) {be};
		\node[gray] () at (.9,2.8) {ze};
	\end{tikzpicture}
	\hspace{-1em}
	\begin{tikzpicture}
		\fill[lightgray!20] (0,3) -- (0,1.5) to[out=0, in=90] (1.5,0) -- (4,0) -- (4,3) -- cycle;
		\draw (4,3) -- (0,3) -- (0,1.5) to[out=0, in=90] (1.5,0) -- (4,0);
		\draw[dashed] (4,3) -- (4,0);
		\node[white] () at (0,.2) {$\{$};
		\node () at (2.75,2.25) {$M$};
		\draw[darkred, ->] (1.6,.1) -- (2.3,.1) node[above] {$E$};
		\draw[darkred, ->] (1.6,.1) to[out=90, in=-60] (1.45,.7) node[right] {$r$};
		\draw[darkred, ->] (.1,1.6) -- (.1,2.2) node[right] {$\sqrt{E} r$};
		\draw[darkred, ->] (.1,1.6) to[out=0,in=160] (.7,1.47) node[right] {$1/r$};
		\node[darkgray]() at (3.25,.2) {be};
		\node[darkgray]() at (-.3,2.5) {ze};
		\node[darkgray]() at (.85,.85) {fe};
	\end{tikzpicture}
	\caption{\textit{Left:} a level set of $z = \sqrt{E} r$ in the compactification $X_1 \hookleftarrow \bbR_E\times \bbR^+_r$ described in \S\ref{subsec:parabolic} and depicted in \Cref{fig:parabolic_X}. \textit{Middle:} the result $X_2$ of blowing up the top endpoint of that curve so as to resolve the regime we are currently studying. \textit{Right:} the compactification $M$ in this case, which can be identified with a neighborhood to the right of the dashed line $\mathrm{cl}_{X_2}\{E=0\}$ in $X_2$.}
	\label{fig:parabolic_X_2}
\end{figure}

The solutions of $Pu=0$ can be written
\begin{equation}
	u = c_1(h) U\Big(  \frac{h}{2}, i\sqrt{\frac{2}{h}} z \Big) + c_2(h)U\Big(- \frac{h}{2},  \sqrt{\frac{2}{h}} z \Big), 
	\label{eq:par}
\end{equation}
for arbitrary $c_1,c_2:(0,\infty)\to \bbC$, where $U$ is the usual parabolic cylinder function \cite[Chp. 6- \S6]{OlverBook}. 
From \cref{eq:par}, it is not apparent where logarithmic terms at $\mathrm{fe}$ might come from. Indeed, if we take $c_1=1$ and $c_2=0$, or vice versa, then $u$ is simply smooth at $\mathrm{fe}^\circ$, due to the fact that $U(a,z) \in C^\infty(\bbR_a\times K^\circ_z;\bbC)$
for $K\subset \bbC$ compact; see \cite[\href{http://dlmf.nist.gov/12.4}{\S12.4}]{NIST}. But this is just because if we fix $c_1,c_2$, then the $u$ given by \cref{eq:par} has, in the $h\to 0^+$ expansion of its initial data at $z=1$, logarithmic terms there. So, if instead of fixing $c_1,c_2$, we instead choose $c_1(h),c_2(h)$ so as to make $u$ have $h$-independent initial data at $\mathrm{ie}=\{z=1\}$, then the $h\to 0^+$ expansion of $c_1(h),c_2(h)$ will have to have logarithmic terms. Then, in the $h\to 0^+$ expansion of $u$, the logarithmic terms will have to be at fe instead.

In order to see the logarithmic terms of \begin{equation} 
	u=U(-h/2,2^{1/2} z h^{-1/2})
\end{equation} 
at $\mathrm{ze}^\circ$, we use the coordinate system $(x,\rho)$ near $\mathrm{fe}\cap\mathrm{ze}$, where $x= z^2$ and $\rho = 2^{-1/2} z^{-1}h^{1/2}$. Here, $x$ is a local bdf for $\mathrm{fe}$, up to a change of smooth structure, and $\rho$ is a local bdf for $\mathrm{ze}$, again up to a change of smooth structure. (These changes of smooth structure do not affect the spaces of polyhomogeneous functions.)
In terms of $x,\rho$, we have
\begin{equation}
	u= U( -x \rho^2 ,1/\rho).
\end{equation}
The large-argument expansion of the parabolic cylinder function \cite[\href{http://dlmf.nist.gov/12.9}{\S12.9}]{NIST} gives the Poincar\'e-type expansion
\begin{equation}
	u \sim e^{- 1/4\rho^2} \rho^{1/2-x\rho^2} \sum_{k=0}^\infty (-1)^k \frac{\Gamma(1/2+2k-x\rho^2)}{k! \Gamma(1/2-x\rho^2)} \Big(\frac{\rho^2}{2} \Big)^k 
	\label{eq:misc_030}
\end{equation}
as $\rho \to 0$. For each $x>0$, the term $\digamma_k=\Gamma(1/2+2k-x\rho^2)/\Gamma(1/2-x\rho^2)$ is smooth down to $\rho=0$, so it can be expanded in powers of $\rho$.
The logarithmic terms in \cref{eq:misc_030} are hidden in the \begin{equation} 
	\rho^{-x\rho^2}=e^{-x\rho^2 \log \rho}
\end{equation} 
term. Indeed, we have the following polyhomogeneous expansion:
\begin{equation}
	\rho^{-x\rho^2}  \sim \sum_{k=0}^\infty \frac{1}{k!} (-x\rho^2 \log \rho)^k .
\end{equation}
So, when organized into the form of an exponential-polyhomogeneous expansion by expanding $\rho^{-x\rho^2}$ and expanding each $\digamma_k$, \cref{eq:misc_030} has logarithmic terms, as claimed.

\printbibliography
\end{document}